\documentclass[a4paper, oneside,11pt,reqno]{amsart}
\usepackage{amsmath}
\usepackage{amsfonts}
\usepackage{amsmath}
\usepackage{amssymb}
\usepackage{amsthm}
\usepackage{geometry}
\usepackage{mathrsfs} 
\usepackage[dvipsnames]{xcolor}
\allowdisplaybreaks
\usepackage{dsfont}
 
\usepackage[pagebackref=true,colorlinks=true, linkcolor=blue, citecolor=blue]{hyperref}
\renewcommand*\backref[1]{\ifx#1\relax \else (Cited on #1) \fi}

\usepackage{enumerate}

\usepackage{tikz}
\usetikzlibrary{arrows, arrows.meta}

\theoremstyle{plain}
\newtheorem{definition}{Definition}
\newtheorem{proposition}[definition]{Proposition}
\newtheorem{lemma}[definition]{Lemma}

\newtheorem{theorem}[definition]{Theorem}

\theoremstyle{definition}
\newtheorem{example}[definition]{Example}
\newtheorem{remark}[definition]{Remark}

\numberwithin{definition}{section}
\numberwithin{equation}{section}

\DeclareMathOperator{\argmin}{arg\,min}

\DeclareMathOperator{\supp}{supp}

\newcommand*{\R}{\mathbb{R}}

\newcommand*{\Z}{\mathbb{Z}}

\newcommand*{\N}{\mathbb{N}}

\newcommand*{\Fcal}{\mathcal{F}}

\newcommand*{\Ncal}{\mathcal{N}}
\newcommand*{\PP}{\mathcal{P}_{\theta}}
\newcommand*{\PPP}{\mathcal{P}_{1,\theta}}

\newcommand{\norm}[1]{\left \lVert  #1 \right \rVert}

\newcommand{\abs}[1]{\left\lvert #1 \right\rvert}

\newcommand{\dd}{\,\mathrm d}
\newcommand{\Om}{\Omega}
\newcommand{\Oms}{\Omega_{\rm s}}
\newcommand{\Omst}{\Omega^*}
\newcommand{\Omsst}{\Omega_{\rm s}^*}

\newcommand*{\GG}{\mathscr{G}}

\newcommand*{\LL}{\mathscr{L}}

\newcommand*{\1}{\mathds{1}}
\newcommand{\x}{\boldsymbol{x}}
\newcommand*{\y}{\boldsymbol{y}}

\renewcommand*{\L}{\Lambda}

\renewcommand*{\d}{\mathrm{d}}
\newcommand*{\e}{\mathrm{e}}

\newcommand*{\range}{R}

\newcommand*{\Ent}{\operatorname{Ent}}

\newcommand{\Ssup}[1]{^{\scriptscriptstyle{({#1}})}}

\renewcommand*{\theta}{\vartheta}
\renewcommand*{\epsilon}{\varepsilon}

\makeatletter
\newcommand{\subalign}[1]{%
  \vcenter{%
    \Let@ \restore@math@cr \default@tag
    \baselineskip\fontdimen10 \scriptfont\tw@
    \advance\baselineskip\fontdimen12 \scriptfont\tw@
    \lineskip\thr@@\fontdimen8 \scriptfont\thr@@
    \lineskiplimit\lineskip
    \ialign{\hfil$\m@th\scriptstyle##$&$\m@th\scriptstyle{}##$\hfil\crcr
      #1\crcr
    }%
  }%
}
\makeatother

\usepackage{zref-clever}
\zcsetup{cap}
\zcRefTypeSetup{equation}{
Name-sg = {} ,
name-sg = {},
Name-pl = {},
name-pl = {}
}
\zcsetup{nameinlink=true}

\AddToHook{env/remark/begin}{%
\zcsetup{countertype={definition=remark}}}
\AddToHook{env/conjecture/begin}{%
\zcsetup{countertype={definition=conjecture}}}
\AddToHook{env/corollary/begin}{%
\zcsetup{countertype={definition=corollary}}}
\AddToHook{env/theorem/begin}{%
\zcsetup{countertype={definition=theorem}}}
\AddToHook{env/lemma/begin}{%
\zcsetup{countertype={definition=lemma}}}
\AddToHook{env/proposition/begin}{%
\zcsetup{countertype={definition=proposition}}}
\AddToHook{env/example/begin}{%
\zcsetup{countertype={definition=example}}}

\title{Space-time large deviations for continuum birth-and-death processes}

\author[B. Jahnel]{Benedikt Jahnel}
\author[J. K\"oppl]{Jonas K\"oppl}
\author[Y. Steenbeck]{Yannic Steenbeck}
\author[A. Zass]{Alexander Zass}

\address[Benedikt Jahnel]{TU Braunschweig, Institut für Mathematische Stochastik, %Universitätsplatz 2, 38106 Braunschweig,
Germany, and Weierstrass Institute, Berlin, Germany.}
\email{jahnel@wias-berlin.de}
\address[Jonas K\"oppl]{Universität Münster, Germany.}
\email{jonas.koeppl@uni-muenster.de}
\address[Yannic Steenbeck]{TU Braunschweig, Institut für Mathematische Stochastik, %Universitätsplatz 2, 38106 Braunschweig, 
Germany.}
\email{yannic.steenbeck@tu-braunschweig.de}
\address[Alexander Zass]{Weierstrass Institute, Berlin,
Germany.}
\email{zass@wias-berlin.de}
\date{\today}

\keywords{Gibbs measures, spatial birth and death processes, space-time large deviation principle, point processes, relative entropy, empirical process, path space.}
\subjclass[2020]{Primary 82C21, 82B21, 60F10; Secondary 60K35, 60G55, 60J25} 

\date{\today}

\definecolor{amethyst}{rgb}{0.6, 0.4, 0.8} 
\usepackage[textwidth=20mm,textsize=small]{todonotes}

\begin{document}

\begin{abstract}
    We establish a space-time large deviation principle for the empirical process of spatial birth-and-death dynamics in the continuum with area interaction at arbitrary inverse temperature, including the free evolution. The good rate function is the dynamical specific relative entropy with respect to the dynamics, and the result holds on the space of space-time stationary laws on the Skorokhod path space, equipped with the topology of weak convergence. For the time marginals of the empirical process, we identify the rate function in terms of the Hellinger distance between the birth and the death measure of the observed point process, which vanishes exactly at the translation-invariant Gibbs point processes. This proves a continuum version of conjectures of Dai Pra and Handa, and connects space-time large deviations with entropy production.
\end{abstract}

\maketitle
\setcounter{tocdepth}{1}
\hypersetup{bookmarksdepth=2}
{
  \hypersetup{linkcolor=black}
  \tableofcontents
}

\section{Introduction}\label{section:introduction}
We study spatial birth-and-death processes in the continuum, that is, Markovian dynamics on the space $\Om$ of locally finite point configurations in $\R^d$ in which particles appear or disappear. Formally, their generator is given by
\begin{equation*}
    (\LL f)(\eta) = \int_{\R^d} \d x \ b(x,\eta)\big( f(\eta+\delta_x) - f(\eta) \big) + \sum_{x\in\eta} \big( f(\eta-\delta_x) - f(\eta)\big), \qquad \eta \in \Om,
\end{equation*}
so that a new particle is added at $x\in \R^d$ at rate $b(x,\eta)$ when the system is in state $\eta$, and every particle present in $\eta$ dies at rate $1$. We focus on the case where the birth rate $b$ is the Papangelou intensity associated with a Hamiltonian, see e.g.~\cite{Georgii1976Canonical,xanh1979integral}. The corresponding Gibbs point processes then satisfy the detailed balance condition, so that the dynamics can be interpreted as a continuum analogue of the Glauber dynamics of lattice spin systems, a well-studied object in the discrete setting~\cite{liggett_interacting_2005}. Concretely, we consider the area interaction at inverse temperature $\beta\in\R$, for which $b(x,\eta)=\e^{-\beta\mathfrak{h}(x,\eta)}$, where $\mathfrak{h}(x,\eta)$ is the volume that a ball of fixed radius around $x$ adds to the union of the balls of the same radius around the points of $\eta$. For $\beta=0$ we obtain the free evolution, in which births form a space-time Poisson point process and particles die independently. Continuum birth-and-death dynamics have recently also proven useful as a tool for equilibrium questions: in~\cite{bowen_2026}, the associated Glauber dynamics is used to establish a phase transition for the hard-sphere model in the hyperbolic plane.
Our aim is to establish a space-time large deviation principle for the empirical process of these dynamics and to characterise the cost of observing a prescribed time marginal. 

Spatial birth-and-death processes were introduced by Preston~\cite{Preston_1975_Announcement} in finite volume. On the analytic side, a substantial body of work studies infinite-volume dynamics via Dirichlet forms on configuration spaces~\cite{AlbeverioKondratievRoeckner1998}, both for birth-and-death dynamics of Glauber type~\cite{KondratievLytvynov2005} and for other jump dynamics such as Kawasaki dynamics~\cite{KondratievLytvynovRoeckner2007}, and via the evolution of correlation functions~\cite{FinkelshteinKondratievKutoviy2012,FinkelshteinKondratievKozitsky2013,Finkelshtein2014Dynamical}. The Dirichlet-form approach is naturally adapted to the reversible case, where it yields the semigroup and an associated process for quasi-every initial configuration; non-equilibrium evolutions are accessible through the correlation functions, but not at the level of paths. Complementary probabilistic constructions of the infinite-volume dynamics were given in~\cite{HS78,Garcia1995Birth,garcia2006spatial}. In particular, \cite{Garcia1995Birth,garcia2006spatial} construct the process pathwise, as a projection of a higher-dimensional Poisson process and as the unique solution of a stochastic equation driven by a space-time Poisson random measure, respectively. It is this pathwise picture that we work with, since space-time large deviations are a statement about the law on path space.

Large deviation principles for empirical measures of Markov processes go back to Donsker and Varadhan~\cite{DonskerVaradhan1975I}; see~\cite{LDPbook,denHollander2000} for textbook accounts, including large deviations of empirical processes. For Gibbsian random fields on the lattice and for Gibbs point processes in the continuum, the empirical field satisfies a large deviation principle whose rate function is a specific relative entropy~\cite{follmerLectureNotes,georgii_solo_93,GeorgiiZessin1993}. For interacting particle systems in discrete time, such as probabilistic cellular automata with positive transition probabilities, the law of the space-time evolution on $\Z^d$ is a Gibbs measure on $\Z^{d+1}$~\cite{LebowitzMaesSpeer1990}, so that large deviation principles for the dynamics reduce to those for $(d+1)$-dimensional Gibbs measures. Motivated by this, Dai Pra studied interacting particle systems on the lattice in continuous time: in~\cite{DaiPra1993}, he established a large deviation principle for the space-time empirical process, with a relative entropy per unit space-time volume with respect to the dynamics as rate function, and in~\cite{DaiPra1993_2,daipra1994minimum} he studied the resulting large deviations of the time marginals.

For interacting particle systems on the lattice, the study of entropy production goes back to Holley~\cite{holley_free_1971} and has been developed further in~\cite{higuchi_results_1975,kunsch_non_1984,jahnel_attractor_2019,jahnel_dynamical_2023}.  In the continuum, entropy-dissipation identities for independent spatial birth-and-death dynamics have been established in~\cite{Dello2024Wasserstein,huesmann2025}, but the study of entropy production in the interacting infinite-volume case has been initiated only recently~\cite{JKSZ25,JKSZ26}. Related ideas also enter~\cite{bowen_2026}, albeit in a geometry where the boundary of a box is comparable to its volume, so that arguments based on entropy per unit volume, and the space-time scaling used here, do not transfer. Another topic of recent interest is the non-local geometry that spatial birth-and-death processes induce on configuration spaces, where the natural transport distances and gradient-flow structures are of jump type rather than diffusive, see e.g.~\cite{Dello2024Wasserstein,huesmann2025}.\medskip

\textbf{Main contributions.}
The main contributions of this manuscript are twofold. First, we prove a space-time large deviation principle for the empirical process of the birth-and-death dynamics with area interaction, for every inverse temperature $\beta\in\R$, and for all initial configurations whose number of particles in $\Lambda_n=[-n/2,n/2)^d$ is of smaller order than $n^{d+1}$, see \zcref{theorem:full-ldp}. The large deviation principle holds at speed $n^{d+1}$, the space-time volume of the observation window $\Lambda_n\times[0,n]$, on the space of space-time stationary laws on $D(\R;\Om)$, equipped with the topology of weak convergence induced by the local Skorokhod $J_1$ topology on the path space. The good rate function $H^\beta$ is the dynamical specific relative entropy with respect to the interacting dynamics. For $\beta=0$, it is the relative entropy per unit space-time volume with respect to the free evolution, and for $\beta\neq0$ an energetic term accounts for the interaction. This extends the results of~\cite{DaiPra1993} to the continuum. The new difficulty is not primarily that the number of particles in a box fluctuates, but that the local state space is non-compact and the number of jumps in a space-time window is unbounded. Controlling these two features is essential for establishing exponential tightness. Moreover, in the continuum there are various natural topologies in which one may work, and the choice turns out to be quite delicate. Two features of our setting are essential: we allow several points at the same location, and we work with the topology of weak convergence. On the subspace of simple configurations, and in a stronger topology that also requires the convergence of intensities, we only obtain weak large deviation principles, see \zcref{theorem:weak-ldps}, and show by counterexamples that, already for the free evolution, they cannot be strengthened to full ones.

Second, we identify the cost of observing a prescribed time marginal. By the contraction principle, the time-zero marginals of the empirical process satisfy a large deviation principle with rate function $\mathfrak{I}^\beta(\mu)=\min\{H^\beta(P)\colon P_0=\mu\}$. We solve this minimum entropy problem and show that $\mathfrak{I}^\beta(\mu)$ is twice the squared Hellinger distance between the birth measure $b(0,\cdot)\mu$ and the reduced Palm measure $\mu_0^!$ of $\mu$. In particular, the rate function vanishes exactly at the translation-invariant Gibbs point processes, see \zcref{thm:marg-ldp}. This establishes the continuum version of a conjecture of Dai Pra~\cite{daipra1994minimum}, formulated for reversible spin-flip systems on the lattice, and of a closely related conjecture of Handa~\cite{handa1996entropy}. It also connects the rate function of the space-time large deviation principle with the entropy production studied in~\cite{holley_free_1971,higuchi_results_1975,JKSZ25,JKSZ26}. In finite volume, a minimiser is given by a reversible birth-and-death process whose birth measure is the geometric mean of the birth and the death measure of $\mu$, and this construction works even if $\mu$ has no Papangelou intensity. We expect that the method also applies to the lattice setting in which the conjecture was originally formulated, but we do not pursue this here.\medskip

\textbf{Organisation of the manuscript.} In \zcref{section:setting-notation} we introduce the setting and state our main results, \zcref{theorem:full-ldp,thm:marg-ldp}. \zcref[S]{section:outlook} contains comments on boundary conditions, the weak large deviation principles in other topologies and further remarks. In \zcref{section:proof-strategy} we explain the strategy of the proofs, state the intermediate results and deduce the main theorems from them. The remaining sections contain the proofs of the intermediate results: \zcref{section:rate-function} collects properties of the dynamical specific entropy, \zcref{section:proof-lower-bound,section:proof-upper-bound} contain the lower and the upper bound for the free evolution, \zcref{section:interacting-case} treats the interacting dynamics, \zcref{section:contraction} the minimum entropy problem, and \zcref{section:counterexamples} the weak large deviation principles and the counterexamples. The appendices collect facts about the topology of the path space and the graphical construction of the dynamics, as well as a sketch of the proof of the one-cell representation from \zcref{subsection:one-cell} of the dynamical specific entropy.

\section{Setting and main results}\label{section:setting-notation}
In this section we introduce the path space on which the dynamics live, the birth-and-death dynamics with area interaction at inverse temperature $\beta\in\R$, the empirical process, and the dynamical specific entropy. We then state our two main results: a space-time large deviation principle for the empirical process, and the identification of the cost of observing a prescribed time marginal.

\subsection{Configurations and paths}\label{subsection:configurations-paths}
Let $\Om$ denote the space of locally finite point configurations on $\R^d$, that is, of locally finite counting measures on $\R^d$; several points may sit at the same location. We equip $\Om$ with the vague topology, generated by the maps $\eta \mapsto \int \eta(\d x)\, g(x)$ for compactly supported continuous functions $g\colon\R^d \to \R$, which makes $\Om$ a Polish space. For a Borel set $B\subset\R^d$ and $\eta\in\Om$, we write $\abs{B}$ for the Lebesgue volume of $B$, $\eta(B)$ for the number of points of $\eta$ in $B$, counted with multiplicity, $\eta_B$ for the restriction of $\eta$ to $B$, and $\Om_B:=\{\eta_B\colon \eta\in\Om\}$. For open $B$, $\Omega_B$ carries the vague topology on $B$, see \zcref{section:appendix-topology}. For $x\in\R^d$, the shifted configuration $\theta_x\eta$ is given by $\theta_x\eta(B):=\eta(B+x)$; it is the configuration $\eta$ seen from $x$.

Let $E=D(\R; \Om)$ denote the space of $\Om$-valued c\`adl\`ag paths $\omega=(\omega_t)_{t\in\R}$, equipped with the local Skorokhod $J_1$ topology, which makes $E$ a Polish space; a metric for this topology and the elementary facts we use about it are collected in \zcref{section:appendix-topology}. For a space-time volume $\Delta\subset\R^d\times\R$, we write $\mathcal{F}_\Delta$ for the $\sigma$-algebra generated by the restrictions of the paths to $\Delta$, that is, by the maps $\omega\mapsto\omega_t(B)$ with $B\times\{t\}\subset\Delta$, and $\mathcal{F}_{(-\infty, t]}:=\mathcal{F}_{\R^d \times (-\infty, t]}$ for $t\in\R$. Space-time volumes will usually be denoted by $\Delta$ and spatial volumes by $\Lambda$. For $(x,s)\in\R^d\times\R$, the space-time translation $\theta_{(x,s)}\colon E\to E$ is given by $(\theta_{(x,s)}\omega)_t:=\theta_x\omega_{s+t}$, so that $\theta_{(x,s)}\omega$ is the path $\omega$ seen from the space-time point $(x,s)$. We write $P[f]$ for the expected value of a function $f$ with respect to a probability measure $P$.

We denote by $\PP(E)$ the set of all probability measures $P$ on $E$ that are invariant under all space-time translations, and we equip it with the topology of weak convergence. We call its elements space-time stationary laws.

\subsection{Birth-and-death dynamics with area interaction}\label{subsection:dynamics}
The dynamics we study are specified by the formal generator
\begin{equation}\label{eq:generator}
    (\LL f)(\eta) := \int_{\R^d} \d x \ b(x,\eta)\big( f(\eta+\delta_x) - f(\eta) \big) + \sum_{x\in\eta} \big( f(\eta-\delta_x) - f(\eta)\big), \quad \eta \in \Om,
\end{equation}
where a new particle is added at $x\in \R^d$ at rate $b(x,\eta)$ when the system is in state $\eta$, and each particle present in $\eta$ dies at constant rate $1$. Throughout, $b$ is the birth rate of the area interaction at inverse temperature $\beta\in\R$, which we now describe. Fix an interaction range $\range>0$ and let $B_r(x)\subset\R^d$ denote the closed ball with radius $r>0$ centred at $x\in\R^d$. For a finite configuration $\eta$, the Hamiltonian
\begin{equation*}
    \mathscr{H}(\eta):=\abs{B_{\range/2}(\eta)} := \Big\lvert \bigcup_{x\in \eta}B_{\range/2}(x) \Big\rvert
\end{equation*}
is the volume covered by the balls of radius $\range/2$ around the points of $\eta$. The conditional energy of a point $x\in\R^d$ in a configuration $\eta\in\Om$,
\begin{equation*}
    \mathfrak{h}(x, \eta) := \Big|B_{\range/2}(x) \setminus \bigcup_{y \in \eta} B_{\range/2}(y)\Big|,
\end{equation*}
is the volume that a point at $x$ adds to $B_{\range/2}(\eta)$; indeed, $\mathfrak{h}(x,\eta)=\mathscr{H}(\eta_\Lambda+\delta_x)-\mathscr{H}(\eta_\Lambda)$ for every bounded $\Lambda\supset B_\range(x)$. For $\beta\in\R$, the \emph{birth rate} at $x$ given $\eta$ is
\begin{equation}\label{eq:birth-rate}
    b(x, \eta) := \e^{-\beta \mathfrak{h}(x, \eta)}.
\end{equation}
In the terminology of the GNZ equations (named after Georgii, Nguyen and Zessin), $b$ is the \emph{Papangelou intensity} associated with the Hamiltonian $\beta\mathscr{H}$, see \zcref{subsection:marginals} and e.g.~\cite{Georgii1976Canonical,xanh1979integral}. Since $0\le\mathfrak{h}\le\abs{B_{\range/2}(0)}$, the birth rate is bounded from above and away from zero; it is continuous, translation covariant in the sense that $b(x,\eta)=b(0,\theta_x\eta)$, and has finite range, since $b(x,\eta)$ depends on $\eta$ only through $\eta_{B_\range(x)}$. See~\zcref{section:interaction-assumptions} for a remark on how these assumptions are used. For $\beta>0$ the interaction is attractive: a point close to existing points adds little volume, so births are more likely there and points tend to cluster. For $\beta<0$ the interaction is repulsive. For $\beta=0$ we have $b\equiv1$, and we call the corresponding dynamics the \emph{free evolution}: births form a space-time Poisson point process with intensity $\d x\,\d t$, and every particle, including those present initially, has an independent unit-exponential lifetime.

For each $\beta\in\R$ and each initial configuration $\xi\in\Om$, the dynamics started from $\xi$ at time $0$ can be constructed pathwise as a thinning of a space-time Poisson random measure~\cite[Theorem~2.13]{garcia2006spatial}; we recall this graphical construction in \zcref{section:appendix-graphical}. We write $Q^\beta_\xi$ for the law of the resulting process on $D(\R_+;\Om)$ and $Q_\xi:=Q^0_\xi$ for the law of the free evolution, which serves as our reference measure. Slightly abusing notation, for a cube $A\subset\R^d$, which may be open, closed or half-open, we write $Q^A_\xi$ for the free evolution in $A$ started from $\xi_A$ at time zero, in which births only occur in $A$.

The paths of the dynamics are almost surely legal birth-and-death paths in the following sense. A path $\omega\in E$ is called a \emph{legal birth-and-death path} if, in every bounded spatial region and over every finite time interval, the configuration changes only at finitely many distinct times, and each change consists of adding one particle (a birth) or removing one particle present immediately beforehand (a death). Particles remain at fixed locations between their respective birth and death. We use the same terminology for paths in $D(\R_+;\Om)$. We write $\mathsf B_\omega(B,(s,t])$ and $\mathsf D_\omega(B,(s,t])$ for the number of births and deaths of $\omega$ in a bounded Borel set $B\Subset\R^d$ during the time interval $(s,t]\subset\R$, and we omit the subscript $\omega$ when the path is clear from the context.

\subsection{The empirical process}\label{subsection:empirical-process}
For $n\in\N$, let $\Lambda_n := [-n/2, n/2)^d$ and $\Delta_n := \Lambda_n \times [0,n]$. For a legal birth-and-death path $\omega \in D(\R_+;\Om)$, we define its \emph{$n$-periodic version} $\omega^{(n)}\in E$ by
\begin{equation*}
    \omega^{(n)}_{t+hn}(B + k\overline{n}) := \omega_t(B), \qquad B\subset\Lambda_n \text{ Borel},\ 0 \leq t < n,\ h\in \Z,\ k \in \Z^d,
\end{equation*}
where $\overline{n} := (n,\dots, n) \in \R^d$. In words, we copy the restriction of the path to $\Lambda_n \times [0, n)$ into all its disjoint translates by $(k\overline{n}, h n)$, $h\in\Z$, $k \in \Z^d$. The path $\omega^{(n)}$ is c\`adl\`ag, but in general it is not a legal birth-and-death path: at the temporal resets, that is, at the times $t\in n\Z$, several particles may change simultaneously, and spatial periodisation repeats every event in all spatial copies. The same construction applied to a configuration $\eta\in\Om$ yields its $n$-periodic version $\eta^{(n)}\in\Om$.

The \emph{$n$-th empirical process} of $\omega$ is the space-time average
\begin{equation}\label{eq:empirical-process}
    R_{n, \omega}
    := \frac{1}{n \abs{\Lambda_n}} \int_{0}^{n} \d s \,\int_{\Lambda_n} \d x \,\, \delta_{\theta_{(x, s)} \omega^{(n)}}.
\end{equation}
Thus $R_{n,\omega}$ is the law of the periodised path seen from a uniformly chosen point of $\Delta_n$, and since $\omega^{(n)}$ is periodic in space and time, $R_{n,\omega}\in\PP(E)$. When $\omega$ is the random path of the dynamics, we simply write $R_n$ for $R_{n,\omega}$.

\subsection{The rate function}\label{subsection:rate-function}
For two probability measures $\mathbb{P},\mathbb{Q}$ on a measurable space $(X, \mathcal{X})$, the \emph{relative entropy} of $\mathbb{P}$ with respect to $\mathbb{Q}$ is
\begin{equation*}
    \Ent(\mathbb{P}\lvert \mathbb{Q}):=
    \begin{cases}
        \int_X \dd\mathbb{P} \log \frac{\dd\mathbb{P}}{\dd\mathbb{Q}} \quad &\text{if } \mathbb{P} \ll \mathbb{Q},
        \\
        \infty &\text{otherwise.}
    \end{cases}
\end{equation*}
For a probability measure $P$ on $E$, a cube $A\subset \R^d$ and $T>0$, let $P_{A,T}$ be the law under $P$ of the restriction $(\omega_t|_A)_{0\le t\le T}$ of the path to $A \times [0,T]$, and let $P_{A,0}$ be its $A\times \{0\}$-marginal. We write $P_{A,0}Q^A := \int P_{A,0}(\d\eta)\,Q_{\eta}^A$ and define
\begin{equation}\label{eq:rate}
    h_{A,T}(P):=\Ent(P_{A,T}\lvert P_{A,0}Q^A),
    \qquad
    H(P):=\sup_{a,T>0}\frac{h_{(0,a)^d,T}(P)}{a^dT}.
\end{equation}
The relative entropy $h_{A,T}(P)$ compares the evolution of $P$ in $A\times[0,T]$ with the free evolution started from the same initial distribution. Consequently, $H$ only charges \emph{dynamical} deviations, for instance keeping particles alive for a long time, and gives no cost to the initial conditions themselves. In analogy with the usual specific entropy of point processes~\cite{GeorgiiZessin1993}, we call $H$ the \emph{dynamical specific entropy}. In \zcref{section:rate-function} we show that the supremum in \eqref{eq:rate} is also a limit along growing space-time boxes, and that $H(P)<\infty$ forces $P$ to be concentrated on legal birth-and-death paths.

For the free evolution, $H$ will be the rate function. In the interacting case there is an additional energetic contribution, whose form can be heuristically derived by a Girsanov-type computation. Consider the dynamics in $\Lambda_n$ with periodic boundary conditions, that is, with birth rate $b(x,\eta^{(n)})$ at $x\in\Lambda_n$, and the free evolution in $\Lambda_n$, both started in the same initial condition and observed during the time interval $[0, n]$. The log-density of the former with respect to the latter is
\begin{equation}\label{eq:girsanov-heuristic}
    \mathfrak{L}_{n}(\omega)
    := \int_0^n \d t\int_{\Lambda_n} \d x\ [1 - b(x, \omega^{(n)}_{t-})]
    + \sum_{\substack{(t, x) \text{ birth}\colon (t, x) \in (0, n] \times \Lambda_n}} \log b(x, \omega_{t-}^{(n)}),
\end{equation}
see \zcref{proposition:change_of_measure_a_la_girsanov}. Up to boundary terms near the times $0$ and $n$, where the periodised path is reset, $\mathfrak{L}_n$ is a sum of contributions of unit space-time cells, so that, heuristically, $\mathfrak{L}_n\approx n\abs{\Lambda_n}R_n[F]$, where, for a legal birth-and-death path $\omega$, the observable $F$ is defined by
\begin{equation}\label{eq:observable-F}
    F(\omega) := 1 - b(0, \omega_0) + \sum_{\substack{(t, x) \text{ birth} \colon (t, x) \in (0, 1] \times [0, 1)^d}} \log b(x, \omega_{t-}).
\end{equation}
Since $\log b=-\beta\mathfrak{h}$, the observable equals $1-\e^{-\beta\mathfrak{h}(0,\omega_0)}-\beta\sum\mathfrak{h}(x,\omega_{t-})$, where the sum runs over the same births as in~\eqref{eq:observable-F}. It is defined only on legal birth-and-death paths, and for $\beta\neq0$ it is neither bounded nor continuous, since it counts births; this is the main technical difficulty of the interacting case, see \zcref{subsection:strategy-interacting}. If $H(P)<\infty$, then $F$ is $P$-integrable, see \zcref{lem:finite-entropy-intensities,lemma:interacting_case_nice_things_from_finite_free_entropy}. This motivates the definition of $H^\beta\colon\PP(E)\to (-\infty,\infty]$ by
\begin{equation}\label{equation:interacting_case_rate_function}
    H^\beta(P) :=
    \begin{cases}
        H(P) - P[F] &\text{if } H(P)<\infty, \\
        \infty &\text{otherwise.}
    \end{cases}
\end{equation}
For $\beta=0$ we have $b\equiv1$, hence $F\equiv0$ and $H^0=H$.

\subsection{Space-time large deviations}\label{section:main-results}
Our main result holds for all initial configurations whose number of points in $\Lambda_n$ is negligible on the scale of the space-time volume $n\abs{\Lambda_n}$ of $\Delta_n$. Denote by $\Omst$ the set of all such \emph{space-time volume-regular} configurations, i.e.
\begin{equation*}
    \Omst := \Big\{\xi \in \Om\colon \lim_{n\to \infty} \frac{\xi(\Lambda_n)}{n|\Lambda_n|} = 0\Big\}.
\end{equation*}

\begin{theorem}[Space-time LDP]\label{theorem:full-ldp}
    Let $\beta\in\R$ and $\xi \in \Omst$. Under $Q^\beta_\xi$, the empirical processes $(R_n)_{n\in\N}$ satisfy a large deviation principle on $\PP(E)$, equipped with the topology of weak convergence, with speed $n\abs{\Lambda_n}$ and good rate function $H^\beta$. That is, the sublevel sets $\{P\in \PP(E)\colon H^\beta(P) \leq c\}$, $c\in\R$, are weakly compact; moreover, for every weakly closed set $F \subset \PP(E)$,
    \begin{equation*}
        \limsup_{n\to \infty}\frac{1}{n|\Lambda_n|} \log Q^\beta_{\xi}(R_n \in F) \leq - \inf_{P\in F} H^\beta (P),
    \end{equation*}
    and for every weakly open set $G \subset \PP(E)$,
    \begin{equation*}
        \liminf_{n\to\infty}\frac1{n|\Lambda_n|} \log Q^\beta_{\xi}(R_n\in G) \geq - \inf_{P\in G}H^\beta(P).
    \end{equation*}
\end{theorem}

\zcref[S]{theorem:full-ldp} extends to the continuum the space-time large deviation principle of Dai Pra~\cite{DaiPra1993} for interacting particle systems on the lattice. The speed $n\abs{\Lambda_n}$ is the space-time volume of the observation window $\Delta_n$, and the rate function measures the entropy cost per unit space-time volume of a stationary deviation of the dynamics. The interaction enters only through the energetic term $P[F]$. The initial configuration does not appear in the rate function: under the assumption $\xi\in\Omst$, the cost of removing or ignoring the initial particles is negligible on the scale $n\abs{\Lambda_n}$. Let us stress that the result holds in regimes of uniqueness as well as non-uniqueness of the Gibbs measures introduced in \zcref{subsection:marginals}, and that it also holds for the dynamics in $\Lambda_n$ with suitable boundary conditions, see \zcref{subsection:boundary-conditions}. Two features of the setting are essential: we allow for several points at the same location, and we work with the topology of weak convergence. On the subspace of simple configurations, and in the stronger topology that additionally requires convergence of intensities, we only obtain weak large deviation principles, and for the free evolution no full large deviation principle holds there with any rate function (on simple configurations for $\xi=\emptyset$, and in the stronger topology for every $\xi\in\Omst$), see \zcref{section:outlook}.

The rate function has the following further properties. For a bounded Borel set $\Lambda\Subset\R^d$, we write $Q^{\Lambda,\beta}$ for the Markov kernel of the birth-and-death dynamics in $\Lambda$ with birth rate $b$, death rate $1$ and empty boundary configuration. As for the free evolution, we write $P_{\Lambda,0}Q^{\Lambda,\beta}:=\int P_{\Lambda,0}(\d\eta)\,Q^{\Lambda,\beta}_\eta$, where $Q^{\Lambda,\beta}_\eta$ is the law started from $\eta$ at time zero, and $Q^{\Lambda,0}=Q^\Lambda$. We recall that $P_{\Lambda_n, [0,n]}$ denotes the law of the restriction of the path to $\Lambda_n\times[0,n]$ under $P$.

\begin{proposition}[Properties of the rate function]\label{proposition:rate-function-properties}
    Let $\beta\in\R$. The rate function $H^\beta$ has the conditional-entropy representation
    \begin{equation}\label{equation:interacting_case_rate_function_conditional_entropy_representation}
        H^\beta(P)
        = \lim_{n \to \infty} \frac{1}{n \abs{\Lambda_n}} \Ent\big(P_{\Lambda_n, [0,n]} \lvert P_{\Lambda_n, 0} Q^{\Lambda_n,\beta}\big), \qquad P\in\PP(E),
    \end{equation}
    %\bjtd{$P_{\Lambda_n, 0} Q^{\Lambda_n,\beta}$ should be globally replaced by $Q^{\Lambda_n,\beta}$, right?}
    and is consequently non-negative; moreover, its minimum is zero. Finally, for every $\delta \in (0,1)$ there is a $C_\delta > 0$ such that
    \begin{equation}\label{equation:interacting_case_lower_bound_via_free_entropy}
        H^\beta(P) \geq (1-\delta)H(P) - C_\delta, \qquad P\in\PP(E).
    \end{equation}
\end{proposition}

By~\eqref{equation:interacting_case_rate_function_conditional_entropy_representation}, $H^\beta$ is the dynamical specific entropy relative to the interacting dynamics, in the same way as $H$ is relative to the free evolution.

\subsection{Costs of prescribed marginals and entropy production}\label{subsection:marginals}
We finally study which time marginals of the empirical process are observed, and at which cost. By the contraction principle, this amounts to minimizing $H^\beta(P)$ under the constraint that the time-zero marginal of $P$ equals a given $\mu\in\PP(\Om)$. We show that the minimal cost is a Hellinger distance that vanishes exactly at the Gibbs measures for the area interaction. The time-zero marginal of $P\in\PP(E)$ is the law $P_0$ of $\omega_0$ under $P$; it belongs to the set $\PP(\Om)$ of translation-invariant probability measures on $\Om$, which we also equip with the topology of weak convergence. For $\mu\in\PP(\Om)$ we write $\rho(\mu):=\mu[\eta([0,1)^d)]\in[0,\infty]$ for its \emph{intensity}, so that $\mu[\eta(B)]=\rho(\mu)\abs{B}$ for every bounded Borel set $B\Subset\R^d$. To state our result precisely, we recall the relevant equilibrium objects.

For $\Lambda\Subset\R^d$ and $\eta\in\Om$, the \emph{(conditional) energy in $\Lambda$},
\begin{equation*}
    \mathscr{H}_\L(\eta):=\lim_{n\to\infty}\big(\mathscr{H}(\eta_{\L_n})-\mathscr{H}(\eta_{\L_n\setminus\L})\big),
\end{equation*}
is well defined, since the sequence is eventually constant by the finite range of the interaction. \emph{Finite-volume Gibbs point processes} in $\L$ with \emph{boundary condition} $\omega\in \Om$ are defined as
\begin{equation*}
    \nu\Ssup{\omega}_\L(\d\eta_\L):=Z_\L(\omega)^{-1}\e^{-\beta\mathscr{H}_\L(\eta_\L\omega_{\L^c})}\pi_\L(\d \eta_\L),
\end{equation*}
where $\eta_\L\omega_{\L^c}$ is the configuration that agrees with $\eta_\L$ in $\L$ and with $\omega$ outside $\L$, the partition function $Z_\L(\omega)$ is the normalization constant that makes \(\nu\Ssup{\omega}_\L\) a probability measure on $\Omega_\L$, and the reference measure $\pi_\L$ is the homogeneous Poisson point process with intensity \(1\) on \(\Omega_\L\). A probability measure $\nu$ on $\Om$ is called an \emph{infinite-volume Gibbs point process} if it satisfies the \emph{DLR equations}
\begin{equation*}
    \int\nu(\d \omega)f(\omega)
    =
    \int \nu(\d \omega) \int\nu_\L\Ssup{\omega}(\d \eta_\L) f(\eta_\L\omega_{\L^c}),\quad \Lambda \Subset \R^d, \, f\ge0 \text{ measurable}.
\end{equation*}
We denote by $\GG^\beta_\theta$ the set of all translation-invariant infinite-volume Gibbs point processes for the area interaction at inverse temperature $\beta$. By standard results for Gibbs point processes, see e.g.~\cite{dereudre2019introduction}, the set $\GG^\beta_\theta$ is never empty. The area-interaction model is known to exhibit a phase transition, in the sense that $\GG^\beta_\theta$ contains more than one element for $d\ge2$, $\beta=1$ and sufficiently large $\range$. Indeed, for $\beta=1$ the area-interaction process is the one-type marginal of the symmetric Widom--Rowlinson model, and a large interaction range corresponds to a large activity, see e.g.~\cite{Ruelle1971WRPT,CCK95,Giacomin1995Agreement,dereudre2019introduction,DereudreHoudebert2021}.

For two \(\sigma\)-finite measures \(\mathfrak{M}\), \(\mathfrak{N}\) on a common measurable space \(\mathcal{X}\), the squared \emph{Hellinger distance} of \(\mathfrak{M}\) and \(\mathfrak{N}\) is
\begin{equation*}
    \mathfrak{H}^2(\mathfrak{M}, \mathfrak{N})
    := \frac{1}{2} \int_{\mathcal{X}} \bigg(\sqrt{\frac{\d\mathfrak{M}}{\d (\mathfrak{M} + \mathfrak{N})}}(x) - \sqrt{\frac{\d\mathfrak{N}}{\d (\mathfrak{M} + \mathfrak{N})}}(x) \bigg)^2 \, (\mathfrak{M} + \mathfrak{N})(\d x) \in [0, \infty].
\end{equation*}
For $\mu\in\PP(\Om)$ we define its \emph{reduced Palm measure} \(\mu_0^!\), without further normalisation, by
\begin{equation*}
    \mu_0^![f]
    := \int \mu(\d\eta) \sum_{x \in \eta_{[0, 1)^d}} f(\theta_x(\eta - \delta_x)), \quad f \geq 0 \text{ measurable}.
\end{equation*}
The measure \(\mu_0^!\) is characterised by the property
\begin{equation*}
    \int \mu(\d\eta) \sum_{x \in \eta} g(x, \theta_x(\eta - \delta_x))
    = \int_{\R^d}\d x \, \mu_0^![g(x, \cdot)], \quad g \geq 0 \text{ measurable}.
\end{equation*}
A measurable function $\kappa\colon\R^d\times\Om\to[0,\infty)$ is called a \emph{Papangelou intensity} of $\mu$ if the \emph{GNZ equation}
\begin{equation*}
    \int \mu(\d\eta) \sum_{x \in \eta} g(x, \eta - \delta_x)
    = \int \mu(\d\eta) \int_{\R^d}\d x \, \kappa(x, \eta)\, g(x, \eta),
\end{equation*}
holds for any measurable $g\geq 0$. The Gibbs measures in $\GG^\beta_\theta$ are characterised as the translation-invariant laws with Papangelou intensity $b$, see \zcref{subsection:strategy-contraction}. With this notation at hand, we can state the large deviation principle for the time-zero marginals of the empirical process.

\begin{theorem}[Marginal LDP and entropy production]\label{thm:marg-ldp}
    Let $\beta\in\R$ and \(\xi \in \Omst\). Under $Q^\beta_\xi$, the time-zero marginals \(((R_n)_0)_{n \in \mathbb{N}}\) satisfy a large deviation principle on \(\PP(\Om)\), equipped with the topology of weak convergence, with speed \(n \abs{\Lambda_n}\) and good rate function
    \begin{equation}\label{equation:contraction_rate_function}
        \mathfrak{I}^\beta(\mu)
        := \min\{H^\beta(P) \colon P \in \PP(E),\, P_0 = \mu \}
        = 2 \mathfrak{H}^2(b(0, \cdot) \mu, \mu_0^!),
    \end{equation}
    if \(\rho(\mu)<\infty\), and $\mathfrak{I}^\beta(\mu)=\infty$ otherwise. In particular, $\mathfrak{I}^\beta$ is non-negative, its minimum is zero, and its minimisers are exactly the translation-invariant Gibbs measures, that is, $\argmin \mathfrak{I}^\beta = \GG^\beta_\theta$.
\end{theorem}

The minimisers are characterised by the equation $b(0, \cdot) \mu=\mu_0^!$, which, by the characterisation of $\mu_0^!$ and translation invariance, is the GNZ equation with Papangelou intensity $b$. It expresses detailed balance: under $\mu$, the rate at which a particle is born at $x$ in a configuration $\eta$ equals the rate at which this particle dies in the configuration $\eta+\delta_x$.

The first result in this direction is contained in~\cite[Theorem~7]{DaiPra1993_2}, where, in the setting of spin-flip processes on the hypercubic lattice, the minimisers of the corresponding contracted rate function are characterised as the stationary measures of the underlying process; the additional characterisation of these stationary measures as Gibbs measures was still missing. For reversible spin-flip systems, Dai Pra~\cite[Theorem~6]{daipra1994minimum} solved the corresponding minimum entropy problem for a dense set of translation-invariant probability measures on the configuration space, by identifying an explicit Markovian minimiser, and he conjectured that the resulting formula holds for all of them. Later, Handa~\cite[Remark]{handa1996entropy} conjectured, by a formal argument and in the discrete setting of~\cite{DaiPra1993_2}, a representation of the contracted rate function very similar to~\eqref{equation:contraction_rate_function}. \zcref[S]{thm:marg-ldp} establishes the continuum version of these conjectures. If $\mu$ has a Papangelou intensity $\kappa$, then the rate function takes the form
\begin{equation*}
\mathfrak{I}^\beta(\mu)=\mu\Big[\int_{[0, 1)^d} \d x\, \Big(\sqrt{b(x, \eta)} - \sqrt{\kappa(x,\eta)}\Big)^2 \Big],
\end{equation*}
see \zcref{lemma:contraction_time_marginal_hellinger_distance_identities}. This expression is comparable to the entropy-dissipation functional of the dynamics studied in~\cite{JKSZ25,JKSZ26}, and in this sense \zcref{thm:marg-ldp} connects the rate function of the space-time large deviation principle with entropy production, see also~\cite{holley_free_1971,higuchi_results_1975}.

\section{Outlook and comments}\label{section:outlook}
This section collects complementary comments and results: a remark on the properties of the interaction that we need, boundary conditions, weak large deviation principles in other topologies, a remark on the topology generated by tame local functions, and a one-cell representation of the dynamical specific entropy. It is not needed for the proofs of \zcref{theorem:full-ldp,thm:marg-ldp} and can be skipped on a first reading.

\subsection{Assumptions on the interaction}\label{section:interaction-assumptions}
We restricted ourselves to the area interaction in order to reduce the technical load and to convey the conceptual ideas more clearly. The properties used are translation covariance; uniform boundedness of the birth rate from above and away from zero; finite range; continuity in the configuration; and uniform continuity in the birth location, uniformly over configurations. 
To be more precise, the fact that the birth rate is translation-covariant and uniformly bounded from above is used throughout. The upper bound and positive lower bound on the birth rate allow us to control the interacting term of the rate function with the free entropy, see~\zcref{lemma:interacting_case_nice_things_from_finite_free_entropy}. Finite range helps control the influence of boundary conditions in~\zcref{section:interacting_case_girsanov_and_irrelevance_of_boundary_configurations,lemma:interacting_case_rate_function_conditional_entropy_representation}, and ensures that only events in a fixed spatial neighbourhood affect the birth rate in the smoothing arguments. Continuity of the birth rate in the configuration is used, for example, to show continuity of the approximants in~\zcref{lemma:interacting_case_approximants_of_rate_function_are_continuous}, while uniform continuity in the birth location, uniformly over configurations, ensures that the spatial smoothing error of~\zcref{lemma:interacting_case_approximation_of_rate_function_deterministic_error} tends to zero. Finally, the upper bound on the birth rate and its finite range are also used for the graphical representation of~\zcref{section:appendix-graphical}.

\subsection{Boundary conditions}\label{subsection:boundary-conditions}
The large deviation principle of \zcref{theorem:full-ldp} also holds for the dynamics in $\Lambda_n$ with periodic boundary conditions, and more generally with any time-dependent boundary conditions outside $\Lambda_n$ that make the birth rates in $\Lambda_n$ predictable; the boundary conditions may also depend on $n$. For $\beta=0$, this is immediate, since the free evolution in $\Lambda_n$ does not depend on the boundary conditions and $R_n$ only depends on the restriction of the path to $\Delta_n$. For $\beta\neq0$, it follows from \zcref{lemma:interacting_case_boundary_conditions_play_no_role}, which shows that boundary conditions play no role on the exponential scale; in fact, our proof of \zcref{theorem:full-ldp} for $\beta\neq0$ proceeds via the periodic dynamics.

\subsection{Weak large deviation principles in other topologies}\label{subsection:other-topologies}
Two modifications of the setting of \zcref{theorem:full-ldp} suggest themselves. First, let $\Oms\subset\Om$ be the subspace of simple configurations, equipped with the induced topology, which is again Polish, and let $E_{\rm s}=D(\R;\Oms)$, which we identify with its canonical image in $E$; its local $J_1$ topology agrees with the subspace topology inherited from $E$. If the initial configuration is simple, the dynamics stay in $E_{\rm s}$, and one may ask for a large deviation principle on the set $\PP(E_{\rm s})$ of space-time stationary laws on $E_{\rm s}$, equipped with the topology of weak convergence. Second, one may strengthen the topology so that it also controls the mean number of particles. Recall the intensity $\rho(P)=\rho(P_0)=P[\omega_0([0,1)^d)]$ of $P\in \PP(E)$. Let $\PPP(E)$ be the set of all $P\in\PP(E)$ with $\rho(P)<\infty$, equipped with the topology $\mathcal{T}_{\rm{occ}}$, which is the topology of weak convergence augmented by requiring continuity of the intensity map $P\mapsto\rho(P)$. Analogously, $\PPP(\Om)$ denotes the set of all $\mu\in\PP(\Om)$ with $\rho(\mu)<\infty$, equipped with the topology of weak convergence augmented by convergence of intensities. Finally, let $\Omsst:=\{\xi\in \Omst\colon \xi\text{ is simple}\}$.

\begin{theorem}[Weak LDPs for simple configurations and with intensity]\label{theorem:weak-ldps}
    Let $\beta\in\R$. The following assertions hold.
    \begin{enumerate}[(i)]
        \item\label{item:weak-ldp-simple} If $\xi \in \Omsst$, then under $Q^\beta_\xi$ the empirical processes $(R_n)_{n\in\N}$ satisfy a weak large deviation principle on $\PP(E_{\rm s})$, equipped with the topology of weak convergence, with speed $n\abs{\Lambda_n}$ and rate function $H^\beta$. That is, for every weakly compact set $F \subset \PP(E_{\rm s})$ and every weakly open set $G \subset \PP(E_{\rm s})$,
        \begin{equation*}
        \begin{split}
            \limsup_{n\to \infty}\frac{1}{n|\Lambda_n|} \log Q^\beta_{\xi}(R_n \in F) &\leq - \inf_{P\in F} H^\beta(P), \\
            \liminf_{n\to\infty}\frac1{n|\Lambda_n|} \log Q^\beta_{\xi}(R_n\in G) &\geq - \inf_{P\in G}H^\beta(P).
        \end{split}
        \end{equation*}
        \item\label{item:weak-ldp-Tocc} If $\xi \in \Omst$, then under $Q^\beta_\xi$ the empirical processes $(R_n)_{n\in\N}$ satisfy a weak large deviation principle on $(\PPP(E),\mathcal{T}_{\rm{occ}})$ with speed $n\abs{\Lambda_n}$ and rate function $H^\beta$. That is, the bounds in~(\ref{item:weak-ldp-simple}) hold for every $\mathcal{T}_{\rm{occ}}$-compact set $F\subset\PPP(E)$ and every $\mathcal{T}_{\rm{occ}}$-open set $G\subset\PPP(E)$.
        \item\label{item:weak-ldp-marginals} Under $Q^\beta_\xi$, the time-zero marginals $((R_n)_0)_{n\in\N}$ satisfy a weak large deviation principle with speed $n\abs{\Lambda_n}$ and rate function $\mathfrak{I}^\beta$ from \eqref{equation:contraction_rate_function} on $\PP(\Oms)$, equipped with the topology of weak convergence, if $\xi\in\Omsst$, and on $\PPP(\Om)$ if $\xi\in\Omst$.
    \end{enumerate}
\end{theorem}

The upper bounds in~(\ref{item:weak-ldp-simple}) and~(\ref{item:weak-ldp-Tocc}) follow from \zcref{theorem:full-ldp}, and those in~(\ref{item:weak-ldp-marginals}) from \zcref{thm:marg-ldp}, by a soft topological argument. If $\xi$ is simple, then $R_n\in\PP(E_{\rm s})$ almost surely, and the canonical embedding of $\PP(E_{\rm s})$ into $\PP(E)$ is continuous, so that it maps compact sets to compact sets. Similarly, $R_n\in\PPP(E)$ almost surely, and the embedding of $(\PPP(E),\mathcal{T}_{\rm{occ}})$ into $\PP(E)$ with the weak topology is continuous. Since $\PP(E_{\rm s})$ carries the topology induced by the weak topology on $\PP(E)$, the lower bound in~(\ref{item:weak-ldp-simple}) is also a soft consequence of \zcref{theorem:full-ldp}: every open subset of $\PP(E_{\rm s})$ is the trace of a weakly open subset of $\PP(E)$, and $R_n\in\PP(E_{\rm s})$ almost surely. The same applies to the time marginals on $\PP(\Oms)$. The lower bound in~(\ref{item:weak-ldp-Tocc}) requires additional work, since $\mathcal{T}_{\rm{occ}}$ is finer than the weak topology: the intensity of the stationary block law $\bar P$ constructed in \zcref{subsection:strategy-lower-bound} has to approximate $\rho(P)$, and that of the empirical process has to converge to $\rho(\bar P)$. We carry this out in \zcref{section:counterexamples}.

The weak large deviation principles in~(\ref{item:weak-ldp-simple}) and~(\ref{item:weak-ldp-Tocc}) cannot be strengthened to full ones in general. The obstruction comes from the temporal resets of the periodised path. A genuine birth or death changes the configuration by a single particle, whereas at a reset all particles alive at time $n$ are replaced at once by the initial configuration. Since the resets occupy a vanishing fraction of space-time, every rate function for which the lower bound holds must be infinite on laws that charge such large downward jumps, see \zcref{lem:periodisation-jumps}. On $\PP(E_{\rm s})$, the empirical processes of paths in which pairs of particles sit at distance of order $1/n$ until the next reset form a closed set whose probability decays only exponentially, because their natural weak limit, a law of paths with double points, lies outside $E_{\rm s}$. In $\mathcal{T}_{\rm{occ}}$, the same happens for empirical processes whose particles are concentrated in a small region: their weak limit, the law of the empty path, is excluded by the constraint on the intensity. In both cases all laws in the closed set charge the large downward jumps created by the resets, which contradicts the upper bound. Hence, already for the free evolution, the empirical processes admit no full large deviation principle at speed $n\abs{\Lambda_n}$, with any rate function, on $\PP(E_{\rm s})$ when $\xi=\emptyset$, and on $(\PPP(E),\mathcal{T}_{\rm{occ}})$ for every $\xi\in\Omst$, see \zcref{example:no-full-ldp-simple,example:no-full-ldp-occupation}. This is compatible with \zcref{theorem:full-ldp}: in the weak topology on $\PP(E)$, these sets are not closed, and their closures contain the limiting laws, which have finite rate. We expect that the weak large deviation principles for $\beta\neq0$ cannot be strengthened in general either. For the time marginals in~(\ref{item:weak-ldp-marginals}) the question remains open, since a single time marginal does not see the resets.

\subsection{The topology of tame local functions}\label{subsection:tau-L}
We call a function $f\colon E\to\R$ \emph{local and tame} if it is $\mathcal{F}_\Delta$-measurable for some bounded space-time volume $\Delta \Subset \R^d \times \R$ and satisfies $\abs{f} \leq c\,(1+\#\{\text{particles present in } \Delta\})$ for some constant $c>0$ that may depend on $f$. Consider the smallest topology on the set of laws $P\in\PP(E)$ under which the expected number of particles present in $\Delta$ is finite for every bounded $\Delta$, such that all evaluation maps $P \mapsto P[f]$ with tame local functions $f$ are continuous. This topology naturally arises in large deviation principles for empirical fields in a purely spatial context, see e.g.~\cite{GeorgiiZessin1993}. A natural first goal would therefore be a large deviation principle in this topology instead of the weak topology. However, the counterexamples in \zcref{section:counterexamples} also rule out a full large deviation principle in this topology. Indeed, by~\eqref{eq:lb-cylinder-agreement}, every bounded Lipschitz function on $E$ is a uniform limit of bounded local functions, and the intensity is the evaluation of a tame local function, so that this topology is finer than $\mathcal{T}_{\rm{occ}}$; hence the argument of \zcref{example:no-full-ldp-occupation} applies verbatim.

\subsection{A one-cell representation of the dynamical specific entropy}\label{subsection:one-cell}
The dynamical specific entropy $H$ coincides with a one-cell conditional entropy, which extends~\cite[Definition~2.3]{DaiPra1993_2} to the continuum. For a probability measure $P'$ on $E$, a space-time box $\Delta=A\times[0,T]$ with a cube $A\subset\R^d$, and $\eta\in\Om$, we write $h_\Delta(P'\lvert Q_\eta):=\Ent(P'_{A,[0,T]}\lvert Q^A_\eta)$ for the relative entropy of the restriction of $P'$ to $\Delta$ with respect to the free evolution in $A$ started from $\eta_A$. Let $C=(0,1)^d$, let $\prec$ denote the strict lexicographic order on $\Z^d$, put $C_z=z+C$ for $z\in\Z^d$, and define
\begin{equation*}
    \mathcal G^-_T=\mathcal F_{(-\infty,0]} \vee\bigvee_{z\prec0}\mathcal F_{C_z\times[0,T]}, \qquad P_\omega=P(\,\cdot\lvert \mathcal G^-_1)(\omega), \qquad P^p_\omega=P(\,\cdot\lvert \mathcal F_{(-\infty,0]})(\omega).
\end{equation*}
Thus $P_\omega$ conditions on the entire past and on the cells preceding $C$ in the lexicographic order during the time interval $[0,1]$, and $P^p_\omega$ only on the past.

\begin{proposition}[One-cell conditional entropy]
\label{prop:cond-entropy-representation}
For every $P\in\PP(E)$ and $T>0$,
 \begin{align*}
 H(P)
 &=
 P\big[ h_{C\times[0,1]}(P_\omega\lvert  Q_{\omega_0})\big]
 =\lim_{n\uparrow\infty}\frac1{n\abs{\Lambda_n}}
       P\big[ h_{\Delta_n}(P^p_\omega\lvert  Q_{\omega_0})\big]=\lim_{n\uparrow\infty}\frac1{T|\Lambda_n|}
       P\big[ h_{\Lambda_n\times[0,T]}
                         (P^p_\omega\lvert  Q_{\omega_0})\big].
\end{align*}
\end{proposition}
The two limits correspond to~\cite[Propositions~4.1 and~5.2]{DaiPra1993_2}, specialised to the free reference measure. As this result is not the main focus of this work, we only sketch its proof in \zcref{section:appendix-one-cell}; the argument follows~\cite[Proof of Proposition~4.1]{DaiPra1993_2} and uses the superadditivity of the local conditional entropies from \zcref{lem:entropy-superadditivity}.

\section{Proof strategy}\label{section:proof-strategy}
The proofs of our main results can roughly be split into the following three major steps.
\begin{enumerate}[(1)]
    \item We first establish the large deviation principle for $(R_n)_{n \in \N}$ in the case without interactions, that is, at inverse temperature $\beta = 0$. Here the lower bound is obtained by an entropic change of measure built from independent space-time blocks, and the upper bound by a variational representation of the local entropies together with exponential tightness.
    \item We then use this as input for a Girsanov-type transformation to deduce the result for $\beta\neq0$.
    \item Finally, we contract the large deviation principle to the time-zero marginals and solve the resulting minimum entropy problem explicitly: its value is a Hellinger distance, whose zeros are precisely the translation-invariant Gibbs measures.
\end{enumerate}
Let us now lay out the strategy for each step in more detail. The proofs of the technical results stated in this section are deferred to \zcref{section:rate-function,section:proof-lower-bound,section:proof-upper-bound,section:interacting-case,section:contraction}.

\subsection{LDP lower bound in the non-interacting case}\label{subsection:strategy-lower-bound}
The main goal of this subsection is the large deviation lower bound for the free evolution, that is, for $\beta = 0$. Recall the notation $\Lambda_n = [-n/2, n/2)^d$ for $n\in \N$.
\begin{theorem}[LDP lower bound, non-interacting]\label{thm:ldp-lower-bound-free}
    Let $\xi\in\Omst$. For every weakly open set $G \subset \PP(E)$ we have
    \begin{equation*}
        \liminf_{n\to \infty}\frac{1}{n|\Lambda_n|} \log Q_{\xi}(R_n \in G) \geq - \inf_{P\in G} H(P).
    \end{equation*}
\end{theorem}
The principal idea behind the lower bound is the same as in the classical proofs of Sanov's theorem for empirical measures~\cite{LDPbook}, and it is also used in many other cases for establishing large deviation principles for empirical processes~\cite{georgii_solo_93,GeorgiiZessin1993,follmerLectureNotes}.
Given a measure $P \in \PP(E)$ and an open neighbourhood $G \subset \PP(E)$ of $P$, we want to perform a change of measure from $Q_\xi$ to a law $\widehat{Q}_n$, such that $\widehat{Q}_n$ is absolutely continuous with respect to $Q_\xi$ and the event $\{R_n \in G\}$ is typical under $\widehat{Q}_n$, that is,
\begin{equation*}
    \lim_{n\to \infty}\widehat{Q}_n(R_n \in G) = 1.
\end{equation*}
Interpreting $\1_{\{R_n \in G\}}$ as a Bernoulli random variable and using that the relative entropy can only decrease when the underlying $\sigma$-algebra is reduced, an elementary calculation shows that
\begin{equation*}
    \begin{split}
    \Ent(\widehat{Q}_n\lvert Q_\xi)
    &\geq
    \widehat{Q}_n(R_n \in G)\log\frac{\widehat{Q}_n(R_n \in G)}{Q_\xi(R_n \in G)}
    +
    \widehat{Q}_n(R_n \notin G)\log\frac{\widehat{Q}_n(R_n \notin G)}{Q_\xi(R_n \notin G)}
    \\
    &\geq
    \widehat{Q}_n(R_n \in G) \log\frac{1}{Q_\xi(R_n \in G)} - \log 2,
    \end{split}
\end{equation*}
where we discarded the non-negative term $\widehat{Q}_n(R_n \notin G)\log(1/Q_\xi(R_n \notin G))$ and used that the entropy of a Bernoulli random variable is at most $\log 2$. Whenever $\widehat Q_n(R_n\in G)>0$, rearranging gives
\begin{equation}\label{eq:sanov-entropy-inequality}
\log Q_\xi(R_n\in G)
\ge
-\frac{
\Ent(\widehat Q_n\lvert Q_\xi)+\log 2
}{
\widehat Q_n(R_n\in G)
},
\end{equation}
so to conclude the lower bound it suffices to control the entropic cost of the change of measure in terms of the dynamical specific entropy $H(P)$. The following intermediate result shows that this can be done.

\begin{proposition}[Change of measure]\label{proposition:change-of-measure-sanov}
    Fix an initial condition $\xi \in \Omst$. Let $P\in \PP(E)$ be such that $H(P)<\infty$, let $G$ be an open neighbourhood of $P$ in $\PP(E)$, and let $\alpha >0$. Then there exists a sequence $(\widehat{Q}_n)_{n \in \N}$ of probability measures on $D(\R_+; \Om)$ with the following properties.
    \begin{enumerate}[(i)]
        \item For every $n\in \N$ we have $\widehat{Q}_n \ll Q_\xi$.
        \item The event $\{R_n \in G\}$ is asymptotically typical under $\widehat{Q}_n$.
        \item The entropic cost of the change of measure is controlled by $H(P)$ up to an error of size $\alpha$, that is,
        \begin{equation*}
            \limsup_{n\to \infty}\frac{1}{n \abs{\Lambda_n}}\Ent(\widehat{Q}_n \lvert Q_\xi) \leq H(P) + \alpha.
        \end{equation*}
    \end{enumerate}
\end{proposition}

Using the calculation preceding the proposition, we can conclude the proof of the lower bound in the non-interacting case.

\begin{proof}[Proof of \zcref{thm:ldp-lower-bound-free}]
    Fix an open set $G\subset \PP(E)$ and let $\alpha>0$. If there is no $P\in G$ with $H(P)<\infty$, the lower bound is trivial. Otherwise, fix such a $P \in G$ and let $(\widehat Q_n)_{n\in\N}$ be the sequence from \zcref{proposition:change-of-measure-sanov}. By properties~(i) and~(ii), the estimate~\eqref{eq:sanov-entropy-inequality} applies for all sufficiently large $n$. Dividing both sides by $n\abs{\Lambda_n}$ and taking the limit $n\to \infty$ yields, by properties~(ii) and~(iii),
\begin{equation*}
    \liminf_{n\to\infty}\frac{1}{n\abs{\Lambda_n}} \log Q_\xi(R_n \in G) \geq - H(P) - \alpha.
\end{equation*}
Since this holds for every $\alpha>0$ and every $P\in G$ with $H(P)<\infty$, sending $\alpha \downarrow 0$ and optimising over $P$ yields the claimed lower bound.
\end{proof}

In the classical Sanov theorem, the empirical measure of independent samples from $P$ is close to $P$ by the law of large numbers, so it is clear how to obtain the desired measures $\widehat Q_n$. Our construction is inspired by the one used in~\cite{georgii_solo_93,GeorgiiZessin1993} for lattice random fields and spatial point processes, where the configuration is sampled independently in disjoint boxes. In space-time, the spatial directions cause no problem, since, conditionally on the initial configuration, the free evolutions in disjoint cubes are independent. However, we cannot simply sample the path independently in consecutive space-time blocks with the same spatial base: the configuration at the end of one block would in general not agree with the configuration at the beginning of the next, so that the concatenated path would make jumps in which many particles change at once. Such paths are not legal birth-and-death paths, and their law would not be absolutely continuous with respect to $Q_\xi$. We therefore let every block path start and end with the empty configuration.

As in~\cite{georgii_solo_93,GeorgiiZessin1993,follmerLectureNotes}, an additional difficulty is that $P$ is stationary but not necessarily ergodic. Sampling the path in a single large space-time window according to $P$ would produce an empirical process close to a random ergodic component of $P$ rather than to $P$ itself. Sampling independently in many blocks avoids this: each block draws afresh from a truncated version of $P$, so that, by a law of large numbers, the empirical process is close to a deterministic stationary law $\bar{P}$, which lies in the neighbourhood $G$ of $P$ for suitable parameters.

More precisely, fix $a,T>0$, write $A=(0,a)^d$ and $S=T+2$; the two additional units of time are used to fill an empty block at its beginning and to empty it at its end. The main idea behind the construction of the measures $\widehat Q_n$ is to cover the observation window $\Delta_n$ with translated copies of the space-time box $A\times[0,S]$, indexed by a finite set $I_n$ of block indices, and to sample the paths $\zeta^{(i)}$, $i\in I_n$, in these blocks independently from a law $\mathcal{C}_{\rm block}$ that behaves like $P$ during the time interval $[1, T+1]$ but starts and ends with the empty configuration. Putting together all the pieces $\zeta^{(i)}$, plus an initial period that removes the particles of the fixed initial condition $\xi$, defines a path on all of $\Delta_n$; outside of $\Delta_n$ we simply let the process evolve according to the free evolution. Each block path consists of three phases, see \zcref{fig:blockpath}: during the \emph{preparation phase} $[0,1]$, particles are created from the empty configuration; during the \emph{bulk phase} $[1,1+T]$, the path follows $P$; and during the \emph{clearance phase} $[1+T,S]$, all particles are removed.

To start a block from the empty configuration, the configuration that $P$ has at the beginning of the bulk phase has to be created by the free evolution during the preparation phase, at finite entropic cost. Under $P$, however, the number of particles in $A$ is unbounded, and their positions may be singular with respect to the Lebesgue measure. We therefore condition the initial configuration of $P$ on having at most $K\in\N$ particles in $A$, and we displace the position of every particle, initial or newly born, independently and uniformly within its cell of a grid of side length $\epsilon\in(0,1]$, where $a/\epsilon\in\N$; the displaced position is kept for the whole lifetime of the particle. This makes the initial law absolutely continuous with respect to a Poisson point process, with a bounded density, at the price of a small error in the approximation of $P$. The following lemma collects all the properties of the resulting block law $\mathcal{C}_{\rm block}$ that we will need. The parameters $a,K,\epsilon,T$ are chosen later, and they remain fixed as $n\to\infty$. Recall that $Q^A_\emptyset$ denotes the free evolution in $A$ started from the empty configuration, and that $\rho(P)=P[\omega_0([0,1)^d)]$ denotes the intensity of $P$, see \zcref{subsection:other-topologies}.

\begin{lemma}[Block paths with empty endpoints]\label{lemma:block-law}
    Let $P \in \PP(E)$ be such that $H(P)<\infty$. Fix $a,T>0$, $K\in\N$ with $P(\omega_0(A) \leq K)>0$, and $\epsilon\in(0,1]$ with $a/\epsilon\in\N$. There exists a law $\mathcal{C}_{\rm block}$ on $D([0,S]; \Om_A)$ that satisfies the following properties.
    \begin{enumerate}[(i)]
        \item The initial and terminal configurations are $\mathcal{C}_{\rm block}$-almost surely empty, that is,
        \begin{equation*}
            \mathcal{C}_{\rm block}(\omega_0 = \omega_S = \emptyset) = 1.
        \end{equation*}
        \item Every finite concatenation of independent paths sampled from $\mathcal{C}_{\rm block}$ is c\`adl\`ag, and its law is absolutely continuous with respect to the free evolution $Q^A_\emptyset$ in $A$ over the same time interval.
        \item The entropic cost of $\mathcal{C}_{\rm block}$ with respect to the free evolution $Q^A_{\emptyset}$, observed during $[0,S]$, is bounded by
        \begin{equation}\label{eq:lb-cycle-entropy}
            \Ent(\mathcal{C}_{\rm block}\lvert Q^A_{\emptyset}) \leq \frac{h_{A,T}(P)}{P(\omega_0(A) \leq K)} + C_{\rm prep} + |A| + \frac{c\, \rho(P)\abs{A}}{P(\omega_0(A) \leq K)},
        \end{equation}
        where $c=-\log(1-\e^{-1})$, and where the preparation cost $C_{\rm prep}$, that is, the relative entropy of the truncated and smoothed initial law with respect to the Poisson point process on $A$ with intensity $1-\e^{-1}$, see~\eqref{eq:lb-preparation-entropy}, satisfies
        \begin{equation*}
            C_{\rm prep}\leq (1-\e^{-1})|A|+\log K! +K\log^+\!\frac1{(1-\e^{-1}) \epsilon^d}<\infty,
        \end{equation*}
        with $\log^+x=\max\{\log x,0\}$. In particular, $C_{\rm prep}$ does not depend on $T$.
    \end{enumerate}
\end{lemma}

\begin{figure}
    \centering
    \includegraphics[width=0.35\linewidth]{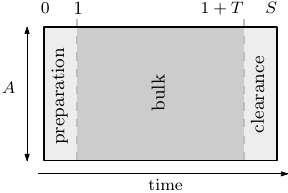}
    \caption{A block path in the space-time box $A\times[0,S]\subset\R^d\times\R$. It starts empty, creates a configuration of at most $K$ particles during the preparation phase $[0,1]$, follows the truncated and smoothed dynamics of $P$ during the bulk phase $[1,1+T]$, and removes all particles during the clearance phase $[1+T,S]$.}
    \label{fig:blockpath}
\end{figure}

In more detail, during the preparation phase we run the free evolution from the empty configuration and tilt its law at time $1$, so that the configuration at time $1$ has the truncated and smoothed initial law of $P$. Without the tilt, the configuration at time $1$ would be a Poisson point process with intensity $1-\e^{-1}$, and the cost of the tilt is $C_{\rm prep}$. During the bulk phase, we follow the conditional law of the truncated and smoothed restriction of $P$ to $A\times[0,T]$, given its initial configuration. Its entropic cost is at most $h_{A,T}(P)$ divided by the conditioning probability $P(\omega_0(A)\le K)$; the smoothing does not increase the cost, since it maps the free evolution onto itself. During the clearance phase, we condition the free evolution on having no births and on all present particles dying; since a particle dies within a unit time interval with probability $1-\e^{-1}$, this costs $|A|$ plus $c:=-\log(1-\e^{-1})$ per particle present at time $1+T$. The costs of the preparation and clearance phases are bounded uniformly in $T$: the preparation cost $C_{\rm prep}$ does not involve $T$ at all, and the clearance cost only depends on the expected number of particles at time $1+T$, which is at most $\rho(P)\abs{A}/P(\omega_0(A)\le K)$ by stationarity. Since a block has space-time volume $\abs{A}(T+2)$, these two costs contribute at most a constant multiple of $1/(T+2)$ per unit space-time volume, see~\eqref{eq:lb-global-entropy} below, and this contribution is small once $T$ is large.

\begin{figure}
    \centering
    \includegraphics[width=0.7\linewidth]{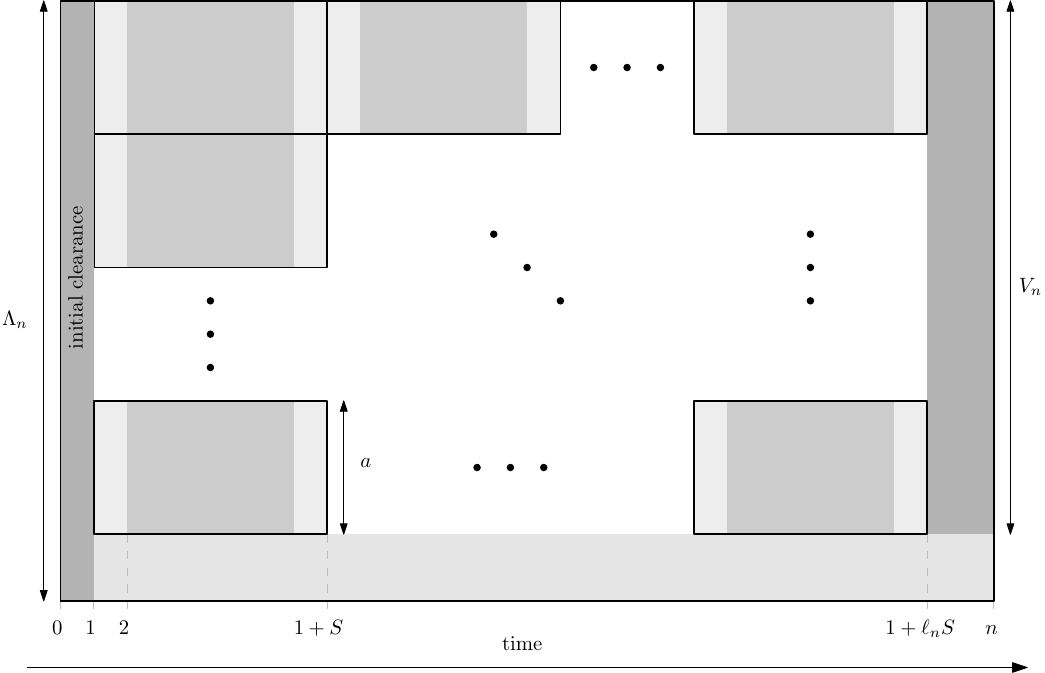}
    \caption{Construction of the auxiliary law $\widehat Q_n$ in the observation window $\Delta_n = \Lambda_n\times[0,n]$, with space drawn vertically. After an initial clearance during $[0,1]$, independent block paths are concatenated in the translates of $A\times[0,S]$ that tile $V_n\times[1,1+\ell_nS]$; the remaining corridor $\Lambda_n\setminus V_n$ and the final time interval stay empty.}
    \label{fig:window}
\end{figure}

Given these building blocks, we construct $\widehat{Q}_n$, still depending on the parameters $a,K,\epsilon,T$, as follows; see \zcref{fig:window}. During the time interval $[0,1]$, we condition the evolution in $\Lambda_n$ on having no births and on all particles of the initial condition $\xi$ in $\Lambda_n$ dying before time $1$. Let $-\overline{n/2}=(-n/2,\ldots,-n/2)$ and $V_n=-\overline{n/2}+[0,\lfloor n/a\rfloor a)^d$, and tile $V_n$ by the half-open cubes $-\overline{n/2}+ak+[0,a)^d$, $k\in\{0,\ldots,\lfloor n/a\rfloor-1\}^d$. In the open interior of each tile, we concatenate $\ell_n=\lfloor(n-1)/S\rfloor$ independent block paths sampled from $\mathcal C_{\mathrm{block}}$, starting at time $1$; block paths in distinct space-time blocks are independent. After time $1$, the corridor $\Lambda_n\setminus V_n$ is kept empty, and births in $V_n$ are suppressed during the final time interval $[1+\ell_nS,n]$. Outside $\Lambda_n$, and after time $n$, we simply continue with the free evolution. The resulting process is defined on the entire positive time line, and we denote its law by $\widehat Q_n$. These laws satisfy the following entropy estimate with respect to $Q_\xi$.

\begin{lemma}[Entropic cost of the auxiliary laws]\label{lemma:auxiliary-entropy}
    Let $\xi\in\Omst$, and let $P$, $a$, $K$, $\epsilon$ and $T$ be as in \zcref{lemma:block-law}. Then $\widehat Q_n\ll Q_{\xi}$ for every $n\in\N$, and
\begin{equation}\label{eq:lb-global-entropy}
\begin{split}
 	\limsup_{n\to\infty}\frac1{n|\Lambda_n|} \Ent(\widehat Q_n\lvert  Q_{\xi})
 	%&\le \frac1{|A|(T+2)}\Big(\frac{h_{A,T}(P)}{P(\omega_0(A)\leq K)}+C_{\mathrm{prep}} +|A|+\frac{c\rho(P)|A|}{P(\omega_0(A)\leq K)} \Big)\\
 	&\le \frac{H(P)}{P(\omega_0(A)\leq K)} + \frac{1}{T+2}\Big(\frac{C_{\mathrm{prep}}}{|A|}+1+\frac{c\rho(P)}{P(\omega_0(A)\leq K)}\Big).
\end{split}
\end{equation}
\end{lemma}

The bound in~\eqref{eq:lb-global-entropy} has a transparent origin: the window $\Delta_n$ contains about $n\abs{\Lambda_n}/(|A|S)$ blocks, each of cost~\eqref{eq:lb-cycle-entropy}; the initial clearance costs $|\Lambda_n|+c\,\xi(\Lambda_n)$, which is negligible on the scale $n\abs{\Lambda_n}$ precisely because $\xi\in\Omst$; and keeping the corridor and the final time interval empty costs a vanishing fraction of $n\abs{\Lambda_n}$. If we first choose $K$ so large that $P(\omega_0(A)\le K)$ is close to one, and then $T$ so large that the second term in~\eqref{eq:lb-global-entropy} is small, the asymptotic entropic cost is at most $H(P)+\alpha$. This takes care of properties~(i) and~(iii) in \zcref{proposition:change-of-measure-sanov}. We now turn towards property~(ii).

We first define the stationary law $\bar{P}$ that the empirical process approaches under $\widehat Q_n$. Let $(\zeta^{(i)})_{i\in\mathbb Z^{d+1}}$ be independent paths with common law $\mathcal C_{\mathrm{block}}$, placed in the blocks
\begin{equation*}
 	(A+ak)\times[\ell S,(\ell+1)S], \quad i=(k,\ell)\in\Z^d\times\Z.
\end{equation*}
Since each of these paths starts and ends with the empty configuration, the paths can be concatenated in time within each spatial tile. Combining the resulting paths over all tiles gives an infinite path $\widetilde\omega$, which belongs to $E$, since the spatial faces of the tiles carry no particles and every bounded space-time box meets only finitely many blocks, each with finitely many particles and events, almost surely. Denote the law of this path by $\mathbf{P}$ and define
\begin{equation}\label{eq:lb-stationary-cycle-law}
	\bar{P} := \frac1{|A|S}\int_A \dd x \int_0^S \dd t \ \mathbf{P} \circ \theta_{(x,t)}^{-1}.
\end{equation}
Since $\mathbf{P}$ is by construction invariant under the shifts by $a\Z^d\times S\Z$, the law $\bar P$ is invariant under all space-time translations. For large $a$ and $T$, most space-time points of a block are far away from its spatial boundary and from the preparation and clearance phases. Seen from such a point, within an observation window of fixed size, a path of $\bar P$ is distributed like a path of $P$ conditioned on $\omega_0(A)\le K$, with particle positions displaced by at most $\epsilon$; this is how $\bar P$ approximates $P$.

To make this precise, we use that the weak topology on $\PP(E)$ is generated by bounded Lipschitz test functions $f\colon E\to \R$, that is, by testing against the family
\begin{equation}\label{eq:BL1}
    \mathrm{BL}_1(E) := \left\{ f\colon E\to\mathbb R\colon \|f\|_\infty\leq 1,\ \abs{f(\omega) - f(\omega')}\leq d_E(\omega,\omega')\ \text{for all } \omega,\omega'\in E \right\},
\end{equation}
where $d_E$ is the bounded metric inducing the local Skorokhod $J_1$ topology on $E$ defined in~\eqref{eqn:definition-j1-metric}. Consequently, every open neighbourhood $G$ of $P$ contains an open set of the form
\begin{equation}\label{eqn:open-u-lipschitz-test}
    U(P;f_1,\dots, f_p;\delta) := \big\{\Tilde{P}\in\PP(E)\colon  \big|\Tilde{P}[f_i]-P[f_i]\big| < \delta, \ 1\leq i \leq p \big\}
\end{equation}
for some $p\in \N$, $\delta >0$, and $f_1, \dots, f_p \in \mathrm{BL}_1(E)$, see \zcref{lemma:open-neighbourhoods}.

The following lemma summarises the needed input about $\bar{P}$: for fixed parameters, $R_n[f]$ converges to $\bar{P}[f]$ in $\widehat{Q}_n$-probability for every $f\in \mathrm{BL}_1(E)$, and the parameters can be chosen such that $\bar P$ approximates $P$ while the entropic cost of $\widehat{Q}_n$ stays below $H(P)+\alpha$. The parameters are chosen in the order $a$, $K$, $\epsilon$, $T$. The side length $a$ comes first, so that most points of a block are far from its spatial boundary on the scale of the test functions $f_1,\dots,f_p$. For this $a$, we choose $K$ so large that $P(\omega_0(A)\le K)$ is close to one, which controls the conditioning error and the factor in front of $H(P)$ in~\eqref{eq:lb-global-entropy}, and then $\epsilon$ so small that the smoothing error is small. Finally, $T$ is chosen so large that the preparation and clearance costs, which depend on $a$, $K$ and $\epsilon$ but not on $T$, are negligible per unit time. This order is possible because the conditioning and smoothing errors are controlled uniformly in $T$, see \zcref{subsec:lb-approximation}.

\begin{lemma}[Law of large numbers and choice of parameters]\label{lemma:G-typical-and-parameters}
Let $\xi\in\Omst$ and let $P \in \PP(E)$ be such that $H(P)<\infty$. Then the following assertions hold.
    \begin{enumerate}[(a)]
        \item For every fixed choice of the parameters $a,K,\epsilon,T$ as in \zcref{lemma:block-law} and every \(f\in\mathrm{BL}_1(E)\),
    \begin{equation*}
           \lim_{n\to \infty}\widehat{Q}_n\big(\abs{R_n[f]-\bar{P}[f]} > \delta\big) = 0 \qquad\text{for every } \delta>0.
    \end{equation*}
    \item For every $\alpha>0$ and every open neighbourhood $G$ of $P$, we can choose the parameters $a,K,\epsilon,T$ such that $\bar{P} \in U \subset G$ for an open set $U = U(P;f_1,\dots, f_p;\delta)$ of the form~\eqref{eqn:open-u-lipschitz-test}, and such that
    \begin{equation*}
        \limsup_{n \to \infty}\frac{1}{n\abs{\Lambda_n}}\Ent(\widehat{Q}_n \lvert Q_\xi) \leq H(P) + \alpha.
    \end{equation*}
    \end{enumerate}
\end{lemma}

With this last technical tool in place, we can now assemble the proof of \zcref{proposition:change-of-measure-sanov}.

\begin{proof}[Proof of \zcref{proposition:change-of-measure-sanov}]
Fix $\alpha >0$ and an open neighbourhood $G$ of $P$. By \zcref{lemma:G-typical-and-parameters}(b), we can choose the parameters $a,K,\epsilon,T$ such that $\bar{P}\in U = U(P;f_1,\dots,f_p;\delta) \subset G$ for some $f_1, \dots, f_p \in \mathrm{BL}_1(E)$ and $\delta>0$, and simultaneously
\begin{equation*}
        \limsup_{n \to \infty}\frac{1}{n\abs{\Lambda_n}}\Ent(\widehat{Q}_n \lvert Q_\xi) \leq H(P) + \alpha.
\end{equation*}
For this fixed choice of parameters, $\widehat{Q}_n \ll Q_\xi$ by \zcref{lemma:auxiliary-entropy}. Since $U$ is open and $\bar P\in U$, there is $\delta'>0$ such that $|R_n[f_i]-\bar P[f_i]|<\delta'$ for all $1\le i\le p$ implies $R_n\in U$. Applying \zcref{lemma:G-typical-and-parameters}(a) to the finitely many test functions $f_1,\dots,f_p$, we therefore obtain
\begin{equation*}
    \lim_{n\to\infty}\widehat{Q}_n(R_n \in G)\ge \lim_{n\to\infty}\widehat{Q}_n(R_n \in U) = 1.
\end{equation*}
So the constructed sequence $(\widehat{Q}_n)_{n \in \N}$ indeed satisfies all of the properties~(i)--(iii).
\end{proof}

The proofs of \zcref{lemma:block-law,lemma:auxiliary-entropy,lemma:G-typical-and-parameters} are given in \zcref{section:proof-lower-bound}.

\subsection{LDP upper bound in the non-interacting case}\label{subsection:strategy-upper-bound}
Let us now move towards the upper bound for the free evolution.
\begin{theorem}[LDP upper bound, non-interacting]\label{thm:ldp-upper-bound}
    Let $\xi\in\Omst$. For every weakly closed set $F \subset \PP(E)$ we have
    \begin{equation*}
        \limsup_{n\to \infty}\frac{1}{n|\Lambda_n|} \log Q_{\xi}(R_n \in F) \leq - \inf_{P\in F} H(P).
    \end{equation*}
\end{theorem}
As usual for large deviation upper bounds, the proof can be decomposed into two separate parts: an upper bound for compact sets, and exponential tightness.
\begin{proposition}[Upper bound for compact sets]\label{prop:compact-upper-bound}
    Let $\xi\in\Om$. For every weakly compact set $F\subset \PP(E)$ we have
    \begin{equation*}
        \limsup_{n\to \infty}\frac{1}{n|\Lambda_n|} \log Q_{\xi}(R_n \in F) \leq - \inf_{P\in F} H(P).
    \end{equation*}
\end{proposition}

\begin{proposition}[Exponential tightness]\label{prop:exponential-tightness}
    Let $\xi\in\Omst$. For every $\alpha,\varepsilon>0$, there is a compact set $\mathscr K\subset E$ with
    \begin{equation}\label{eq:pathcompact}
    Q_{\xi}(R_n(\mathscr K^c)>\varepsilon)\le \e^{-\alpha n|\Lambda_n|}
    \qquad\text{for every }n\ge1.
    \end{equation}
    Consequently, the empirical processes $R_n$ are exponentially tight on $\PP(E)$, that is, for every $\alpha \in (0,\infty)$ there is a weakly compact set $K_\alpha \subset \PP(E)$ such that
    \begin{equation}\label{eq:exponential-tightness}
        \limsup_{n \to \infty}\frac{1}{n|\Lambda_n|}\log Q_{\xi}(R_n \not\in K_\alpha) \leq -\alpha.
    \end{equation}
\end{proposition}

We first show how \zcref{thm:ldp-upper-bound} follows from these two parts via standard arguments.
\begin{proof}[Proof of~\zcref{thm:ldp-upper-bound}]
    Let $F \subset \PP(E)$ be closed and let $\alpha >0$. For every $n\in \N$,
    \begin{equation*}
        Q_{\xi}(R_n \in F) \leq Q_{\xi}(R_n \in F \cap K_\alpha) + Q_\xi(R_n \not\in K_\alpha),
    \end{equation*}
    where $K_\alpha$ is the compact set from \zcref{prop:exponential-tightness}. The set $F\cap K_\alpha$ is compact, so we can apply \zcref{prop:compact-upper-bound} to obtain
    \begin{equation*}
        \limsup_{n \to \infty}\frac{1}{n|\Lambda_n|}\log Q_{\xi}(R_n \in F) \leq \max\Big\{-\inf_{F\cap K_\alpha}H, -\alpha\Big\} \leq \max\Big\{-\inf_F H, -\alpha\Big\}.
    \end{equation*}
    The claimed upper bound follows by sending $\alpha$ to infinity.
\end{proof}
Note that \zcref{prop:compact-upper-bound} holds for every $\xi\in\Om$; the assumption $\xi\in\Omst$ enters the upper bound only through exponential tightness.

For the upper bound on compact sets, we follow the general strategy of~\cite{DaiPra1993,varadhanLectureNotes}, the main input being a variational representation of the local entropies $h_{A,T}$. For $A=(0,a)^d$, $T>0$ and a bounded continuous function $f$ on $D([0,T];\Om_A)$, let
\begin{equation*}
    g_f(\eta):=\log Q^A_\eta[\e^f],\qquad
    \psi_f(\omega):=f(\omega|_{A\times[0,T]})-g_f(\omega_0|_A).
\end{equation*}
Then $h_{A,T}(P)=\sup_f P[\psi_f]$, and the maps $P\mapsto P[\psi_f]$ are weakly continuous, so that $H$ is a supremum of weakly continuous affine functionals (\zcref{lemma:continuity-properties,lemma:nice-variational-representation}). The subtraction of $g_f$ normalises the exponential moments, $Q^A_\eta[\e^{f-g_f(\eta)}]=1$ for every $\eta$. Conditionally on the configuration at a given time, the free evolutions in disjoint cubes are independent, and the free evolution is Markovian. Hence the exponential moment of the sum of the functions $\psi_f$, evaluated at the path seen from the corners of a lattice of disjoint space-time blocks in $\Delta_n$, equals one, and averaging over the offsets of this lattice yields
\begin{equation*}
    \limsup_{n\to\infty}\frac{1}{n|\Lambda_n|}
    \log Q_{\xi}\left[\exp\left\{\frac{n|\Lambda_n|}{|A|T}R_n[\psi_f]\right\}\right]\le0,
\end{equation*}
see \zcref{lemma:upper-bound-for-such-laplace-transforms}. An exponential Chebyshev inequality in a neighbourhood of each law in the compact set, together with a finite covering, then proves \zcref{prop:compact-upper-bound}.

Exponential tightness is proved by quite different means, adapted to the continuum setting, where the number of points in a volume of fixed size is not uniformly bounded. A bound on the number of particles alone does not give compactness in the $J_1$ topology on $E$: we also have to make sure that no two jumps occur arbitrarily close together in time, and that no jumps occur too close to the endpoints of an observation interval, since $J_1$ time changes fix the endpoints, so that jumps accumulating at an endpoint would destroy compactness. For a bounded open set $B\Subset\R^d$, a time interval $[u,v]$, $M\in\N$ and $\delta>0$, the set of piecewise constant paths in $D([u,v];\Om_B)$ with at most $M$ points at every time, at most $2M$ jumps, consecutive jump times separated by at least $\delta$, and all jump times in $[u+\delta,v-\delta]$, is compact (\zcref{lemma:compact-sets-building-blocks}). Imposing such constraints on a growing sequence of space-time boxes defines compact subsets of $E$ (\zcref{lem:compactness}). It remains to show that, outside an event of probability at most $\e^{-\alpha n\abs{\Lambda_n}}$, where $\alpha$ can be made arbitrarily large by tightening the constraints, $R_n$ puts little mass on paths violating these constraints. In other words, only a small fraction of the space-time points in $\Delta_n$ may have a local window with too many particles, with two jumps that are too close in time, or with a jump too close to the endpoints. These three local defects are controlled by elementary properties of Poisson random variables: exponential moments of the total occupation time and of the total number of births in $\Delta_n$, and a comparison of the jumps in windows of bounded population with Poisson clocks, see \zcref{subsection:probabilistic-control}.

The proofs of \zcref{prop:compact-upper-bound,prop:exponential-tightness} are given in \zcref{section:proof-upper-bound}.

The lower and upper bounds of \zcref{thm:ldp-lower-bound-free,thm:ldp-upper-bound} are the large deviation principle of \zcref{theorem:full-ldp} for $\beta=0$, except for the goodness of the rate function. The latter follows from exponential tightness and the lower bound by a standard argument, see e.g.~\cite[Lemma~1.2.18]{LDPbook}, which we include for completeness.

\begin{proof}[Proof of \zcref{theorem:full-ldp} for $\beta=0$]
The lower and upper bounds are \zcref{thm:ldp-lower-bound-free,thm:ldp-upper-bound}. By \zcref{lemma:nice-variational-representation}, $H$ is weakly lower semicontinuous, so its sublevel sets are closed. Let $c\in\R$ and $\alpha>\max\{c,0\}$, and let $K_\alpha$ be the compact set from \zcref{prop:exponential-tightness}. If $P\notin K_\alpha$, then $G=\PP(E)\setminus K_\alpha$ is an open neighbourhood of $P$, and \zcref{thm:ldp-lower-bound-free} gives
\begin{equation*}
    -H(P)\le-\inf_G H\le\liminf_{n\to\infty}\frac{1}{n|\Lambda_n|}\log Q_\xi(R_n\notin K_\alpha)\le-\alpha<-c.
\end{equation*}
Hence $\{H\le c\}\subset K_\alpha$, and $\{H\le c\}$ is compact as a closed subset of a compact set.
\end{proof}

\subsection{LDP in the interacting case}\label{subsection:strategy-interacting}
For $\beta\neq0$, we deduce \zcref{theorem:full-ldp} from the free case. Our strategy follows the architecture of the corresponding large deviation principle for lattice spin-flip systems in~\cite{DaiPra1993,DaiPra1993_2}. Throughout this subsection and in \zcref{section:interacting-case}, we suppress the dependence of all laws on the initial condition $\xi$, because we never compare dynamics with different initial conditions; we write $Q^\beta_n$ for the law of the dynamics in $\Lambda_n$ with periodic boundary conditions, and $Q^0_n=Q^{\Lambda_n}_\xi$ for the free evolution in $\Lambda_n$, both started from $\xi_{\Lambda_n}$.

The centrepiece of the transfer is the following explicit expression for the density of one birth-and-death dynamics with respect to another with the same death rate. It also shows that $\mathfrak{L}_n$ from~\eqref{eq:girsanov-heuristic} is the log-density of $Q^\beta_n$ with respect to $Q^0_n$ on the time interval $[0,n]$.
\begin{proposition}[Change of measure à la Girsanov]\label{proposition:change_of_measure_a_la_girsanov}
    Let \(u_t(x, \omega)\), \(v_t(x, \omega)\) be two predictable birth rates in a bounded volume \(\Lambda \Subset \R^d\). Assume that they are uniformly elliptic, that is, bounded from below by \(m > 0\) and bounded from above by \(M < \infty\). For \(T > 0\), consider the associated path laws \(Q^u, Q^v\) on \(\Lambda \times [0, T]\), both with death rate \(1\) and the same initial distribution.
    Then
    \begin{equation*}
        \frac{\d Q^u}{\d Q^v}
        = \exp\bigg(\int_0^T\d t \int_{\Lambda} \d x\, [v_t(x, \cdot) - u_t(x, \cdot)] + \sum_{\substack{(t, x) \text{ birth}\colon (t, x) \in (0, T] \times \Lambda}} \log\frac{u_t(x, \cdot)}{v_t(x, \cdot)} \bigg).
    \end{equation*}
    If \(u = v\) off a set \(D \subset \Lambda\), then for all \(p > 1\) we have the moment bound
    \begin{equation}\label{equation:interacting_case_girsanov_moment_bound}
        Q^v\Big[\Big(\frac{\d Q^u}{\d Q^v}\Big)^p \Big]
        \leq \e^{C_p T\abs{D}},
        \qquad C_p = \sup_{s, t \in [m, M]} \{s^p t^{1-p} - ps + (p-1) t\}.
    \end{equation}
    In particular, $\mathfrak{L}_n=\log(\d Q^\beta_n/\d Q^0_n)$ on the time interval $[0,n]$, and
    \begin{equation}\label{equation:interacting_case_girsanov_moment_bound_periodic_interacting_wrt_free}
        Q_{n}^0[\e^{p \mathfrak{L}_{n}}]
        \leq \e^{c_p n \abs{\Lambda_n}},
        \qquad c_p = \sup_{s \in [1/\Vert 1/b \Vert_\infty, \Vert b\Vert_\infty]} \{s^p - 1 - p(s-1)\},
        \qquad p > 1.
    \end{equation}
\end{proposition}

By the heuristics in \zcref{subsection:rate-function}, $\mathfrak{L}_n\approx n\abs{\Lambda_n}R_n[F]$, so that the large deviations of the interacting dynamics should be obtained from those of the free evolution by tilting with the density $\e^{n\abs{\Lambda_n}R_n[F]}$; Varadhan's lemma would then give the rate function $H-P[F]=H^\beta$. It cannot be applied directly, for two reasons that are already present for the lattice systems in~\cite{DaiPra1993,DaiPra1993_2}: the identity $\mathfrak{L}_n= n\abs{\Lambda_n}R_n[F]$ holds only up to boundary terms, and $F$ is neither bounded nor continuous, since it counts births in a unit space-time cell. We therefore approximate $F$ by bounded continuous functions $F_j$, in which the birth count is replaced by a smoothed increment of the number of particles near the origin, and we replace Varadhan's lemma by three estimates: a superexponential approximation of $\mathfrak{L}_n/(n\abs{\Lambda_n})$ by $R_n[F_j]$ under the free evolution (\zcref{proposition:interacting_case_approximation_of_rate_function_approximately_functional_of_R_n}), the convergence $P[F_j]\to P[F]$ uniformly on the sublevel sets of $H$ (\zcref{proposition:interacting_case_approximation_of_rate_function}), and the moment bound~\eqref{equation:interacting_case_girsanov_moment_bound_periodic_interacting_wrt_free}. In contrast to the lattice setting, births may occur arbitrarily close to each other in space and time, and the number of particles per volume is unbounded, so that the superexponential approximation requires a probabilistic control of births that occur close to other events.

\begin{definition}[Approximation scheme]\label{definition:interacting_case_approximation_of_rate_function}
    Fix, for every \(\delta > 0\), a non-negative function \(\varphi_\delta \in C_c(\R^d)\) with \(\supp \varphi_\delta \subset B_\delta(0)\) and \(\int \varphi_\delta(x) \, \d x = 1\).
    For \(\epsilon, \delta, K > 0\), set
    \begin{equation*}
        U_{\epsilon, \delta, K}(\omega)
        := \int_0^1 \d t\, \log(b(0, \omega_t)) \,[D_{\epsilon, \delta}(\omega, t) \land K],
        \quad
        D_{\epsilon, \delta}(\omega, t)
        := \frac{1}{\epsilon}\Big(\sum_{x \in \omega_{t+\epsilon}} \varphi_\delta(x) - \sum_{x \in \omega_{t}} \varphi_\delta(x) \Big)_+,
    \end{equation*}
    and write \(U_{\epsilon, \delta} := U_{\epsilon, \delta, \infty}\) for the same expression without clipping \(D_{\epsilon, \delta}\). Moreover, set
    \begin{equation*}
        F_{\epsilon, \delta, K}(\omega)
        := \int_{0}^{1}\d t\, [1 - b(0, \omega_t)] + U_{\epsilon, \delta, K}(\omega),
    \end{equation*}
    and define \(F_{\epsilon, \delta}\) analogously, with \(U_{\epsilon, \delta}\) in place of \(U_{\epsilon, \delta, K}\).
\end{definition}

Here $D_{\epsilon,\delta}(\omega,t)$ is the increase during $[t,t+\epsilon]$ of the number of particles near the origin, weighted by $\varphi_\delta$ and divided by $\epsilon$. For a legal path whose births near the origin are well separated in time, $U_{\epsilon,\delta}(\omega)$ is therefore approximately the $\varphi_\delta$-weighted sum of $\log b(y,\omega_{s-})$ over the births $(s,y)$ with $s\in(0,1]$ and $y$ close to the origin. After averaging over space-time shifts, as under $R_n$ or under a stationary law, this becomes the birth term of $F$, since $\int\varphi_\delta(x)\,\d x=1$. The time integral and the mollifier $\varphi_\delta$ make $U_{\epsilon,\delta,K}$ continuous, the clipping at level $K$ makes it bounded, and the time integral in the first term of $F_{\epsilon,\delta,K}$ replaces the evaluation at time zero in $F$, which has the same expectation under stationary laws but is not continuous. (The parameters $\epsilon$, $\delta$ and $K$ have a different meaning in this subsection than in \zcref{subsection:strategy-lower-bound}.) The approximants $F_{\epsilon,\delta,K}$ are bounded and continuous (\zcref{lemma:interacting_case_approximants_of_rate_function_are_continuous}). The following two propositions show in which sense they approximate $F$: superexponentially well under the free evolution along the empirical process, and in expectation, uniformly on the sublevel sets of $H$.

\begin{proposition}[The Girsanov density is approximately a functional of \(R_n\)]\label{proposition:interacting_case_approximation_of_rate_function_approximately_functional_of_R_n}
    There is a choice of \(\epsilon_j \to 0\), \(\delta_j \to 0\) and \(K_j \to \infty\) such that, for
    \(F_{j} := F_{\epsilon_j, \delta_j, K_j}\) and every tolerance \(\tau > 0\),
    \begin{equation}\label{equation:interacting_case_approximation_of_rate_function_approximately_functional_of_R_n}
        \lim_{j \to \infty} \limsup_{n \to \infty} \frac{1}{n \abs{\Lambda_n}} \log Q^0_n\Big(\big| \tfrac{1}{n\abs{\Lambda_n}}\mathfrak{L}_n - R_n[F_j] \big| > \tau \Big)
        = -\infty.
    \end{equation}
    Moreover, the sequences can be chosen such that
    \begin{equation}\label{equation:interacting_case_particular_apprximation_parameters}
        \frac{\epsilon_j K_j}{\Vert \varphi_{\delta_j} \Vert_\infty} \xrightarrow[j \to \infty]{} \infty.
    \end{equation}
\end{proposition}
Since a single birth raises $D_{\epsilon,\delta}$ by at most $\Vert\varphi_\delta\Vert_\infty/\epsilon$, condition~\eqref{equation:interacting_case_particular_apprximation_parameters} ensures that the clipping at level $K_j$ only affects time intervals with many births.

\begin{proposition}[Approximation of \(F\)]\label{proposition:interacting_case_approximation_of_rate_function}
    Let $(F_j)_{j\in\N}$ be as in \zcref{proposition:interacting_case_approximation_of_rate_function_approximately_functional_of_R_n}. For every \(P\in\PP(E)\) with \(H(P) < \infty\), we have
    \begin{equation}\label{equation:interacting_case_pointwise_regarding_measure_convergence_of_rate_function_approximation}
        P[F_j]
        \xrightarrow[j \to \infty]{} P[F].
    \end{equation}
    The convergence is uniform on every sublevel set \(\{H \leq \alpha\}\), \(\alpha \geq 0\), so that in particular the map \(P \mapsto P[F]\) is continuous on each such sublevel set with respect to the relative weak topology.
    Moreover, for every \(\mathcal{K} \subset \PP(E)\),
    \begin{equation}\label{equation:interacting_case_rate_function_approximation_upper_bound_set_of_measures}
        \limsup_{j \to \infty} \sup_{P \in \mathcal{K}} \big(P[F_j] - H(P)\big)
        \leq -\inf_{P \in \mathcal{K}} H^\beta(P).
    \end{equation}
\end{proposition}

Finally, the same change of measure shows that boundary conditions play no role on the exponential scale.
\begin{lemma}[Irrelevance of boundary conditions]\label{lemma:interacting_case_boundary_conditions_play_no_role}
    Let \((A_n)_{n \in \N}\) be a sequence of events with \(A_n \in \mathcal{F}_{\Lambda_n \times [0, n]}\) for every \(n \in \N\). The limits
    \begin{equation*}
        \liminf_{n \to \infty} \frac{\log P_n(A_n)}{n\abs{\Lambda_n}}
        \qquad\text{and}\qquad \limsup_{n \to \infty} \frac{\log P_n(A_n)}{n\abs{\Lambda_n}}
    \end{equation*}
    each have identical values for \((P_n)_{n \in \N}\) being any of the laws of the dynamics in $\Lambda_n$ with periodic boundary conditions, of the dynamics in $\Lambda_n$ with time-dependent boundary conditions as in \zcref{subsection:boundary-conditions}, or of the infinite-volume dynamics, all started from the same initial configuration \(\xi\).
\end{lemma}

We can now assemble the proof of the main result for $\beta\neq0$.

\begin{proof}[Proof of \zcref{theorem:full-ldp} for $\beta\neq0$]
Since $\{R_n\in G\}\in\mathcal{F}_{\Lambda_n\times[0,n]}$ for every Borel set $G\subset\PP(E)$, \zcref{lemma:interacting_case_boundary_conditions_play_no_role} shows that it suffices to prove the large deviation principle under the periodic dynamics $Q^\beta_n$. The goodness of $H^\beta$ is proved in \zcref{lemma:interacting_case_rate_function_further_properties}.

\emph{Step 1: Exponential tightness.} Let $K_\alpha$ be the compact set from \zcref{prop:exponential-tightness}; since $R_n$ only depends on the restriction of the path to $\Delta_n$, on which $Q_\xi$ and $Q^0_n$ agree, \eqref{eq:exponential-tightness} also holds under $Q^0_n$. Let $p,q>1$ with $1/p+1/q=1$. By \zcref{proposition:change_of_measure_a_la_girsanov} and H\"older's inequality,
\begin{equation*}
\begin{split}
    Q^\beta_n(R_n\notin K_\alpha)&=Q^0_n\big[\e^{\mathfrak{L}_n}\1_{\{R_n\notin K_\alpha\}}\big]\le Q^0_n[\e^{p\mathfrak{L}_n}]^{1/p}\,Q^0_n(R_n\notin K_\alpha)^{1/q}\\
    &\le\exp\Big\{n\abs{\Lambda_n}\Big(\frac{c_p}{p}+\frac{1}{q}\cdot\frac{\log Q^0_n(R_n\notin K_\alpha)}{n\abs{\Lambda_n}}\Big)\Big\},
\end{split}
\end{equation*}
so that $\limsup_{n}(n\abs{\Lambda_n})^{-1}\log Q^\beta_n(R_n\notin K_\alpha)\le c_p/p-\alpha/q$, which tends to $-\infty$ as $\alpha\to\infty$. Hence the empirical processes are exponentially tight under $Q^\beta_n$ as well, and it suffices to prove the upper bound for compact sets.

\emph{Step 2: Lower bound.} Let $G$ be open, $P\in G$ with $H(P)<\infty$, and $\tau>0$. By \zcref{proposition:interacting_case_approximation_of_rate_function_approximately_functional_of_R_n}, we can choose $j$ so large that
\begin{equation*}
    \limsup_{n \to \infty} \frac{1}{n \abs{\Lambda_n}} \log Q^0_n\Big(\abs{ \tfrac{1}{n\abs{\Lambda_n}}\mathfrak{L}_n - R_n[F_j] } > \tau \Big)<-H(P).
\end{equation*}
Since $F_j$ is bounded and continuous, $G_j:=G\cap\{\widetilde P\colon \widetilde P[F_j]>P[F_j]-\tau\}$ is an open neighbourhood of $P$. On the event $\{R_n\in G_j\}\cap\{\abs{\mathfrak{L}_n/(n\abs{\Lambda_n})-R_n[F_j]}\le \tau\}$, we have $\mathfrak{L}_n\ge n\abs{\Lambda_n}(P[F_j]-2\tau)$, and therefore
\begin{equation*}
    Q^\beta_n(R_n\in G)\ge \e^{n\abs{\Lambda_n}(P[F_j]-2\tau)}\Big(Q^0_n(R_n\in G_j)-Q^0_n\big(\big| \tfrac{1}{n\abs{\Lambda_n}}\mathfrak{L}_n - R_n[F_j] \big| > \tau \big)\Big).
\end{equation*}
By \zcref{thm:ldp-lower-bound-free}, the first probability in the bracket decays at most at exponential rate $H(P)$, whereas the second decays strictly faster by the choice of $j$. Hence
\begin{equation*}
    \liminf_{n \to \infty} \frac{1}{n \abs{\Lambda_n}} \log Q_n^\beta(R_n \in G)
    \geq  -H(P) + P[F_j] - 2\tau.
\end{equation*}
Letting $j \to \infty$, using $P[F_j]\to P[F]$ from \zcref{proposition:interacting_case_approximation_of_rate_function}, then $\tau \downarrow 0$, and finally optimising over $P \in G$ yields the lower bound with rate function $H^\beta$.

\emph{Step 3: Upper bound for compact sets.} Let $\mathcal{K}\subset\PP(E)$ be compact and $\tau>0$. We decompose
\begin{equation*}
    Q_n^\beta(R_n \in \mathcal{K})
    \leq Q_n^\beta\big(R_n \in \mathcal{K},  \big\vert \tfrac{1}{n\abs{\Lambda_n}}\mathfrak{L}_n - R_n[F_j] \big\vert \leq \tau\big) + Q_n^\beta\big(\big\vert \tfrac{1}{n\abs{\Lambda_n}}\mathfrak{L}_n - R_n[F_j] \big\vert > \tau\big).
\end{equation*}
The first term is at most $Q^0_n[\exp\{n\abs{\Lambda_n}(R_n[F_j]+\tau)\}\1_{\{R_n\in \mathcal{K}\}}]$. Since $F_j$ is bounded and continuous, the map $\widetilde P\mapsto \widetilde P[F_j]$ is weakly continuous. Hence, for every $\eta>0$, the set $\mathcal{K}$ is covered by finitely many compact sets of the form $\{\widetilde P\in\mathcal{K}\colon|\widetilde P[F_j]-P_i[F_j]|\le\eta\}$ with $P_i\in\mathcal{K}$, and applying \zcref{prop:compact-upper-bound} to each of them and letting $\eta\downarrow0$ shows that the exponential rate of the first term is at most $\tau+\sup_{P\in \mathcal{K}}(P[F_j]-H(P))$. By H\"older's inequality, \eqref{equation:interacting_case_girsanov_moment_bound_periodic_interacting_wrt_free} and \eqref{equation:interacting_case_approximation_of_rate_function_approximately_functional_of_R_n}, the second term satisfies
\begin{equation*}
    \delta_j(\tau)
    := \limsup_{n \to \infty} \frac{1}{n\abs{\Lambda_n}} \log Q_n^\beta\big(\big\vert \tfrac{1}{n\abs{\Lambda_n}}\mathfrak{L}_n - R_n[F_j] \big\vert > \tau\big)
    \xrightarrow[j \to \infty]{} -\infty.
\end{equation*}
Altogether, using~\eqref{equation:interacting_case_rate_function_approximation_upper_bound_set_of_measures},
\begin{equation*}
\begin{split}
    \limsup_{n \to \infty} \frac{1}{n \abs{\Lambda_n}} \log Q_n^\beta(R_n \in \mathcal{K})
    &\leq \limsup_{\tau \downarrow 0} \limsup_{j \to \infty }\max\Big\{\tau + \sup_{P \in \mathcal{K}}\big(P[F_j] - H(P)\big), \delta_j(\tau) \Big\} \leq -\inf_{P \in \mathcal{K}} H^\beta(P).
\end{split}
\end{equation*}
Together with Step~1, this proves the upper bound for all closed sets.
\end{proof}

\begin{proof}[Proof of \zcref{proposition:rate-function-properties}]
The conditional-entropy representation~\eqref{equation:interacting_case_rate_function_conditional_entropy_representation} is \zcref{lemma:interacting_case_rate_function_conditional_entropy_representation}, and the remaining assertions are contained in \zcref{lemma:interacting_case_rate_function_further_properties}.
\end{proof}

The proofs of the results stated in this subsection are given in \zcref{section:interacting-case}. The proof of \zcref{lemma:interacting_case_boundary_conditions_play_no_role}, as well as the proofs of three auxiliary lemmas used for \zcref{proposition:interacting_case_approximation_of_rate_function_approximately_functional_of_R_n}, see \zcref{lemma:interacting_case_approximation_of_rate_function_deterministic_error,lemma:interacting_case_approximation_of_rate_function_nearby_events,lemma:interacting_case_approximation_of_rate_function_exponential_clipping}, are only sketched there.

\subsection{Prescribed time marginals}\label{subsection:strategy-contraction}
The large deviation principle in \zcref{thm:marg-ldp} follows from \zcref{theorem:full-ldp} by the contraction principle, since the map $P\mapsto P_0$ is continuous from $\PP(E)$ to $\PP(\Om)$: the evaluation at time zero is continuous at every path without a jump at time zero, and a stationary law charges a jump at a fixed time with probability zero, see the proof of \zcref{lemma:continuity-properties}. The work lies in the identification of the contracted rate function, that is, in the minimum entropy problem
\begin{equation}\label{equation:contraction_time_marginal_minimum_entropy_problem}
    \min\{H^\beta(P)\colon P \in \PP(E),\, P_0 = \mu \}
    = 2 \mathfrak{H}^2(b(0, \cdot) \mu, \mu_0^!), \qquad \mu\in\PP(\Om).
\end{equation}

It is natural to look for a minimiser among reversible birth-and-death processes with time marginal $\mu$: for these, the stationarity of $\mu$ reduces to a pointwise detailed balance relation, so that the cost can be minimised pointwise. We first compute the optimal birth and death rates formally. Suppose that $P$ is a space-time stationary Markov law with $P_0 = \mu$ and with uniformly elliptic, translation-covariant birth rate $q(x, \eta)$ and death rate $d(x, \eta)$. The conditional-entropy representation~\eqref{equation:interacting_case_rate_function_conditional_entropy_representation} and a finite-volume Girsanov computation similar to \zcref{proposition:change_of_measure_a_la_girsanov} show that
\begin{equation}\label{equation:contraction_time_marginal_entropy_for_B_D_process}
    H^\beta(P)
    = \mu\Big[\int_{[0, 1)^d}\d x \, \ell(q(x, \eta) \vert b(x, \eta))  \Big] + \mu\Big[\sum_{x \in \eta_{[0, 1)^d}} \ell(d(x, \eta) \vert 1) \Big],
\end{equation}
where \(\ell(s \vert t) = s \log(s / t) - s + t\).
If we further assume that \(\mu\) has a Papangelou intensity \(\kappa(x, \eta)\) in the sense of \zcref{subsection:marginals}, which for the sake of this argument is bounded away from zero and infinity, then the GNZ equation allows us to write \eqref{equation:contraction_time_marginal_entropy_for_B_D_process} as
\begin{equation*}
    H^\beta(P)
    = \mu\Big[\int_{[0, 1)^d}\d x \, \big[ \ell(q(x, \eta) \vert b(x, \eta)) + \kappa(x, \eta) \ell(d(x, \eta + \delta_x) \vert 1) \big]  \Big].
\end{equation*}
For $P$ to be time-stationary and reversible, the rates have to satisfy the detailed balance relation \(q(x, \eta) = \kappa(x, \eta) d(x, \eta + \delta_x)\), and substituting it gives
\begin{equation*}
    H^\beta(P)
    = \mu\Big[\int_{[0, 1)^d}\d x \, \big[ \ell(q(x, \eta) \vert b(x, \eta)) + \ell(q(x, \eta) \vert \kappa(x, \eta))  \big]  \Big].
\end{equation*}
For fixed \(b, \kappa > 0\), the minimum \((\sqrt{b} - \sqrt{\kappa})^2\) of \(s \mapsto \ell(s\vert b) + \ell(s \vert \kappa)\) is attained at \(s = \sqrt{b\kappa}\). Consequently, the candidate birth and death rates are
\begin{equation}\label{equation:contraction_time_marginal_entropy_for_B_D_process_optimal_rates}
    q(x, \eta) = \sqrt{b(x, \eta) \kappa(x, \eta)},
    \qquad d(x, \eta + \delta_x)
    = \frac{q(x, \eta)}{\kappa(x, \eta)}
    = \sqrt{\frac{b(x, \eta)}{\kappa(x, \eta)}},
\end{equation}
and with these rates
\begin{equation*}
    H^\beta(P)
    = \mu\Big[\int_{[0, 1)^d} \d x\, \Big(\sqrt{b(x, \eta)} - \sqrt{\kappa(x,\eta)}\Big)^2\Big].
\end{equation*}

To turn this heuristics into a proof for general $\mu$, which need not have a Papangelou intensity, we work with the measures \(B_\mu\) and \(D_\mu\) on \(\Omega \times [0, 1)^d\) given by
\begin{equation*}
\begin{split}
    \int B_\mu(\d \eta , \d x)\, f(\eta, x)  &= \int\mu(\d\eta) \int_{[0,1)^d} \d x \, b(x, \eta) \, f(\eta, x), \\
    \int D_\mu(\d \eta, \d x)\, f(\eta, x) &=  \int\mu(\d\eta)  \sum_{x \in \eta_{[0,1)^d}} f(\eta - \delta_x, x)
\end{split}
\end{equation*}
for measurable \(f \geq 0\), where the sum over \(x \in \eta\) counts multiplicities. Thus $B_\mu$ describes the births of the original dynamics in a unit cell, and $D_\mu$ the deaths, both in terms of the configuration without the particle in question. If \(\mu\) has a Papangelou intensity \(\kappa\), then $D_\mu(\d\eta,\d x)=\mu(\d\eta)\kappa(x,\eta)\1_{[0,1)^d}(x)\,\d x$, and the heuristic minimal cost is $2\mathfrak{H}^2(B_\mu,D_\mu)$.

\begin{lemma}[Hellinger distance identities]\label{lemma:contraction_time_marginal_hellinger_distance_identities}
    Let $\mu\in\PP(\Om)$. Then
    \begin{equation}\label{equation:contraction_time_marginal_hellinger_distance_identities}
        \mathfrak{H}^2(b(0, \cdot) \mu, \mu_0^!)
        = \mathfrak{H}^2(B_\mu, D_\mu).
    \end{equation}
    If \(\mu\) has a Papangelou intensity \(\kappa(x, \eta)\), then
    \begin{equation}\label{equation:contraction_time_marginal_hellinger_distance_identity_with_papangelou}
        2 \mathfrak{H}^2(B_\mu, D_\mu )
        = \mu\Big[\int_{[0, 1)^d}\d x \, \Big(\sqrt{b(x, \eta)} - \sqrt{\kappa(x,\eta)}\Big)^2 \Big].
    \end{equation}
\end{lemma}

To lift finite-volume computations and bounds, we also use, for \(\Lambda \Subset \R^d\), the measures \(B^{(\Lambda)}_\mu\) and \(D^{(\Lambda)}_\mu\) on \(\Omega_{\Lambda} \times \Lambda\) given by
\begin{equation*}
\begin{split}
    \int \, f(\eta, x) \, B^{(\Lambda)}_\mu(\d \eta , \d x) &= \int\mu_\Lambda(\d\eta_\Lambda) \int_\Lambda \d x \, b(x, \eta_\Lambda) \, f(\eta_\Lambda, x), \\
    \int \, f(\eta, x) \, D^{(\Lambda)}_\mu(\d \eta, \d x) &=  \int\mu_\Lambda(\d\eta_\Lambda)  \sum_{x \in \eta_\Lambda} f(\eta_\Lambda - \delta_x, x)
\end{split}
\end{equation*}
for measurable \(f \geq 0\), where $\mu_\Lambda$ denotes the law of $\eta_\Lambda$ under $\mu$, and similarly $\nu_\Lambda$ for other probability measures $\nu$ on $\Om$. They describe the births and deaths of the dynamics in $\Lambda$ with empty boundary configuration, and their Hellinger distance per unit volume converges to the infinite-volume one.

\begin{lemma}[Thermodynamic limit]\label{lemma:contraction_time_marginal_minimum_entropy_problem_hellinger_thermodynamic_limit}
    Let \(\mu \in \PP(\Om)\).
    Then
    \begin{equation}\label{equation:contraction_time_marginal_minimum_entropy_problem_hellinger_thermodynamic_limit}
        \mathfrak{H}^2(B_\mu, D_\mu)
        = \lim_{n \to \infty} \frac{1}{\abs{\Lambda_n}} \mathfrak{H}^2(B_\mu^{(\Lambda_n)}, D_\mu^{(\Lambda_n)}).
    \end{equation}
    If \(\mu\) is locally absolutely continuous with respect to the Poisson point process and \(\nu \in \GG^\beta_\theta\), then
    \begin{equation}\label{equation:contraction_time_marginal_dirichlet_form_identity_hellinger_distance}
        2\mathfrak{H}^2(B_\mu, D_\mu)
        = \lim_{n \to \infty} \frac{1}{\abs{\Lambda_n}} \mathcal{E}^{\nu}_{\Lambda_n}\Big(\sqrt{\tfrac{\d \mu_{\Lambda_n}}{\d \nu_{\Lambda_n}}}, \sqrt{\tfrac{\d \mu_{\Lambda_n}}{\d \nu_{\Lambda_n}}}\Big),
    \end{equation}
    with the formal Dirichlet form
    \begin{equation*}
        \mathcal{E}^\nu_\Lambda(g, g)
        = \nu\big[g(-\mathscr{L}g) \big]
        = \int \nu_\Lambda(\d\eta_\Lambda) \sum_{x \in \eta_\Lambda} [g(\eta_\Lambda - \delta_x) - g(\eta_\Lambda)]^2
        \in [0, \infty], \quad g \geq 0, g \in L^2(\nu_\Lambda).
    \end{equation*}
\end{lemma}
The identity~\eqref{equation:contraction_time_marginal_dirichlet_form_identity_hellinger_distance} might remind the reader of the rate function in the classical large deviation principle for the empirical measure of continuous-time Markov chains, see e.g.~\cite[Theorem~IV.14 and Exercise~IV.24]{denHollander2000}.

The heuristics suggests that $2\mathfrak{H}^2(B_\mu,D_\mu)$ is the minimal cost. The lower bound is proved in finite volume, via a variational characterisation of the Hellinger distance in terms of the generator of the dynamics in $\Lambda$ and an exponential martingale, and then lifted by \zcref{lemma:contraction_time_marginal_minimum_entropy_problem_hellinger_thermodynamic_limit}. Recall that $Q^{\Lambda,\beta}$ denotes the dynamics in $\Lambda$ with birth rate $b(x,\eta_\Lambda)$, that is, with empty boundary configuration.

\begin{lemma}[Finite-volume lower bound]\label{lemma:contraction_time_marginal_minimum_entropy_problem_hellinger_finite_volume_lower_bound}
    Let \(\mu \in \PP(\Om)\), \(\Lambda \Subset \R^d\) and \(T > 0\), and let \(P_{[0, T]}\) be a probability measure on \(D([0, T]; \Omega_\Lambda)\) whose time marginals are all equal to \(\mu_\Lambda\).
    Then
    \begin{equation}\label{equation:contraction_time_marginal_minimum_entropy_problem_hellinger_finite_volume_lower_bound}
        \Ent(P_{[0, T]} \vert \mu_\Lambda Q^{\Lambda,\beta})
        \geq 2T \mathfrak{H}^2(B_\mu^{(\Lambda)}, D_\mu^{(\Lambda)}).
    \end{equation}
\end{lemma}

The upper bound is achieved by a rigorous version of the optimal reversible dynamics~\eqref{equation:contraction_time_marginal_entropy_for_B_D_process_optimal_rates}, constructed first in finite volume and then extended to infinite volume by a stationary block construction, as for $\bar P$ in \zcref{subsection:strategy-lower-bound}, together with a compactness argument.

\begin{lemma}[Infinite-volume attainment]\label{lemma:contraction_time_marginal_minimum_entropy_problem_hellinger_existence_of_minimizer_in_infinite_volume}
    For every \(\mu \in \PP(\Om)\) with $\rho(\mu)<\infty$, there is a \(P \in \PP(E)\) with \(P_0 = \mu\) such that
    \begin{equation}\label{equation:contraction_time_marginal_minimum_entropy_problem_hellinger_existence_of_minimizer_in_infinite_volume}
        H^\beta(P)
        = 2\mathfrak{H}^2(B_\mu, D_\mu).
    \end{equation}
\end{lemma}

Let us summarise these observations.
\begin{proposition}[Minimum entropy problem]\label{proposition:contraction_time_marginal_minimum_entropy_problem}
    Let \(\mu \in \PP(\Om)\) with \(\rho(\mu) < \infty\). Then the identity~\eqref{equation:contraction_time_marginal_minimum_entropy_problem} holds, and the minimum is attained. If \(\rho(\mu) = \infty\), then both sides of~\eqref{equation:contraction_time_marginal_minimum_entropy_problem} are infinite.
\end{proposition}
\begin{proof}
    If \(\rho(\mu) = \infty\), then every space-time stationary \(P\) with \(P_0 = \mu\) satisfies \(H(P) = \infty\), since \(H(P)<\infty\) would imply \(\rho(P)<\infty\) by \zcref{lem:finite-entropy-intensities}, and hence \(H^\beta(P) = \infty\). On the other hand, \(B_\mu[1] \leq \Vert b\Vert_\infty\), whereas \(D_\mu[1]=\rho(\mu) = \infty\), so that the Hellinger distance \(\mathfrak{H}^2(B_\mu, D_\mu)\) is infinite as well; by~\eqref{equation:contraction_time_marginal_hellinger_distance_identities}, the right-hand side of~\eqref{equation:contraction_time_marginal_minimum_entropy_problem} is infinite.

    Now let \(\rho(\mu) < \infty\), and let $P\in\PP(E)$ with $P_0=\mu$. If $H^\beta(P)=\infty$, there is nothing to prove. Otherwise $H(P)<\infty$, so that $P$-almost every path is a legal birth-and-death path by \zcref{lem:finite-entropy-intensities}, and $P_{\Lambda_n, [0,n]}$ is a law on $D([0,n];\Om_{\Lambda_n})$. By stationarity, all its time marginals are equal to $\mu_{\Lambda_n}$, so that \zcref{lemma:contraction_time_marginal_minimum_entropy_problem_hellinger_finite_volume_lower_bound} applies with $\Lambda=\Lambda_n$ and $T=n$. Dividing by $n\abs{\Lambda_n}$ and using the conditional-entropy representation~\eqref{equation:interacting_case_rate_function_conditional_entropy_representation} together with \zcref{lemma:contraction_time_marginal_minimum_entropy_problem_hellinger_thermodynamic_limit}, we obtain
    \begin{equation}\label{equation:contraction_time_marginal_minimum_entropy_problem_hellinger_infinite_volume_lower_bound}
        H^\beta(P) \geq 2\mathfrak{H}^2(B_\mu, D_\mu).
    \end{equation}
    On the other hand, \zcref{lemma:contraction_time_marginal_minimum_entropy_problem_hellinger_existence_of_minimizer_in_infinite_volume} shows that there exists a space-time stationary \(P\) with time marginal \(P_0 = \mu\) and
    \(H^\beta(P) = 2\mathfrak{H}^2(B_\mu, D_\mu)\).
    Together with~\eqref{equation:contraction_time_marginal_hellinger_distance_identities}, this proves~\eqref{equation:contraction_time_marginal_minimum_entropy_problem}.
\end{proof}

It remains to identify the zeros. The Hellinger distance in~\eqref{equation:contraction_time_marginal_minimum_entropy_problem} vanishes if and only if \(b(0, \cdot) \mu = \mu_0^!\). In view of the characterisation of the reduced Palm measure and translation invariance, this is the GNZ equation for $\mu$ with Papangelou intensity $b$, which characterises the Gibbs measures. Equivalently, $B_\mu=D_\mu$: under $\mu$, the dynamics creates a particle at $x$ in a configuration $\eta$ at the same rate as it removes this particle from $\eta+\delta_x$.

\begin{proof}[Proof of \zcref{thm:marg-ldp}]
    The map \(P \mapsto P_0\) is continuous from $\PP(E)$ to $\PP(\Om)$, both equipped with the weak topology, as explained at the beginning of this subsection, so \zcref{theorem:full-ldp} and the contraction principle give the large deviation principle for $((R_n)_0)_{n\in\N}$ with good rate function $\mu\mapsto\inf\{H^\beta(P)\colon P_0 = \mu\}$. By the goodness of $H^\beta$, the infimum is attained whenever it is finite, and \zcref{proposition:contraction_time_marginal_minimum_entropy_problem} identifies it with $2 \mathfrak{H}^2(b(0, \cdot) \mu, \mu_0^!)$ if $\rho(\mu)<\infty$, and with $\infty$ otherwise.

    The rate function $\mathfrak{I}^\beta$ is non-negative, and $\mathfrak{I}^\beta(\mu)=0$ holds if and only if \(\rho(\mu)<\infty\) and
\begin{equation}\label{equation:contraction_time_marginal_hellinger_distance_gibbs}
        b(0, \cdot) \mu
        = \mu_0^!.
    \end{equation}
    By~\cite[Lemma~3.4]{JKSZ25}, \eqref{equation:contraction_time_marginal_hellinger_distance_gibbs} holds for \(\mu \in \PP(\Om)\) if and only if \(\mu \in \GG^\beta_\theta\). Every $\mu\in\GG^\beta_\theta$ has finite intensity, since the GNZ equation with $g(x,\eta)=\1_{[0,1)^d}(x)$ gives $\rho(\mu)=\mu\big[\int_{[0,1)^d}\d x\, b(x,\cdot)\big]\le\Vert b\Vert_\infty$. Hence $\argmin \mathfrak{I}^\beta = \GG^\beta_\theta$. Since $\GG^\beta_\theta\neq\emptyset$, the minimum of $\mathfrak{I}^\beta$ is zero.
\end{proof}

The proofs of \zcref{lemma:contraction_time_marginal_hellinger_distance_identities,lemma:contraction_time_marginal_minimum_entropy_problem_hellinger_thermodynamic_limit,lemma:contraction_time_marginal_minimum_entropy_problem_hellinger_finite_volume_lower_bound,lemma:contraction_time_marginal_minimum_entropy_problem_hellinger_existence_of_minimizer_in_infinite_volume} are given in \zcref{section:contraction}.

\section{Properties of the dynamical specific entropy}\label{section:rate-function}
In this section we collect two properties of the dynamical specific entropy $H$ that are used throughout the proofs. First, the local entropies $h_{A,T}$ are superadditive in space and in time, which identifies $H$ as a limit along growing space-time boxes and shows that integer spatial and temporal scales suffice in its definition. Second, a law with finite dynamical specific entropy is concentrated on legal birth-and-death paths and has finite occupation, birth and death intensities.

Throughout the proofs, we use the variational representation of the relative entropy: for probability measures $\mathbb{P},\mathbb{Q}$ on a measurable space $(X, \mathcal{X})$,
\begin{equation}\label{def:entropy-variational-formulation}
    \Ent(\mathbb{P}\lvert \mathbb{Q})
    =
    \sup_{f \in B(\mathcal{X}),\, f \geq 0}\Big(\mathbb{P}[f]-\log\mathbb{Q}\big[\e^f\big]\Big),
\end{equation}
where $B(\mathcal{X})$ denotes the set of all bounded $\mathcal{X}$-measurable functions $f\colon X \to \R$. When $X$ is a metric space and $\mathcal{X}$ its Borel $\sigma$-algebra, the supremum can be restricted to bounded continuous functions.

\begin{lemma}[Superadditivity and the entropy limit]
\label{lem:entropy-superadditivity}
Let $P\in\PP(E)$.
For every open cube $A\subset\mathbb R^d$, every finite family of pairwise disjoint open cubes $A_1,\ldots,A_m\subset A$, and every $T>0$,
\begin{equation}\label{eq:entropy-spatial-superadditivity}
    h_{A,T}(P)\geq\sum_{i=1}^m h_{A_i,T}(P).
\end{equation}
Moreover, for every open cube $A$ and every $S,T>0$,
\begin{equation}\label{eq:entropy-temporal-superadditivity}
    h_{A,S+T}(P)\geq h_{A,S}(P)+h_{A,T}(P).
\end{equation}
Consequently, for every $\tau>0$,
\begin{equation}\label{eq:specific-entropy-limit}
    H(P)
    =\sup_{a,T\in\mathbb N_{\geq1}}
       \frac{h_{(0,a)^d,T}(P)}{a^dT} = \lim_{n\to\infty}
       \frac{h_{\Lambda_n^\circ,\tau n}(P)}{\tau n|\Lambda_n|},
\end{equation}
where $\Lambda_n^\circ=(-n/2,n/2)^d$.
Finally, if $H(P)<\infty$, then for all $a,T>0$,
\begin{equation*}
    h_{(0,a)^d,T}(P) = h_{[0,a)^d,T}(P) = h_{[0,a]^d,T}(P).
\end{equation*}
\end{lemma}

\begin{proof}
For the spatial inequality, write
\begin{equation*}
    X=(\omega_t|_A)_{0\leq t\leq T},
    \qquad X_i=(\omega_t|_{A_i})_{0\leq t\leq T},
    \qquad \eta=\omega_0|_A,
    \qquad \eta_i=\eta|_{A_i}.
\end{equation*}
Conditionally on the complete initial configuration $\eta$, the restrictions of the free evolution to the disjoint cubes $A_i$ are independent, with laws $Q^{A_i}_{\eta_i}$ on $[0,T]$. Disintegrating over $\eta$, restricting to the coordinates $(X_1,\dots,X_m)$, and using the superadditivity of the relative entropy with respect to a product reference measure, we obtain
\begin{equation}\label{eq:entropy-spatial-conditioning}
\begin{split}
    h_{A,T}(P)
    &= P\left[
       \Ent\bigl(P(X\in\cdot\lvert \eta)\lvert  Q^A_\eta\bigr)
       \right]
    \geq\sum_{i=1}^m P\left[
       \Ent\bigl(P(X_i\in\cdot\lvert \eta)
                  \lvert  Q^{A_i}_{\eta_i}\bigr)\right]\\
    &\geq\sum_{i=1}^m P \left[
       \Ent\bigl(P(X_i\in\cdot\lvert \eta_i)
                  \lvert  Q^{A_i}_{\eta_i}\bigr)\right]
     =\sum_{i=1}^m h_{A_i,T}(P).
\end{split}
\end{equation}
The superadditivity with respect to a product reference measure follows from the variational formula~\eqref{def:entropy-variational-formulation} by testing sums of functions of the separate coordinates. The second inequality follows from the convexity of the relative entropy: we average the conditional law given $\eta$ over $\eta$, conditionally on $\eta_i$, while the reference measure $Q^{A_i}_{\eta_i}$ remains fixed. In particular, neither independence nor finite entropy of the initial marginal is required.

For the temporal inequality, put
\begin{equation*}
    X=(\omega_t|_A)_{0\leq t\leq S},
    \qquad Y=(\omega_{S+t}|_A)_{0\leq t\leq T},
    \qquad \eta=\omega_S|_A.
\end{equation*}
The Markov property of the reference measure and the chain rule for the relative entropy yield
\begin{equation}\label{eq:entropy-temporal-conditioning}
\begin{split}
    h_{A,S+T}(P)
    &=h_{A,S}(P)
      + P \left[
       \Ent\bigl(P(Y\in\cdot\lvert  X)\lvert  Q^A_\eta\bigr)
       \right]\\
    &\geq h_{A,S}(P)
      + P \left[
       \Ent\bigl(P(Y\in\cdot\lvert \eta)\lvert  Q^A_\eta\bigr)
       \right]
    =h_{A,S}(P)+h_{A,T}(P),
\end{split}
\end{equation}
where $Q^A_\eta$ is observed during $[0,T]$. Here, convexity removes the conditioning on the first segment beyond its terminal configuration, and time stationarity gives the last equality. All terms are non-negative, so these arguments also apply when an entropy is infinite.

Fix $a,T,\tau>0$. The cube $\Lambda_n^\circ$ contains $\lfloor n/a\rfloor^d$ disjoint translates of $(0,a)^d$, and $[0,\tau n]$ contains $\lfloor\tau n/T\rfloor$ consecutive intervals of length $T$. Applying the two superadditivity inequalities, discarding the remaining space and time, and using spatial stationarity gives, for all sufficiently large $n$,
\begin{equation}\label{eq:entropy-tiling-bound}
    h_{\Lambda_n^\circ,\tau n}(P)
    \geq
    \left\lfloor\frac na\right\rfloor^d
    \left\lfloor\frac{\tau n}{T}\right\rfloor
    h_{(0,a)^d,T}(P).
\end{equation}
Thus,
\begin{equation*}
    \liminf_{n\to\infty}
    \frac{h_{\Lambda_n^\circ,\tau n}(P)}{\tau n|\Lambda_n|}
    \geq\frac{h_{(0,a)^d,T}(P)}{a^dT},
\end{equation*}
and taking the supremum over $a,T>0$ gives the lower bound $H(P)$. Conversely, spatial stationarity and the definition of $H$ give
\begin{equation*}
    \frac{h_{\Lambda_n^\circ,\tau n}(P)}{\tau n|\Lambda_n|}
    =\frac{h_{(0,n)^d,\tau n}(P)}{n^d(\tau n)}
    \leq H(P).
\end{equation*}
This proves the limit, including the case $H(P)=\infty$. Taking $\tau=1$ shows that integer sides and durations suffice: the supremum over integer parameters contains the diagonal sequence and is bounded from above by $H(P)$.

Finally, assume $H(P)<\infty$ and fix $a,T>0$, and let $A$ be one of the cubes $(0,a)^d$, $[0,a)^d$ and $[0,a]^d$. By \zcref{lem:finite-entropy-intensities} below,
\begin{equation*}
    P[\omega_0(\partial A)]=0,
    \qquad
    P[\mathsf B(\partial A,(0,T])]=0,
\end{equation*}
since $|\partial A|=0$. Since $P$-almost every path is a legal birth-and-death path, whose particles do not move, almost surely the boundary $\partial A$ carries no particles at any time in $[0,T]$. Likewise, under the reference measure $P_{A,0}Q^A$, the boundary carries no particles initially and no births occur there. Restriction to $(0,a)^d$ and extension by zero on the boundary therefore identify the laws $P_{A,T}$ for the three cubes, as well as the corresponding reference laws. The relative entropy is unchanged under this identification, which proves the final assertion.
\end{proof}

Next we show that a finite dynamical specific entropy controls the mean number of particles, births and deaths.
\begin{lemma}[Finite intensities]
\label{lem:finite-entropy-intensities}
Let $P\in\PP(E)$ satisfy $H(P)<\infty$.
Then $P$-almost every path is a legal birth-and-death path, the intensity $\rho(P)$ is finite, and there is a finite constant $\chi(P)\geq0$, the \emph{birth intensity} of $P$, such that, for every bounded Borel set $B\subset\mathbb R^d$, every $t\in\mathbb R$, and every bounded interval $I=(s,t]\subset\mathbb R$,
\begin{equation}\label{eq:finite-entropy-intensities}
    P[\omega_t(B)]=\rho(P)|B|,
    \qquad
    P[\mathsf B(B,I)]=P[\mathsf D(B,I)]=\chi(P)|B||I|.
\end{equation}
Moreover,
\begin{equation}\label{eq:entropy-intensity-bounds}
    H(P)\geq\bigl(\sqrt{\rho(P)}-1\bigr)^2,
    \qquad\chi(P)\leq H(P)+\e-1.
\end{equation}
In particular, only finitely many individual particles meet any bounded space-time box, almost surely.
\end{lemma}

\begin{proof}
Fix $A=(0,a)^d$ and $T>0$. Since
\begin{equation*}
    h_{A,T}(P)=\Ent(P_{A,T}\lvert  P_{A,0}Q^A)\leq |A|TH(P)<\infty,
\end{equation*}
we have $P_{A,T}\ll P_{A,0}Q^A$. Under $P_{A,0}Q^A$, the configuration in $A$ is constant between finitely many jumps, each consisting of the addition or removal of one particle, and the jump times are almost surely distinct. Absolute continuity, stationarity and a countable exhaustion by translated space-time boxes therefore show that $P$-almost every path is a legal birth-and-death path.

For $\gamma\in\mathbb R$ and $m\in\mathbb N$, set $\phi_m(k)=\gamma(k\wedge m)$. Under $P_{A,0}Q^A$, the particle-count process $(\omega_t(A))_{t\ge0}$ is a Markov chain with generator
\begin{equation*}
    (\mathcal L_Au)(k) = |A|\bigl(u(k+1)-u(k)\bigr) + k\bigl(u(k-1)-u(k)\bigr), \qquad k\in\N_0,
\end{equation*}
where the second term vanishes for $k=0$. Thus
\begin{equation*}
    G_m(k)
    :=\e^{-\phi_m(k)} \mathcal L_A(\e^{\phi_m})(k)=|A|(\e^\gamma-1)\1_{\{k<m\}} + k(\e^{-\gamma}-1)\1_{\{1\leq k\leq m\}}.
\end{equation*}
The positive local martingale
\begin{equation*}
    Z_{m,t} =\exp\left\{ \phi_m(\omega_t(A))-\phi_m(\omega_0(A)) - \int_0^t\d s\, G_m(\omega_s(A)) \right\}
\end{equation*}
is bounded on $[0,T]$, since $\phi_m$ and $G_m$ are bounded for fixed $m$ and $\gamma$. Hence $P_{A,0}Q^A[Z_{m,T}]=1$.
The variational formula~\eqref{def:entropy-variational-formulation} applied to the bounded function $\log Z_{m,T}$, together with stationarity, yields
\begin{equation}\label{eq:truncated-population-entropy}
    \frac{h_{A,T}(P)}{T}
    \geq \frac1T P[\log Z_{m,T}]
    =(1-\e^{-\gamma}) P\big[\omega_0(A)\1_{\{\omega_0(A)\leq m\}}\big] - |A|(\e^\gamma-1)P(\omega_0(A)<m),
\end{equation}
where the endpoint terms cancel because $\phi_m$ is bounded and $\omega_0(A)$ and $\omega_T(A)$ have the same law.

First choose $\gamma=\log2$ in~\eqref{eq:truncated-population-entropy}. Since $h_{A,T}(P)\leq |A|TH(P)$, we obtain
\begin{equation*}
    P\big[\omega_0(A)\1_{\{\omega_0(A)\leq m\}}\big]\leq2|A|(H(P)+1).
\end{equation*}
Letting $m\to\infty$ shows $P[\omega_0(A)]<\infty$, so that the intensity $\rho(P)=P[\omega_0(A)]/|A|$ is finite, and $P[\omega_t(B)]=\rho(P)|B|$ for every bounded Borel set $B$ and $t\in\R$ by stationarity.
Now let $m\to\infty$ in~\eqref{eq:truncated-population-entropy} for an arbitrary fixed $\gamma\in\mathbb R$. Dividing by $|A|$, we obtain
\begin{equation*}
    H(P)\geq 1+\rho(P)-\e^\gamma-\rho(P)\e^{-\gamma}.
\end{equation*}
For $\rho(P)>0$, the choice $\gamma=\frac12\log\rho(P)$ yields $H(P)\geq(\sqrt{\rho(P)}-1)^2$.
For $\rho(P)=0$, the same bound follows by sending $\gamma\to-\infty$.

Finally, put $B_{A,T}=\mathsf B_\omega(A,(0,T])$ and $D_{A,T}=\mathsf D_\omega(A,(0,T])$. Under $P_{A,0}Q^A$, the birth count $B_{A,T}$ is Poisson-distributed with mean $|A|T$, independently of the initial configuration. Applying the entropy inequality $P[f]\le\Ent(P_{A,T}\lvert P_{A,0}Q^A)+\log P_{A,0}Q^A[\e^{f}]$ to $f=B_{A,T}\wedge m$ and then using monotone convergence shows
\begin{equation*}
    P[B_{A,T}]\leq h_{A,T}(P)+|A|T(\e-1)\leq |A|T(H(P)+\e-1)<\infty.
\end{equation*}
From the pathwise identity $\omega_T(A)-\omega_0(A)=B_{A,T}-D_{A,T}$, we have $D_{A,T}\leq \omega_0(A)+B_{A,T}$, so that the death count is integrable, and stationarity gives $P[D_{A,T}]=P[B_{A,T}]$. The mean birth and death measures are therefore locally finite and invariant under space-time translations, and hence equal to a common constant multiple $\chi(P)$ of the space-time Lebesgue measure, which satisfies the asserted bound.

The number of individual particles present in $B$ at some time during $[s,t]$ is bounded by $\omega_s(B)+\mathsf B_\omega(B,(s,t])$, whose expectation is $|B|\bigl(\rho(P)+(t-s)\chi(P)\bigr)<\infty$. A countable exhaustion proves the final assertion simultaneously for all bounded space-time boxes.
\end{proof}

\section{Proof of the lower bound in the non-interacting case}\label{section:proof-lower-bound}
In this section we prove \zcref{lemma:block-law,lemma:auxiliary-entropy,lemma:G-typical-and-parameters}, which complete the proof of the lower bound in \zcref{thm:ldp-lower-bound-free}. Throughout, $P\in\PP(E)$ satisfies $H(P)<\infty$. By \zcref{lem:finite-entropy-intensities}, $P$-almost every path is a legal birth-and-death path, and $P$ has finite intensity $\rho(P)$ and finite birth intensity $\chi(P)$, so that, for every bounded Borel set $B\subset\mathbb R^d$ and every $u<v$,
\begin{equation*}
    P[\omega_u(B)]=\rho(P)|B|, \qquad P[\mathsf B(B,(u,v])]=\chi(P)|B|(v-u).
\end{equation*}
In particular, the number of individual particles present in $B$ at some time during $[u,v]$ is integrable, with expectation $|B|\bigl(\rho(P)+(v-u)\chi(P)\bigr)$.

\subsection{Truncation and smoothing of the initial configuration}\label{subsec:lb-truncation}
Fix the open cube $A=(0,a)^d$, let $P_{A,T}(\cdot\lvert \eta)$ be the conditional law of the path in $A\times[0,T]$ under $P$, given the initial configuration $\eta$ in $A$, and let $K\in\N$ be such that $P(\omega_0(A)\leq K) > 0$. We denote the conditioned initial and path laws by
\begin{equation*}
 	P_{A,0}^{(K)} = P_{A,0}(\,\cdot\lvert  \eta(A)\leq K),
 	\qquad
	P_{A,T}^{(K)} = P_{A,T}(\,\cdot\lvert  \omega_0(A)\leq K).
\end{equation*}
Conditioning on a small initial population costs at most a factor $1/P(\omega_0(A)\le K)$ in entropy, and it does not increase the mean population at later times by more than this factor.
\begin{lemma}[Truncation]
\label{lem:lb-truncation}
The conditioned laws satisfy
\begin{equation}\label{eq:lb-initial-entropy}
 	\Ent\bigl(P_{A,T}^{(K)}\lvert  P_{A,0}^{(K)}Q^A\bigr)
 	= \frac1{P(\omega_0(A)\leq K)}\int_{\{\eta(A)\leq K\}}\hspace{0cm}P_{A,0}(\d\eta) \Ent\bigl(P_{A,T}(\cdot\lvert \eta)\lvert  Q^A_\eta\bigr)\leq \frac{h_{A,T}(P)}{P(\omega_0(A)\leq K)}.
\end{equation}
Moreover,
\begin{equation}\label{eq:lb-terminal-mean}
	P[\omega_T(A)\lvert  \omega_0(A)\leq K] \leq \frac{\rho(P)|A|}{P(\omega_0(A)\leq K)}, \qquad T\geq 0.
\end{equation}
\end{lemma}

\begin{proof}
The chain rule for the relative entropy, applied to the disintegrations of $P^{(K)}_{A,T}$ and $P^{(K)}_{A,0}Q^A$ over the initial configuration, gives~\eqref{eq:lb-initial-entropy}, where the inequality follows by extending the integral to all $\eta$. The bound~\eqref{eq:lb-terminal-mean} follows from $P[\omega_T(A)\1_{\{\omega_0(A)\le K\}}]\le P[\omega_T(A)]=\rho(P)|A|$ by stationarity.
\end{proof}

Next, we smooth the particle positions on a fixed spatial scale. Choose $\epsilon\in(0,1]$ such that $a/\epsilon\in\N$, and partition the cube $A$ into the disjoint cells
\begin{equation*}
	C_j = A\cap\prod_{\ell=1}^d (j_\ell\epsilon,(j_\ell+1)\epsilon], \qquad j=(j_1,\ldots,j_d)\in\{0,\ldots,a/\epsilon-1\}^d,
\end{equation*}
of volume $\epsilon^d$. We now choose, independently for each initial or newly born individual particle, a uniform position within its original cell, so that each particle is displaced by at most $\epsilon$ in the sup-norm. For locations with multiplicities, we label the copies exchangeably and assign each death uniformly among the remaining labels at the original death location. The new position is retained throughout the lifetime of the particle, and all birth and death times remain unchanged.
We denote by $J_\epsilon$ the resulting kernel on paths, and by $J_{\epsilon,0}$ the corresponding kernel on initial configurations. Write
\begin{equation*}
 	P_{A,0}^{(K,\epsilon)} := P_{A,0}^{(K)}J_{\epsilon,0},
 	\qquad
 	P_{A,T}^{(K,\epsilon)} := P_{A,T}^{(K)}J_\epsilon,
\end{equation*}
and, for $\upsilon>0$, let $\Pi^\upsilon_A$ be the law of the Poisson point process on $A$ with intensity $\upsilon$. Smoothing maps the free evolution to itself, so it does not increase the entropy, and it makes the truncated initial law absolutely continuous with respect to a Poisson point process, with a bounded density.

\begin{lemma}[Smoothing]
\label{lem:lb-smoothing}
The smoothing kernel satisfies
\begin{equation}\label{eq:lb-reference-smoothing}
 	(P_{A,0}^{(K)}Q^A)J_\epsilon = (P_{A,0}^{(K)}J_{\epsilon,0})Q^A.
\end{equation}
Moreover,
\begin{equation}\label{eq:lb-smoothed-entropy}
 	\Ent \bigl(P_{A,T}^{(K,\epsilon)} \lvert  P_{A,0}^{(K,\epsilon)} Q^A\bigr) \leq \frac{h_{A,T}(P)}{P(\omega_0(A)\leq K)}.
\end{equation}
The smoothed paths have simple configurations at every time almost surely, and smoothing conserves $\omega_t(A)$ at every time $t$.
Finally, for every $\upsilon>0$,
\begin{equation*}
	\Big\| \frac{\d P_{A,0}^{(K,\epsilon)}}{\d\Pi^\upsilon_A} \Big\|_\infty \leq \e^{\upsilon|A|}K! \max\{1,(\upsilon \epsilon^d)^{-K}\},
\end{equation*}
and, in particular,
\begin{equation}\label{eq:lb-smoothed-initial-entropy}
 \Ent\bigl(P_{A,0}^{(K,\epsilon)} \lvert  \Pi^\upsilon_A\bigr) \leq \upsilon|A|+\log K! +K\log^+\!\frac1{\upsilon \epsilon^d} <\infty.
\end{equation}
\end{lemma}

\begin{proof}
The displacement kernel preserves the Lebesgue measure, and hence the marked birth process of the free evolution: births with their lifetimes still form a marked Poisson point process with intensity measure $\d x\,\d t\,\d r\, \e^{-r}$. Since the displacements are independent of the lifetimes, we obtain~\eqref{eq:lb-reference-smoothing}. The data-processing inequality applied to~\eqref{eq:lb-initial-entropy}, together with~\eqref{eq:lb-reference-smoothing}, yields~\eqref{eq:lb-smoothed-entropy}. The displaced positions are independent and uniformly distributed within their cells, so that the smoothed configurations are simple at every time almost surely, and smoothing does not change the number of particles in $A$.

Conditionally on their cell-count vectors, both $P_{A,0}^{(K,\epsilon)}$ and $\Pi^\upsilon_A$ have independent uniform positions within the cells. Let $p_{\mathbf m}$ be the probability under $P_{A,0}^{(K,\epsilon)}$ of a cell-count vector $\mathbf m=(m_j)_j$. For configurations with cell-count vector $\mathbf m$, the Radon--Nikodym derivative of $P_{A,0}^{(K,\epsilon)}$ with respect to $\Pi^\upsilon_A$ is
\begin{equation*}
	p_{\mathbf m}\,\e^{\upsilon|A|} \frac{\prod_jm_j!}{(\upsilon \epsilon^d)^{\sum_jm_j}}.
\end{equation*}
Since $p_{\mathbf m}\le1$ and $\sum_jm_j\leq K$, this yields the stated density bound, and~\eqref{eq:lb-smoothed-initial-entropy} follows, since the relative entropy is bounded by the logarithm of the supremum of the density.
\end{proof}

\subsection{Block paths with empty endpoints}\label{subsec:lb-cycles}
Recall that $S=T+2$. We construct paths in $D([0,S];\Om_A)$ that start and end with the empty configuration, using a preparation interval of length one, a bulk interval of length $T$, and a clearance interval of length one, see \zcref{fig:blockpath}.

Let $\upsilon=1-\e^{-1}$ and $c=-\log\upsilon$. Under $Q^A_\emptyset$, a particle born at time $s\in[0,1]$ is still alive at time $1$ with probability $\e^{-(1-s)}$, so that the configuration at time $1$ is a Poisson point process with intensity $\int_0^1\d s\, \e^{-(1-s)}=\upsilon$, that is, it has law $\Pi^\upsilon_A$. The preparation cost
\begin{equation}\label{eq:lb-preparation-entropy}
 	C_{\mathrm{prep}} := \Ent \bigl(P_{A,0}^{(K,\epsilon)} \lvert  \Pi^\upsilon_A\bigr)
\end{equation}
is finite by~\eqref{eq:lb-smoothed-initial-entropy}, and bounded as stated in \zcref{lemma:block-law}.

\begin{proof}[Proof of \zcref{lemma:block-law}]
The block law is the concatenation of three phases.

\emph{Preparation.} We start the free evolution in $A$ from the empty configuration and tilt its time-one marginal from $\Pi^\upsilon_A$ to $P_{A,0}^{(K,\epsilon)}$, leaving its conditional path law given that endpoint unchanged. That is, we define the preparation law $\mathcal C_{\mathrm{prep}}$ on $D([0,1];\Om_A)$ by
\begin{equation*}
 \frac{\d\mathcal C_{\mathrm{prep}}} {\d Q^A_{\emptyset}}(\omega) = \frac{\d P_{A,0}^{(K,\epsilon)}}{\d\Pi^\upsilon_A}(\omega_1),
\end{equation*}
where $Q^A_\emptyset$ is observed during $[0,1]$. By the chain rule for the relative entropy, the cost of this operation is $\Ent(\mathcal C_{\mathrm{prep}}\lvert Q^A_\emptyset)=C_{\mathrm{prep}}$, and at time $1$ there are at most $K$ particles.

\emph{Bulk.} Given the configuration at the end of the preparation phase, we sample the bulk path from the conditional law of $P_{A,T}^{(K,\epsilon)}$ given its initial configuration. By~\eqref{eq:lb-smoothed-entropy}, its mean entropic cost relative to the free evolution in $A$ started from the same configuration is at most $h_{A,T}(P)/P(\omega_0(A)\leq K)$.

\emph{Clearance.} Finally, to clear a configuration $\eta$, we condition the free evolution in $A$ started from $\eta$ during a unit time interval on having no births and on all particles of $\eta$ dying within that interval. This event has probability $\e^{-|A|}\upsilon^{\eta(A)}$, so that its conditional entropy is $|A|+c\,\eta(A)$. The configuration at the beginning of the clearance phase has the law of $\omega_T$ under $P^{(K,\epsilon)}_{A,T}$. Since smoothing conserves the number of particles in $A$, \eqref{eq:lb-terminal-mean} shows that the mean entropic cost of the clearance phase is at most $|A|+c\rho(P)|A|/P(\omega_0(A)\leq K)$, uniformly in $T$.

We concatenate these three paths over the intervals $[0,1]$, $[1,1+T]$ and $[1+T,S]$, respectively, and denote the law of the resulting path by $\mathcal C_{\mathrm{block}}$. The adjacent endpoints agree by construction, and the path starts and ends with the empty configuration, which proves~(i). The chain rule for the relative entropy and the Markov property of the free evolution give~\eqref{eq:lb-cycle-entropy}, and the bound on $C_{\rm prep}$ in~(iii) is~\eqref{eq:lb-smoothed-initial-entropy} with $\upsilon=1-\e^{-1}$. Matching empty endpoints make every finite concatenation of independent block paths c\`adl\`ag. Each preparation, bulk and clearance kernel is absolutely continuous with respect to the corresponding kernel of the free evolution, and the Markov property of the free evolution therefore gives absolute continuity for every finite concatenation, which proves~(ii).
\end{proof}

\subsection{The auxiliary laws in the observation window}\label{subsec:lb-finite-cylinder}
We now bound the entropic cost of the auxiliary laws $\widehat Q_n$ constructed in \zcref{subsection:strategy-lower-bound}, see \zcref{fig:window}.

\begin{proof}[Proof of \zcref{lemma:auxiliary-entropy}]
During the time interval $[0,1]$, the evolution in $\Lambda_n$ is conditioned on having no births and on all initial particles of $\xi$ in $\Lambda_n$ dying before time $1$. This event has probability $\e^{-|\Lambda_n|}\upsilon^{\xi(\Lambda_n)}$, so that the entropic cost of the conditioning is $|\Lambda_n|+c\, \xi(\Lambda_n)$. The tile faces remain empty after time one: they have zero Lebesgue measure, and the particle positions in the block paths have distributions that are absolutely continuous with respect to the Lebesgue measure. Each of the $\lfloor n/a\rfloor^d\ell_n$ block paths has entropic cost $\Ent(\mathcal C_{\mathrm{block}}\lvert Q^A_\emptyset)$ relative to the free evolution. Keeping the corridor $\Lambda_n\setminus V_n$, which is empty at time $1$, free of births during $[1,n]$ costs $(n-1)|\Lambda_n\setminus V_n|$, and suppressing births in $V_n$, which is empty at time $1+\ell_nS$, during the remaining time interval of length $r_n=n-1-\ell_nS\in[0,S)$ costs $r_n|V_n|$. Outside $\Lambda_n$ and after time $n$, the law $\widehat Q_n$ agrees with the free evolution. Each of these modifications is absolutely continuous with respect to the free evolution, so that $\widehat Q_n\ll Q_\xi$, and by the chain rule for the relative entropy and the Markov property,
\begin{equation*}
 	\Ent(\widehat Q_n\lvert Q_\xi)\le|\Lambda_n|+c\, \xi(\Lambda_n) + \lfloor n/a\rfloor^d\ell_n \Ent(\mathcal C_{\mathrm{block}}\lvert Q^A_\emptyset) + (n-1)|\Lambda_n\setminus V_n|+r_n|V_n|.
\end{equation*}
For fixed $a,K,\epsilon,T$, we have
\begin{equation*}
	\frac{\lfloor n/a\rfloor^d\ell_n}{n|\Lambda_n|}\longrightarrow\frac1{|A|S}, \qquad \frac{|\Lambda_n\setminus V_n|}{|\Lambda_n|}\longrightarrow0, \qquad \frac{\xi(\Lambda_n)}{n|\Lambda_n|}\longrightarrow0,
\end{equation*}
where the last limit holds since $\xi\in\Omst$, and $(|\Lambda_n|+r_n|V_n|)/(n|\Lambda_n|)\le(1+S)/n\to0$. Estimate ~\eqref{eq:lb-global-entropy} now follows from~\eqref{eq:lb-cycle-entropy} and $h_{A,T}(P)\leq |A|T H(P)$, which holds by the definition of $H$.
\end{proof}

\subsection{Approximation of \texorpdfstring{$P$}{P} by the stationary block law}\label{subsec:lb-approximation}
Recall the stationary block law $\bar P$ from~\eqref{eq:lb-stationary-cycle-law}. In order to show that $\bar P$ approximates $P$ in the weak topology, we control the conditioning and smoothing errors uniformly in the bulk duration $T$, so that we can choose $K$ and $\epsilon$ before $T$.

\begin{lemma}[Uniform conditioning estimate]
\label{lem:lb-conditioning}
For every bounded measurable function $f$ on $E$,
\begin{equation}\label{eq:lb-bounded-conditioning}
 	\sup_{z\in\R^d\times\R}\left| P[f\circ\theta_z\lvert  \omega_0(A)\leq K]-P[f] \right| \leq 2\|f\|_\infty P(\omega_0(A)>K).
\end{equation}
\end{lemma}

\begin{proof}
The total variation distance $\sup_B|P(B\lvert \omega_0(A)\leq K)-P(B)|$ between $P(\cdot\lvert  \omega_0(A)\leq K)$ and $P$ is $1-P(\omega_0(A)\leq K)$. Together with the stationarity of $P$, which gives $P[f\circ\theta_z]=P[f]$, this yields~\eqref{eq:lb-bounded-conditioning}.
\end{proof}

To compare paths, we use the metric $d_E$ from \zcref{section:appendix-topology}, and in particular the estimates in \zcref{lem:lb-local-metric}. Let $R,J,r$ be as there, that is, $R>0$ and $J\in\N$ are fixed, $r>1$ is such that the supports of the test functions $g_1,\ldots,g_J$ defining the metric $d_\Omega$ lie in $(-r+1,r-1)^d$, and $B_r=(-r,r)^d$ is the open box of side length $2r$ (not to be confused with the ball $B_r(x)$). Assume $a>2r$, $T>2R$ and $0<\epsilon\le1$, and abbreviate
\begin{equation*}
    C_{r,R}:=|B_r|(\rho(P)+2R\chi(P)).
\end{equation*} A space-time point $(x,s)\in A\times[0,S]$ is called a \emph{block-interior point} if
\begin{equation*}
 	x\in(r,a-r)^d, \qquad s\in[1+R,1+T-R].
\end{equation*}
Seen from a block-interior point, the observation window $B_r\times[-R,R]$ lies inside the bulk phase of a single block. The proportion of space-time points in $A\times[0,S]$ that are not block-interior is
\begin{equation}\label{eq:lb-boundary-phases}
 	b_{a,T} = 1 - \left(1-\frac{2r}{a}\right)^d \frac{T-2R}{T+2} \leq \frac{2dr}{a}+\frac{2(R+1)}{T+2}.
\end{equation}

\begin{lemma}[Approximation by the stationary block law]
\label{lem:lb-bl-approximation}
Under these assumptions, the stationary block law $\bar P$ satisfies
\begin{equation}\label{eq:lb-bl-approximation}
    \sup_{f\in\mathrm{BL}_1(E)} |\bar P[f]-P[f]| \leq \e^{-R}+2^{-J} + \frac{\epsilon |B_r|(\rho(P)+2R\chi(P))}{P(\omega_0(A)\leq K)} + 2P(\omega_0(A)>K)+2b_{a,T}.
\end{equation}
\end{lemma}
The first two terms come from the size of the observation window, the third from the smoothing, the fourth from the truncation, and the last from the space-time points that are not block-interior. Only the last term depends on $T$; by~\eqref{eq:lb-boundary-phases}, it is at most $4dr/a+4(R+1)/(T+2)$, which is small once $a$ and then $T$ are large.

\begin{proof}
Fix $f\in\mathrm{BL}_1(E)$. From the definition of $\bar P$,
\begin{equation*}
    \bar P[f]-P[f] = \frac1{|A|S}\int_A\d x\int_0^S \d s\, \left(\mathbf P[f(\theta_{(x,s)}\widetilde\omega)]-P[f]\right).
\end{equation*}
The points that are not block-interior occupy the proportion $b_{a,T}$ of $A\times[0,S]$, and the integrand at such points is bounded in absolute value by $2$, since $\|f\|_\infty\leq 1$.

Let $(x,s)$ be a block-interior point. We couple the infinite path $\widetilde\omega$ under $\mathbf P$ with a path $\omega^*$ sampled from $P(\cdot\lvert  \omega_0(A)\leq K)$, so that the bulk path $(\widetilde\omega_{1+t}|_A)_{0\le t\le T}$ of the block $A\times[0,S]$ is obtained from $(\omega^*_t|_A)_{0\le t\le T}$ using $J_\epsilon$. Since the preparation phase occupies the first unit time interval, time $s$ in the block path corresponds to time $s-1$ in the path $\omega^*$.

We estimate the $d_E$-distance between the shifted paths $\theta_{(x,s)}\widetilde\omega$ and $\theta_{(x,s-1)}\omega^*$ by comparing their restrictions to the space-time box $B_r\times[-R,R]$. By the interior conditions, $x+B_r\subset A$ and $[s-1-R,s-1+R]\subset[0,T]$. After these shifts, corresponding birth and death events in the observed bulk interval occur at the same times, while the corresponding spatial positions differ by at most $\epsilon$. The number of individual particles of $\omega^*$ that are present in $x+B_r$ at some time during $[s-1-R,s-1+R]$ is
\begin{equation*}
    L_{x,s}(\omega^*) = \omega^*_{s-1-R}(x+B_r) + \mathsf B_{\omega^*} \bigl(x+B_r,(s-1-R,s-1+R]\bigr).
\end{equation*}
Since the occupation and birth intensities $\rho(P)$ and $\chi(P)$ are finite, space-time stationarity gives $P[L_{x,s}]=C_{r,R}<\infty$, and, since $L_{x,s}\geq0$,
\begin{equation}\label{eq:lb-local-lifeline-mean}
    P[L_{x,s}\lvert  \omega_0(A)\leq K] = \frac{P[L_{x,s}\1_{\{\omega_0(A)\leq K\}}]}{P(\omega_0(A)\leq K)} \leq\frac{C_{r,R}}{P(\omega_0(A)\leq K)}.
\end{equation}
For $P(\omega_0(A)\leq K)\geq1/2$, this is bounded by $2C_{r,R}$, independently of $T$, $a$ and the block-interior point $(x,s)$.

By construction, the coupled shifted paths satisfy the matching condition of \zcref{lem:lb-local-metric}, with $\theta_{(x,s-1)}\omega^*$ as the original path and $L=L_{x,s}(\omega^*)$. Applying~\eqref{eq:lb-local-metric} and~\eqref{eq:lb-lifeline-metric}, we obtain, almost surely,
\begin{equation*}
    d_E\bigl(\theta_{(x,s)}\widetilde\omega, \theta_{(x,s-1)}\omega^*\bigr) \leq \e^{-R}+2^{-J}+\epsilon L_{x,s}(\omega^*).
\end{equation*}
Taking expectations under the coupling and using the Lipschitz property of $f$ and~\eqref{eq:lb-local-lifeline-mean} therefore gives
\begin{equation*}
    \big| \mathbf P[f(\theta_{(x,s)}\widetilde\omega)] - P[f\circ\theta_{(x,s-1)}\lvert  \omega_0(A)\leq K] \big| \leq \e^{-R}+2^{-J}+\frac{\epsilon C_{r,R}}{P(\omega_0(A)\leq K)}.
\end{equation*}
By the uniform conditioning estimate~\eqref{eq:lb-bounded-conditioning},
\begin{equation*}
 \big| P[f\circ\theta_{(x,s-1)}\lvert  \omega_0(A)\leq K]-P[f]\big| \leq 2P(\omega_0(A)>K).
\end{equation*}
Thus, at every block-interior point $(x,s)$,
\begin{equation*}
 \big|\mathbf P[f(\theta_{(x,s)}\widetilde\omega)]-P[f]\big| \leq \e^{-R}+2^{-J} + \frac{\epsilon C_{r,R}}{P(\omega_0(A)\leq K)}+2P(\omega_0(A)>K).
\end{equation*}
Together with the bound $2$ at the remaining points, this yields~\eqref{eq:lb-bl-approximation}, since the right-hand side does not depend on $f\in\mathrm{BL}_1(E)$.
\end{proof}

We can now choose the parameters. Since the conditioning and smoothing errors are controlled uniformly in $T$, we can choose $K$ and $\epsilon$ before taking $T$ sufficiently large.

\begin{proof}[Proof of \zcref{lemma:G-typical-and-parameters}(b)]
Let $\alpha>0$ and let $G$ be an open neighbourhood of $P$. By \zcref{lemma:open-neighbourhoods}, there are $f_1,\ldots,f_p\in\mathrm{BL}_1(E)$ and $\delta>0$ such that $U=U(P;f_1,\dots,f_p;\delta)\subset G$; we may assume $\delta\le1$. We show that the parameters can be chosen such that
\begin{equation}\label{eq:lb-strict-margins}
 	\big|\bar P[f_i]-P[f_i]\big|<\delta/2, \qquad 1\leq i\leq p,
\end{equation}
and such that the right-hand side of~\eqref{eq:lb-global-entropy} is at most $H(P)+\alpha$, which gives the entropy bound in~(b) by \zcref{lemma:auxiliary-entropy}; in particular, $\bar P\in U$.

Choose $R$ and $J$ such that $\e^{-R}\le\delta/16$ and $2^{-J}\le\delta/16$, then $r$ as in \zcref{lem:lb-local-metric}, and $a>2r$ such that $4dr/a\le\delta/16$. For the fixed cube $A=(0,a)^d$, we have $P(\omega_0(A)\leq K)\uparrow1$ as $K\to\infty$. We can therefore choose $K$ such that
\begin{equation*}
    P(\omega_0(A)\leq K)\geq\frac12,\qquad
    2P(\omega_0(A)> K)\leq\frac{\delta}{16},\qquad
    H(P)\Big(\frac1{P(\omega_0(A)\leq K)}-1\Big) \leq \frac{\alpha}{2}.
\end{equation*}
Next, choose $0<\epsilon\leq 1$ with $a/\epsilon\in\N$ and $\epsilon C_{r,R}/P(\omega_0(A)\leq K)\le\delta/16$.
This fixes a finite preparation cost $C_{\mathrm{prep}}$, which does not depend on $T$. Finally, choose $T>2R$ such that
\begin{equation*}
    \frac{4(R+1)}{T+2}\le\frac{\delta}{16},\qquad
    \frac{1}{T+2}\Big(\frac{C_{\mathrm{prep}}}{|A|}+1+\frac{c\rho(P)}{P(\omega_0(A)\leq K)}\Big)
    \leq\frac{\alpha}{2}.
\end{equation*}
By~\eqref{eq:lb-boundary-phases}, these choices give $2b_{a,T}\le 2\delta/16$, and therefore, by~\eqref{eq:lb-bl-approximation},
\begin{equation*}
    \sup_{f\in\mathrm{BL}_1(E)}
    |\bar P[f]-P[f]|
    \le\frac{6\delta}{16}<\frac{\delta}{2},
\end{equation*}
which implies~\eqref{eq:lb-strict-margins}. Moreover, the right-hand side of~\eqref{eq:lb-global-entropy} is at most $H(P)+\alpha/2+\alpha/2$.
\end{proof}

\subsection{A law of large numbers under the auxiliary laws}\label{subsec:lb-concentration}
For fixed parameters $a,K,\epsilon,T$, we now show that $R_n$ converges in $\widehat Q_n$-probability to $\bar P$ when tested against bounded Lipschitz functions. We approximate the block integrals of such functions by functions of finitely many neighbouring block paths, and we apply a law of large numbers for random fields with finite-range dependence.

Let
\begin{equation*}
	I_n=\{0,\ldots,\lfloor n/a\rfloor-1\}^d \times\{0,\ldots,\ell_n-1\},
\end{equation*}
and note that $|I_n||A|S/(n|\Lambda_n|)\to1$. Let $\widetilde\omega$ be the infinite path assembled from the independent block paths $(\zeta^{(i)})_{i\in\mathbb Z^{d+1}}$ under $\mathbf P$, as in the definition~\eqref{eq:lb-stationary-cycle-law} of $\bar P$. For $f\in\mathrm{BL}_1(E)$, define the block integrals
\begin{equation*}
	Y_i(f)=\int_A\d x\int_0^S\d s\, f(\theta_{(ak+x,\ell S+s)}\widetilde\omega), \qquad i=(k,\ell)\in\Z^{d+1}.
\end{equation*}

\begin{lemma}[Law of large numbers for block averages]
\label{lem:lb-cycle-field-lln}
For fixed block parameters $a,K,\epsilon,T$, and every $f\in\mathrm{BL}_1(E)$,
\begin{equation}\label{eq:lb-cycle-field-lln}
 	\frac1{|I_n||A|S}\sum_{i\in I_n}Y_i(f) \longrightarrow \bar P[f] \qquad\text{in }L^1(\mathbf P).
\end{equation}
\end{lemma}

\begin{proof}
The law $\mathbf P$ is invariant under index translations, and $\mathbf P[Y_i(f)]=|A|S\bar P[f]$ by~\eqref{eq:lb-stationary-cycle-law}. For $i\in\Z^{d+1}$ and $M\in\N$, we construct the path $\widetilde\omega^{i,M}$ by retaining the block paths $\zeta^{(j)}$ with $\|j-i\|_\infty\leq M$ and replacing every other block path by the empty path. Since every retained block path starts and ends with the empty configuration, this again gives a path in $E$. Let
\begin{equation*}
	Y_i^{[M]}(f) :=\int_A\d x\int_0^S\d s\, f(\theta_{(ak+x,\ell S+s)}\widetilde\omega^{i,M}),  \qquad i=(k,\ell)\in\Z^{d+1}.
\end{equation*}
Fix $\tau>0$. By~\eqref{eq:lb-cylinder-agreement}, there exists a space-time box
\begin{equation*}
	\mathcal W_\tau = B_{r_\tau}\times[-R_\tau,R_\tau]
\end{equation*}
such that any two paths that agree on this box are at $d_E$-distance at most $\tau$. These auxiliary radii do not change the block parameters. Choose $M=M(\tau)$ so large that every translate of $\mathcal W_\tau$ centred in the $i$-th block is contained in the blocks retained in $\widetilde\omega^{i,M}$, for example $Ma>r_\tau$ and $MS>R_\tau$. Then, for every space-time translation in the $i$-th block, the paths $\widetilde\omega$ and $\widetilde\omega^{i,M}$ agree on the corresponding translate of $\mathcal W_\tau$, and the Lipschitz property of $f$ yields
\begin{equation}\label{eq:lb-localized-block-error}
 	|Y_i(f)-Y_i^{[M]}(f)|\leq |A|S\tau
\end{equation}
for every realisation of $(\zeta^{(i)})_{i\in\Z^{d+1}}$.

Since $Y_i^{[M]}(f)$ depends only on the block paths with indices $j$ satisfying $\|j-i\|_\infty\leq M$, the random variables $Y_i^{[M]}(f)$ and $Y_j^{[M]}(f)$ are independent whenever $\|i-j\|_\infty>2M$. Furthermore, $|Y_i^{[M]}(f)|\leq |A|S$, and these random variables have the same mean. Thus, for every non-empty finite index set $I$,
\begin{equation}\label{eq:lb-finite-dependence-variance}
	\operatorname{Var}_{\mathbf P}\left(\frac1{|I||A|S}\sum_{i\in I}Y_i^{[M]}(f)\right) \leq \frac{(4M+1)^{d+1}}{|I|},
\end{equation}
since each index has at most $(4M+1)^{d+1}$ indices with which its covariance can be non-zero, and each covariance has absolute value at most $(|A|S)^2$. By~\eqref{eq:lb-localized-block-error},
\begin{equation*}
	\left|(|A|S)^{-1}\mathbf P[Y_0^{[M]}(f)] - \bar P[f]\right| \leq \tau,
\end{equation*}
which, together with~\eqref{eq:lb-localized-block-error} and~\eqref{eq:lb-finite-dependence-variance}, yields
\begin{equation*}
 \mathbf P\Big[\Big|\frac1{|I||A|S}\sum_{i\in I}Y_i(f)-\bar P[f] \Big|\Big] \leq 2\tau+\sqrt{\frac{(4M+1)^{d+1}}{|I|}}.
\end{equation*}
We apply this estimate with $I=I_n$. Letting first $n\to\infty$ and then $\tau\to 0$ yields~\eqref{eq:lb-cycle-field-lln}.
\end{proof}

It remains to compare the infinite block path used to define $\bar P$ with the periodised path that defines the empirical process under $\widehat Q_n$.

\begin{proof}[Proof of \zcref{lemma:G-typical-and-parameters}(a)]
Fix $f\in\mathrm{BL}_1(E)$ and $\tau>0$, and choose $\mathcal W_\tau$ as in the proof of \zcref{lem:lb-cycle-field-lln}. We couple the two constructions by using the same block paths: for each $i=(k,\ell)\in I_n$, the path $\zeta^{(i)}$ is placed in $(A+ak)\times[\ell S,(\ell+1)S]$ in the infinite construction, and in $(-\overline{n/2}+A+ak)\times[1+\ell S,1+(\ell+1)S]$ in the construction of $\widehat Q_n$. We complete the latter construction by sampling its remaining parts according to the conditional laws prescribed in the definition of $\widehat Q_n$, and denote the resulting path by $\widehat\omega$. Since $(\zeta^{(i)})_{i\in I_n}$ has the prescribed product law, $\widehat\omega$ has law $\widehat Q_n$, while the law of the infinite block family is unchanged. Let $\widehat\omega^{(n)}$ be the $n$-periodic version of $\widehat\omega$, so that $R_n=R_{n,\widehat\omega}$.

Under this coupling, $\widehat\omega^{(n)}$ agrees with $\theta_{(\overline{n/2},-1)}\widetilde\omega$ on
\begin{equation*}
    \Gamma_n=V_n\times[1,1+\ell_nS).
\end{equation*}
Let $\Gamma_{n,\tau}$ consist of the space-time points $(x,s)\in\Gamma_n$ whose translated observation window $(x,s)+\mathcal W_\tau$ lies strictly inside $\Gamma_n$. For fixed block parameters and fixed $\mathcal W_\tau$,
\begin{equation*}
    |\Delta_n\setminus\Gamma_n|=O(n^d), \qquad |\Gamma_n\setminus\Gamma_{n,\tau}|=O_\tau(n^d).
\end{equation*}
By the definition of $Y_i(f)$ and a change of variables, $\sum_{i\in I_n}Y_i(f) = \int_{\Gamma_n}\d x\, \d s\,f(\theta_{(x+\overline{n/2},s-1)}\widetilde\omega)$.
For every $(x,s)\in\Gamma_{n,\tau}$, the shifted paths $\theta_{(x,s)}\widehat\omega^{(n)}$ and $\theta_{(x+\overline{n/2},s-1)}\widetilde\omega$ agree on $\mathcal W_\tau$, so that their $d_E$-distance is at most $\tau$, and their $f$-values differ by at most $\tau$. On $\Gamma_n\setminus\Gamma_{n,\tau}$, the two integrands differ by at most $2$, whereas on $\Delta_n\setminus\Gamma_n$ only the integrand of the empirical process remains, with absolute value at most $1$. Splitting the integrals accordingly yields, almost surely,
\begin{equation}\label{eq:lb-finite-cylinder-comparison}
    \left| R_{n}[f]-\frac1{n|\Lambda_n|}\sum_{i\in I_n}Y_i(f) \right|
    \leq \tau+ \frac{ 2|\Gamma_n\setminus\Gamma_{n,\tau}| +|\Delta_n\setminus\Gamma_n| }{n|\Lambda_n|}=\tau+O_\tau(n^{-1}).
\end{equation}
Since $|I_n||A|S/(n|\Lambda_n|)\to1$ and $|Y_i(f)|\le|A|S$, the convergence~\eqref{eq:lb-cycle-field-lln} together with~\eqref{eq:lb-finite-cylinder-comparison} shows, by first letting $n\to\infty$ and then $\tau\to 0$, that $R_n[f]\to\bar P[f]$ in $\widehat Q_n$-probability.
\end{proof}

\section{Proof of the upper bound in the non-interacting case}\label{section:proof-upper-bound}
In this section we prove \zcref{prop:compact-upper-bound,prop:exponential-tightness}, which complete the proof of the upper bound in \zcref{thm:ldp-upper-bound}.

\subsection{The upper bound for compact sets}\label{subsection:compact-upper-bound}
We follow the general strategy of~\cite{DaiPra1993,varadhanLectureNotes}, but we have to be careful with some of the technical details. Fix $A=(0,a)^d$, $T>0$, and a bounded continuous function $f$ on $D([0,T];\Om_A)$. Define
\begin{equation}\label{eq:test}
g_f(\eta):=\log Q^A_\eta[\e^f],\qquad
\psi_f(\omega):=f(\omega|_{A\times[0,T]})-g_f(\omega_0|_A).
\end{equation}
Subtracting $g_f$ implies the following normalisation: for every initial configuration $\eta$,
\begin{equation}\label{eq:normalization}
Q^A_\eta\bigl[\e^{f-g_f(\eta)}\bigr]=1.
\end{equation}
We first show some technical properties of $g_f$ and $\psi_f$.

\begin{lemma}[Continuity properties]\label{lemma:continuity-properties}
The function $\eta \mapsto g_f(\eta)$ is bounded and continuous on $\Om_A$, and the map $P\mapsto P[\psi_f]$ is continuous on $\PP(E)$.
\end{lemma}
\begin{proof}
\emph{Step 1: Continuity of the free evolution in its initial configuration.}
Suppose that $\eta_i\to\eta$ vaguely in $\Om_A$. We couple the free evolutions in $A$ started from $\eta_i$ and from $\eta$. Choose a vague metric on $\Om_A$, which we also denote by $d_\Omega$, of the form
\begin{equation*}
    d_\Omega(\eta, \eta') = \sum_{k=1}^\infty 2^{-k}\left(1 \wedge \abs{\langle \eta, h_k\rangle - \langle \eta', h_k\rangle}\right),
\end{equation*}
analogous to the metric $d_\Omega$ on $\Om$ from \zcref{section:appendix-topology}, where $(h_k)_{k \in \N}$ is a countable family in $C_c(A)$ that determines vague convergence in $\Om_A$, and $\langle\eta,h\rangle=\int\eta(\d x)\,h(x)$.
Given $\varepsilon>0$, choose $N$ so large that $\sum_{k>N} 2^{-k}\leq \varepsilon$, and choose a compact set $K\subset A$ that contains the supports of $h_1, \dots, h_N$ in its interior and whose boundary carries no atom of $\eta$. Then, for $i$ sufficiently large, $\eta$ and $\eta_i$ put the same mass on $K$, and the finitely many atoms $x_1,x_2,\ldots$ of $\eta$ in $K$ can be paired, with multiplicity, with the atoms $x^{(i)}_1,x^{(i)}_2,\ldots$ of $\eta_i$ in $K$, in such a way that $x^{(i)}_\ell\to x_\ell$ as $i\to\infty$. To couple the two evolutions, we give the two particles of each pair the same exponential lifetime, and we use the same births with the same lifetimes in both processes. Under this coupling, the two paths differ only in the positions of the paired initial particles, which die at the same times, and in the initial particles outside $K$, which do not contribute to $h_1,\ldots,h_N$. Hence, for $k\le N$ and every time $t$,
\begin{equation*}
\big|\langle \omega^{(i)}_t, h_k\rangle - \langle \omega_t, h_k\rangle\big|\leq \sum_{\ell}\big|h_k(x^{(i)}_\ell)-h_k(x_\ell)\big|\longrightarrow0,
\end{equation*}
where $\omega^{(i)}$ and $\omega$ denote the coupled paths. For each $\varepsilon>0$, this gives a coupling whose path distance, with the identity time change, is eventually at most $2\varepsilon$. The Prokhorov distance between $Q^A_{\eta_i}$ and $Q^A_\eta$ therefore tends to zero, which implies the weak convergence $Q^A_{\eta_i}\Rightarrow Q^A_\eta$. Thus $g_f$ is continuous, and $\norm{g_f}_\infty\le\norm f_\infty$.

\emph{Step 2: Continuity on stationary laws.}
Restricting configurations to the open set $A$ is a vaguely continuous map, because a function in $C_c(A)$ can be extended to a function in $C_c(\mathbb{R}^d)$ by setting it equal to zero outside of $A$. The induced restriction of paths is continuous, and the further restriction to the time interval $[0,T]$ is continuous at every path that has no jump at the times $0$ and $T$. Every stationary law on $E$ assigns zero probability to a jump at a fixed deterministic time. Indeed, each c\`adl\`ag path has at most countably many discontinuities, so that, by Fubini's theorem, the probability of a jump at time $t$ vanishes for Lebesgue-almost every $t$, and by stationarity this probability does not depend on $t$. Therefore $\psi_f$, which is bounded by $2\norm f_\infty$, is continuous outside a null set for every stationary limiting law, and the version of the Portmanteau theorem for bounded, almost everywhere continuous functions gives the continuity of $P\mapsto P[\psi_f]$.
\end{proof}

\begin{lemma}[Variational representation]\label{lemma:nice-variational-representation}
    For every $P\in\PP(E)$,
        \begin{equation}\label{eq:variational}
        h_{A,T}(P)=\sup_{f\in C_b(D([0,T];\Om_A))}P[\psi_f].
    \end{equation}
    Consequently, $H$ is a supremum of weakly continuous affine functionals on $\PP(E)$, and in particular it is weakly lower semicontinuous.
\end{lemma}

\begin{proof}
Disintegration over the initial configuration gives
\begin{equation*}
h_{A,T}(P)
=\int P_{A,0}(\mathrm d\eta)\,\Ent\bigl(P_{A,T}(\cdot\lvert \eta)\lvert  Q^A_\eta\bigr).
\end{equation*}
The variational formula~\eqref{def:entropy-variational-formulation} applied for each $\eta$, together with~\eqref{eq:normalization}, shows $h_{A,T}(P)\ge P[\psi_f]$ for every $f$. Conversely, Jensen's inequality gives
\begin{equation*}
P_{A,0}[g_f]\le\log (P_{A,0}Q^A)[\e^f],
\end{equation*}
so that $P[\psi_f]=P_{A,T}[f]-P_{A,0}[g_f]\ge P_{A,T}[f]-\log (P_{A,0}Q^A)[\e^f]$. Taking the supremum over bounded continuous $f$, the right-hand side gives the variational formula for $\Ent(P_{A,T}\lvert  P_{A,0}Q^A)$ on the metric space $D([0,T];\Om_A)$. This proves~\eqref{eq:variational}, and the last assertion follows from \zcref{lemma:continuity-properties} and the definition~\eqref{eq:rate} of $H$.
\end{proof}

The decisive estimate is the following bound on exponential moments, which follows along the lines of~\cite[Lemmas~5.2 and~5.3]{DaiPra1993}, adapted to the continuum; for the comparison of the empirical field with spatial averages, which corresponds to Step~3 of the proof, see also~\cite[Remark~2.4]{GeorgiiZessin1993}.

\begin{lemma}[Exponential moments]\label{lemma:upper-bound-for-such-laplace-transforms}
    For every $\xi\in\Om$ and every function $\psi_f$ as in~\eqref{eq:test},
    \begin{equation}\label{eq:normalizedbound}
    \limsup_{n\to\infty}\frac{1}{n|\Lambda_n|}
    \log Q_{\xi}\left[\exp\left\{\frac{n|\Lambda_n|}{|A|T}R_n[\psi_f]\right\}\right]\le0.
\end{equation}
\end{lemma}

\begin{proof}
Recall that $A=(0,a)^d$, so that $|A|=a^d$, and fix $n>\max\{a,T\}$.

\emph{Step 1: An exact identity for each fixed lattice offset.}
Fix an offset $(y,w)\in[0,a)^d\times[0,T)$. The spatial cubes and time intervals of the corresponding block lattice have base points
\begin{equation*}
x_k :=y+ak,
\qquad
s_\ell :=w+\ell T,
\qquad
k\in\mathbb Z^d,\quad \ell\in\mathbb Z_{\ge0}.
\end{equation*}
We only keep the indices whose corresponding boxes $(A+x_k)\times[s_\ell,s_\ell+T]$ lie in $\Delta_n$, that is,
\begin{equation*}
I_n(y)
:=
\big\{
k\in\mathbb Z^d\colon
y+ak\in[-n/2,n/2-a)^d
\big\},
\qquad
J_n(w)
:=
\left\{
\ell\in\mathbb Z_{\ge0}\colon
w+\ell T<n-T
\right\}.
\end{equation*}
Every retained box lies inside $\Delta_n$, and in particular its terminal time is strictly less than $n$. Define the random sum
\begin{equation*}
\Phi_n(y,w)
:=
\sum_{\ell\in J_n(w)}
\sum_{k\in I_n(y)}
\psi_f(\theta_{(x_k,s_\ell)}\omega),
\end{equation*}
where an empty sum is understood to be zero. We claim that
\begin{equation}
\label{eq:upper-fixed-offset}
Q_{\xi}
\left[
\exp\{\Phi_n(y,w)\}
\right]
=1.
\end{equation}
To prove this, we group the factors according to their time indices,
\begin{equation*}
Z_\ell
:=
\exp\Big\{
\sum_{k\in I_n(y)}
\psi_f(\theta_{(x_k,s_\ell)}\omega)
\Big\},
\end{equation*}
and let $\mathcal F_t:=\sigma(\omega_r\colon 0\le r\le t)$ be the $\sigma$-algebra of the full history up to time $t$.
For each retained spatial cube, let $\eta_{k,\ell}\in\Om_A$ be its configuration at time $s_\ell$, translated back to $A$, that is,
\begin{equation*}
\eta_{k,\ell}(C)
:=
\omega_{s_\ell}(x_k+C),
\qquad C\subset A\ \text{measurable}.
\end{equation*}
Conditionally on $\mathcal F_{s_\ell}$, the evolutions in the disjoint cubes $A+x_k$ are independent, and their translated laws on the next time interval of length $T$ are $Q^A_{\eta_{k,\ell}}$. Indeed, births in disjoint cubes are independent, particles do not move, and the memoryless property of exponential lifetimes gives independent residual lifetimes to the particles present at time $s_\ell$. The normalisation~\eqref{eq:normalization} therefore yields
\begin{equation*}
Q_{\xi}
\left[
Z_\ell\lvert \mathcal F_{s_\ell}
\right]
=
\prod_{k\in I_n(y)}
Q^A_{\eta_{k,\ell}}
\left[
\e^{f-g_f(\eta_{k,\ell})}
\right]
=1.
\end{equation*}
For every earlier time index $r<\ell$, the terminal time satisfies $s_r+T\le s_\ell$, so that all earlier factors $Z_r$ are $\mathcal F_{s_\ell}$-measurable. Conditioning at the beginning of the last retained time interval thus removes its factor from the expectation, and repeating this backwards over the finitely many time intervals proves~\eqref{eq:upper-fixed-offset}. In particular, no independence between successive time intervals is required.

\emph{Step 2: Averaging over the lattice offset.}
Let $\pi$ be the uniform probability measure on $[0,a)^d \times [0,T)$, that is,
\begin{equation*}
\pi(\d y,\d w)
=
\frac{1}{|A|T}\,\dd y\dd w,
\qquad
(y,w)\in[0,a)^d\times[0,T).
\end{equation*}
Jensen's inequality, followed by~\eqref{eq:upper-fixed-offset}, gives
\begin{equation}
\label{eq:upper-offset-average}
Q_{\xi}
\Big[
\exp\Big\{
\int\pi(\d y,\d w)\Phi_n(y,w)
\Big\}
\Big]
\le
\int\pi(\d y,\d w)
Q_{\xi}
\Big[\e^{\Phi_n(y,w)}\Big]
=1.
\end{equation}
To identify the averaged sum, we introduce the region of retained base points
\begin{equation*}
\hat{\Delta}_n
:=
[-n/2,n/2-a)^d\times[0,n-T) \subset \Delta_n.
\end{equation*}
The translates $[0,a)^d+ak$, $k\in\mathbb Z^d$, partition $\mathbb R^d$, while the intervals $[0,T)+\ell T$, $\ell\in\mathbb Z_{\ge0}$, partition $[0,\infty)$. Consequently, the change of variables $x=y+ak$ and $s=w+\ell T$ gives
\begin{equation*}
\begin{split}
\int\pi(\d y,\d w)\Phi_n(y,w)
&=
\frac{1}{|A|T}
\int_{[0,a)^d}\d y\int_0^T\d w
\sum_{\ell\in J_n(w)}
\sum_{k\in I_n(y)}
\psi_f(\theta_{(y+ak,w+\ell T)}\omega)
\\
&=
\frac{1}{|A|T}
\int_{\hat{\Delta}_n}\d x\, \d s\,
\psi_f(\theta_{(x,s)}\omega).
\end{split}
\end{equation*}

\emph{Step 3: Comparison with the empirical process.}
By the definition of $R_n$,
\begin{equation*}
\frac{n|\Lambda_n|}{|A|T}R_n[\psi_f]
=
\frac{1}{|A|T}
\int_{\Delta_n}\d x\, \d s\,
\psi_f(\theta_{(x,s)}\omega^{(n)}).
\end{equation*}
Thus, to use the bound obtained in Steps~1 and~2, it suffices to bound the error induced by the periodisation. For $(x,s)\in\hat{\Delta}_n$, the whole observation box $(A+x)\times[s,s+T]$ lies within $\Delta_n$, and $s+T<n$. Therefore, its restriction is unaffected by the periodisation, and $\psi_f(\theta_{(x,s)}\omega^{(n)})=\psi_f(\theta_{(x,s)}\omega)$. On $\Delta_n \setminus \hat{\Delta}_n$, we use the bound $\|\psi_f\|_\infty\le2\|f\|_\infty$, so that
\begin{equation*}
\left|
\frac{n|\Lambda_n|}{|A|T}R_n[\psi_f]
-
\int\pi(\d y,\d w)\Phi_n(y,w)
\right|
\le
\frac{2\|f\|_\infty}{|A|T}
\,|\Delta_n\setminus\hat{\Delta}_n|\le
\frac{2(da+T)\|f\|_\infty}{|A|T}\,n^d.
\end{equation*}
Combining this deterministic error bound with~\eqref{eq:upper-offset-average}, we obtain
\begin{equation*}
\frac{1}{n|\Lambda_n|}
\log
Q_{\xi}\left[
\exp\left\{
\frac{n|\Lambda_n|}{|A|T}R_n[\psi_f]
\right\}\right]
\le
\frac{2(da+T)\|f\|_\infty}{|A|T\,n}.
\end{equation*}
Letting $n\to\infty$ proves the claim.
\end{proof}

We can now put these steps together.
\begin{proof}[Proof of~\zcref{prop:compact-upper-bound}]
    Let $F\subset \PP(E)$ be compact and fix an arbitrary $\alpha <\inf_F H$. It suffices to show that the $\limsup$ is at most $-\alpha$. For each $P \in F$ we have $H(P) > \alpha$, so by the definition~\eqref{eq:rate} of $H$ and the variational formula of \zcref{lemma:nice-variational-representation}, there are a cube $A=(0,a)^d$, a duration $T>0$ and a bounded continuous function $f$ on $D([0,T];\Om_A)$ with $P[\psi_f] / (\abs{A} T) > \alpha$. By \zcref{lemma:continuity-properties}, the set
    \begin{equation*}
        O_P := \big\{P' \in \PP(E)\colon P'[\psi_f] / (\abs{A} T) > \alpha\big\}
    \end{equation*}
    is an open neighbourhood of $P$, and by the exponential Markov inequality,
    \begin{equation*}
        Q_{\xi}(R_n \in O_P)
        \leq
        Q_{\xi}\Big(\frac{n|\Lambda_n|}{\abs{A}T}R_n[\psi_f] > \alpha n|\Lambda_n|\Big)
        \leq
        \e^{-\alpha n|\Lambda_n|}\, Q_{\xi}\left[\exp\left\{\frac{n|\Lambda_n|}{\abs{A}T}R_n[\psi_f]\right\}\right].
    \end{equation*}
    Combining this with \zcref{lemma:upper-bound-for-such-laplace-transforms} gives
    \begin{equation*}
        \limsup_{n \to \infty}\frac{1}{n|\Lambda_n|}\log Q_{ \xi}(R_n \in O_P) \leq -\alpha.
    \end{equation*}
    The sets $(O_P)_{P \in F}$ form an open cover of the compact set $F$, so we can extract finitely many $P_1, \dots, P_k \in F$ such that $O_{P_1},\dots, O_{P_k}$ already cover $F$, and a union bound gives
    \begin{equation*}
        \limsup_{n \to \infty}\frac{1}{n|\Lambda_n|}\log Q_{ \xi}(R_n \in F)
        \leq
        \max_{j=1,\dots,k} \limsup_{n \to \infty}\frac{1}{n|\Lambda_n|}\log Q_{ \xi}(R_n \in O_{P_j}) \leq -\alpha.
    \end{equation*}
    Since this holds for every $\alpha <\inf_F H$, the claim follows.
\end{proof}

\subsection{Compact sets of paths}\label{subsection:compact-paths}
To show exponential tightness, we construct compact subsets of $\PP(E)$ in which $R_n$ takes its values with high probability. By Prokhorov's theorem, this amounts to constructing compact subsets of $E$ on which $R_n$ puts most of its mass, so we first have to understand which subsets of $E$ are compact in the local $J_1$ topology. A bound on the number of particles alone does not give compactness. We also have to make sure that no two jumps occur arbitrarily close together in time, and, for technical reasons, that no jumps occur too close to the endpoints of an observation interval. As it turns out, these requirements suffice.

For a bounded open set $B\Subset \R^d$, a time interval $[u,v]$, $M\in\mathbb N$, and $\delta>0$, let $\Xi_{B,[u,v]}(M,\delta)$ be the set of all piecewise constant paths $\gamma \in D([u,v];\Om_B)$ with at most $M$ atoms in every state, that is, $\gamma_t(B) \leq M$ for all $t\in [u,v]$, with at most $2M$ jumps, with consecutive jump times separated by at least $\delta$, and with all jump times in $[u+\delta,v-\delta]$.

\begin{lemma}[Compact sets of local paths]\label{lemma:compact-sets-building-blocks}
    For every bounded open set $B\Subset \R^d$, time interval $[u,v]$, $M\in\mathbb N$, and $\delta>0$, the set $\Xi_{B,[u,v]}(M,\delta)$ is compact in the $J_1$ topology on $D([u,v];\Om_B)$.
\end{lemma}

Here it is important that we take multiplicities into account and do not work only with simple point configurations, for which the corresponding set would in general fail to be closed.

\begin{proof}
We show that $\Xi_{B,[u,v]}(M,\delta)$ is the image of a compact set under a continuous map. First, the set
\begin{equation*}
\Sigma_M(B)=\{\eta\in\Om_B\colon \eta(B)\le M\}
\end{equation*}
is the image, under the restriction to $B$, of the compact set of counting measures on $\overline B$ with at most $M$ atoms. Allowing multiplicities makes the latter set closed, and since $B$ is open, the restriction map is continuous; hence $\Sigma_M(B)$ is compact.

For a fixed number $0\leq q \leq2M$ of jumps, the set of admissible ordered jump times,
\begin{equation*}
    \Theta_q := \{(s_1, \dots, s_q) \in [u+\delta, v-\delta]^q \colon s_{i+1}\geq s_i + \delta,\ 1\leq i < q\},
\end{equation*}
is a compact subset of $\R^q$, where $\Theta_0$ consists of the empty tuple. Hence $\Theta_q \times \Sigma_M(B)^{q+1}$ is compact in the product topology. On this parameter space, define the reconstruction map $((s_i)_{i=1}^q, (\eta^{(i)})_{i=0}^{q}) \mapsto (\gamma_t)_{t\in [u,v]}$ by
\begin{equation*}
    \gamma_t := \eta^{(i)} \quad \text{for } t \in [s_i, s_{i+1}),\ 0 \leq i \le q, \qquad \gamma_v:=\eta^{(q)},
\end{equation*}
with the conventions $s_0=u$ and $s_{q+1}=v$. The reconstruction maps are continuous, and $\Xi_{B,[u,v]}(M,\delta)$ is the union of the images of the sets $\Theta_q \times \Sigma_M(B)^{q+1}$, $0\leq q \leq 2M$. This yields the claim.
\end{proof}

From these local sets we construct compact subsets of the full path space $E$. Recall the box $B_r=(-r,r)^d$ for $r>0$. For $j\ge1$, put
\begin{equation*}
\Xi_j(M,\delta):=\Xi_{B_{j+1},[-j-1,j+1]}(M,\delta),
\end{equation*}
and let $G_j(M,\delta)$ be the set of paths in $E$ whose restriction to the space-time box $B_{j+1}\times[-j-1,j+1]$ belongs to $\Xi_j(M,\delta)$. These sets are Borel measurable, since the restriction of paths is a Borel measurable map. We use their closures $\overline{G}_j(M,\delta):=\overline{G_j(M,\delta)}^{\,E}$, because closed constraints will be needed for the compactness of sets of laws.

\begin{lemma}[Two compactness facts]\label{lem:compactness}
The following assertions hold.
\begin{enumerate}[(i)]
\item If $\omega\in \overline{G}_j(M,\delta)$, then there is $\gamma\in\Xi_j(M,\delta)$ with
\begin{equation*}
\omega_t|_{B_{j+1}}=\gamma_t\qquad\text{for } -j-1<t<j+1.
\end{equation*}
\item Fix $M_j\in\N$ and $\delta_j>0$ for every $j$. If a sequence $(\omega_i)_{i\in\N}$ belongs to $\overline G_j(M_j,\delta_j)$ for all sufficiently large $i$, separately for each $j$, then it has a subsequence converging in $E$.
\end{enumerate}
\end{lemma}
Note that the identity in~(i) is asserted on the open time interval, not at its endpoints.

\begin{proof}
For~(i), choose $\omega_i\in G_j(M,\delta)$ converging to $\omega$. The compactness of $\Xi_j(M,\delta)$ gives a subsequential limit $\gamma$ of their local restrictions. At every common continuity time of $\omega$ and $\gamma$ in the open interval, convergence in the two path spaces identifies their values after restriction to $B_{j+1}$. These common continuity times are dense, and right continuity identifies the values at all times in the open interval.

For~(ii), the main idea is to choose representatives as provided by~(i), and to extract a diagonal subsequence that converges in $\Xi_j(M_j,\delta_j)$ for every $j$. We then check that these limits are consistent on their overlapping domains, so that they define a unique c\`adl\`ag path $\omega_\infty$, and finally that the subsequence converges to $\omega_\infty$ in $E$.

By~(i), for every $j$ and all sufficiently large $i$, there is $\gamma_i^{(j)}\in\Xi_j(M_j,\delta_j)$ such that
\begin{equation*}
\omega_i(t)|_{B_{j+1}}
=
\gamma_i^{(j)}(t),
\qquad -j-1<t<j+1.
\end{equation*}
The compactness of each $\Xi_j(M_j,\delta_j)$ and a diagonal extraction give a subsequence, still indexed by $i$, for which $\gamma_i^{(j)}\to\gamma^{(j)}$ in the $J_1$ topology on $D([-j-1,j+1];\Om_{B_{j+1}})$, for every $j$, with $\gamma^{(j)}\in\Xi_j(M_j,\delta_j)$.

\emph{Step 1: Compatibility and construction of the limit.}
Fix $k<j$. At every common continuity time $t\in(-k-1,k+1)$ of $\gamma^{(k)}$ and $\gamma^{(j)}$, the convergence of the evaluations gives
\begin{equation*}
\gamma^{(k)}(t)
=
\gamma^{(j)}(t)|_{B_{k+1}}.
\end{equation*}
Indeed, the corresponding identity holds for the approximating representatives, since both are restrictions of $\omega_i(t)$, and restriction to the open set $B_{k+1}$ is continuous for the vague topology, so that this identity passes to the limit. The common continuity times are dense, and right continuity extends the identity to every $t\in(-k-1,k+1)$.

Having this consistency, we can patch the local paths together. For each $t\in\mathbb R$, the compatible paths $(\gamma^{(j)})_{j \in \N}$ define a unique configuration $\omega_\infty(t)\in\Om$ with
\begin{equation*}
\omega_\infty(t)|_{B_{j+1}}
=
\gamma^{(j)}(t),
\qquad |t|<j+1.
\end{equation*}
The resulting configuration is locally finite, because every compact set $D \subset \R^d$ lies in $B_{j+1}$ for sufficiently large $j$.
It remains to check that $\omega_\infty$ is a c\`adl\`ag path in the vague topology. We use the test functions $(g_\ell)_{\ell\ge1}$ and the metric $d_\Omega$ from \zcref{section:appendix-topology}. For any finite collection $g_1,\ldots,g_K$, choose $j$ so large that their supports lie in $B_{j+1}$ and the time under consideration lies in $(-j-1,j+1)$. The right continuity of $\gamma^{(j)}$ gives the right continuity of all these pairings. The remaining terms in the metric contribute at most $2^{-K}$, so sending $K \to\infty$ proves the right continuity of $\omega_\infty$ with respect to $d_\Omega$. Likewise, the local left limits $\gamma^{(j)}(t-)$ are compatible under restriction and define a locally finite configuration on $\mathbb R^d$, and the same argument shows that $\omega_\infty(s)$ converges to this configuration as $s\uparrow t$. Thus $\omega_\infty\in E$.

\emph{Step 2: Convergence on finite time intervals.}
Fix $u>0$ such that $\omega_\infty$ is continuous at both $-u$ and $u$. For a given $K\in\mathbb N$, choose $j$ so large that
\begin{equation*}
[-u,u]\subset(-j-1,j+1),
\qquad
\bigcup_{\ell=1}^K\operatorname{supp}(g_\ell)
\subset B_{j+1}.
\end{equation*}
Since $\gamma^{(j)}$ agrees near these endpoints with the spatial restriction of $\omega_\infty$, it is continuous at $-u$ and $u$. Restricting the $J_1$ convergence to $[-u,u]$ therefore gives
\begin{equation*}
\gamma_i^{(j)}|_{[-u,u]}
\longrightarrow
\gamma^{(j)}|_{[-u,u]}
\quad\text{in }
D([-u,u];\Om_{B_{j+1}}).
\end{equation*}
Choose the corresponding Skorokhod time changes $\Psi_i\colon[-u,u]\to[-u,u]$, which are increasing homeomorphisms fixing the endpoints, whose distance to the identity tends to zero. All states of the paths under consideration belong to the compact set $\Sigma_{M_j}(B_{j+1})$ from the proof of \zcref{lemma:compact-sets-building-blocks}. On this set, each continuous map $\eta\mapsto\langle\eta,g_\ell\rangle$, $1\le\ell\le K$, is uniformly continuous. The uniform convergence of the time-changed paths therefore implies
\begin{equation*}
\max_{1\le\ell\le K}
\sup_{|t|\le u}
\left|
\left\langle
\omega_i(\Psi_i(t))-\omega_\infty(t),g_\ell
\right\rangle
\right|
\longrightarrow0.
\end{equation*}
Here the same time change controls all $K$ retained test functions, because their supports lie in the same spatial window. Using these time changes in the full configuration space, the definition of $d_\Omega$ gives
\begin{equation*}
\sup_{|t|\le u}
d_\Omega\bigl(
\omega_i(\Psi_i(t)),\omega_\infty(t)
\bigr)\le
\sum_{\ell=1}^K
2^{-\ell}
\Big(
1\wedge
\sup_{|t|\le u}
\left|
\left\langle
\omega_i(\Psi_i(t))-\omega_\infty(t),g_\ell
\right\rangle
\right|
\Big)
+
2^{-K}.
\end{equation*}
The finite sum tends to zero, while $2^{-K}$ bounds the remaining tail of the series defining $d_\Omega$, irrespective of the number of particles outside the chosen window. Consequently,
\begin{equation*}
\limsup_{i\to\infty}
d_{J_1,u}\bigl(
\omega_i|_{[-u,u]},
\omega_\infty|_{[-u,u]}
\bigr)
\le 2^{-K},
\end{equation*}
and since $K$ was arbitrary, the $\limsup$ vanishes.

\emph{Step 3: Convergence in $E$.}
A c\`adl\`ag path has at most countably many discontinuities, so that the preceding convergence holds for almost every $u>0$. The definition~\eqref{eqn:definition-j1-metric} of the metric $d_E$ and dominated convergence therefore give
\begin{equation*}
d_E(\omega_i,\omega_\infty)
=
\int_0^\infty\d u\,
\e^{-u}
\left(
1\wedge
d_{J_1,u}\bigl(
\omega_i|_{[-u,u]},
\omega_\infty|_{[-u,u]}
\bigr)
\right)
\to 0.
\end{equation*}
This proves~(ii).
\end{proof}

\subsection{Probabilistic control of local defects}\label{subsection:probabilistic-control}
To show exponential tightness, we have to show that only a small fraction of the space-time points in $\Delta_n$ exhibit one of the following three local defects in their observation window:
\begin{enumerate}[(1)]
    \item too many particles,
    \item two jumps that are too close in time,
    \item a jump that is too close in time to one of the endpoints of the observation interval.
\end{enumerate}
The required probabilistic estimates are the goal of this subsection. Throughout, $\xi\in\Omst$.

We start with two elementary bounds that serve as the backbone of the subsequent analysis. Define
\begin{equation*}
A_n:=\int_0^n\d t\, \omega_t(\Lambda_n),
\qquad B_n:=\mathsf{B}(\Lambda_n,(0,n]).
\end{equation*}
The first quantity is the total occupation time in the space-time box $\Delta_n$, and the second the total number of births in $\Delta_n$; they are not to be confused with the cube $A$ and the boxes $B_r$. We begin by estimating their exponential moments, which will later control the fraction of space-time points in $\Delta_n$ whose observation window has too many particles or jumps close to the endpoints.

\begin{lemma}[Occupation time and births]\label{lem:scalar}
For $0<\lambda<1$ and $\psi\ge0$,
\begin{equation}\label{eq:mgf}
    \limsup_{n\to\infty}\frac1{n|\Lambda_n|}
    \log Q_{\xi} \left[\e^{\lambda A_n+\psi B_n}\right]
    \le\frac{\e^\psi}{1-\lambda}-1.
\end{equation}
Consequently, for $R\ge0$ and $\sigma>0$,
\begin{equation}\label{eq:occupationtail}
\limsup_{n\to\infty}\frac1{n|\Lambda_n|}
\log Q_{\xi}(A_n+2RB_n>\sigma n|\Lambda_n|)
\le-\frac{\sigma}{2}+2\e^R-1.
\end{equation}
Moreover, for $\sigma\ge1$,
\begin{equation}\label{eq:birthtail}
\limsup_{n\to\infty}\frac1{n|\Lambda_n|}
\log Q_{\xi}(B_n>\sigma n|\Lambda_n|)
\le-\sigma\log\sigma+\sigma-1.
\end{equation}
\end{lemma}
\begin{proof}
\emph{Step 1: The joint exponential moment.}
Particles do not move, so that only the initial particles in $\Lambda_n$ and the births in $\Delta_n$ contribute to $A_n$. Each of the $\xi(\Lambda_n)$ initial particles has an independent exponential lifetime with mean one, births form a Poisson point process with intensity one on $\Delta_n$, and each birth has its own independent lifetime. For an $\operatorname{Exp}(1)$-distributed random variable $X$, put
\begin{equation*}
m_\lambda(t)=\mathbb{E}\Big[\e^{\lambda\min(X,t)}\Big].
\end{equation*}
An initial particle contributes $\min(X,n)$ to $A_n$, and a particle born at time $u$ contributes $\min(X,n-u)$ to $A_n$ and one to $B_n$. Thus the independence of the initial lifetimes and the exponential formula for marked Poisson point processes give
\begin{equation}\label{eq:exactmgf}
    Q_{\xi} \Big[\e^{\lambda A_n+\psi B_n}\Big]
    =m_\lambda(n)^{\xi(\Lambda_n)}
    \exp\Big\{n^d\int_0^n\d u\,
    \bigl(\e^\psi m_\lambda(n-u)-1\bigr)\Big\}.
\end{equation}
Since $0<\lambda<1$,
\begin{equation*}
m_\lambda(t)\le\mathbb{E}\big[\e^{\lambda X}\big]
=\int_0^\infty\d v\, \e^{-(1-\lambda)v}
=\frac1{1-\lambda}.
\end{equation*}
Consequently, for every $n$,
\begin{equation*}
\log Q_{\xi}\big[\e^{\lambda A_n+\psi B_n}\big]
\le  \xi(\Lambda_n)\log\frac1{1-\lambda}
+n|\Lambda_n|\left(\frac{\e^\psi}{1-\lambda}-1\right).
\end{equation*}
Dividing by $n|\Lambda_n|$ and using that $\xi(\Lambda_n)/(n|\Lambda_n|)\to0$, since $\xi\in\Omst$, we obtain~\eqref{eq:mgf}.

\emph{Step 2: Exponential Markov inequality for occupation and births.}
For $R\ge0$, the choice $\lambda=1/2$ and $\psi = R$ in the preceding moment estimate gives
\begin{equation*}
\begin{split}
Q_{\xi}(A_n+2RB_n>\sigma n|\Lambda_n|)
&\le \e^{-\sigma n|\Lambda_n|/2}
   Q_{\xi}\big[\e^{A_n/2+RB_n}\big]
\\
&\le
\exp\left\{\xi(\Lambda_n)\log2+
n|\Lambda_n|\left(-\frac{\sigma}{2}+2\e^R-1\right)\right\}.
\end{split}
\end{equation*}
Taking the logarithm and dividing by $n|\Lambda_n|$ proves~\eqref{eq:occupationtail}.

\emph{Step 3: The Poisson upper tail.}
The birth count is Poisson-distributed with mean $n|\Lambda_n|$, so that $Q_{\xi}[\e^{\psi B_n}]=\exp\{n|\Lambda_n|(\e^\psi-1)\}$. For any $\psi > 0$, the exponential Markov inequality therefore gives
\begin{equation*}
Q_{\xi}(B_n>\sigma n|\Lambda_n|)
\le\exp\{n|\Lambda_n|(\e^\psi-1-\psi\sigma)\}.
\end{equation*}
For $\sigma>1$, the function $\psi \mapsto \e^\psi-1-\psi\sigma$ is minimised at $\psi=\log\sigma > 0$, and substituting this value gives
\begin{equation*}
Q_{\xi}(B_n>\sigma n|\Lambda_n|)
\le\exp\{-n|\Lambda_n|(\sigma\log\sigma-\sigma+1)\}.
\end{equation*}
For $\sigma=1$, the claimed bound is the trivial bound by one. Taking logarithms and dividing by $n|\Lambda_n|$ proves~\eqref{eq:birthtail} in both cases.
\end{proof}

The second elementary estimate controls the number of points of a Poisson process that are very close to another point.

\begin{lemma}[Close points of a Poisson clock]\label{lem:poisson}
Let $\Pi^\lambda$ be a Poisson point process with intensity $\lambda>0$ on $[0,n]$, and let $Z_\delta$ be the number of its points that have another point within distance $\delta>0$.
Then, for every $\psi>0$, there is a finite constant $c= c_{\psi,\lambda}$ such that
\begin{equation}\label{eq:poissonpair}
\log \Pi^\lambda[\e^{\psi Z_\delta}]
\le c(n\delta+\delta^2),
\qquad 0<\delta\le(3\lambda)^{-1}.
\end{equation}
\end{lemma}

\begin{proof}
\emph{Step 1: Two or more points in a short interval.}
If $X$ is a Poisson random variable with mean $m\le1$, then
\begin{equation*}
\mathbb{E}\Big[\e^{\psi X\1_{\{X\ge2\}}}\Big]
=\e^{-m}\Big(1+m+\sum_{j\ge2}\frac{(\e^\psi m)^j}{j!}\Big).
\end{equation*}
The elementary inequalities $\e^{-m}(1+m)\le1$ and $\sum_{j\ge2}x^j/j!\le x^2\e^x/2$ give
\begin{equation*}
\mathbb{E}\Big[\e^{\psi X\1_{\{X\ge2\}}}\Big]
\le1+c_\psi m^2,
\qquad c_\psi = \tfrac12\e^{2\psi+\e^\psi},
\end{equation*}
and hence $\log\mathbb{E}[\e^{\psi X\1_{\{X\ge2\}}}]\le c_\psi m^2$.

\emph{Step 2: Covering all close pairs of points.}
Partition the interval $[0,n]$ into $\lceil n/\delta\rceil$ successive subintervals $I_\ell$ of length $\delta$, except possibly the last. Let $\widetilde I_\ell$ be the union of $I_\ell$ and its nearest neighbours $I_{\ell-1}$ and $I_{\ell+1}$, where available, and define
\begin{equation*}
Y_\ell:=\Pi^\lambda(\widetilde I_\ell)
\1_{\{\Pi^\lambda(\widetilde I_\ell)\ge2\}}.
\end{equation*}
A point in $I_\ell$ with a partner within distance $\delta$ forces at least two points in $\widetilde I_\ell$. Thus $Z_\delta\le\sum_\ell Y_\ell$ almost surely.

\emph{Step 3: Separating the dependent intervals.}
If $\ell_1 \neq \ell_2$ but $\ell_1 \equiv \ell_2 \bmod 3$, then the intervals $\widetilde{I}_{\ell_1}$ and $\widetilde{I}_{\ell_2}$ are disjoint up to endpoints, so that their Poisson counts are independent. H\"older's inequality across these three residue classes, followed by independence, gives
\begin{equation*}
\log {\Pi^\lambda}\left[\e^{\psi Z_\delta}\right]
\le\frac13\sum_\ell\log {\Pi^\lambda}\left[\e^{3\psi Y_\ell}\right].
\end{equation*}
Each count has mean at most $3\lambda\delta\le1$, so that Step~1 implies
\begin{equation*}
\log {\Pi^\lambda}\left[\e^{\psi Z_\delta}\right]
\le\frac13\left(\frac n\delta+1\right)
c_{3\psi}(3\lambda\delta)^2.
\end{equation*}
This proves~\eqref{eq:poissonpair}, for example with $c_{\psi,\lambda}=3c_{3\psi}\lambda^2$.
\end{proof}

For controlling the mass that the empirical process $R_n$ puts on the complement of the sets $\overline{G}_j(M,\delta)$, the key estimate will be
\begin{equation*}
R_n(G_j(M,\delta)^c)
\le\frac{2(j+1)}n+\Upsilon_n^{\mathrm{cnt}}(M)
+\Upsilon_n^{\mathrm{gap}}(M,\delta)
+\Upsilon_n^{\mathrm{end}}(\delta),
\end{equation*}
see the proof of \zcref{lem:single} below. Here, the first term is the normalised space-time volume of the seams, that is, of the space-time points near the temporal boundary of $\Delta_n$, which we discard. The quantity $\Upsilon_n^{\mathrm{cnt}}$ is the fraction of space-time points in $\Delta_n$ whose observation window contains too many particles, $\Upsilon_n^{\mathrm{gap}}$ the fraction whose window contains two jumps that are too close in time, and $\Upsilon_n^{\mathrm{end}}$ the fraction whose window contains a jump too close to one of its endpoints. These quantities are defined precisely below. We bound them separately, but we have to do this in the right order, since the choices of parameters depend on each other.

Fix $r,R>0$, recall the box $B_r=(-r,r)^d$, and consider $n>2R$, sufficiently large so that no translate of $B_r$ overlaps itself on the spatial torus $\mathbb R^d/n\mathbb Z^d$. The \emph{non-seam} space-time points are the points $(x,s)\in\Lambda_n\times(R,n-R)$. The observation windows $(x+B_r)\times[s-R,s+R]$ around such points do not touch the temporal boundary of $\Delta_n$. The remaining \emph{seam} points have normalised volume $n^{-(d+1)} \cdot 2R n^d = 2R/n$ and are treated separately. The space-time points $(x,s)$ with $s\in \{R,n-R\}$ have zero Lebesgue measure, so the space-time averages below may equivalently be written over the closed interval $[R,n-R]$. All spatial translates below are interpreted modulo $n$, that is, we consider periodic boundary conditions in space.

If the window around a non-seam point $(x,s)$ contains $L$ individual particles at some time, counting the particles present initially and those born later, then at every fixed time within the window the population does not exceed $L$, and the window contains at most $2L$ jumps after its initial time. Indeed, every particle present during the window is either present initially or born later, and each such particle can die at most once.

We now partition the spatial torus $\Lambda_n$ into cells of unit volume,
\begin{equation*}
u_{n,k}:=-\overline{n/2}+k+[0,1)^d,
\qquad k\in\{0,\ldots,n-1\}^d.
\end{equation*}
Let $U_{n,k}$ be the union of the cells that meet $B_r + u_{n,k}$, with indices read modulo $n$, see \zcref{fig:unk_boxes}. For fixed $r$ and all sufficiently large $n$, each $U_{n,k}$ consists of the same number $q_r$ of cells, and every cell belongs to exactly $q_r$ of these unions. In particular, $B_r+x\subset U_{n,k}$ whenever $x\in u_{n,k}$. For $R\le s\le n-R$, define
\begin{equation*}
L_{k,s}
:=\omega_{s-R}(U_{n,k})
+\mathsf{B}(U_{n,k},(s-R,s+R]).
\end{equation*}
Thus $L_{k,s}$ counts all particles that are present in the enlarged window $U_{n,k}$, either already at the initial time or by being born during the observation interval. In particular, a bound on $L_{k,s}$ controls the smaller observation window $B_r + x$ for every spatial point $x$ in the cell $u_{n,k}$.

\begin{figure}
    \centering
    \includegraphics[width=0.5\linewidth]{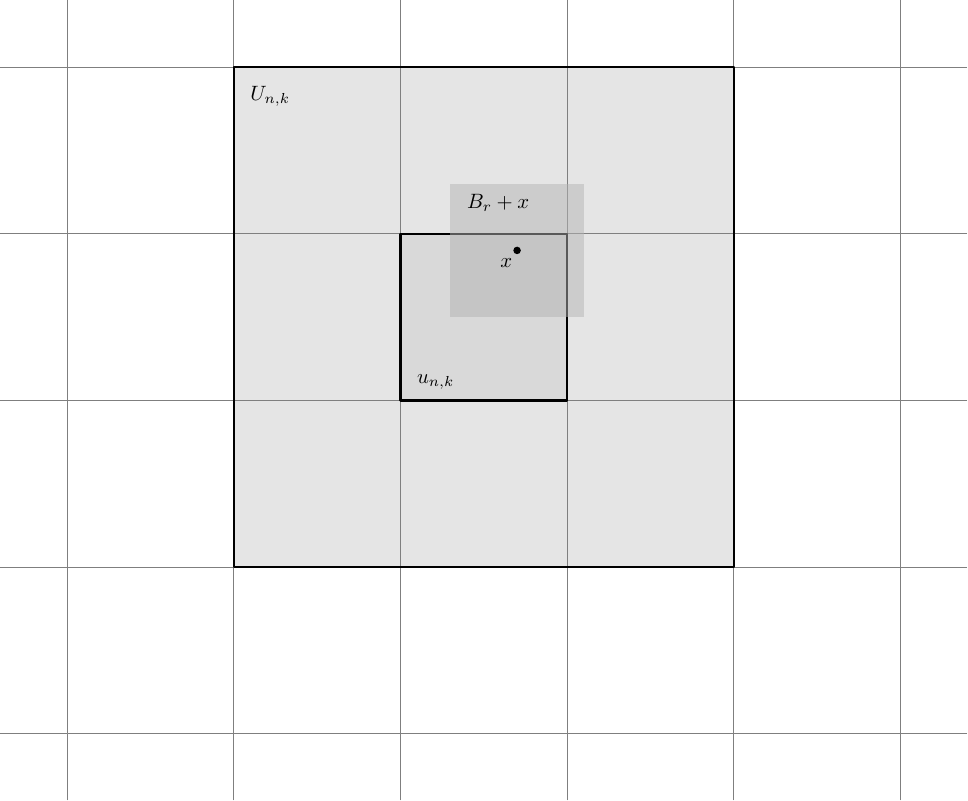}
    \caption{Two-dimensional illustration of the enlargement $U_{n,k}$ of the cell $u_{n,k}$, the union of the cells meeting $B_r+u_{n,k}$. For every $x\in u_{n,k}$, the observation window $B_r+x$ lies in $U_{n,k}$.}
    \label{fig:unk_boxes}
\end{figure}

As a first step, we bound the space-time averages of $L_{k,s}$ in terms of the quantities $A_n$ and $B_n$, which we already know how to control.
\begin{lemma}[Space-time count]\label{lem:space-time-count}
For fixed $r,R>0$ and all sufficiently large $n$,
\begin{equation}\label{eq:space-time-count}
\sum_k\int_R^{n-R}\d s\, L_{k,s}
\le q_r(A_n+2RB_n),
\end{equation}
where the sum is over $k\in\{0,\ldots,n-1\}^d$.
\end{lemma}
\begin{proof}
We treat the two terms in the definition of $L_{k,s}$ separately. At every fixed time $t$, each particle belongs to exactly $q_r$ of the unions $U_{n,k}$. Hence
\begin{equation*}
\sum_k\int_R^{n-R}\d s\, \omega_{s-R}(U_{n,k})
=q_r\int_0^{n-2R}\d t\, \omega_t(\Lambda_n)\le q_r A_n.
\end{equation*}
For the birth term, fix a birth at spatial position $y$ and time $u\in(0,n]$. Its cell belongs to exactly $q_r$ of the unions $U_{n,k}$, and for each such union, the birth is counted in $L_{k,s}$ only if $u\in(s-R,s+R]$, a condition that restricts $s$ to a set of Lebesgue measure at most $2R$. Summing this contribution over all births yields
\begin{equation*}
\sum_k\int_R^{n-R}\d s\,
\mathsf{B}(U_{n,k},(s-R,s+R])
\le 2Rq_rB_n.
\end{equation*}
Adding the two bounds proves~\eqref{eq:space-time-count}.
\end{proof}

We now define the first fraction of bad space-time points, namely of those where $L_{k,s}$ exceeds a threshold $M>0$,
\begin{equation}\label{eq:countbad}
\Upsilon_n^{\mathrm{cnt}}(M):
=\frac1{n|\Lambda_n|}\sum_k
 \int_R^{n-R}\d s\, \1_{\{L_{k,s}>M\}},
\end{equation}
where the sum is over $k\in\{0,\dots, n-1\}^d$. In words, $\Upsilon_n^{\mathrm{cnt}}(M)$ is the normalised volume of the non-seam space-time points whose \emph{enlarged} window has more than $M$ particles. We use the enlarged count here, because the same cutoff is required for the control of jumps that are too close to each other, see below. The dependence on the fixed radii $r,R>0$ is suppressed in the notation.

\begin{lemma}[Exponential estimate for $\Upsilon^{\mathrm{cnt}}$]\label{lem:counttail}
Fix $r,R>0$. For every $\varepsilon>0$ and every $M>0$,
\begin{equation}\label{eq:counttail}
\limsup_{n\to\infty}\frac1{n|\Lambda_n|}
\log Q_{\xi}\bigl(\Upsilon_n^{\mathrm{cnt}}(M)>\varepsilon\bigr)
\le-\frac{\varepsilon M}{2q_r}+2\e^R-1
\xrightarrow[M\to\infty]{}-\infty.
\end{equation}
In particular, given $\alpha>0$, one can choose $M\in\mathbb N$ such that this $\limsup$ is strictly smaller than $-\alpha$.
\end{lemma}
\begin{proof}
Since $\1_{\{L_{k,s}>M\}}\le L_{k,s}/M$, \zcref{lem:space-time-count} gives the pathwise estimate
\begin{equation*}
\Upsilon_n^{\mathrm{cnt}}(M)
\le\frac{q_r}{Mn|\Lambda_n|}(A_n+2RB_n).
\end{equation*}
Therefore
\begin{equation*}
\{\Upsilon_n^{\mathrm{cnt}}(M)>\varepsilon\}
\subset
\Big\{A_n+2RB_n>
\frac{\varepsilon M}{q_r}n|\Lambda_n|\Big\},
\end{equation*}
and~\eqref{eq:occupationtail} with $\sigma=\varepsilon M/q_r$ gives~\eqref{eq:counttail}. Finally, any integer $M>2q_r(\alpha+2\e^R-1)/\varepsilon$ makes the limiting upper bound strictly smaller than $-\alpha$.
\end{proof}

From now on, we fix $M\in\mathbb N$ and only consider windows with $L_{k,s}\le M$. Their populations stay below $M$ throughout the observed time interval. This permits a comparison with Poisson clocks whose rates do not depend on the random population, which allows us to use \zcref{lem:poisson}.

For a non-seam space-time point $(x,s)\in \Delta_n$, let
\begin{equation*}
\mathcal J_n(x,s)
=\{t\in(s-R,s+R]\colon
\text{a genuine birth or death occurs in }B_r+x
\text{ at time }t\}.
\end{equation*}
This is a finite set, and it contains only genuine jumps, not the resets that occur at the end of the time interval $[0,n]$. We define
\begin{equation*}
\Upsilon_n^{\mathrm{gap}}(M,\delta)
:=\frac1{n|\Lambda_n|}\sum_k
\int_R^{n-R}\d s\int_{u_{n,k}}\d x\,
\1_{\{L_{k,s}\le M\}}
\1_{\{\exists\,t,t'\in\mathcal J_n(x,s)\colon
0<|t-t'|\le\delta\}}.
\end{equation*}
In words, $\Upsilon_n^{\mathrm{gap}}(M,\delta)$ is the fraction of non-seam space-time points for which the enlarged count is at most $M$, but the smaller observation window contains two distinct jumps separated by at most $\delta$. The space-time points with $L_{k,s}>M$ are already accounted for by $\Upsilon_n^{\mathrm{cnt}}(M)$.

To control $\Upsilon_n^{\mathrm{gap}}$, we compare the jumps in windows of bounded population with Poisson clocks. For each unit cell $u=u_{n,k}$, take an independent Poisson random measure on $(0,\infty)\times(0,\infty)\times[0,1]$ with intensity $\mathrm{d}t\,\mathrm{d}y\,\mathrm{d}m$. At a point $(t,y,m)$, we accept a death if $y\le \omega_{t-}(u)$, and we use $m$ to choose one of the present particles uniformly. The total death rate is $\omega_{t-}(u)$, and every present particle has death rate one, so that this realises the deaths of the free evolution. Restricting this random measure to $y\le M$ and projecting its points onto the time axis gives a Poisson clock with rate $M$. We superpose it with the independent birth clock of the cell, which has rate one. The resulting clock has rate $1+M$, and it contains every genuine jump in the cell whenever the population of the cell immediately before that jump is at most $M$. This construction is made without conditioning on an event of low population: the points with $y>M$ supply any further deaths when the population is larger. For $U_{n,k}$, we superpose the clocks of its $q_r$ cells and call the resulting clock $\Pi_{n,k}$; it has rate $q_r(1+M)$. In a window with $L_{k,s}\le M$, all genuine jumps in $U_{n,k}$ after the initial time of the window are points of $\Pi_{n,k}$. Finally, let $Z_{k,\delta}$ be the number of points of $\Pi_{n,k}$ in $[0,n]$ that have another point of $\Pi_{n,k}$ within distance $\delta$, as in \zcref{lem:poisson}. The clocks are realised on an enlarged probability space, and we keep writing $Q_\xi$ for the joint law of the path and the clocks; this does not change the probabilities of events that only involve the path.

\begin{lemma}[Exponential estimate for $\Upsilon^{\mathrm{gap}}$]\label{lem:gaptail}
Fix $r,R>0$ and $M\in\mathbb N$. For every $\delta>0$ and all sufficiently large $n$,
\begin{equation}\label{eq:gapfraction}
\Upsilon_n^{\mathrm{gap}}(M,\delta)
\le\frac{2R}{n|\Lambda_n|}\sum_k Z_{k,\delta}.
\end{equation}
Moreover, for every $\varepsilon>0$,
\begin{equation}\label{eq:gaptail}
\lim_{\delta\downarrow0}\limsup_{n\to\infty}\frac1{n|\Lambda_n|}
\log Q_{\xi}
\bigl(\Upsilon_n^{\mathrm{gap}}(M,\delta)>\varepsilon\bigr)
=-\infty.
\end{equation}
\end{lemma}

\begin{proof}
\emph{Step 1: The total number of close pairs of clock points.}
Clocks associated with disjoint unions of cells are independent, but a single cell belongs to several unions, see \zcref{fig:unk_boxes}, so we do not have full independence. Consider the graph whose vertices are the unions of cells, with an edge between $U_{n,k}$ and $U_{n,k'}$ whenever they share a cell. For fixed $r$, this graph has uniformly bounded degree, so that we can choose a proper colouring of it, in which no two neighbours share a colour, with at most $p_r$ colours, independently of $n$. Within each colour class $\mathcal I_1,\ldots,\mathcal I_{p_r}$, the Poisson clocks are independent.

For $b>0$ and $0<\delta\le[3q_r(1+M)]^{-1}$, H\"older's inequality across the $p_r$ colour classes, followed by the independence within each class, gives
\begin{equation*}
\log Q_{\xi}\Big[
 \exp\Big\{b\sum_k Z_{k,\delta}\Big\}\Big]
\le\frac1{p_r}\sum_{a=1}^{p_r}
\log Q_{\xi}\Big[
\exp\Big\{bp_r\sum_{k\in\mathcal I_a}Z_{k,\delta}\Big\}\Big]=\frac1{p_r}\sum_k
\log Q_{\xi}\left[\e^{bp_r Z_{k,\delta}}\right].
\end{equation*}
Applying \zcref{lem:poisson} to each of the $n^d$ terms yields
\begin{equation}\label{eq:colored}
\log Q_{\xi}\Big[\exp\Big\{b\sum_k Z_{k,\delta}\Big\}\Big]
\le c_{b,M,r}(n|\Lambda_n|\delta+n^d\delta^2),
\end{equation}
where the finite constant $c_{b,M,r}$ does not depend on $n$ and $\delta$.

\emph{Step 2: Charging each bad space-time point to a close clock point.}
Fix $k$ and a space-time point $(x,s)$ counted in $\Upsilon_n^{\mathrm{gap}}(M,\delta)$ with $x\in u_{n,k}$. Since $L_{k,s}\le M$, both genuine jumps in its window are points of $\Pi_{n,k}$. Each of them therefore has another clock point within distance $\delta$, and is counted by $Z_{k,\delta}$. A fixed clock point at time $t$ can belong to the window around $(x,s)$ only if $s\in[t-R,t+R]$, a set of central times of length at most $2R$, and all $x\in u_{n,k}$ together contribute spatial volume one. A union bound over the $Z_{k,\delta}$ counted points therefore bounds the bad space-time volume associated with the cell $u_{n,k}$ by $2R Z_{k,\delta}$. Summing over $k$ and normalising by $n|\Lambda_n|$ gives~\eqref{eq:gapfraction}.

\emph{Step 3: Exponential Markov inequality.}
For any $b>0$ and $0<\delta\le[3q_r(1+M)]^{-1}$, \eqref{eq:gapfraction} and~\eqref{eq:colored} give
\begin{equation*}
Q_{\xi}\bigl(\Upsilon_n^{\mathrm{gap}}(M,\delta)>\varepsilon\bigr)
\le Q_{\xi}\Big(\sum_k Z_{k,\delta}>
\frac{\varepsilon n|\Lambda_n|}{2R}\Big)\le\exp\Big\{-\frac{b\varepsilon n|\Lambda_n|}{2R}
+c_{b,M,r}(n|\Lambda_n|\delta+n^d\delta^2)\Big\}.
\end{equation*}
Hence,
\begin{equation}\label{ineq:proof-upsilon-gap}
\limsup_{n\to\infty}\frac1{n|\Lambda_n|}
\log Q_{\xi}
\bigl(\Upsilon_n^{\mathrm{gap}}(M,\delta)>\varepsilon\bigr)
\le-\frac{b\varepsilon}{2R}+c_{b,M,r}\delta.
\end{equation}

\emph{Step 4: Choice of the parameters.}
Let $\alpha>0$ be the desired decay exponent. First choose $b=2R(\alpha+1)/\varepsilon$, which fixes the constant $c_{b,M,r}$ from Step~1. Next choose $\delta>0$ so small that $\delta\le[3q_r(1+M)]^{-1}$ and $c_{b,M,r}\delta<1/2$. With this choice of parameters, the bound in~\eqref{ineq:proof-upsilon-gap} is smaller than $-\alpha-1/2$, and the same holds for every smaller $\delta>0$. Since $\alpha$ was arbitrary, this proves~\eqref{eq:gaptail}.
\end{proof}

It remains to make sure that no jumps occur near the endpoints of the observation interval. We prove a slightly more general estimate for the absence of jumps near a finite family of prescribed times. Fix $p\ge1$ and times $t_1,\ldots,t_p\in[-R,R]$, and define
\begin{equation*}
\Upsilon_n^{\mathrm{end}}(\delta)
:=\frac1{n|\Lambda_n|}\sum_k
\int_R^{n-R}\d s\int_{u_{n,k}}\d x\,
\1_{\{\exists\,t\in\mathcal J_n(x,s)\colon
\min_{1\le i\le p}|t-s-t_i|\le\delta\}}.
\end{equation*}
This is the fraction of non-seam space-time points whose observation window contains a genuine jump within distance $\delta$ of at least one of the translated times $s+t_i$. In contrast to $\Upsilon^{\text{gap}}$, we do not need to impose a population cutoff in this definition. The dependence of $\Upsilon^{\text{end}}$ on $r,R$ and on the times $t_1,\dots, t_p$ is suppressed. Later we take $p=2$ and $t_1=-R$, $t_2=R$, in order to protect the endpoints of the observed time interval.

\begin{lemma}[Exponential estimate for $\Upsilon^{\mathrm{end}}$]\label{lem:endtail}
Fix $r,R>0$ and $t_1,\ldots,t_p\in[-R,R]$. For every $\delta>0$ and all sufficiently large $n$,
\begin{equation}\label{eq:endfraction}
\Upsilon_n^{\mathrm{end}}(\delta)
\le\frac{2p\delta|B_r|}{n|\Lambda_n|}\bigl(2B_n+ \xi(\Lambda_n)\bigr).
\end{equation}
Consequently, for every $\varepsilon>0$,
\begin{equation}\label{eq:endtail}
\lim_{\delta\downarrow0}\limsup_{n\to\infty}\frac1{n|\Lambda_n|}
\log Q_{\xi}\bigl(\Upsilon_n^{\mathrm{end}}(\delta)>\varepsilon\bigr)
=-\infty.
\end{equation}
\end{lemma}
\begin{proof}
\emph{Step 1: The space-time volume affected by one jump.}
Every genuine jump in $\Delta_n$ is either a birth or the death of a particle that was present initially or born during $[0,n]$, so that their total number is at most $2B_n+\xi(\Lambda_n)$. Fix such a jump at spatial position $y$ and time $u$. The spatial points $x$ with $y\in B_r+x$ form a set of torus volume $|B_r|$. For a fixed $i$, the temporal condition $|u-s-t_i|\le\delta$ restricts $s$ to an interval of length $2\delta$, and the union over $i=1,\ldots,p$ has length at most $2p\delta$. The additional requirements $s\in[R,n-R]$ and $u\in(s-R,s+R]$ can only reduce this measure. Thus each genuine jump affects space-time volume at most $2p\delta|B_r|$. Summing over all jumps and dividing by $n|\Lambda_n|$ gives~\eqref{eq:endfraction}.

\emph{Step 2: Reduction to a large birth count.}
Fix $\varepsilon,\delta>0$. Since $\xi(\Lambda_n)/(n|\Lambda_n|)\to0$, we have, for all sufficiently large $n$,
\begin{equation*}
\frac{\xi(\Lambda_n)}{n|\Lambda_n|}
\le\frac{\varepsilon}{4p|B_r|\delta}.
\end{equation*}
If $\Upsilon_n^{\mathrm{end}}(\delta)>\varepsilon$, then~\eqref{eq:endfraction} implies
\begin{equation*}
2\frac{B_n}{n|\Lambda_n|}+\frac{\xi(\Lambda_n)}{n|\Lambda_n|}
>\frac{\varepsilon}{2p|B_r|\delta}.
\end{equation*}
Subtracting the preceding bound on $\xi(\Lambda_n)/(n|\Lambda_n|)$ and dividing by two shows that
\begin{equation*}
\{\Upsilon_n^{\mathrm{end}}(\delta)>\varepsilon\}
\subset
\{B_n>\sigma_\delta n|\Lambda_n|\},
\qquad
\sigma_\delta:=\frac{\varepsilon}{8p|B_r|\delta},
\end{equation*}
for all sufficiently large $n$, with $\delta$ fixed.

\emph{Step 3: Sending the time tolerance to zero.}
For $\delta\le\varepsilon/(8p|B_r|)$, we have $\sigma_\delta\ge1$, so that~\eqref{eq:birthtail} gives
\begin{equation*}
\limsup_{n\to\infty}\frac1{n|\Lambda_n|}
\log Q_{\xi}\bigl(\Upsilon_n^{\mathrm{end}}(\delta)>\varepsilon\bigr)
\le-\sigma_\delta\log\sigma_\delta+\sigma_\delta-1.
\end{equation*}
As $\delta\downarrow0$, the threshold $\sigma_\delta$ tends to infinity, and the right-hand side tends to $-\infty$. This proves~\eqref{eq:endtail}. The order of the limits is important here: the tolerance $\delta$ is fixed while $n\to\infty$, so that the contribution of the initial particles can be removed first.
\end{proof}

Together, the three lemmas give the order of the parameters needed below: first we choose $M$ to control $\Upsilon_n^{\mathrm{cnt}}(M)$, which also fixes the clock rate for $\Upsilon_n^{\mathrm{gap}}$; then we choose a gap size $\delta>0$ that controls both $\Upsilon_n^{\mathrm{gap}}(M,\delta)$ and $\Upsilon_n^{\mathrm{end}}(\delta)$.

\subsection{Exponential tightness}\label{subsection:exponential-tightness}
\zcref[S]{prop:exponential-tightness} follows from the following estimate for a single scale $j$.

\begin{lemma}[Single scale]\label{lem:single}
For every $j\ge1$ and all $\alpha,\varepsilon>0$, there are $M\in\mathbb N$ and $\delta>0$ such that
\begin{equation}\label{eq:single}
\limsup_{n\to\infty}\frac1{n|\Lambda_n|}\log Q_{\xi}\bigl(R_n(\overline{G}_j(M,\delta)^c)>\varepsilon\bigr)<-\alpha.
\end{equation}
\end{lemma}

\begin{proof}
We apply the estimates of \zcref{subsection:probabilistic-control} with $r=R=j+1$, protecting the two times $t_1 = -R$ and $t_2 = R$ in \zcref{lem:endtail}. Consider a non-seam space-time point $(x,s)$ that is not counted in any of $\Upsilon_n^{\mathrm{cnt}}(M)$, $\Upsilon_n^{\mathrm{gap}}(M,\delta)$ and $\Upsilon_n^{\mathrm{end}}(\delta)$. Its enlarged window contains at most $M$ particles in total, so that the smaller window $(x+B_{j+1})\times[s-j-1,s+j+1]$ contains at most $M$ particles at every time and at most $2M$ jumps. These jumps are separated by more than $\delta$ and avoid the two endpoints by more than $\delta$. The restriction of $\theta_{(x,s)}\omega^{(n)}$ to $B_{j+1}\times[-j-1,j+1]$ therefore belongs to $\Xi_j(M,\delta)$, and
\begin{equation*}
R_n(G_j(M,\delta)^c)
\le\frac{2(j+1)}n+\Upsilon_n^{\mathrm{cnt}}(M)
+\Upsilon_n^{\mathrm{gap}}(M,\delta)
+\Upsilon_n^{\mathrm{end}}(\delta).
\end{equation*}
First choose $M$ by~\eqref{eq:counttail} at threshold $\varepsilon/4$, with a decay exponent strictly larger than $\alpha$; this fixes the clock rate. Next choose $\delta\in(0,1/4)$ by~\eqref{eq:gaptail} and~\eqref{eq:endtail} at the same threshold. Finally, for all sufficiently large $n$, the seam term is smaller than $\varepsilon/4$. A union bound gives~\eqref{eq:single} with $G_j(M,\delta)$ in place of $\overline G_j(M,\delta)$, and replacing $G_j(M,\delta)$ by its closure only reduces the complement.
\end{proof}

The passage from a single scale to all scales needs some care, since the threshold in $n$ may depend on $j$. The following construction covers the finitely many exceptional values of $n$ at each scale by ordinary tightness.

\begin{proof}[Proof of~\zcref{prop:exponential-tightness}]
Fix $\alpha,\varepsilon>0$, and choose $\varepsilon_j>0$ with $\sum_j\varepsilon_j\le\varepsilon$. By \zcref{lem:single}, applied with the decay exponent $\alpha+j+1$, we can choose $M_j$, $\delta_j$ and $n_j$ such that, with $\overline{G}_j=\overline G_j(M_j,\delta_j)$,
\begin{equation*}
Q_{\xi}\bigl(R_n(\overline{G}_j^c)>\varepsilon_j\bigr)
\le2^{-j}\e^{-(\alpha+j)n|\Lambda_n|}\qquad\text{for } n\ge n_j.
\end{equation*}
For each of the finitely many $n<n_j$, the mean empirical process $Q_{\xi}[R_n]$ is a probability measure on the Polish space $E$. Since finite families of probability measures are tight, there is a compact set $\mathscr L_j\subset E$ such that
\begin{equation*}
Q_{\xi} \left[R_n(\mathscr L_j^c)\right]
\le\varepsilon_j2^{-j}\e^{-(\alpha+j)n|\Lambda_n|}
\qquad\text{for } n<n_j.
\end{equation*}
Set $\mathscr H_j:=\overline{G}_j\cup\mathscr L_j$ and $\mathscr K:=\bigcap_j\mathscr H_j$. Markov's inequality for the finitely many small $n$ and the preceding exponential estimate for the large $n$ give the common bound
\begin{equation*}
    Q_\xi\bigl(R_n(\mathscr H_j^c)>\varepsilon_j\bigr)\le 2^{-j}\e^{-(\alpha+j)n|\Lambda_n|}, \qquad n\ge1,
\end{equation*}
and summing over $j$ proves~\eqref{eq:pathcompact}, since $R_n(\mathscr K^c)\le\sum_jR_n(\mathscr H_j^c)$.

To check that $\mathscr K$ is compact, consider any sequence in $\mathscr K$. If some $\mathscr L_j$ contains a subsequence, the compactness of $\mathscr L_j$ supplies a convergent further subsequence. Otherwise, for every $j$, the sequence eventually belongs to $\overline{G}_j$, and \zcref{lem:compactness}(ii) applies. In either case the limit belongs to $\mathscr K$, which is closed. This proves the first claim.

For the second claim, choose a sequence $(\varepsilon_k)_{k \in \N}$ with $\varepsilon_k\downarrow0$, and apply~\eqref{eq:pathcompact} with the threshold $\varepsilon_k$ and the exponent $\alpha+k$, which gives compact sets $\mathscr K_k \subset E$. Define
\begin{equation*}
K_\alpha
:=\{P\in\mathcal{P}_{\theta}(E)\colon P(\mathscr K_k^c)\le\varepsilon_k\text{ for every }k\}.
\end{equation*}
Each constraint is weakly closed by the Portmanteau theorem. Moreover, the constraints imply uniform tightness, so that Prokhorov's theorem yields relative compactness, which, together with the closedness of $\PP(E)$ in the space of probability measures on $E$, shows that $K_\alpha$ is compact. Finally,
\begin{equation*}
Q_{\xi}(R_n\notin K_\alpha)
\le\sum_{k\ge1}\e^{-(\alpha+k)n|\Lambda_n|}
\le \e^{-\alpha n|\Lambda_n|},
\end{equation*}
where the last inequality uses $n|\Lambda_n|\ge1$ and $\sum_{k\ge1}\e^{-k}<1$.
\end{proof}

\section{Proofs for the interacting dynamics}\label{section:interacting-case}
In this section we prove the results stated in \zcref{subsection:strategy-interacting}, as well as the properties of the rate function in \zcref{proposition:rate-function-properties}. Recall that $b(x,\eta)=\e^{-\beta\mathfrak{h}(x,\eta)}$ denotes the birth rate at the fixed inverse temperature $\beta\in\R$, that $\range$ denotes its range, that we suppress the dependence of all laws on the initial condition $\xi\in\Omst$, and that $Q^\beta_n$ and $Q^0_n$ denote the dynamics in $\Lambda_n$ with periodic boundary conditions and the free evolution in $\Lambda_n$, respectively.

\subsection{Entropy bounds}\label{section:interacting_case_entropy_bounds}
Recall from \zcref{lem:finite-entropy-intensities} that every \(P\) with \(H(P) < \infty\) is concentrated on legal birth-and-death paths and has finite intensity \(\rho(P)\) and finite birth intensity \(\chi(P)\), which equals its death intensity. The proof of \zcref{lem:finite-entropy-intensities}, with the entropy inequality applied to $\lambda\mathsf B(A,(0,T])\wedge m$ in place of $\mathsf B(A,(0,T])\wedge m$, shows more generally that
\begin{equation}\label{equation:interacting_case_bound_free_entropy_beta}
    \lambda \chi(P)
    \leq H(P) + \e^\lambda - 1, \qquad \lambda > 0.
\end{equation}
We record the following upgrade of \eqref{equation:interacting_case_bound_free_entropy_beta}, and how \(H(P)\) also bounds the interaction term \(P[F]\) of the rate function.

\begin{lemma}[Entropy bounds]\label{lemma:interacting_case_nice_things_from_finite_free_entropy}
    Let $P\in\PP(E)$ with $H(P)<\infty$. The bound \eqref{equation:interacting_case_bound_free_entropy_beta} also holds for \(\lambda \leq 0\), and in particular, by setting \(\lambda = \log\chi(P)\) if \(\chi(P) \neq 0\), and by letting \(\lambda \downarrow - \infty\) otherwise,
    \begin{equation*}
        H(P) \geq \chi(P) \log \chi(P) - \chi(P) + 1.
    \end{equation*}
    Furthermore, for every \(\delta > 0\) there exists \(C_\delta < \infty\), not depending on $P$, such that
    \begin{equation}\label{equation:interacting_case_entropy_bound_of_interacting_part}
        \abs{P[F]}
        \leq \Vert 1-b\Vert_\infty + \Vert \log b \Vert_\infty \chi(P)
        \leq \delta H(P) + C_\delta.
    \end{equation}
\end{lemma}
\begin{proof}
    The extension of \eqref{equation:interacting_case_bound_free_entropy_beta} to \(\lambda \leq 0\) follows from the same entropy inequality, once the integrability \(\chi(P) < \infty\) has been established with \(\lambda = 1\). The triangle inequality and the stationarity of \(P\) give \(\abs{P[F]} \leq \Vert 1-b\Vert_\infty + \Vert \log b \Vert_\infty \chi(P)\). Given \(\delta > 0\), the choice \(\lambda =  \Vert \log b \Vert_\infty/ \delta\) in \eqref{equation:interacting_case_bound_free_entropy_beta} yields
    \begin{equation*}
        \Vert 1-b\Vert_\infty + \Vert \log b \Vert_\infty \chi(P)
        \leq \delta H(P) + C_\delta
    \end{equation*}
    with \(C_\delta = \Vert 1-b\Vert_\infty + \delta (\e^{\Vert \log b\Vert_\infty/\delta} -1)\).
\end{proof}

\subsection{Change of measure and boundary conditions}\label{section:interacting_case_girsanov_and_irrelevance_of_boundary_configurations}

\begin{proof}[Proof of \zcref{proposition:change_of_measure_a_la_girsanov}]
    Consider the random birth-counting measure $\mathsf B_\omega(\d x, \d s)$, that is,
    \begin{equation*}
        \mathsf B_\omega(B,(s,t])
        = \#\{\text{births in } B \text{ during } (s, t] \}, \qquad 0 \leq s \leq t, \quad B \subset \Lambda.
    \end{equation*}
    Under \(Q^v\), its predictable compensator is \(v_t(x) \, \d t \, \d x\), so that
    \(\mathsf B_\omega(B,(0,t]) - \int_0^t\d s \int_B \d x\, v_s(x)\)
    is a martingale under \(Q^v\).
    Now write \(r_t(x) := u_t(x)/v_t(x)\) and set
    \begin{align*}
        Z_t
        &:= \exp\bigg(\int_0^t \d s\int_D\d x\,  [v_s(x) - u_s(x) ]  + \sum_{\substack{(s, x) \text{ birth}\colon (s,x) \in (0, t] \times D}} \log \frac{u_s(x)}{v_s(x)} \bigg) \\
        &= \exp\bigg(\int_0^t \int_D \mathsf B_\omega(\d x, \d s)\, \log r_s(x) - \int_0^t \d s\int_D\d x\,  [r_s(x) - 1] v_s(x)\bigg).
    \end{align*}
    This process solves
    \begin{equation*}
        Z_0 = 1,
        \qquad
        \d Z_t = Z_{t-} \, \int_{D} \big(\mathsf B_\omega(\d x, \d t) - \d x\, \d t\, v_t(x)\big)[r_t(x) - 1] ,
    \end{equation*}
    and is therefore a positive local martingale. We now check that it is in fact a true martingale by establishing the desired moment bounds.
    For some fixed \(p > 1\), let
    \begin{equation*}
        L_t
        = \exp\bigg(\int_0^t  \int_D  \mathsf B_\omega(\d x, \d s) \,\log r^p_s(x)- \int_0^t\d s \int_D\d x\,  [r^p_s(x) - 1] v_s(x) \bigg),
    \end{equation*}
    which is \(Z_t\) with \(r_t\) replaced by \(r_t^p\), so that \(L_t\) is a positive local martingale as well, and hence a supermartingale. In particular, \(Q^v[L_\sigma] \leq Q^v[L_0] = 1\) for every stopping time \(\sigma \leq T\).
    Comparing \(\log Z_t^p\) with \(\log L_t\), we get
    \begin{equation*}
    \begin{split}
        \log Z_t^p - \log L_t
        &= \int_0^t \d s\int_D \d x\, [r_s^p(x) - 1 - p(r_s(x) - 1)]  v_s(x)\\
        &= \int_0^t \d s\int_D \d x\, [u_s(x)^p v_s(x)^{1-p} - p u_s(x) + (p-1) v_s(x)],
    \end{split}
    \end{equation*}
    where the integrand is non-negative by the convexity of \(r \mapsto r^p\), and bounded from above by \(C_p\).
    It follows that, for every stopping time \(\sigma \leq T\),
    \begin{equation*}
        Q^v[Z^p_\sigma] = Q^v\big[\e^{\log Z_\sigma^p - \log L_\sigma } L_\sigma\big]
        \leq \e^{C_p T \abs{D}} Q^v[L_\sigma]
        \leq \e^{C_p T \abs{D}}.
    \end{equation*}
    This gives the moment bound \eqref{equation:interacting_case_girsanov_moment_bound}, and also shows that \(Z\) is a true martingale with \(Q^v[Z_t] = 1\) for \(0 \leq t \leq T\).
    Finally, \(\d Q^u/\d Q^v = Z_T\): indeed, \(Z_T \, \d Q^v\) changes the birth compensator in \(D\) from \(v_t(x) \, \d t \, \d x\) to \(r_t(x) v_t(x) \, \d t \, \d x = u_t(x) \, \d t \, \d x\), see~\cite[Theorem~4.5 and the remark after it]{Jacod1975}, and outside \(D\) the birth rates agree by assumption, so that \(Z_T\,  \d Q^v = \d Q^u\).

    For the last assertion, we apply the first part with $\Lambda=D=\Lambda_n$, $T=n$, $u_t(x)=b(x,\omega^{(n)}_{t-})$ and $v_t(x)=1$. Then $\log Z_n=\mathfrak{L}_n$, and the constant $C_p$ with $m=1/\Vert 1/b\Vert_\infty$ and $M=\Vert b\Vert_\infty$, restricted to $t=1$, is $c_p$.
\end{proof}

One direct consequence of \zcref{proposition:change_of_measure_a_la_girsanov} is that boundary conditions play no role for the large deviations.
\begin{proof}[Proof sketch of \zcref{lemma:interacting_case_boundary_conditions_play_no_role}]
    \emph{Finite-volume dynamics with boundary conditions.} For two finite-volume dynamics with different boundary conditions, the equalities follow from the finite range of \(b\), the moment bound \eqref{equation:interacting_case_girsanov_moment_bound} and H\"older's inequality, since the corresponding birth rates only differ on a set of volume of order \(\abs{\partial^\range \Lambda_n}\), where \(\partial^\range \Lambda_n := \{x \in  \Lambda_n \colon \mathrm{dist}(x, \Lambda_n^c) \leq \range\}\).

    \emph{Infinite-volume dynamics.} To compare the infinite-volume dynamics with the periodic dynamics, we introduce the modified birth rate
    \begin{equation*}
        \widetilde{b}_n(x, \eta) = \begin{cases}
            b(x, \eta^{(n)}), &x \in \Lambda_n, \\
            b(x, \eta), &x \not\in \Lambda_n,
        \end{cases}
    \end{equation*}
    which agrees with the birth rate $b(x, \eta^{(n)})$ with periodic boundary conditions inside \(\Lambda_n\), and with the usual infinite-volume birth rate \(b(x, \eta)\) outside \(\Lambda_n\).
    Write \(Q^{\mathrm{inf}}_n\) for the infinite-volume dynamics and \(Q^{\mathrm{mod}}_n\) for the dynamics with the modified birth rate, both considered as laws on \(D([0, n]; \Om)\), whereas \(Q^\beta_n\) is a law on \(D([0, n]; \Om_{\Lambda_n})\).
    Then \(Q^\beta_n(A_n) = Q^{\mathrm{mod}}_n(A_n)\), since under the modified dynamics the configuration in $\Lambda_n$ evolves autonomously with periodic birth rates. It therefore suffices to compare the infinite-volume and the modified dynamics, whose birth rates differ, by the finite range of \(b\), only in the strip \(\partial^\range \Lambda_n\).
    A proof similar to that of \zcref{proposition:change_of_measure_a_la_girsanov} shows that
    \begin{equation*}
    \begin{split}
        \frac{\d Q_n^{\mathrm{inf}}}{\d Q_n^{\mathrm{mod}}}(\omega)
        = \exp\bigg(&\int_0^n \d t\int_{\partial^\range \Lambda_n}\hspace{-0.5cm}\d x\,  [\widetilde{b}_n(x, \omega_{t-}) - b(x, \omega_{t-})]+ \sum_{\substack{(t, x) \text{ birth}\colon (t, x) \in (0, n] \times \partial^\range \Lambda_n}} \hspace{-0.5cm}\log\frac{b(x, \omega_{t-})}{\widetilde{b}_n(x, \omega_{t-})} \bigg),
    \end{split}
    \end{equation*}
    and we can conclude with moment bounds as before.
\end{proof}

\subsection{Approximation of the rate function}\label{section:interacting_case_approximation_of_rate_function}
We first show that the approximants from \zcref{definition:interacting_case_approximation_of_rate_function} are legitimate test functions for the weak topology.

\begin{lemma}[Continuity of the approximants]\label{lemma:interacting_case_approximants_of_rate_function_are_continuous}
    For all \(\epsilon, \delta, K > 0\), the maps \(U_{\epsilon, \delta, K}\) and \(\omega \mapsto \int_0^1 [1 - b(0, \omega_t)] \, \d t\) are bounded and continuous with respect to the local Skorokhod topology.
    Moreover, for every \(P \in \PP(E)\) that is concentrated on legal birth-and-death paths and has finite birth intensity $\chi(P)$,
    \begin{equation}\label{equation:interacting_case_bounds_approximants_of_rate_function}
        \abs{P[U_{\epsilon, \delta, K}]} \leq \Vert \log b\Vert_\infty \chi(P),
        \qquad \abs{P[F_{\epsilon, \delta, K}]} \leq \Vert 1-b\Vert_\infty + \Vert \log b\Vert_\infty \chi(P).
    \end{equation}
\end{lemma}
\begin{proof}
    \emph{Continuity.} Convergence in the Skorokhod topology implies the convergence of the evaluation maps at every continuity time of the limiting path. Since a c\`adl\`ag path has at most countably many discontinuities, this applies to \(\d t\)-almost every \(t\) and \(t + \epsilon\) simultaneously. Moreover, the maps
    \begin{equation*}
        \eta \mapsto \log b(0, \eta), \qquad (\eta, \widetilde{\eta}) \mapsto \Big(\sum_{x \in \widetilde{\eta}} \varphi_\delta(x) - \sum_{x \in \eta} \varphi_\delta(x) \Big)_+
    \end{equation*}
    are continuous. Since the integrand of \(U_{\epsilon, \delta, K}\) is bounded by $K\Vert\log b\Vert_\infty$, dominated convergence shows that \(U_{\epsilon, \delta, K}\) is continuous.
    The continuity of \(\omega \mapsto \int_{0}^{1} \d t\, [1 - b(0, \omega_t)]\) is proved analogously.

    \emph{Bounds.} For a legal birth-and-death path, the increase of $\sum_{x\in\omega_t}\varphi_\delta(x)$ during $(t,t+\epsilon]$ is at most the $\varphi_\delta$-weighted number of births during $(t,t+\epsilon]$. Hence, by stationarity,
    \begin{equation*}
    \begin{split}
        \abs{P[U_{\epsilon, \delta, K}]}
        &\leq \frac{\Vert \log b\Vert_\infty}{\epsilon} \int_0^1 \d t\,  P\Big[\sum_{\substack{(s, y) \text{ birth}\colon t < s \leq t+\epsilon}} \varphi_\delta(y) \Big] \\
        &= \frac{\Vert \log b\Vert_\infty}{\epsilon} \Big(\epsilon \int_{\R^d}\d x\,  \varphi_\delta(x)\Big) \chi(P)
        = \Vert \log b\Vert_\infty\chi(P).
    \end{split}
    \end{equation*}
    The other bound follows from the triangle inequality.
\end{proof}

We now decompose the error that we make by replacing $\mathfrak{L}_n/(n\abs{\Lambda_n})$ with \(R_n[F_{\epsilon, \delta}]\) into manageable pieces. Fix \(\epsilon, \delta \in (0,1)\). A birth of the periodised path is called \emph{genuine} if it is neither the reappearance of an initial particle nor otherwise caused by the temporal reset introduced by the periodisation. A genuine birth is called \((n, \epsilon)\)-\emph{bad} if its own lifetime is smaller than \(\epsilon\), or if, within time distance \(\epsilon\) and spatial distance \(\range + 2\) on the periodised box \(\Lambda_n\), there is another birth, the death of a particle born before it, or the death of an initial particle. We let \(B_{n, \epsilon}^{\mathrm{bad}}\) be the number of \((n, \epsilon)\)-bad genuine births, \(B_{n, \epsilon}^{\mathrm{seam}}\) the number of genuine births in the time intervals \([0, 2 \epsilon]\) and \([n - 2\epsilon, n]\), and \(B_n\) the total number of genuine births during $[0,n]$. Moreover, let
\begin{equation*}
    w_b(\delta) := \sup_{\eta\in\Om} \sup_{\abs{x - y}\leq \delta} \abs{\log b(x, \eta) - \log b(y, \eta)},
\end{equation*}
which tends to zero as $\delta\downarrow0$, since $\abs{\mathfrak{h}(x,\eta)-\mathfrak{h}(y,\eta)}\le\abs{B_{\range/2}(x)\triangle B_{\range/2}(y)}$.

\begin{lemma}[Deterministic error from smoothing]\label{lemma:interacting_case_approximation_of_rate_function_deterministic_error}
    For every legal birth-and-death path, with the counts defined above for its periodised version,
    \begin{equation*}
        \abs{\mathfrak{L}_n - n \abs{\Lambda_n} R_n[F_{\epsilon, \delta}]}
        \leq w_b(\delta) B_n + 2 \Vert\log b\Vert_{\infty}  \big(B^{\mathrm{bad}}_{n, \epsilon} + B_{n, \epsilon}^{\mathrm{seam}} + \xi(\Lambda_n)\big).
    \end{equation*}
\end{lemma}
We omit the proof, which consists only of elementary but tedious calculations: the contributions of the term $\int_0^1\d t\, [1-b(0,\omega_t)]$ and of $\int_0^n\d t\int_{\Lambda_n}\d x\, [1-b(x,\omega^{(n)}_{t-})]$ agree by periodicity, a genuine birth that is neither bad nor in the seams contributes to $n\abs{\Lambda_n}R_n[U_{\epsilon,\delta}]$ a $\varphi_\delta$-average of $\log b$ over a $\delta$-neighbourhood of its position, and every other birth contributes at most $\Vert\log b\Vert_\infty$ to each of the two expressions.

Next, we bound the number of bad births on the exponential scale.
\begin{lemma}[Nearby events]\label{lemma:interacting_case_approximation_of_rate_function_nearby_events}
    For every \(a > 0\),
    \begin{equation*}
        \lim_{\epsilon \downarrow 0} \limsup_{n \to \infty} \frac{1}{n\abs{\Lambda_n}} \log Q_n^0\big[\e^{a B_{n, \epsilon}^{\mathrm{bad}}}\big]
        = 0.
    \end{equation*}
\end{lemma}
\begin{proof}[Proof sketch]
    Nearby births can be controlled via \zcref{lem:poisson}, and short lifetimes by Poisson thinning.
    We sketch how to control nearby deaths, as the other events are handled similarly or more easily. Nearby deaths of initial particles can be handled by first conditioning on their lifetimes and then observing that the births close to these deaths lie in space-time boxes of total volume at most \(2 \epsilon \abs{B_{\range+2}(0)} \xi(\Lambda_n)\). It remains to bound the nearby deaths of particles born during the observation window. The spatial colouring and H\"older argument is similar to Step~1 of the proof of \zcref{lem:gaptail}.

    First, we partition \(\Lambda_n\) into unit cubes and enlarge them by the dependence radius \(\range + 2\). Each enlarged cube intersects only a uniformly bounded number of other enlarged cubes, so that independence and H\"older's inequality allow us to work in a fixed enlarged cube of volume \(\kappa\), which only depends on \(\range\) and \(d\). The births in this cube form a Poisson point process with rate \(\kappa\), marked with independent unit-exponential lifetimes.
    To account for the deaths of older particles, we order the particles by their birth times \(s_i\), let \(r_i\) be their lifetimes, and put
    \begin{equation*}
        D_i :=  \{j < i \colon \abs{s_j + r_j - s_i} \leq \epsilon \},
    \end{equation*}
    so that \(D_i\) contains the indices of all older particles that die within time distance $\epsilon$ of the birth at time \(s_i\).
    Now consider a marked Poisson point process with rate \(2 \kappa\), and flip independent fair coins to assign colour \(1\) or colour \(2\) to each of its points. We realise the marked Poisson point process with rate \(\kappa\) as the points of colour \(1\). The number of births with nearby deaths that we are interested in is then given by \(M_{1,1}\) in \zcref{lemma:interacting_case_color_switching_domination} below. The stochastic domination of \(M_{1,1}\) by \(M_{1,2} + M_{2, 1}\) established there, H\"older's inequality, and the symmetry of the two colours, which shows that $M_{1,2}$ and $M_{2,1}$ have the same law, allow us to only bound \(\mathbb{E}[\e^{a M_{1,2}}]\) for \(a > 0\). Conditioning on the points of colour $2$, we get
    \begin{equation*}
        \log \mathbb{E}[\e^{a M_{1, 2}}]
        \leq \kappa n \{\exp(2\kappa \epsilon (\e^a - 1) ) - 1\},
    \end{equation*}
    and the right-hand side is of the right order: after division by \(n\), it converges to zero as \(\epsilon \downarrow 0\), for fixed \(a > 0\).
\end{proof}

In the proof sketch above, we used the following stochastic domination.
\begin{lemma}[Colour switching]\label{lemma:interacting_case_color_switching_domination}
    Let \(k \in \N\) and let \(D_i \subset \{1, \dots, i-1\}\) for \(1 \leq i \leq k\).
    For an assignment of colours \(a \in \{1, 2\}^k\) and \(c_1,c_2\in\{1,2\}\), set
    \begin{equation*}
        M_{c_1, c_2}(a) = \#\{i \colon a_i = c_1 \text{ and } D_i \text{ contains some } j \text{ with } a_j = c_2 \}.
    \end{equation*}
    Then there is a bijection \(S \colon \{1,2\}^k \to \{1,2\}^k\) such that
    \begin{equation}\label{equation:interacting_case_stochastic_domation_pointwise_color_switching_bound}
        M_{1,1}(a)
        \leq M_{1,2}(Sa) + M_{2, 1}(Sa), \qquad a\in\{1,2\}^k.
    \end{equation}
    In particular, \(M_{1,1}\) is stochastically dominated by \(M_{1,2} + M_{2,1}\) if \(a\) is chosen uniformly from \(\{1, 2\}^k\).
\end{lemma}
\begin{proof}
    We construct \(S(a)\) iteratively. Keep the first colour, \((S(a))_1 := a_1\). For \(i \geq 2\), we fix the first predecessor of colour $1$, or the index $1$ if there is none, that is,
    \begin{equation*}
        p_i = \begin{cases}
            \min\{j \in D_i  \colon  a_j = 1\}, & \text{ if } \{j \in D_i  \colon  a_j = 1\} \neq \emptyset, \\
            1, & \text{otherwise},
        \end{cases}
    \end{equation*}
    and set
    \begin{equation}\label{equation:interacting_case_definition_color_switching_map}
        (S(a))_i := \begin{cases}
            3 - (S(a))_{p_i}, &\text{if } a_i = 1, \\
            (S(a))_{p_i}, & \text{if } a_i = 2.
        \end{cases}
    \end{equation}
    That is, if \(a_i=2\), the new colour of \(i\) is the new colour of its chosen predecessor, and if \(a_i=1\), it is the opposite colour.
    The map \(S\) is invertible, because for every \(a' \in \{1,2\}^k\) the iterative construction determines a unique preimage \(a \in \{1,2\}^k\).
    Regarding \eqref{equation:interacting_case_stochastic_domation_pointwise_color_switching_bound}, suppose that \(i\) is counted by \(M_{1,1}(a)\). Then \(a_i = 1\), and there is a predecessor \(j \in D_i\) with \(a_j = 1\), so that $p_i\in D_i$ is a predecessor of colour $1$. By~\eqref{equation:interacting_case_definition_color_switching_map}, \((S(a))_i \neq (S(a))_{p_i}\), so that \(i\) is counted by \(M_{1,2}(S(a))\) or by \(M_{2, 1}(S(a))\).
\end{proof}

Next, we control the error introduced by clipping the smoothed birth counts at the level $K$.
\begin{lemma}[Exponential clipping]\label{lemma:interacting_case_approximation_of_rate_function_exponential_clipping}
    For fixed \(\epsilon, \delta > 0\) and \(a > 0\),
    \begin{equation}\label{lemma:interacting_case_exponential_clipping_removal}
        \lim_{K \to \infty} \limsup_{n \to \infty} \frac{1}{n \abs{\Lambda_n}} \log Q^0_n\big[\e^{a n \abs{\Lambda_n} R_n[\abs{F_{\epsilon, \delta} - F_{\epsilon, \delta, K}}] }\big]
        = 0.
    \end{equation}
\end{lemma}
\begin{proof}[Proof sketch]
    We only provide a rough sketch, in order not to repeat the same arguments.
    Partition the space-time torus corresponding to \(\Lambda_n \times [0, n]\) into unit space-time boxes \(C_i\), and enlarge them to
    \begin{equation*}
        D_i
        := C_i + ([-\delta, \delta]^d \times [0, \epsilon]),
    \end{equation*}
    of space-time volume \(\kappa\), depending only on $\epsilon,\delta$ and $d$, interpreted periodically at the boundaries of the space-time torus. One can show that, deterministically,
    \begin{equation}\label{equation:interacting_case_exponential_clipping_removal_deterministic_bound}
    \begin{split}
        n \abs{\Lambda_n} R_n[\abs{F_{\epsilon, \delta} - F_{\epsilon, \delta, K}}]
        &= n \abs{\Lambda_n} R_n[\abs{U_{\epsilon, \delta} - U_{\epsilon, \delta, K}}] \\
        &\leq \Vert \log b\Vert_\infty \xi(\Lambda_n) + \frac{\Vert \log b\Vert_\infty \Vert \varphi_\delta\Vert_\infty}{\epsilon} \sum_{i} N_i \1_{\{N_i > \epsilon K / \Vert \varphi_\delta\Vert_\infty\}},
    \end{split}
    \end{equation}
    where \(N_i\) is the number of genuine births in \(D_i\).
    For the first summand on the right-hand side we use the assumption $\xi\in\Omst$, and for the second summand we use a colouring or bounded-overlap argument and H\"older's inequality, which reduces \eqref{lemma:interacting_case_exponential_clipping_removal} to controlling \(N \1_{\{N > \epsilon K / \Vert \varphi_\delta\Vert_\infty\}}\) for of order \(n \abs{\Lambda_n}\) independent Poisson random variables \(N\) with parameter of order \(\kappa\); compare with Step~1 of the proof of \zcref{lem:gaptail}.
\end{proof}

We are now in a position to prove the superexponential approximation.
\begin{proof}[Proof of \zcref{proposition:interacting_case_approximation_of_rate_function_approximately_functional_of_R_n}]
    We establish
    \begin{equation}\label{equation:interacting_case_approximation_of_rate_function_approximately_functional_of_R_n_MGF_bound}
        \lim_{j \to \infty} \limsup_{n \to \infty} \frac{1}{n \abs{\Lambda_n}} \log Q^0_n\big[\e^{a \abs{\mathfrak{L}_n - n\abs{\Lambda_n} R_n[F_j] }} \big]
        = 0, \qquad a > 0,
    \end{equation}
    from which \eqref{equation:interacting_case_approximation_of_rate_function_approximately_functional_of_R_n} follows by a Chernoff bound.
    By \zcref{lemma:interacting_case_approximation_of_rate_function_deterministic_error} and the triangle inequality,
    \begin{equation*}
    \begin{split}
        \abs{\mathfrak{L}_n - n \abs{\Lambda_n} R_n[F_{\epsilon, \delta, K}]}
        &\leq n \abs{\Lambda_n} \abs{R_n[F_{\epsilon, \delta}] - R_n[F_{\epsilon, \delta, K}]} + w_b(\delta) B_n + 2 \Vert\log b\Vert_{\infty}  \big({B}^{\mathrm{bad}}_{n, \epsilon} + B_{n, \epsilon}^{\mathrm{seam}} + \xi(\Lambda_n)\big).
    \end{split}
    \end{equation*}
    Under \(Q_n^0\), the counts \(B_n\) and \(B_{n, \epsilon}^{\mathrm{seam}}\) are Poisson-distributed with means \(n \abs{\Lambda_n}\) and \(4 \epsilon \abs{\Lambda_n}\), respectively. H\"older's inequality with four factors therefore gives
    \begin{equation*}
    \begin{split}
       & \frac{1}{n \abs{\Lambda_n}} \log Q^0_n\big[\e^{a \abs{\mathfrak{L}_n - n\abs{\Lambda_n} R_n[F_{\epsilon, \delta, K}] }} \big] \leq 2 a \Vert\log b \Vert_{\infty} \frac{\xi(\Lambda_n)}{n \abs{\Lambda_n}} + \frac{1}{4} \big(\e^{4 a w_b(\delta)} - 1\big)\\
        &\quad 
        + \frac{1}{4 n\abs{\Lambda_n}} \log Q_n^0\big[\e^{8 a \Vert\log b \Vert_{\infty} B_{n, \epsilon}^{\mathrm{bad}}}\big]  + \frac{\epsilon}{n}\big(\e^{8a  \Vert\log b \Vert_{\infty}} - 1\big)
        + \frac{1}{4 n\abs{\Lambda_n}} \log Q_n^0\big[\e^{4 a n \abs{\Lambda_n} \abs{R_n[F_{\epsilon, \delta}] - R_n[F_{\epsilon, \delta, K}]} }\big].
    \end{split}
    \end{equation*}
    The first and the fourth summands on the right-hand side vanish as \(n \to \infty\), for fixed \(\epsilon, \delta, K\), the first one since $\xi\in\Omst$. For the other terms, we have to take care of the order of the limits \(K \to \infty\), \(\epsilon \to 0\) and \(\delta \to 0\), and we choose the parameters diagonally.
    We first choose \(\delta_j > 0\) so small that \(\e^{4 j w_b(\delta_j)} - 1 \leq 1/j\).
    Next, we choose \(\epsilon_j > 0\) so small that
    \begin{equation*}
        \limsup_{n \to \infty} \frac{1}{n\abs{\Lambda_n}} \log Q_n^0\big[\e^{8 j \Vert\log b\Vert_{\infty} B_{n, \epsilon_j}^{\mathrm{bad}}}\big]
        \leq \frac{1}{j},
    \end{equation*}
    which is possible by \zcref{lemma:interacting_case_approximation_of_rate_function_nearby_events}.
    Finally, for fixed \(\epsilon_j\) and \(\delta_j\), we choose \(K_j \in \N\) such that \(\epsilon_j K_j \geq j \Vert \varphi_{\delta_j} \Vert_\infty\) and
    \begin{equation*}
        \limsup_{n \to \infty} \frac{1}{n\abs{\Lambda_n}} \log Q_n^0\Big[\e^{4 j n \abs{\Lambda_n} \big|R_n\big[F_{\epsilon_j, \delta_j}\big] - R_n\big[F_{\epsilon_j, \delta_j, K_j}\big]\big|}\Big]
        \leq \frac{1}{j},
    \end{equation*}
    which is possible by \zcref{lemma:interacting_case_approximation_of_rate_function_exponential_clipping}. With this choice of parameters, \eqref{equation:interacting_case_approximation_of_rate_function_approximately_functional_of_R_n_MGF_bound} holds, since for every \(a > 0\) we have \(j \geq a\) for all sufficiently large $j$, and~\eqref{equation:interacting_case_particular_apprximation_parameters} holds by construction.
\end{proof}

We close this subsection with the approximation of the interaction term in expectation.
\begin{proof}[Proof of \zcref{proposition:interacting_case_approximation_of_rate_function}]
    \emph{Pointwise convergence.} If \(\chi(P) = 0\), then $P$-almost surely there are no births, so that $U_{\epsilon,\delta,K}=0$ and, by stationarity, $P[F_j]=P[F]$ for all $j$.
    If \(\chi(P) > 0\), we use as a computational tool the \emph{birth Palm law} \(P^+\) of \(P\). Analogous to the usual Palm measure, \(P^+\) is uniquely defined by
    \begin{equation*}
        \chi(P)P^+[g]
        = P\Big[\sum_{\substack{(t, x) \text{ birth}\colon (t, x) \in (0, 1] \times [0,1)^d}} g(\theta_{(x, t)}\omega) \Big], \qquad g \geq 0 \text{ measurable},
    \end{equation*}
    and is interpreted as the law of the process seen from a typical birth, which occurs at time \(0\) at the origin.
    Campbell's formula gives
    \begin{equation*}
        P[F]
        = P\Big[\int_{0}^{1}\d t\,  [1 - b(0, \omega_t)]\Big]
        + \chi(P) P^+[\log b(0, \omega_{0-})].
    \end{equation*}
    We first compare with the unclipped approximation $F_{\epsilon,\delta}$ and show that
    \begin{equation}\label{equation:interacting_case_approximation_of_rate_function_pointwise_unclipped}
        \abs{P[F_{\epsilon, \delta}] - P[F]}
        = \abs{P[U_{\epsilon, \delta}] - \chi(P) P^+[\log b(0, \omega_{0-})]}
        \leq \chi(P) w_b(\delta) + 2 \Vert \log b\Vert_\infty \chi(P) P^+(A_\epsilon),
    \end{equation}
    with the Palm event
    \begin{equation*}
        A_\epsilon
        = \{\text{a birth or death other than the distinguished birth occurs in } B_{\range+2}(0) \times [-\epsilon, \epsilon]\}.
    \end{equation*}
    Indeed, again by Campbell's formula,
    \begin{equation}\label{equation:interacting_case_campbell_expansion_approximation}
        P[U_{\epsilon, \delta}]
        = \frac{\chi(P)}{\epsilon} \int_{\R^d}\d x \, \varphi_\delta(-x) \, \int_{-\epsilon}^{0}\d t \, P^+\Big[\log(b(0, \theta_x\omega_t)) \, \frac{(W_{x, t}^{\mathrm{birth}} - W_{x, t}^{\mathrm{death}})_+}{W_{x, t}^{\mathrm{birth}}} \Big],
    \end{equation}
    where
    \begin{equation*}
        W_{x, t}^{\mathrm{birth}}
        = \sum_{\substack{(s, y) \text{ birth}\colon t < s \leq t+\epsilon}} \varphi_\delta(y-x),
    \end{equation*}
    and \(W_{x, t}^{\mathrm{death}}\) is defined analogously, with deaths instead of births.
    On \(A_\epsilon^c\) and whenever \(\varphi_\delta(-x) > 0\), the only jump relevant for the integrand in \eqref{equation:interacting_case_campbell_expansion_approximation} is the distinguished birth at the origin at time \(0\), so that there
    \begin{equation*}
        \log(b(0, \theta_x\omega_t)) \, \frac{(W_{x, t}^{\mathrm{birth}} - W_{x, t}^{\mathrm{death}})_+}{W_{x, t}^{\mathrm{birth}}}
        = \log(b(0, \theta_x\omega_{0-})).
    \end{equation*}
    Since the fraction in~\eqref{equation:interacting_case_campbell_expansion_approximation} takes values in $[0,1]$, it follows that
    \begin{equation*}
    \begin{split}
        &\abs{P[U_{\epsilon, \delta}] - \chi(P) P^+[\log b(0, \omega_{0-})]} \\
        &\leq \frac{\chi(P)}{\epsilon} \int_{\R^d}\hspace{-0.15cm} \d x\, \varphi_\delta(-x)  \int_{-\epsilon}^{0} \hspace{-0.15cm}\d t\, \big\{P^+\big[\abs{\log(b(0, \theta_x\omega_{0-})) - \log(b(0, \omega_{0-}))} \1_{A_\epsilon^c}\big]+ 2 \Vert \log b \Vert_\infty P^+(A_\epsilon) \big\}  \\
        &\leq  \chi(P) w_b(\delta) + 2 \Vert \log b \Vert_\infty \chi(P) P^+(A_\epsilon),
    \end{split}
    \end{equation*}
    which is \eqref{equation:interacting_case_approximation_of_rate_function_pointwise_unclipped}.

    To remove the clipping, we first note that, by stationarity,
    \begin{equation*}
        \abs{P[F_{\epsilon, \delta}] - P[F_{\epsilon, \delta, K}]}
        = \abs{P[U_{\epsilon, \delta}] - P[U_{\epsilon, \delta, K}]}
        \leq \Vert \log b \Vert_{\infty} P[(D_{\epsilon, \delta}(\omega, 0) - K)_+].
    \end{equation*}
    Since
    \begin{equation*}
        D_{\epsilon, \delta}(\omega, 0)
        = \frac{1}{\epsilon}\Big(\sum_{x \in \omega_{\epsilon}} \varphi_\delta(x) - \sum_{x \in \omega_{0}} \varphi_\delta(x) \Big)_+
        \leq \frac{1}{\epsilon} \int_{\R^d \times (0, \epsilon]} \mathsf{B}_\omega(\d y, \d s)\, \varphi_\delta(y),
    \end{equation*}
    we have
    \begin{equation*}
        (D_{\epsilon, \delta}(\omega, 0) - K)_+
        \leq \frac{1}{\epsilon} \int_{\R^d \times (0, \epsilon]}\mathsf{B}_\omega(\d y, \d s)\,  \varphi_\delta(y) \1_{\{\mathsf{B}_\omega(B_\delta(0), (0, \epsilon]) > \epsilon K / \Vert\varphi_\delta\Vert_\infty\}}.
    \end{equation*}
    Again by Campbell's formula,
    \begin{equation*}
         P[(D_{\epsilon, \delta}(\omega, 0) - K)_+]
         \leq \chi(P) P^+\big(\mathsf{B}_\omega(B_2(0), (-1, 1]) > \epsilon K / \Vert \varphi_\delta \Vert_\infty\big),
    \end{equation*}
    so that altogether
    \begin{equation*}
    \begin{split}
        \abs{P[F_{\epsilon, \delta, K}] - P[F]}
        &\leq \chi(P) w_b(\delta) + 2 \Vert \log b \Vert_\infty \chi(P) P^+(A_\epsilon) \\
        &\qquad + \Vert\log b \Vert_\infty \chi(P) P^+\big(\mathsf{B}_\omega(B_2(0), (-1, 1])  > \epsilon K / \Vert \varphi_\delta \Vert_\infty\big).
    \end{split}
    \end{equation*}
    As $\epsilon\downarrow0$, $P^+(A_\epsilon)\to0$, since $P^+$-almost every path has only finitely many events in a bounded space-time box, none of which occurs simultaneously with the distinguished birth. Together with~\eqref{equation:interacting_case_particular_apprximation_parameters}, this proves the pointwise convergence~\eqref{equation:interacting_case_pointwise_regarding_measure_convergence_of_rate_function_approximation} for our choice of \((F_j)_{j \in \N}\).

    \emph{Uniform convergence on entropy sublevel sets.}
    Fix \(\alpha \in [0, \infty)\). We show that \((P[F_j])_{j \in \N}\) is uniformly Cauchy on the sublevel set \(\{P \colon H(P) \leq \alpha\}\), which suffices, since we have already identified the limit.
    Let \(a > 0\). We show that there is a \(J \in \N\) such that, for all \(j, k \geq J\),
    \begin{equation}\label{equation:interacting_case_cauchyness_approximation_of_rate_function}
        \sup_{H(P) \leq \alpha} \abs{P[F_j] - P[F_k]}
        \leq 2 a.
    \end{equation}
    By \zcref{proposition:interacting_case_approximation_of_rate_function_approximately_functional_of_R_n}, we can choose \(J\) such that, for all $j\ge J$,
    \begin{equation}\label{equation:interacting_case_exponential_upper_bound_approximation_applied}
        \limsup_{n \to \infty} \frac{1}{n \abs{\Lambda_n}} \log Q^0_n\big(\abs{ \tfrac{1}{n\abs{\Lambda_n}}\mathfrak{L}_n - R_n[F_j] } > a \big) \leq -\alpha - 1.
    \end{equation}
    Suppose that there were some \(P\) with $H(P)\le\alpha$ and \(j, k \geq J\) such that \(\abs{P[F_j] - P[F_k]} > 2a\). Then
    \(G := \{Q \colon  \abs{Q[F_j] - Q[F_k]} > 2a\}\)
    is an open neighbourhood of \(P\), and on the event $\{R_n\in G\}$ at least one of the two approximation errors exceeds $a$. By the lower bound of \zcref{thm:ldp-lower-bound-free} and the continuity of \(F_j, F_k\),
    \begin{equation*}
    \begin{split}
        &\liminf_{n \to \infty} \frac{1}{n \abs{\Lambda_n}} \log\Big[ Q^0_n\big(\abs{ \tfrac{1}{n\abs{\Lambda_n}}\mathfrak{L}_n - R_n[F_j] } > a \big) + Q^0_n\big(\abs{ \tfrac{1}{n\abs{\Lambda_n}}\mathfrak{L}_n - R_n[F_k] } > a \big) \Big] \\
        &\qquad\geq
        \liminf_{n \to \infty} \frac{1}{n\abs{\Lambda_n}} \log Q^0_n\big(R_n \in G\big)
        \geq -H(P)
        \geq -\alpha,
    \end{split}
    \end{equation*}
    which contradicts \eqref{equation:interacting_case_exponential_upper_bound_approximation_applied}. Hence \eqref{equation:interacting_case_cauchyness_approximation_of_rate_function} holds.

    \emph{Upper bound on sets of laws.}
    The inequality \eqref{equation:interacting_case_rate_function_approximation_upper_bound_set_of_measures} is clear if \(\mathcal{K}\) is contained in a sublevel set \(\{P \colon H(P) \leq \alpha\}\) for some \(\alpha \in [0, \infty)\), because we have established the uniform convergence \(P[F_j] \to P[F]\) there. In the general case, we additionally use \eqref{equation:interacting_case_bounds_approximants_of_rate_function} and then the bound \eqref{equation:interacting_case_bound_free_entropy_beta} on \(\chi(P)\), which yield
    \begin{equation*}
        P[F_j] - H(P)
        \leq C_b' - C_b H(P)
    \end{equation*}
    for constants \(C_b', C_b \in (0, \infty)\) that do not depend on $j$ and $P$. Hence, for every \(\alpha > 0\),
    \begin{equation*}
        \sup_{P \in \mathcal{K}} \big(P[F_j] - H(P)\big)
        \leq \Big(\sup_{P \in \mathcal{K},\, H(P) \leq \alpha} \big(P[F_j] - H(P)\big)\Big) \lor (C_b' - C_b \alpha),
    \end{equation*}
    and \eqref{equation:interacting_case_rate_function_approximation_upper_bound_set_of_measures} follows by first letting $j \to \infty$ and then \(\alpha \to \infty\).
\end{proof}

\subsection{Properties of the rate function}\label{section:interacting_case_properties_of_the_rate_function}
We finally prove the properties of the rate function \(H^\beta\) stated in \zcref{proposition:rate-function-properties}. We start with the conditional-entropy representation, the analogue of the limit in \zcref{lem:entropy-superadditivity} for the free evolution.
\begin{lemma}[Conditional-entropy representation of \(H^\beta\)]\label{lemma:interacting_case_rate_function_conditional_entropy_representation}
    For every \(P \in \PP(E)\), the representation \eqref{equation:interacting_case_rate_function_conditional_entropy_representation} holds.
\end{lemma}
\begin{proof}
    Assume first that \(H(P) < \infty\). Then \zcref{proposition:change_of_measure_a_la_girsanov} and the chain rule for the relative entropy give
    \begin{equation}\label{equation:interacting_case_conditional_entropy_rep_finite_volume}
    \begin{split}
        \Ent&\big(P_{\Lambda_n, [0,n]} \lvert  P_{\Lambda_n, 0} Q^{\Lambda_n,\beta}\big)
        = \Ent\big(P_{\Lambda_n, [0,n]} \lvert P_{\Lambda_n, 0} Q^{\Lambda_n}\big)\\
        &\quad- P\Big[\int_0^n \d t \int_{\Lambda_n}\d x\,  \big(1 - b(x, (\omega_t)_{\Lambda_n})\big)+ \sum_{\substack{(t, x) \text{ birth}\colon (t, x) \in (0, n] \times \Lambda_n }} \log b(x, (\omega_{t-})_{\Lambda_n}) \Big].
    \end{split}
    \end{equation}
    By the finite range of \(b\), the birth rates with and without the restriction to $\Lambda_n$ only differ at points of the strip $\partial^\range\Lambda_n$, so that stationarity gives
    \begin{equation*}
    \begin{split}
        &\Big\vert n \abs{\Lambda_n} P[F] -  P\Big[\int_0^n \d t \int_{\Lambda_n} \d x\, \big(1 - b(x, (\omega_t)_{\Lambda_n})\big) +\hspace{-0.3cm} \sum_{\substack{(t, x) \text{ birth}\colon \\ (t, x) \in (0, n] \times \Lambda_n }} \hspace{-0.3cm}\log b(x, (\omega_{t-})_{\Lambda_n}) \Big] \Big\vert \lesssim (1 + \chi(P)) n\abs{\partial^\range \Lambda_n}.
    \end{split}
    \end{equation*}
    Dividing \eqref{equation:interacting_case_conditional_entropy_rep_finite_volume} by \(n\abs{\Lambda_n}\), letting \(n \to \infty\), and using \zcref{lem:entropy-superadditivity} for the first term on its right-hand side yields \eqref{equation:interacting_case_rate_function_conditional_entropy_representation}.

    Now assume that \(H(P) = \infty\). Uniformly over the initial distribution,
    \begin{equation*}
        \log \big(P_{\Lambda_n, 0} Q^{\Lambda_n}\big)\Big[ \Big(\frac{\d(P_{\Lambda_n, 0} Q^{\Lambda_n,\beta})}{\d(P_{\Lambda_n, 0} Q^{\Lambda_n})}\Big)^2 \Big]
        \lesssim n\abs{\Lambda_n}
    \end{equation*}
    by \eqref{equation:interacting_case_girsanov_moment_bound}. H\"older's inequality and the variational formula for the relative entropy therefore give
    \begin{equation*}
        \Ent\big(P_{\Lambda_n, [0,n]} \lvert P_{\Lambda_n, 0} Q^{\Lambda_n}\big)
        \lesssim 2 \Ent\big(P_{\Lambda_n, [0,n]} \lvert  P_{\Lambda_n, 0} Q^{\Lambda_n,\beta}\big) + n \abs{\Lambda_n},
    \end{equation*}
    so that \(H(P) = \infty\), and hence \(H^\beta(P) = \infty\) by definition, implies
    \begin{equation*}
        \lim_{n \to \infty} \frac{1}{n\abs{\Lambda_n}}  \Ent\big(P_{\Lambda_n, [0,n]} \lvert P_{\Lambda_n, 0} Q^{\Lambda_n,\beta}\big)
        = \infty
    \end{equation*}
    as desired. 
\end{proof}

We also stated the following continuity and boundedness properties of \(H^\beta\).
\begin{lemma}[Further properties of \(H^\beta\)]\label{lemma:interacting_case_rate_function_further_properties}
    The rate function \(H^\beta\) from \eqref{equation:interacting_case_rate_function} has the following properties.
    \begin{enumerate}[(i)]
        \item It is lower semicontinuous on \(\PP(E)\), equipped with the weak topology.
        \item It is a good rate function on \(\PP(E)\), equipped with the weak topology.
        \item It is non-negative and has minimum \(0\).
        \item The bound \eqref{equation:interacting_case_lower_bound_via_free_entropy} holds.
    \end{enumerate}
\end{lemma}
\begin{proof}
    For~(i), fix $\alpha\in\R$, choose $\delta \in (0,1)$, and let $C_\delta$ be as in \eqref{equation:interacting_case_entropy_bound_of_interacting_part}. Then $H^\beta=H-P[F]\ge(1-\delta)H-C_\delta$ on $\{H<\infty\}$, so that
    \begin{equation*}
        \{H^\beta \leq \alpha\} \subset \{H \leq \alpha'\},
        \qquad \alpha' = (\alpha + C_\delta) / (1-\delta).
    \end{equation*}
    By \zcref{proposition:interacting_case_approximation_of_rate_function}, \(P \mapsto P[F]\) is continuous on \(\{H \leq \alpha'\}\), which is closed by the lower semicontinuity of $H$, see \zcref{lemma:nice-variational-representation}. Hence $H^\beta=H-P[F]$ is lower semicontinuous on $\{H\le\alpha'\}$, and \(\{H^\beta \leq \alpha\} = \{P\in\{H\le\alpha'\}\colon H(P) - P[F] \leq \alpha\}\) is closed.
    For~(ii), \(\{H^\beta \leq \alpha\}\) is a closed subset of the weakly compact set \(\{H \leq \alpha'\}\), see \zcref{theorem:full-ldp} for $\beta=0$.
    For~(iii), the non-negativity follows from \eqref{equation:interacting_case_rate_function_conditional_entropy_representation}, and \(\inf H^\beta \leq 0\) follows by applying the upper bound of \zcref{theorem:full-ldp}, whose proof does not use the goodness of $H^\beta$, to the closed set \(\PP(E)\). The infimum is attained by the goodness of \(H^\beta\).
    Finally, (iv) follows from \eqref{equation:interacting_case_entropy_bound_of_interacting_part}, since $H^\beta=H-P[F]$ if $H<\infty$, and both sides of~\eqref{equation:interacting_case_lower_bound_via_free_entropy} are infinite otherwise.
\end{proof}

\section{Proofs for the minimum entropy problem}\label{section:contraction}
In this section we prove \zcref{lemma:contraction_time_marginal_hellinger_distance_identities,lemma:contraction_time_marginal_minimum_entropy_problem_hellinger_thermodynamic_limit,lemma:contraction_time_marginal_minimum_entropy_problem_hellinger_finite_volume_lower_bound,lemma:contraction_time_marginal_minimum_entropy_problem_hellinger_existence_of_minimizer_in_infinite_volume}, which complete the identification of the rate function in \zcref{thm:marg-ldp}. Throughout, $\beta\in\R$ is fixed and $b$ denotes the corresponding birth rate. For two \(\sigma\)-finite measures \(\mathfrak{M}\), \(\mathfrak{N}\) on a common measurable space \(\mathcal{X}\), we use their \emph{geometric mean} \(\sqrt{\mathfrak{M}\mathfrak{N}}\), defined by
\begin{equation*}
    \d\sqrt{\mathfrak{M}\mathfrak{N}}
    := \sqrt{\frac{\d\mathfrak{M}}{\d (\mathfrak{M} + \mathfrak{N})}\frac{\d\mathfrak{N}}{\d (\mathfrak{M} + \mathfrak{N})}}\, \d (\mathfrak{M} + \mathfrak{N}).
\end{equation*}
If \(\mathfrak{M}\) and \(\mathfrak{N}\) are finite measures, then expanding the square in the definition of the Hellinger distance gives
\begin{equation*}
    \mathfrak{H}^2(\mathfrak{M}, \mathfrak{N})
    = \frac{\mathfrak{M}(\mathcal{X}) + \mathfrak{N}(\mathcal{X})}{2} - \sqrt{\mathfrak{M}\mathfrak{N}}(\mathcal{X}).
\end{equation*}

\subsection{Hellinger distance identities}\label{section:contraction_time_marginal_minimum_entropy_problem_basic_facts}
We first prove that \( \mathfrak{H}^2(b(0, \cdot) \mu, \mu_0^!) = \mathfrak{H}^2(B_\mu, D_\mu)\), and we express this quantity in terms of Papangelou intensities, when they exist.
\begin{proof}[Proof of \zcref{lemma:contraction_time_marginal_hellinger_distance_identities}]
    For measurable \(f \geq 0\), translation invariance gives
    \begin{equation*}
    \begin{split}
        \int f(\eta, x) \, B_\mu(\d\eta, \d x)
        &= \int \mu(\d\eta)  \int_{[0, 1)^d} \d x \, b(x, \eta) \, f(\eta, x) = \int_{[0, 1)^d} \d x \, \int \mu(\d\eta) \, b(0, \theta_x\eta) \, f(\eta, x)  \\
        &= \int_{[0, 1)^d} \d x \, \int \mu(\d\eta) \, b(0, \eta) \, f(\theta_{-x}\eta, x) = \int (b(0, \cdot) \mu)(\d\eta) \int_{[0, 1)^d} \d x\, f(\theta_{-x}\eta, x) ,
    \end{split}
    \end{equation*}
    that is, \(B_\mu = ((b(0, \cdot) \mu) \otimes \1_{[0, 1)^d}(x)\, \d x) \circ i^{-1}\) with \(i \colon (\eta, x) \mapsto (\theta_{-x}\eta, x)\).
    Similarly, the characterisation of the reduced Palm measure gives \(D_\mu = (\mu_0^! \otimes \1_{[0, 1)^d}(x)\, \d x) \circ i^{-1}\).
    Since $i$ is a measurable bijection with measurable inverse, and the Lebesgue measure of $[0,1)^d$ is one, it follows that
    \begin{equation*}
        \mathfrak{H}^2(B_\mu, D_\mu)
        = \mathfrak{H}^2\big((b(0, \cdot) \mu) \otimes \1_{[0, 1)^d}(x)\, \d x, \mu_0^! \otimes \1_{[0, 1)^d}(x)\, \d x \big)
        = \mathfrak{H}^2\big(b(0, \cdot) \mu , \mu_0^!\big).
    \end{equation*}
    Now suppose that \(\mu\) has a Papangelou intensity \(\kappa\). Then, by the definitions of \(B_\mu\) and of the Papangelou intensity,
    \begin{equation*}
        B_\mu(\d\eta, \d x) = \mu(\d\eta)\d x\, b(x, \eta) \1_{[0, 1)^d}(x),
        \qquad
        D_\mu(\d\eta, \d x) = \mu(\d\eta)\d x\, \kappa(x, \eta) \1_{[0, 1)^d}(x),
    \end{equation*}
    and therefore
    \begin{equation*}
    \begin{split}
        2\mathfrak{H}^2(B_\mu, D_\mu)
        &= \int\mu(\d\eta)\d x\,  \Big(\sqrt{b(x,\eta)} - \sqrt{\kappa(x,\eta)}  \Big)^2 \1_{[0, 1)^d}(x) = \mu\Big[\int_{[0, 1)^d}\d x \, \Big(\sqrt{b(x, \eta)} - \sqrt{\kappa(x,\eta)}\Big)^2\Big].
    \end{split}
    \end{equation*}
        as desired.
\end{proof}

Next, we establish the representation of \(\mathfrak{H}^2(B_\mu, D_\mu)\) as a thermodynamic limit. For measurable sets $A, \Lambda\Subset\R^d$, we write
\begin{equation*}
    B_\mu\vert_{A}(\d\eta,\d x)
    := \mu(\d\eta)\d x \, b(x, \eta) \, \1_A(x),
\end{equation*}
and \(B_\mu\vert_{\Lambda, A}\) for the image of \(B_\mu\vert_{A}\) under \(i_\Lambda \colon (\eta, x) \mapsto (\eta_\Lambda, x)\); the measures \(D_\mu\vert_{A}\) and \(D_\mu\vert_{\Lambda, A}\) are defined analogously. Moreover, \(B_\mu^{(\Lambda)}\vert_A\) and \(D_\mu^{(\Lambda)}\vert_A\) denote the restrictions of \(B_\mu^{(\Lambda)}\) and \(D_\mu^{(\Lambda)}\) to $\Om_\Lambda\times A$ for $A\subset\Lambda$, and $\partial^\range\Lambda:=\{y \in \Lambda \colon \mathrm{dist}(y, \Lambda^c) \leq \range\}$.

\begin{proof}[Proof of \zcref{lemma:contraction_time_marginal_minimum_entropy_problem_hellinger_thermodynamic_limit}]
    We can assume \(\rho(\mu) < \infty\). Indeed, if \(\rho(\mu) = \infty\), then the total mass of \(D_\mu\) is infinite, whereas \(B_\mu[1] \leq \Vert b\Vert_\infty < \infty\), so that \(\mathfrak{H}^2(B_\mu, D_\mu) = \infty\), and the same argument applies to the finite-volume versions.

    \emph{Upper bound on the limit superior.}
    For \(A \subset \Lambda_n\), the Hellinger distance is additive over the decomposition of the spatial coordinate $x$, so that
    \begin{equation}\label{equation:contraction_time_marginal_minimum_entropy_problem_hellinger_thermodynamic_limit_lower_bound_decomposition}
        \mathfrak{H}^2(B_\mu^{(\Lambda_n)}, D^{(\Lambda_n)}_\mu)
        = \mathfrak{H}^2(B_\mu^{(\Lambda_n)}\vert_{A}, D^{(\Lambda_n)}_\mu \vert_{A}) + \mathfrak{H}^2(B_\mu^{(\Lambda_n)}\vert_{\Lambda_n \setminus A}, D^{(\Lambda_n)}_\mu \vert_{\Lambda_n \setminus A}).
    \end{equation}
    We apply this with \(A = \partial^\range \Lambda_n\), because \(b(x, \eta_{\Lambda_n}) = b(x, \eta)\) for \(x \in \Lambda_n \setminus \partial^\range \Lambda_n\) and \(\eta \in \Om\).
    The boundary part is of boundary order, as
    \begin{equation}\label{equation:contraction_time_marginal_minimum_entropy_problem_hellinger_thermodynamic_limit_lower_bound_boundary}
    \begin{split}
        \mathfrak{H}^2(B_\mu^{(\Lambda_n)}\vert_{\partial^\range \Lambda_n}, D^{(\Lambda_n)}_\mu \vert_{\partial^\range \Lambda_n})
        &\leq B_\mu^{(\Lambda_n)}(\Omega_{\Lambda_n} \times \partial^\range \Lambda_n) + D_\mu^{(\Lambda_n)}(\Omega_{\Lambda_n} \times \partial^\range \Lambda_n) \\
        &\leq \Vert b\Vert_\infty \abs{\partial^\range \Lambda_n} + \rho(\mu) \abs{\partial^\range \Lambda_n}.
    \end{split}
    \end{equation}
    On the other hand, the finite range of \(b\) gives
    \begin{equation*}
        B_\mu^{(\Lambda_n)}\vert_{\Lambda_n \setminus \partial^\range \Lambda_n}
        = B_\mu\vert_{\Lambda_n \setminus \partial^\range \Lambda_n} \circ i_{\Lambda_n}^{-1},
    \end{equation*}
    and similarly \(D^{(\Lambda)}_\mu\vert_A = D_\mu\vert_{A} \circ i_\Lambda^{-1}\) for all \(A\subset\Lambda\). The Hellinger distance contracts under images of measures, and stationarity gives
\begin{equation}\label{equation:contraction_time_marginal_minimum_entropy_problem_hellinger_thermodynamic_limit_lower_bound_bulk}
\mathfrak{H}^2(B_\mu^{(\Lambda_n)}\vert_{\Lambda_n \setminus \partial^\range \Lambda_n} , D^{(\Lambda_n)}_\mu \vert_{\Lambda_n \setminus \partial^\range \Lambda_n})
        \leq \mathfrak{H}^2( B_\mu\vert_{\Lambda_n \setminus \partial^\range \Lambda_n},  D_\mu\vert_{\Lambda_n \setminus \partial^\range \Lambda_n})
        = \abs{\Lambda_n \setminus \partial^\range \Lambda_n}  \mathfrak{H}^2( B_\mu,  D_\mu).
    \end{equation}
    It follows from \eqref{equation:contraction_time_marginal_minimum_entropy_problem_hellinger_thermodynamic_limit_lower_bound_decomposition}, \eqref{equation:contraction_time_marginal_minimum_entropy_problem_hellinger_thermodynamic_limit_lower_bound_boundary} and \eqref{equation:contraction_time_marginal_minimum_entropy_problem_hellinger_thermodynamic_limit_lower_bound_bulk} that
    \begin{equation*}
        \limsup_{n \to \infty} \frac{\mathfrak{H}^2( B_\mu^{(\Lambda_n)},  D_\mu^{(\Lambda_n)})}{\abs{\Lambda_n}}
        \leq  \mathfrak{H}^2( B_\mu,  D_\mu).
    \end{equation*}

    \emph{Lower bound on the limit inferior.}
    We show that
    \begin{equation}\label{equation:contraction_time_marginal_minimum_entropy_problem_hellinger_thermodynamic_limit_upper_bound_k}
        \liminf_{n \to \infty} \frac{\mathfrak{H}^2( B_\mu^{(\Lambda_n)},  D_\mu^{(\Lambda_n)})}{\abs{\Lambda_n}}
        \geq \mathfrak{H}^2(B_{\mu}\vert_{\Lambda_k, [0, 1)^d}, D_{\mu}\vert_{\Lambda_k, [0,1)^d})
    \end{equation}
    for all \(k \geq \range\). This implies
    \begin{equation*}
        \liminf_{n \to \infty} \frac{\mathfrak{H}^2( B_\mu^{(\Lambda_n)},  D_\mu^{(\Lambda_n)})}{\abs{\Lambda_n}}
        \geq \mathfrak{H}^2(B_{\mu}, D_{\mu}),
    \end{equation*}
    since \(\mathfrak{H}^2(B_{\mu}\vert_{\Lambda_k, [0, 1)^d}, D_{\mu}\vert_{\Lambda_k, [0,1)^d}) \to \mathfrak{H}^2(B_{\mu}, D_{\mu})\) as $k \to \infty$, by the martingale convergence theorem; see~\cite[Theorem~6.6]{nguyen1979ergodic} for the corresponding statement for the relative entropy.
    To prove \eqref{equation:contraction_time_marginal_minimum_entropy_problem_hellinger_thermodynamic_limit_upper_bound_k}, we decompose \(\Lambda_n\), up to boundary order, into shifted copies \(j + [0, 1)^d\) of \([0, 1)^d\) such that all \(j + \Lambda_k\) have distance at least \(\range\) from \(\partial \Lambda_n\), so that \(b(x, \eta_{\Lambda_n}) = b(x, \eta)\) for \(x \in j + \Lambda_k\).
    By \eqref{equation:contraction_time_marginal_minimum_entropy_problem_hellinger_thermodynamic_limit_lower_bound_decomposition}, and since the Hellinger distance contracts under images of measures,
    \begin{equation*}
    \begin{split}
        \mathfrak{H}^2( B_\mu^{(\Lambda_n)},  D_\mu^{(\Lambda_n)})
        &\geq \sum_j \mathfrak{H}^2\big( B_\mu^{(\Lambda_n)}\vert_{j + [0,1)^d},  D_\mu^{(\Lambda_n)}\vert_{j + [0, 1)^d}\big)\geq \sum_j \mathfrak{H}^2\big( B_\mu\vert_{j + \Lambda_k, j + [0, 1)^d},  D_\mu\vert_{j + \Lambda_k, j + [0, 1)^d}\big).
    \end{split}
    \end{equation*}
    Finally, by stationarity, all summands are equal to \(\mathfrak{H}^2(B_{\mu}\vert_{\Lambda_k, [0, 1)^d}, D_{\mu}\vert_{\Lambda_k, [0,1)^d})\), so that
    \begin{equation*}
        \mathfrak{H}^2( B_\mu^{(\Lambda_n)},  D_\mu^{(\Lambda_n)})
        \geq (1 - o(1))\abs{\Lambda_n} \mathfrak{H}^2(B_{\mu}\vert_{\Lambda_k, [0, 1)^d}, D_{\mu}\vert_{\Lambda_k, [0,1)^d}),
    \end{equation*}
    which proves \eqref{equation:contraction_time_marginal_minimum_entropy_problem_hellinger_thermodynamic_limit_upper_bound_k}.

    \emph{Dirichlet form expression with local densities.}
    It remains to prove \eqref{equation:contraction_time_marginal_dirichlet_form_identity_hellinger_distance} for a fixed \(\nu \in \GG^\beta_\theta\).
    By the GNZ equations,
    \begin{equation*}
    \begin{split}
        \mathcal{E}^{\nu}_{\Lambda}(g, g)
        &= \int \nu_\Lambda(\d\eta_\Lambda) \sum_{x \in \eta_\Lambda} [g(\eta_\Lambda - \delta_x) - g(\eta_\Lambda)]^2
        = \int \nu(\d\eta) \int_{\Lambda} \d x \, b(x, \eta) [g(\eta_\Lambda) - g(\eta_\Lambda + \delta_x)]^2 \\
        &= \int \nu_\Lambda(\d\eta_\Lambda) \int_{\Lambda} \d x \, b_\Lambda^\nu(x, \eta) [g(\eta_\Lambda) - g(\eta_\Lambda + \delta_x)]^2,
    \end{split}
    \end{equation*}
    with the averaged birth rate
    \begin{equation*}
        b_\Lambda^\nu(x, \eta)
        := \int\nu_{\Lambda^c \lvert \Lambda}(\d \zeta_{\Lambda^c} \vert \eta)\,  b(x, \eta_{\Lambda} \zeta_{\Lambda^c}),
    \end{equation*}
    where $\nu_{\Lambda^c\lvert\Lambda}$ denotes the conditional law of the configuration outside $\Lambda$ given the configuration inside $\Lambda$ under $\nu$.
    For \(g_\Lambda = \d\mu_{\Lambda}/\d\nu_\Lambda\), the GNZ equations show more generally that
    \begin{equation*}
        D_\mu^{(\Lambda)}(\d\eta, \d x)
        = g_\Lambda(\eta + \delta_x) \,  b_\Lambda^\nu(x, \eta) \, \nu_\Lambda(\d\eta) \, \d x,
    \end{equation*}
    so that, with
    \begin{equation*}
        B_\mu^{(\Lambda, \nu)}(\d \eta, \d x)
        := g_\Lambda(\eta) \,  b_\Lambda^\nu(x, \eta) \, \nu_\Lambda(\d\eta) \, \d x,
    \end{equation*}
    we obtain
    \begin{equation*}
        \mathcal{E}^\nu_{\Lambda}\big(\sqrt{g_\Lambda},  \sqrt{g_\Lambda}\big)
        = \int \nu_\Lambda(\d\eta) \int_{\Lambda} \d x \, b_\Lambda^\nu(x, \eta) \big[\sqrt{g_\Lambda}(\eta) - \sqrt{g_\Lambda}(\eta + \delta_x)\big]^2
        = 2\mathfrak{H}^2(B_\mu^{(\Lambda, \nu)}, D_\mu^{(\Lambda)}).
    \end{equation*}
    Since we have already shown \eqref{equation:contraction_time_marginal_minimum_entropy_problem_hellinger_thermodynamic_limit}, it remains to bound
    \begin{equation*}
        \big\vert \mathfrak{H}^2(B_\mu^{(\Lambda, \nu)}, D_\mu^{(\Lambda)}) - \mathfrak{H}^2(B_\mu^{(\Lambda)}, D_\mu^{(\Lambda)}) \big\vert
        \leq (\Vert b\Vert_\infty + \rho(\mu)) \abs{\partial^\range \Lambda}.
    \end{equation*}
    This follows from computations analogous to \eqref{equation:contraction_time_marginal_minimum_entropy_problem_hellinger_thermodynamic_limit_lower_bound_decomposition} and \eqref{equation:contraction_time_marginal_minimum_entropy_problem_hellinger_thermodynamic_limit_lower_bound_boundary}, because \(b(x, \eta_\Lambda) = b_\Lambda^\nu(x, \eta)\) for \(x \in \Lambda \setminus \partial^\range \Lambda\).
\end{proof}

\subsection{The finite-volume lower bound}\label{section:contraction_time_marginal_minimum_entropy_problem_lower_bound}
In this subsection we prove the finite-volume lower bound of \zcref{lemma:contraction_time_marginal_minimum_entropy_problem_hellinger_finite_volume_lower_bound}. To connect the Hellinger distance \(\mathfrak{H}^2(B_\mu^{(\Lambda)}, D_\mu^{(\Lambda)})\) with the dynamics, we use a variational characterisation involving the formal generator of the dynamics in $\Lambda$ with empty boundary configuration,
\begin{equation*}
    (\mathscr{L}_\Lambda g)(\eta_\Lambda)
    = \int_\Lambda\d x\,  b(x, \eta_\Lambda) [g(\eta_\Lambda + \delta_x) - g(\eta_\Lambda)] + \sum_{x \in \eta_\Lambda} [g(\eta_\Lambda - \delta_x) - g(\eta_\Lambda)].
\end{equation*}
Let \(\gamma_\Lambda = \nu^{(\emptyset)}_\Lambda\) be the finite-volume Gibbs measure in \(\Lambda\) with empty boundary configuration, and consider the formal Dirichlet form
\begin{equation*}
    \mathcal{E}^{\emptyset}_\Lambda(g, g)
    = \gamma_\Lambda[g (-\mathscr{L}_\Lambda g)] \in [0, \infty]
\end{equation*}
for non-negative \(g \in L^2(\gamma_\Lambda)\). If \(\mu_\Lambda \ll \gamma_\Lambda\) and \(f = \d \mu_\Lambda/\d \gamma_\Lambda\) is bounded away from zero and infinity, then the computation in the proof of \eqref{equation:contraction_time_marginal_dirichlet_form_identity_hellinger_distance} shows that
\begin{equation}\label{equation:contraction_time_marginal_minimum_entropy_problem_hellinger_finite_volume_lower_bound_variational_formulation_using_generator_regular_mu}
    2\mathfrak{H}^2(B_\mu^{(\Lambda)}, D_\mu^{(\Lambda)})
    = \mathcal{E}^{\emptyset}_\Lambda(\sqrt{f}, \sqrt{f})
    = \gamma_\Lambda[\sqrt{f} (-\mathscr{L}_\Lambda \sqrt{f})]
    = \mu_\Lambda\Big[\frac{-\mathscr{L}_\Lambda \sqrt{f}}{\sqrt{f}} \Big].
\end{equation}
The following variational characterisation holds even if \(\mu_\Lambda\) is singular with respect to $\gamma_\Lambda$.
\begin{lemma}[Variational characterisation of the Hellinger distance]\label{lemma:contraction_time_marginal_minimum_entropy_problem_hellinger_finite_volume_lower_bound_variational_formulation_using_generator}
    For every \(\mu\in\PP(\Om)\) and \(\Lambda \Subset \R^d\),
    \begin{equation}\label{equation:contraction_time_marginal_minimum_entropy_problem_hellinger_finite_volume_lower_bound_variational_formulation_using_generator}
        2\mathfrak{H}^2(B_\mu^{(\Lambda)}, D_\mu^{(\Lambda)})
        = \sup \Big\{ \mu_\Lambda\Big[\frac{-\mathscr{L}_\Lambda u}{u} \Big]\colon u \text { measurable},\ 0 <\inf u \leq \sup u < \infty, \ \sup \abs{\mathscr{L}_\Lambda u} < \infty\Big\}.
    \end{equation}
\end{lemma}
\begin{proof}
    We first show that both sides of \eqref{equation:contraction_time_marginal_minimum_entropy_problem_hellinger_finite_volume_lower_bound_variational_formulation_using_generator} are infinite if \(\mu[\eta(\Lambda)] = \infty\). Indeed,
    \begin{equation*}
        \mu[\eta(\Lambda)]
        = D_\mu^{(\Lambda)}[1]
        \leq 4 \mathfrak{H}^2(B^{(\Lambda)}_\mu, D^{(\Lambda)}_\mu) + 2 B_\mu^{(\Lambda)}[1]
        \leq 4 \mathfrak{H}^2(B^{(\Lambda)}_\mu, D^{(\Lambda)}_\mu) + 2 \Vert b\Vert_\infty \abs{\Lambda},
    \end{equation*}
    so that \(\mathfrak{H}^2(B^{(\Lambda)}_\mu, D^{(\Lambda)}_\mu) = \infty\) if \(\mu[\eta(\Lambda)] = \infty\). For \(u_m(\eta) = \e^{\eta(\Lambda) \land m}\), we have
    \begin{equation*}
    \begin{split}
        \mu_\Lambda\Big[\frac{-\mathscr{L}_\Lambda u_m}{u_m} \Big]
        &= \mu_\Lambda\Big[(1-\e)\1_{\{\eta(\Lambda) < m\}} \int_\Lambda b(x, \eta)\, \d x + (1-\e^{-1}) \eta(\Lambda) \1_{\{\eta(\Lambda) \leq m\}} \Big] \\
        &\geq - (\e - 1) \Vert b\Vert_\infty \abs{\Lambda}  + (1-\e^{-1}) \mu\big[\eta(\Lambda) \1_{\{\eta(\Lambda) \leq m\}}\big],
    \end{split}
    \end{equation*}
    which tends to infinity as $m\to\infty$ if \(\mu[\eta(\Lambda)] = \infty\). We can therefore assume \(\mu[\eta(\Lambda)] < \infty\) from now on.

    \emph{Upper bound on the supremum.}
    Let \(u\) be measurable with \(0 < \inf u \leq \sup u < \infty\) and \(\sup \abs{\mathscr{L}_\Lambda u }< \infty\). Then, with \(r(\eta, x) = u(\eta + \delta_x)/u(\eta)\),
    \begin{equation*}
    \begin{split}
        \mu_\Lambda\Big[\frac{-\mathscr{L}_\Lambda u}{u} \Big]
        &= \int  [1-r(\eta, x)] \,B^{(\Lambda)}_\mu(\d\eta, \d x) + \int [1 - r(\eta, x)^{-1}]\,D^{(\Lambda)}_\mu(\d \eta, \d x) \\
        &= B^{(\Lambda)}_\mu[1] + D^{(\Lambda)}_\mu[1] - \int \big\{ r(\eta, x)\, B^{(\Lambda)}_\mu(\d\eta, \d x) + r(\eta, x)^{-1} D^{(\Lambda)}_\mu(\d \eta, \d x) \big\}  \\
        &\leq B^{(\Lambda)}_\mu[1] + D^{(\Lambda)}_\mu[1] - 2\int  \sqrt{B^{(\Lambda)}_\mu(\d\eta, \d x)\,   D^{(\Lambda)}_\mu(\d \eta, \d x)}
        = 2 \mathfrak{H}^2(B^{(\Lambda)}_\mu, D^{(\Lambda)}_\mu),
    \end{split}
    \end{equation*}
    where we used $rs+r^{-1}t\ge2\sqrt{st}$ for the densities $s,t$ of the two measures with respect to their sum.

    \emph{Lower bound on the supremum.}
     We extend \eqref{equation:contraction_time_marginal_minimum_entropy_problem_hellinger_finite_volume_lower_bound_variational_formulation_using_generator_regular_mu} to singular \(\mu_\Lambda\) by approximating with absolutely continuous measures \(\mu_\Lambda^{(n)}\).
    Partition \(\Lambda\) into cells \(C_1^{(n)}, \dots, C_{2^{nd}}^{(n)}\) of positive volume, using dyadic partitions of \(\R^d\) to ensure that the partitions are nested. Define \(\mu_\Lambda^{(n)}\) by first sampling from \(\mu_\Lambda\) and then resampling the particle positions in each cell uniformly and independently, as in the smoothing of \zcref{subsec:lb-truncation}.
    Let
    \begin{equation*}
        \mathfrak{C}_n(\eta)
        := \big\{\widetilde{\eta} \in \Omega_\Lambda \,\colon\,  \widetilde{\eta}(C_j^{(n)}) =  \eta(C_j^{(n)})\ \text{for all } 1 \leq j \leq 2^{nd}\big\}.
    \end{equation*}
    An everywhere-defined version of \(\d \mu_\Lambda^{(n)}/\d \pi_\Lambda\) is given by
    \begin{equation*}
        g_n(\eta)
        := \frac{\mu_\Lambda(\mathfrak{C}_n(\eta))}{\pi_\Lambda(\mathfrak{C}_n(\eta))},
    \end{equation*}
    and an everywhere-defined version of \(\d \mu_\Lambda^{(n)}/\d \gamma_\Lambda\) by
    \begin{equation*}
        f_n(\eta)
        := Z_\Lambda(\emptyset) \e^{\beta\mathscr{H}_\Lambda(\eta)} g_n(\eta).
    \end{equation*}
    We show that
    \begin{equation}\label{equation:contraction_time_marginal_minimum_entropy_problem_hellinger_finite_volume_lower_bound_variational_formulation_using_generator_maximizing_sequence}
        \mu_\Lambda\Big[\frac{-\mathscr{L}_\Lambda \sqrt{f_n}}{\sqrt{f_n}}\Big] \xrightarrow[n \to \infty]{} 2\mathfrak{H}^2(B_\mu^{(\Lambda)}, D_\mu^{(\Lambda)}).
    \end{equation}
    First, observe that
    \begin{equation*}
    \begin{split}
        &\mu_\Lambda\Big[\frac{-\mathscr{L}_\Lambda \sqrt{f_n}}{\sqrt{f_n}}\Big]
        = B_\mu^{(\Lambda)}[1] + D_\mu^{(\Lambda)}[1] - \int \sqrt{\frac{f_n(\eta + \delta_x)}{f_n(\eta)}}  \, B_\mu^{(\Lambda)}(\d\eta, \d x) - \int \sqrt{\frac{f_n(\eta)}{f_n(\eta + \delta_x)}}  \, D_\mu^{(\Lambda)}(\d\eta, \d x) \\
        &\qquad= B_\mu^{(\Lambda)}[1] + D_\mu^{(\Lambda)}[1] - \int \sqrt{b(x, \eta)} \, r_n^+(\eta, x) \, \widetilde{B}_\mu^{(\Lambda)}(\d\eta, \d x) - \int \sqrt{b(x, \eta)} \, r_n^-(\eta, x) \, D_\mu^{(\Lambda)}(\d\eta, \d x),
    \end{split}
    \end{equation*}
    with
    \begin{equation*}
        \widetilde{B}_\mu^{(\Lambda)}(\d\eta, \d x)
        = \mu_\Lambda(\d \eta) \, \1_\Lambda(x) \, \d x
    \end{equation*}
    and
    \begin{equation*}
        r_n^+(\eta, x)
        = \begin{cases}
            \sqrt{\frac{g_n(\eta + \delta_x)}{g_n(\eta)}}, &g_n(\eta) > 0,\\
            0, &g_n(\eta) = 0,
        \end{cases}
        \qquad
        r_n^-(\eta, x)
        = \begin{cases}
            \sqrt{\frac{g_n(\eta)}{g_n(\eta + \delta_x)}}, &g_n(\eta + \delta_x) > 0,\\
            0, &g_n(\eta + \delta_x) = 0.
        \end{cases}
    \end{equation*}
    By the boundedness of \(b\), it suffices for \eqref{equation:contraction_time_marginal_minimum_entropy_problem_hellinger_finite_volume_lower_bound_variational_formulation_using_generator_maximizing_sequence} to show that
    \begin{equation}\label{equation:contraction_time_marginal_minimum_entropy_problem_hellinger_finite_volume_lower_bound_variational_formulation_using_generator_measure_convergence}
    \begin{split}
        r_n^+(\eta, x) \, \widetilde{B}_\mu^{(\Lambda)}(\d\eta, \d x)
        &\xrightarrow[n \to \infty]{} \sqrt{\widetilde{B}_\mu^{(\Lambda)}(\d\eta, \d x)\,  D_\mu^{(\Lambda)}(\d\eta, \d x)}, \\
        r_n^-(\eta, x) \, D_\mu^{(\Lambda)}(\d\eta, \d x)
        &\xrightarrow[n \to \infty]{} \sqrt{\widetilde{B}_\mu^{(\Lambda)}(\d\eta, \d x) \, D_\mu^{(\Lambda)}(\d\eta, \d x)}
    \end{split}
    \end{equation}
    in total variation.
    Writing \(C^{(n)}(x)\) for the unique cell \(C_j^{(n)}\) with \(x \in C_j^{(n)}\), we have
    \begin{equation*}
    \begin{split}
        \widetilde{B}^{(\Lambda)}_\mu(\mathfrak{C}_n(\eta) \times  C^{(n)}(x))
        &= \abs{C^{(n)}(x)}\, \mu_\Lambda(\mathfrak{C}_n(\eta)), \\
        D^{(\Lambda)}_\mu(\mathfrak{C}_n(\eta) \times C^{(n)}(x))
        &= [\eta(C^{(n)}(x)) + 1]\, \mu_\Lambda(\mathfrak{C}_n(\eta + \delta_x)),
    \end{split}
    \end{equation*}
    and therefore, if the denominator is non-zero,
    \begin{equation*}
    \begin{split}
        r^+_n(\eta, x)
        &= \sqrt{\tfrac{\mu_\Lambda(\mathfrak{C}_n(\eta + \delta_x))}{\mu_\Lambda(\mathfrak{C}_n(\eta))}} \sqrt{\tfrac{\pi_\Lambda(\mathfrak{C}_n(\eta))}{\pi_\Lambda(\mathfrak{C}_n(\eta + \delta_x))}}= \sqrt{\tfrac{\mu_\Lambda(\mathfrak{C}_n(\eta + \delta_x))}{\mu_\Lambda(\mathfrak{C}_n(\eta))}} \sqrt{\tfrac{\eta(C^{(n)}(x)) + 1}{\abs{C^{(n)}(x)}}}
        = \sqrt{\tfrac{D^{(\Lambda)}_\mu(\mathfrak{C}_n(\eta)\times  C^{(n)}(x))}{\widetilde{B}^{(\Lambda)}_\mu(\mathfrak{C}_n(\eta)\times  C^{(n)}(x))}},
    \end{split}
    \end{equation*}
    and analogously for \(r_n^-\). Rewriting everything in terms of densities with respect to \(D^{(\Lambda)}_\mu + \widetilde{B}^{(\Lambda)}_\mu\), we see that \eqref{equation:contraction_time_marginal_minimum_entropy_problem_hellinger_finite_volume_lower_bound_variational_formulation_using_generator_measure_convergence} holds by the classical martingale convergence argument for the Radon--Nikodym theorem and Scheff\'e's lemma.
Finally, $\sqrt{f_n}$ is in general neither bounded away from zero nor such that
$\sup|\mathscr{L}_\Lambda\sqrt{f_n}|<\infty$, and we remedy this by truncation. For
$r\ge1$ let $c_r(t):=(t\vee r^{-1})\wedge r$, $t\ge0$, and for $m\in\mathbb{N}$ set
\[
  u_{n,r}:=c_r\bigl(\sqrt{f_n}\bigr),\qquad
  u_{n,m,r}(\eta):=u_{n,r}(\eta)\,\1_{\{\eta(\Lambda)\le m\}}
  +\1_{\{\eta(\Lambda)>m\}}.
\]
For fixed $m,r$, the function $u_{n,m,r}$ takes values in $[r^{-1},r]$, and
$\mathscr{L}_\Lambda u_{n,m,r}$ vanishes on $\{\eta(\Lambda)>m+1\}$ and is bounded
elsewhere, so $u_{n,m,r}$ satisfies the requirements in~\eqref{equation:contraction_time_marginal_minimum_entropy_problem_hellinger_finite_volume_lower_bound_variational_formulation_using_generator}.
We show that
\begin{equation}\label{eq:trunc-limit}
  \mu_\Lambda\Big[\frac{-\mathscr{L}_\Lambda\sqrt{f_n}}{\sqrt{f_n}}\Big]
  =\lim_{r\to\infty}\lim_{m\to\infty}
   \mu_\Lambda\Big[\frac{-\mathscr{L}_\Lambda u_{n,m,r}}{u_{n,m,r}}\Big].
\end{equation}
For any positive function $u$ on $\Omega_\Lambda$,
\begin{equation}\label{eq:gen-ratio}
  \frac{-\mathscr{L}_\Lambda u}{u}(\eta)
  =\int_\Lambda\mathrm{d}x\,b(x,\eta)\Bigl[1-\frac{u(\eta+\delta_x)}{u(\eta)}\Bigr]
  +\sum_{x\in\eta}\Bigl[1-\frac{u(\eta-\delta_x)}{u(\eta)}\Bigr].
\end{equation}
For fixed $r$, all ratios in \eqref{eq:gen-ratio}
with $u=u_{n,m,r}$ lie in $[r^{-2},r^2]$, so the integrand is bounded by
$(1+r^2)(\|b\|_\infty|\Lambda|+\eta(\Lambda))$, which is $\mu_\Lambda$-integrable.
Moreover, $u_{n,m,r}$ and $u_{n,r}$ agree at $\eta$ and at all $\eta\pm\delta_x$ as
soon as $\eta(\Lambda)<m$. By dominated convergence, the inner limit in
\eqref{eq:trunc-limit} equals $\mu_\Lambda[-\mathscr{L}_\Lambda u_{n,r}/u_{n,r}]$.

For the limit $r\to\infty$, the ratios in \eqref{eq:gen-ratio} are of order
$r^2$ in general, and we use instead the elementary inequality
\begin{equation}\label{eq:clamp}
  \frac{c_r(s)}{c_r(t)}\le 1+\frac{s}{t},\qquad s\ge0,\ t>0,\ r\ge1.
\end{equation}
Indeed, if $t\in[r^{-1},r]$, then $c_r(s)\le s+r^{-1}\le s+t$; if $t<r^{-1}$, then
$c_r(s)/c_r(t)=r\,c_r(s)\le\max\{rs,1\}\le s/t+1$; and if $t>r$, the ratio is at most
one. Since $g_n(\eta)=0$ means $\mu_\Lambda(\mathfrak{C}_n(\eta))=0$ and there are only
countably many sets $\mathfrak{C}_n(\cdot)$, we have $f_n>0$ $\mu_\Lambda$-almost
surely. By \eqref{eq:clamp}, every bracket in \eqref{eq:gen-ratio} with $u=u_{n,r}$ is
then bounded in absolute value by $2+\sqrt{f_n(\zeta)/f_n(\eta)}$, $\zeta=\eta\pm\delta_x$,
uniformly in $r$. The resulting dominating function is $\mu_\Lambda$-integrable: by the
computation following \eqref{equation:contraction_time_marginal_minimum_entropy_problem_hellinger_finite_volume_lower_bound_variational_formulation_using_generator_maximizing_sequence}, its integral equals
\[
  2B^{(\Lambda)}_\mu[1]+2D^{(\Lambda)}_\mu[1]
  +\int\sqrt{b}\,r_n^+\,\mathrm{d}\widetilde{B}^{(\Lambda)}_\mu
  +\int\sqrt{b}\,r_n^-\,\mathrm{d}D^{(\Lambda)}_\mu .
\]
The functions $r_n^\pm$ are constant on each of the countably many sets
$\mathfrak{C}_n(\eta)\times C^{(n)}(x)$, with values
$(D^{(\Lambda)}_\mu/\widetilde{B}^{(\Lambda)}_\mu)^{1/2}$ and
$(\widetilde{B}^{(\Lambda)}_\mu/D^{(\Lambda)}_\mu)^{1/2}$ of that set, so the
Cauchy--Schwarz inequality over these sets bounds each of the last two integrals by
\[
  \|b\|_\infty^{1/2}\bigl(\widetilde{B}^{(\Lambda)}_\mu[1]\,D^{(\Lambda)}_\mu[1]\bigr)^{1/2}
  =\|b\|_\infty^{1/2}\bigl(|\Lambda|\,\mu_\Lambda[\eta(\Lambda)]\bigr)^{1/2}<\infty.
\]
Finally, $u_{n,r}(\zeta)\to\sqrt{f_n(\zeta)}$ for every $\zeta$ as $r\to\infty$, and
$u_{n,r}(\eta)=\sqrt{f_n(\eta)}$ for all large $r$ whenever $f_n(\eta)>0$. Hence the
integrand converges $\mu_\Lambda$-almost everywhere to
$-\mathscr{L}_\Lambda\sqrt{f_n}/\sqrt{f_n}$, and dominated convergence proves
\eqref{eq:trunc-limit}. Together with \eqref{equation:contraction_time_marginal_minimum_entropy_problem_hellinger_finite_volume_lower_bound_variational_formulation_using_generator_maximizing_sequence}, this shows that the
supremum in \eqref{equation:contraction_time_marginal_minimum_entropy_problem_hellinger_finite_volume_lower_bound_variational_formulation_using_generator} is at least
$2\mathfrak{H}^2(B^{(\Lambda)}_\mu,D^{(\Lambda)}_\mu)$.
    % \newpage
    % Finally, $\sqrt{f_n}$ is in general neither bounded away from zero nor such that \(\sup \abs{\mathscr{L}_\Lambda \sqrt{f_n}} < \infty\), and we remedy this by another approximation step. By dominated convergence,
    % \begin{equation*}
    %     \mu_\Lambda\Big[\frac{-\mathscr{L}_\Lambda \sqrt{f_n}}{\sqrt{f_n}}\Big]
    %     = \lim_{r \to \infty} \lim_{m \to \infty} \mu_\Lambda\Big[\frac{-\mathscr{L}_\Lambda u_{n, m, r}}{u_{n, m, r}}\Big]
    % \end{equation*}
    % for
    % \begin{equation*}
    %     u_{n, m, r}(\eta)
    %     = \big[(\sqrt{f_n} \lor r^{-1}) \land r\big]\1_{\{\eta(\Lambda) \leq m\}} + \1_{\{\eta(\Lambda) > m\}},
    % \end{equation*}
    % which, for fixed \(m, r \geq 1\), satisfies the requirements in~\eqref{equation:contraction_time_marginal_minimum_entropy_problem_hellinger_finite_volume_lower_bound_variational_formulation_using_generator}.
\end{proof}

We put the variational characterisation to use with the following entropy bound.
\begin{lemma}[Entropy bound in generator form]\label{lemma:contraction_time_marginal_minimum_entropy_problem_hellinger_finite_volume_lower_bound_generator_form}
    Let \(P_{[0,T]}\) be as in \zcref{lemma:contraction_time_marginal_minimum_entropy_problem_hellinger_finite_volume_lower_bound}, and let \(u \colon \Omega_\Lambda \to [0, \infty)\) be measurable with \(0 < \inf u \leq \sup u < \infty\) and \(\sup \abs{\mathscr{L}_\Lambda u }< \infty\).
    Then
    \begin{equation}\label{equation:contraction_time_marginal_minimum_entropy_problem_hellinger_finite_volume_lower_bound_generator_form}
        T\mu_\Lambda\Big[\frac{-\mathscr{L}_\Lambda u}{u}\Big]
        \leq \Ent(P_{[0, T]} \vert \mu_\Lambda Q^{\Lambda,\beta}).
    \end{equation}
\end{lemma}
\begin{proof}
    We use the entropy inequality \(\mathbb{P}[f] \leq \Ent(\mathbb{P} \vert \mathbb{Q}) + \log \mathbb{Q}[\e^f] \leq \Ent(\mathbb{P} \vert \mathbb{Q})\) for measurable $f$ with \(\mathbb{Q}[\e^f] \leq 1\), with \(\mathbb{P} = P_{[0, T]}\) and \(\mathbb{Q} = \mu_\Lambda Q^{\Lambda,\beta}\).
    Since all time marginals of $P_{[0,T]}$ are equal to \(\mu_\Lambda\), the left-hand side of~\eqref{equation:contraction_time_marginal_minimum_entropy_problem_hellinger_finite_volume_lower_bound_generator_form} equals
    \begin{equation*}
        \int_0^T P\Big[\frac{(-\mathscr{L}_\Lambda u)(\omega_t)}{u(\omega_t)}\Big]\, \d t
        = P\Big[\log\frac{u(\omega_T)}{u(\omega_0)} -\int_0^T \frac{\mathscr{L}_\Lambda u(\omega_t)}{u(\omega_t)}\, \d t\Big]
        = P[\log Z_T],
    \end{equation*}
    where
    \begin{equation*}
        Z_t
        = \frac{u(\omega_t)}{u(\omega_0)} \exp\Big( -\int_0^t \frac{\mathscr{L}_\Lambda u(\omega_s)}{u(\omega_s)}\, \d s\Big).
    \end{equation*}
    It is classical, see e.g.~\cite[Chapter~4, Lemma~3.2]{EK86}, that \((Z_t)_{t \in [0, T]}\) is a martingale under the Markov process \(\mu_\Lambda Q^{\Lambda,\beta}\) with generator \(\mathscr{L}_\Lambda\) on bounded measurable functions. Hence
    \((\mu_\Lambda Q^{\Lambda,\beta})[\e^{\log Z_T}]
        = (\mu_\Lambda Q^{\Lambda,\beta})[Z_0] = 1\),
    and the entropy inequality yields \eqref{equation:contraction_time_marginal_minimum_entropy_problem_hellinger_finite_volume_lower_bound_generator_form}.
\end{proof}

\begin{proof}[Proof of \zcref{lemma:contraction_time_marginal_minimum_entropy_problem_hellinger_finite_volume_lower_bound}]
    Combining \zcref{lemma:contraction_time_marginal_minimum_entropy_problem_hellinger_finite_volume_lower_bound_variational_formulation_using_generator,lemma:contraction_time_marginal_minimum_entropy_problem_hellinger_finite_volume_lower_bound_generator_form}, we obtain
    \begin{equation*}
        2\mathfrak{H}^2(B_\mu^{(\Lambda)}, D_\mu^{(\Lambda)})
        = \sup_u \mu_\Lambda\Big[\frac{-\mathscr{L}_\Lambda u}{u} \Big]
        \leq \frac{1}{T}\Ent(P_{[0, T]} \vert \mu_\Lambda Q^{\Lambda,\beta}),
    \end{equation*}
    where the supremum is over all $u$ as in \eqref{equation:contraction_time_marginal_minimum_entropy_problem_hellinger_finite_volume_lower_bound_variational_formulation_using_generator}.
\end{proof}

\subsection{Attainment of the lower bound}\label{section:lemma:contraction_time_marginal_minimum_entropy_problem_attainment_of_lower_bound}
In this subsection we prove \zcref{lemma:contraction_time_marginal_minimum_entropy_problem_hellinger_existence_of_minimizer_in_infinite_volume}. We first construct a process in finite volume that attains the lower bound, and then build a sequence of infinite-volume processes from these finite-volume processes, from which we extract a minimiser by compactness.
\begin{lemma}[Finite-volume attainment]\label{lemma:contraction_time_marginal_minimum_entropy_problem_hellinger_finite_volume_attainment}
    Let \(\mu \in \PP(\Om)\) with $\rho(\mu)<\infty$, and let \(\Lambda \Subset \R^d\). Then there is a time-stationary probability measure \(W^{(\Lambda)}\) on \(D([0, \infty); \Omega_\Lambda)\) with time marginals \(W^{(\Lambda)}_t = \mu_\Lambda\), \(t \geq 0\), such that
    \begin{equation}\label{equation:contraction_time_marginal_minimum_entropy_problem_hellinger_finite_volume_attainment}
        \Ent\big(W^{(\Lambda)}_{[0, T]} \big\vert \mu_\Lambda Q^{\Lambda,\beta}\big)
        = 2T \mathfrak{H}^2(B_\mu^{(\Lambda)}, D_\mu^{(\Lambda)}), \qquad T > 0.
    \end{equation}
\end{lemma}

The construction of the process in \zcref{lemma:contraction_time_marginal_minimum_entropy_problem_hellinger_finite_volume_attainment} is guided by the optimal birth and death rates~\eqref{equation:contraction_time_marginal_entropy_for_B_D_process_optimal_rates} for regular \(\mu\). In that case, the jump kernel \(k_\Lambda(\eta, \d \zeta)\) would be given by
\begin{equation*}
    \int k_\Lambda(\eta, \d\zeta) f(\zeta)
    = \int_\Lambda \sqrt{b(x, \eta) \kappa_\Lambda(x, \eta)}\, f(\eta + \delta_x)  \, \d x
    + \sum_{x \in \eta_\Lambda}  \sqrt{\frac{b(x, \eta - \delta_x)}{\kappa_\Lambda(x, \eta - \delta_x)}}\, f(\eta - \delta_x),
\end{equation*}
where \(\kappa_\Lambda\) is the Papangelou intensity of \(\mu_\Lambda\).
In terms of
\begin{equation*}
    p_\Lambda
    = \frac{\d B^{(\Lambda)}_\mu}{\d(B^{(\Lambda)}_\mu + D^{(\Lambda)}_\mu)},
\end{equation*}
we have \(\kappa_\Lambda/b_\Lambda = (1 - p_\Lambda)/p_\Lambda\), where $b_\Lambda(x,\eta)=b(x,\eta_\Lambda)$, since \(b_\Lambda/\kappa_\Lambda = \d B^{(\Lambda)}_\mu/\d D^{(\Lambda)}_\mu\). Consequently, the right-hand side of
\begin{equation*}
    \int k_\Lambda(\eta, \d\zeta) f(\zeta)
    = \int_\Lambda b(x, \eta) \sqrt{\frac{1 - p_\Lambda(\eta, x)}{p_\Lambda(\eta, x)}}\, f(\eta + \delta_x)  \, \d x
    + \sum_{x \in \eta_\Lambda}  \sqrt{\frac{p_\Lambda(\eta - \delta_x, x)}{1- p_\Lambda(\eta - \delta_x, x)}}\, f(\eta - \delta_x)
\end{equation*}
makes sense even if \(\mu_\Lambda\) has no Papangelou intensity. The resulting jump kernel satisfies the detailed balance condition \(\mu_\Lambda(\d \eta) k_\Lambda(\eta, \d\zeta) = \mu_\Lambda(\d\zeta) k_\Lambda(\zeta, \d\eta)\). Indeed,
\begin{equation}\label{equation:contraction_time_marginal_minimum_entropy_problem_hellinger_finite_volume_attainment_symmetry_density_calculation}
    \sqrt{\frac{1- p_\Lambda}{p_\Lambda}}\, B^{(\Lambda)}_\mu
    = \sqrt{p_\Lambda(1 - p_\Lambda)}\, (B^{(\Lambda)}_\mu + D^{(\Lambda)}_\mu)
    = \sqrt{B^{(\Lambda)}_\mu  D^{(\Lambda)}_\mu},
    \qquad \sqrt{\frac{p_\Lambda}{1 - p_\Lambda}}\, D^{(\Lambda)}_\mu
    = \sqrt{B^{(\Lambda)}_\mu  D^{(\Lambda)}_\mu},
\end{equation}
and hence
\begin{equation}\label{equation:contraction_time_marginal_minimum_entropy_problem_hellinger_finite_volume_attainment_symmetry_calculation}
\begin{split}
    \int \mu_\Lambda(\d\eta) \int k_\Lambda(\eta, \d\zeta) f(\eta, \zeta)
    &= \int B_\mu^{(\Lambda)}(\d\eta, \d x)\, f(\eta, \eta + \delta_x) \sqrt{\frac{1 - p_\Lambda(\eta, x)}{p_\Lambda(\eta, x)}} \\
    &\qquad + \int D_\mu^{(\Lambda)}(\d\eta, \d x)\, f(\eta + \delta_x, \eta) \sqrt{\frac{p_\Lambda(\eta, x)}{1- p_\Lambda(\eta, x)}}  \\
    &= \int \sqrt{B^{(\Lambda)}_\mu  D^{(\Lambda)}_\mu}(\d\eta, \d x)\, \big[f(\eta, \eta + \delta_x) + f(\eta + \delta_x, \eta)\big],
\end{split}
\end{equation}
which is symmetric under exchanging $\eta$ and $\zeta$ in $f(\eta,\zeta)$.

The following technical lemma shows that this gives a well-defined Markov jump process.
\begin{lemma}[The optimal process in finite volume]\label{lemma:optimal-jump-process}
    Let \(\mu_\Lambda\) satisfy \(\mu_\Lambda[\eta(\Lambda)] < \infty\), and fix an everywhere-defined version of \(p_\Lambda\) with values in \([0, 1]\).
    Then there is a non-explosive birth-and-death Markov jump process \(W^{(\Lambda)}\) with values in \(D([0, \infty); \Om_\Lambda)\), started from \(\mu_\Lambda\), whose jump kernel coincides, for \(\mu_\Lambda\)-almost every \(\eta\), with
    \begin{equation}\label{equation:equation:contraction_time_marginal_minimum_entropy_problem_hellinger_finite_volume_attainment_jump_kernel_definition}
        \int k_\Lambda(\eta, \d\zeta) f(\zeta)
        = \int_\Lambda \d x\, b(x, \eta) \sqrt{\frac{1 - p_\Lambda(\eta, x)}{p_\Lambda(\eta, x)}}\, f(\eta + \delta_x)
        + \sum_{x \in \eta_\Lambda}  \sqrt{\frac{p_\Lambda(\eta - \delta_x, x)}{1- p_\Lambda(\eta - \delta_x, x)}}\, f(\eta - \delta_x).
    \end{equation}
    It is time-stationary and reversible.
\end{lemma}
\begin{proof}
    We first give an everywhere-defined version of the jump kernel \(k_\Lambda\) with finite total jump rate everywhere. In \eqref{equation:equation:contraction_time_marginal_minimum_entropy_problem_hellinger_finite_volume_attainment_jump_kernel_definition}, we interpret the square root in the birth term as \(\infty\) if \(p_\Lambda(\eta, x) = 0\), and the square root in the death term as \(\infty\) if \(p_\Lambda(\eta - \delta_x, x) = 1\). The concrete values assigned here are irrelevant, since \(p_\Lambda > 0\) holds $B^{(\Lambda)}_\mu$-almost everywhere and \(p_\Lambda < 1\) holds $D^{(\Lambda)}_\mu$-almost everywhere; we only need an everywhere-defined measurable version of these terms. Let \(\widetilde{k}_\Lambda\) be defined by the right-hand side of \eqref{equation:equation:contraction_time_marginal_minimum_entropy_problem_hellinger_finite_volume_attainment_jump_kernel_definition}, and set
    \begin{equation*}
        k_\Lambda(\eta, \cdot)
        = \begin{cases}
            \widetilde{k}_\Lambda(\eta, \cdot), &\widetilde{k}_\Lambda(\eta, \Omega_\Lambda) < \infty, \\
            0, &\text{ otherwise}.
        \end{cases}
    \end{equation*}
    Then \(\mu_\Lambda(\d\eta) k_\Lambda(\eta, \cdot) = \mu_\Lambda(\d\eta) \widetilde{k}_\Lambda(\eta, \cdot) \), since, by \eqref{equation:contraction_time_marginal_minimum_entropy_problem_hellinger_finite_volume_attainment_symmetry_calculation},
\begin{equation}\label{equation:contraction_time_marginal_minimum_entropy_problem_hellinger_finite_volume_attainment_finiteness_jump_expectation}
         \mu_\Lambda\big[\widetilde k_\Lambda(\eta, \Omega_\Lambda)\big]
        = 2\sqrt{B^{(\Lambda)}_\mu D^{(\Lambda)}_\mu}[1]
        \leq  \Vert b\Vert_\infty \abs{\Lambda} + \mu_\Lambda[\eta(\Lambda)]
        < \infty.
    \end{equation}

    For a jump-hold construction of the process, see~\cite[Theorem~3.2.1, Proposition~4.3.4, and Section~7.2]{Jacobsen2006}. Alternatively, one can use a graphical construction with a Poisson random measure as in~\cite{JKSZ25}. Since the jump rates are possibly unbounded, both constructions only work up to an explosion time \(\tau_\infty\), which is a priori not necessarily infinite.
    Let \(W^{(\Lambda)}\) be the law of the Markov jump process killed at explosion. We now verify non-explosion; reversibility and hence time-stationarity then follow from the detailed balance~\eqref{equation:contraction_time_marginal_minimum_entropy_problem_hellinger_finite_volume_attainment_symmetry_calculation}.
  Writing $q_t$ for the transition kernel of the killed process, we have
$\mu_\Lambda(\mathrm{d}\eta)\,q_t(\eta,\mathrm{d}\zeta)
=\mu_\Lambda(\mathrm{d}\zeta)\,q_t(\zeta,\mathrm{d}\eta)$. Indeed, by the jump-hold
construction \cite[Section~7.2]{Jacobsen2006}, $q_t$ is the increasing limit of the kernels
obtained by allowing at most $N$ jumps in $[0,t]$, and each of these is
$\mu_\Lambda$-symmetric by \eqref{equation:contraction_time_marginal_minimum_entropy_problem_hellinger_finite_volume_attainment_symmetry_calculation}, since reversing the sequence of
visited states and holding times preserves their joint weight. Together with
\eqref{equation:contraction_time_marginal_minimum_entropy_problem_hellinger_finite_volume_attainment_finiteness_jump_expectation}, we get
    % Writing \(q_t\) for the transition kernel of the killed process, the detailed balance gives \(\mu_\Lambda(\d\eta) q_t(\eta, \d\zeta) = \mu_\Lambda(\d\zeta) q_t(\zeta, \d\eta)\), and together with \eqref{equation:contraction_time_marginal_minimum_entropy_problem_hellinger_finite_volume_attainment_finiteness_jump_expectation} we get
    \begin{equation*}
    \begin{split}
        W^{(\Lambda)}\Big[\int_0^{T \land \tau_\infty}\d t\,  k_\Lambda(\omega_t, \Omega_\Lambda)\Big]
        &= \int_0^T  \d t\, W^{(\Lambda)}\big[k_\Lambda(\omega_t, \Omega_\Lambda) \1_{\{\tau_\infty > t\}}\big] \\
        &= \int_0^T  \d t\,  \int \mu_\Lambda(\d\eta) \int q_t(\eta, \d\zeta)\, k_\Lambda(\zeta, \Omega_\Lambda) \\
        &= \int_0^T \d t\,   \int \mu_\Lambda(\d\eta)\, k_\Lambda(\eta, \Omega_\Lambda) \int q_t(\eta, \d\zeta) 
        \leq T \mu_\Lambda[k_\Lambda(\eta, \Omega_\Lambda)]
        < \infty.
    \end{split}
    \end{equation*}
    By jump compensation, the total number of jumps up to time \(T\) is therefore finite, which shows non-explosion.
\end{proof}

\begin{remark}
    Let us illustrate the construction of the process \(W^{(\Lambda)}\). The measure \(\mu_\Lambda\) has a Papangelou intensity if and only if \(D_\mu^{(\Lambda)} \ll B_\mu^{(\Lambda)}\). In our construction, this absolute continuity is not needed, and \(W^{(\Lambda)}\) is \(\mu_\Lambda\)-almost surely constant if and only if \(B_\mu^{(\Lambda)} \perp D_\mu^{(\Lambda)}\). Indeed, for every \(T > 0\) and measurable \(f \geq 0\), jump compensation, stationarity and \eqref{equation:contraction_time_marginal_minimum_entropy_problem_hellinger_finite_volume_attainment_symmetry_calculation} give
    \begin{equation*}
    \begin{split}
        W^{(\Lambda)}\Big[\sum_{0 < t \leq T\colon \omega_{t-} \neq \omega_t} f(\omega_{t-}, \omega_t) \Big]
        &= T \int \mu_\Lambda(\d\eta) \int k_\Lambda(\eta, \d\zeta) f(\eta, \zeta) \\
        &= T \int \sqrt{B^{(\Lambda)}_\mu D^{(\Lambda)}_\mu}(\d\eta, \d x)[f(\eta, \eta + \delta_x) + f(\eta + \delta_x, \eta)],
    \end{split}
    \end{equation*}
    and \(\sqrt{B^{(\Lambda)}_\mu D^{(\Lambda)}_\mu} = 0\) if and only if \(B^{(\Lambda)}_\mu \perp D^{(\Lambda)}_\mu\). In particular, as in the formal calculation with Papangelou intensities, the birth measure \(\mu_\Lambda(\d\eta) \, q(x, \eta) \, \1_\Lambda(x) \, \d x\) with the new birth rates \(q(x, \eta) = b(x, \eta) \sqrt{(1 - p_\Lambda(\eta, x))/p_\Lambda(\eta, x)}\) is the geometric mean \(\sqrt{B^{(\Lambda)}_\mu D^{(\Lambda)}_\mu}\), which minimises the dynamical entropy with respect to the process started in \(\mu_\Lambda\) with birth rates \(b(x, \eta_\Lambda)\) and death rate \(1\), and the corresponding death measure is then essentially determined by the requirement of time-stationarity, or even reversibility.

    We compute the minimising birth and death rates in the following concrete example of a measure \(\mu_\Lambda\) with \(B_\mu^{(\Lambda)} \not\perp D_\mu^{(\Lambda)}\) but \(D_\mu^{(\Lambda)} \not\ll B_\mu^{(\Lambda)}\). For simplicity of the formulas, we consider the free case \(b \equiv 1\), that is, $\beta=0$, and set
    \(\mu_\Lambda = \pi_\Lambda(\cdot \lvert \eta(\Lambda) \geq 1)\),
    so that \(\mu_\Lambda\) is a Poisson point process in \(\Lambda\) conditioned to be non-empty. Then, by Mecke's formula,
    \begin{equation*}
    \begin{split}
        B_\mu^{(\Lambda)}(\d\eta, \d x) &= \mu_\Lambda(\d\eta) \, \1_\Lambda(x)\,\d x, \\
        D_\mu^{(\Lambda)}(\d\eta, \d x) &= \frac{1}{1- \e^{-\abs{\Lambda}}}\, \pi_\Lambda(\d\eta) \, \1_\Lambda(x)\,\d x
        = B_\mu^{(\Lambda)}(\d\eta, \d x) + \frac{1}{\e^{\abs{\Lambda}} - 1}\, \delta_\emptyset(\d\eta) \, \1_\Lambda(x)\,\d x.
    \end{split}
    \end{equation*}
    Since \(\mu_\Lambda\) does not charge the empty configuration, the last term is singular with respect to \(B_\mu^{(\Lambda)}\), so that \(D_\mu^{(\Lambda)} \not\ll B_\mu^{(\Lambda)}\), and \(\mu_\Lambda\) has no Papangelou intensity. Nevertheless, \(\sqrt{B_\mu^{(\Lambda)} D_\mu^{(\Lambda)}} = B_\mu^{(\Lambda)} \neq 0\).
    It follows that, on the set of non-empty configurations, which has \(\mu_\Lambda\)-probability one, the process \(W^{(\Lambda)}\) has birth rate \(q(x, \eta) = 1\) and death rate \(d(x, \eta) = \1_{\{\eta(\Lambda) \geq 2\}}\).
\end{remark}

We now check that the finite-volume minimum is attained by the Markov process \(W^{(\Lambda)}\).
\begin{proof}[Proof of \zcref{lemma:contraction_time_marginal_minimum_entropy_problem_hellinger_finite_volume_attainment}]
    Let $W^{(\Lambda)}$ be the process from \zcref{lemma:optimal-jump-process}, which is time-stationary with time marginals $\mu_\Lambda$.
    To compute the relative entropy \(\Ent(W^{(\Lambda)}_{[0, T]} \vert \mu_\Lambda Q^{\Lambda,\beta})\), we express the density\linebreak \(\d W^{(\Lambda)}_{[0, T]}/\d(\mu_\Lambda Q^{\Lambda,\beta})\) in terms of the jump kernel \(k^b_\Lambda\) of the reference process \(\mu_\Lambda Q^{\Lambda,\beta}\), which is given by
    \begin{equation*}
        \int k^b_\Lambda(\eta, \d\zeta) f(\zeta)
        = \int_\Lambda b(x, \eta) f(\eta + \delta_x)  \, \d x
        + \sum_{x \in \eta_\Lambda}  f(\eta - \delta_x).
    \end{equation*}
    A martingale computation similar to the proof of \zcref{proposition:change_of_measure_a_la_girsanov}, now first localised by the total jump rates and using non-explosion, see also~\cite[Theorem~7.9.1 and (7.127)]{Jacobsen2006}, shows that
    \begin{equation*}
        \frac{\d W^{(\Lambda)}_{[0, T]}}{\d(\mu_\Lambda Q^{\Lambda,\beta})}(\omega)
        = \exp\Big(\int_0^T \d t\, (\lambda^b - \lambda)(\omega_t)  \Big) \prod_{0 < t \leq T \colon \omega_t \neq \omega_{t-}} r(\omega_{t-}, \omega_t),
    \end{equation*}
    with the total jump rates \(\lambda(\eta) = k_\Lambda(\eta, \Omega_\Lambda)\) and \(\lambda^b(\eta) = k^b_\Lambda(\eta, \Omega_\Lambda)\), and the rate ratio \(r\) given by
    \begin{equation*}
        r(\eta, \eta + \delta_x)
        = \sqrt{\frac{1 - p_\Lambda(\eta, x)}{p_\Lambda(\eta, x)}} ,
        \qquad r(\eta + \delta_x, \eta) = \sqrt{\frac{p_\Lambda(\eta, x)}{1 - p_\Lambda(\eta, x)}}.
    \end{equation*}
    It follows, by jump compensation and stationarity, that
    \begin{equation}\label{equation:contraction_time_marginal_minimum_entropy_problem_hellinger_finite_volume_attainment_entropy_calculation}
    \begin{split}
        \Ent(W^{(\Lambda)}_{[0, T]} \vert \mu_\Lambda Q^{\Lambda,\beta})
        &= W^{(\Lambda)}\Big[\int_0^T \d t\,  (\lambda^b - \lambda)(\omega_t) \Big] + W^{(\Lambda)}\Big[\sum_{0 < t \leq T \colon \omega_t \neq \omega_{t-}} \log r(\omega_{t-}, \omega_t) \Big] \\
        &= T \mu_\Lambda[\lambda^b - \lambda] + T \mu_\Lambda\Big[\int k_\Lambda(\eta, \d\zeta) \log r(\eta, \zeta) \Big].
    \end{split}
    \end{equation}
    By \eqref{equation:contraction_time_marginal_minimum_entropy_problem_hellinger_finite_volume_attainment_symmetry_calculation}, the first term on the right-hand side of \eqref{equation:contraction_time_marginal_minimum_entropy_problem_hellinger_finite_volume_attainment_entropy_calculation} is
    \begin{equation*}
    \begin{split}
        T \mu_\Lambda[\lambda^b - \lambda]
        &= T \Big(\mu_\Lambda\Big[\int_\Lambda b(x, \eta)\, \d x + \eta(\Lambda)\Big] - (\mu_\Lambda k_\Lambda)[1] \Big) \\
        &= T \Big(B^{(\Lambda)}_\mu[1] + D^{(\Lambda)}_\mu[1] - 2 \int  \sqrt{B^{(\Lambda)}_\mu  D^{(\Lambda)}_\mu}(\d\eta, \d x)  \Big)
        = 2 T \mathfrak{H}^2(B_\mu^{(\Lambda)}, D_\mu^{(\Lambda)}).
    \end{split}
    \end{equation*}
    For the second term, again by \eqref{equation:contraction_time_marginal_minimum_entropy_problem_hellinger_finite_volume_attainment_symmetry_calculation},
    \begin{equation}\label{equation:contraction_time_marginal_minimum_entropy_problem_hellinger_finite_volume_attainment_entropy_calculation_cancellation}
    \begin{split}
        &T \mu_\Lambda\Big[\int k_\Lambda(\eta, \d\zeta) \log r(\eta, \zeta) \Big] = T \int \sqrt{B^{(\Lambda)}_\mu  D^{(\Lambda)}_\mu}(\d\eta, \d x) \Big[\log \sqrt{\tfrac{1 - p_\Lambda(\eta, x)}{p_\Lambda(\eta, x)}} + \log \sqrt{\tfrac{p_\Lambda(\eta, x)}{1 - p_\Lambda(\eta, x)}} \Big]
        = 0,
    \end{split}
    \end{equation}
    so that \eqref{equation:contraction_time_marginal_minimum_entropy_problem_hellinger_finite_volume_attainment} follows. This cancellation is justified, since \(\sqrt{B_\mu^{(\Lambda)}D_\mu^{(\Lambda)}}\) is concentrated on \(\{0 < p_\Lambda < 1\}\), and, by \eqref{equation:contraction_time_marginal_minimum_entropy_problem_hellinger_finite_volume_attainment_symmetry_density_calculation},
    \begin{equation*}
    \begin{split}
        \mu_\Lambda\Big[\int k_\Lambda(\eta, \d\zeta) \abs{\log r(\eta, \zeta)} \Big]
        &=  \int \sqrt{p_\Lambda (1- p_\Lambda)} \Big\vert \log \frac{1- p_\Lambda}{p_\Lambda}\Big\vert \, \d(B^{(\Lambda)}_\mu + D^{(\Lambda)}_\mu) \\
        &\leq \frac{2}{\e} (B^{(\Lambda)}_\mu + D^{(\Lambda)}_\mu)[1]
        \leq \frac{2}{\e} \big(\Vert b\Vert_\infty \abs{\Lambda} + \mu_\Lambda[\eta(\Lambda)]\big)
        < \infty,
    \end{split}
    \end{equation*}
    where we used the scalar inequality
    \begin{equation*}
    \begin{split}
        \sup_{p \in (0,1)} \sqrt{p(1-p)}\, \abs{\log((1-p)/p)}
        &= \sup_{p \in (0,1/2)} \sqrt{p(1-p)}\, \abs{\log((1-p)/p)} \\
        &\leq \sup_{p \in (0,1/2)} \sqrt{p} \log(1/p)
        = 2 \sup_{p \in (0,1/2)} \sqrt{p} \log(1/\sqrt{p})
        \leq 2/\e,
    \end{split}
    \end{equation*}
    as desired. 
\end{proof}

To pass to infinite volume, we use the following bound for the dynamical specific entropy of a stationarised block construction, which is comparable with~\cite[Lemma~5.5]{GeorgiiZessin1993} in the static case.
\begin{lemma}[Entropy of a stationarised block construction]\label{lemma:contraction_time_marginal_minimum_entropy_problem_hellinger_existence_of_minimizer_entropy_bound_block_construction}
    Let \(\Lambda = [-\ell, \ell)^d\) be a half-open cube, and let \(W\) be a time-stationary probability measure on \(D(\R; \Om_\Lambda)\). Let \(P \in \PP(E)\) be the law obtained by tiling \(\R^d\) with shifted copies of \(\Lambda\), sampling independent copies of \(W\) in each of these blocks, and then shifting the origin uniformly in $\Lambda$.
    Then
    \begin{equation*}
        H(P)
        \leq \sup_{T > 0} \frac{\Ent(W_{[0, T]} \lvert W_0 Q^\Lambda)}{T \abs{\Lambda}}.
    \end{equation*}
\end{lemma}
\begin{proof}
    Write
    \begin{equation*}
        P
        = \frac{1}{\abs{\Lambda}} \int_\Lambda \d u \, P^u,
    \end{equation*}
    where \(P^u\) is the law obtained by sampling independent copies of \(W\) in the shifted blocks
    \(C_j^u = u + 2\ell j + \Lambda\), \(j \in \Z^d\).
    The joint convexity of the relative entropy gives, for every cube \(A \subset \R^d\) and \(T > 0\),
    \begin{equation*}
        \Ent(P_{A, T} \lvert P_{A, 0} Q^{A})
        \leq \frac{1}{\abs{\Lambda}} \int_\Lambda \d u\, \Ent(P^u_{A, T} \lvert P^u_{A, 0} Q^{A}).
    \end{equation*}
    For a given cube \(A\), let \(C(A) := \bigcup_{j \in \Z^d \colon C_j^u \cap A \neq \emptyset} C_j^u\). Monotonicity in the volume and independence give
    \begin{equation*}
    \begin{split}
        \Ent(P^u_{A, T} \lvert P^u_{A ,0} Q^{A})
        &\leq \Ent(P^u_{C(A), T} \lvert P^u_{C(A), 0} Q^{C(A)})= \#\{j \in \Z^d \,\colon\, C_j^u \cap A \neq \emptyset \}\, \Ent(W_{[0, T]} \lvert W_0 Q^\Lambda).
    \end{split}
    \end{equation*}
    For $u\in\Lambda$ and $a>0$, let
$N_a(u):=\#\{j\in\mathbb{Z}^d\colon C^u_j\cap(0,a)^d\neq\emptyset\}$. Since the cubes
$C^u_j$ have side length $2\ell$ and those meeting $(0,a)^d$ cover it,
\begin{equation*}
  \Bigl(\frac{a}{2\ell}\Bigr)^d\le N_a(u)\le\Bigl(\frac{a}{2\ell}+2\Bigr)^d,
  \qquad u\in\Lambda,\ a>0,
\end{equation*}
so that $N_a(u)/a^d\to(2\ell)^{-d}=1/|\Lambda|$ as $a\to\infty$, boundedly in $u$
for $a\ge1$. Combining the two preceding displays with $A=(0,a)^d$ gives
\[
  \frac{h_{(0,a)^d,T}(P)}{Ta^d}
  \le\frac{\mathrm{Ent}(W_{[0,T]}|W_0Q^\Lambda)}{T|\Lambda|}
     \int_\Lambda\mathrm{d}u\,\frac{N_a(u)}{a^d},
\]
and, by dominated convergence, the right-hand side converges to
$\mathrm{Ent}(W_{[0,T]}|W_0Q^\Lambda)/(T|\Lambda|)$ as $a\to\infty$.

It remains to identify $H(P)$ with a limit in $a$ for fixed $T$. For $a\ge a_0>0$,
the cube $(0,a)^d$ contains $\lfloor a/a_0\rfloor^d$ disjoint translates of
$(0,a_0)^d$, so \eqref{eq:entropy-spatial-superadditivity} and spatial stationarity give
$h_{(0,a)^d,T}(P)\ge\lfloor a/a_0\rfloor^d\,h_{(0,a_0)^d,T}(P)$, and hence
\[
  \liminf_{a\to\infty}\frac{h_{(0,a)^d,T}(P)}{Ta^d}
  \ge\sup_{a_0>0}\frac{h_{(0,a_0)^d,T}(P)}{Ta_0^d}.
\]
Taking the supremum over $T>0$ and using the definition \eqref{eq:rate} of $H$, we
obtain
\[
  H(P)\le\sup_{T>0}\liminf_{a\to\infty}\frac{h_{(0,a)^d,T}(P)}{Ta^d}
  \le\sup_{T>0}\frac{\mathrm{Ent}(W_{[0,T]}|W_0Q^\Lambda)}{T|\Lambda|},
\]
as desired.
    % \newpage
    % Using \zcref{lem:entropy-superadditivity}, it follows that
    % \begin{equation*}
    % \begin{split}
    %     H(P)
    %     &= \sup_{T > 0} \liminf_{a \to \infty} \frac{h_{(0, a)^d, T}(P)}{Ta^d} \\
    %     &\leq \sup_{T > 0}  \frac{\Ent(W_{[0, T]} \lvert W_0 Q^\Lambda) }{T \abs{\Lambda}} \int_\Lambda\d u \, \liminf_{a \to \infty} \frac{\#\{j \in \Z^d \colon C_j^u \cap (0,a)^d \neq \emptyset \}}{a^d} =  \sup_{T > 0}  \frac{\Ent(W_{[0, T]} \lvert W_0 Q^\Lambda) }{T \abs{\Lambda}},
    % \end{split}
    % \end{equation*}
    % as desired.
\end{proof}

We are now in a position to construct the minimiser in infinite volume.
\begin{proof}[Proof of \zcref{lemma:contraction_time_marginal_minimum_entropy_problem_hellinger_existence_of_minimizer_in_infinite_volume}]
    Recall the process \(W^{(\Lambda_n)}\) constructed in \zcref{lemma:contraction_time_marginal_minimum_entropy_problem_hellinger_finite_volume_attainment}, and denote, by abuse of notation, also by \(W^{(\Lambda_n)}\) its time-stationary extension to the time interval $\R$.
    For \(n \in \N\), we tile \(\R^d\) with copies of \(\Lambda_n\), and define \(P^{(n)}\) by first placing independent copies of \(W^{(\Lambda_n)}\) into these and then shifting the origin uniformly in \(\Lambda_n\), so that \(P^{(n)}\in\PP(E)\).

    We show that \(\abs{\Lambda_n} H^\beta(P^{(n)})\) is bounded above by \(\Ent(W^{(\Lambda_n)}_{[0, 1]} \lvert  \mu_{\Lambda_n} Q^{\Lambda_n,\beta})\) plus a boundary-order term, that is,
    \begin{equation}\label{equation:contraction_time_marginal_minimum_entropy_problem_hellinger_existence_of_minimizer_in_infinite_volume_entropy_comparison}
        H^\beta(P^{(n)})
        \leq \frac{\Ent(W^{(\Lambda_n)}_{[0, 1]} \lvert  \mu_{\Lambda_n} Q^{\Lambda_n,\beta})}{\abs{\Lambda_n}} + c_{b, \mu} \frac{\abs{\partial^\range \Lambda_n}}{\abs{\Lambda_n}}
    \end{equation}
    for some constant \(c_{b, \mu} \in (0, \infty)\). By \zcref{lemma:contraction_time_marginal_minimum_entropy_problem_hellinger_finite_volume_attainment} and \zcref{lemma:contraction_time_marginal_minimum_entropy_problem_hellinger_thermodynamic_limit}, the right-hand side is bounded in $n$. The compactness of the sublevel sets of \(H^\beta\) therefore implies that there are a space-time stationary \(P\) and a subsequence \((P^{(n_k)})_{k \in \N}\) with \(P^{(n_k)} \to P\) weakly as $k\to\infty$. The limit \(P\) has time marginal \(P_0 = \mu\), by the continuity of \(\widetilde{P} \mapsto \widetilde{P}_0\) on \(\PP(E)\) and the weak convergence \((P^{(n)})_0 \to \mu\), which holds since a bounded observation window lies in one of the shifted copies of $\Lambda_n$ with probability tending to one as \(n \to \infty\).
    The lower semicontinuity of \(H^\beta\), the estimate \eqref{equation:contraction_time_marginal_minimum_entropy_problem_hellinger_existence_of_minimizer_in_infinite_volume_entropy_comparison}, \zcref{lemma:contraction_time_marginal_minimum_entropy_problem_hellinger_finite_volume_attainment} and \zcref{lemma:contraction_time_marginal_minimum_entropy_problem_hellinger_thermodynamic_limit} give
    \begin{equation*}
    \begin{split}
        H^\beta(P)
        &\leq \liminf_{k \to \infty} H^\beta(P^{(n_k)})
        \leq \liminf_{k \to \infty} \bigg( \frac{\Ent(W^{(\Lambda_{n_k})}_{[0, 1]} \lvert  \mu_{\Lambda_{n_k}} Q^{\Lambda_{n_k},\beta})}{\abs{\Lambda_{n_k}}} + c_{b, \mu} \frac{\abs{\partial^\range \Lambda_{n_k}}}{\abs{\Lambda_{n_k}}} \bigg) \\
        &= \liminf_{k \to \infty} \frac{2\mathfrak{H}^2(B_\mu^{(\Lambda_{n_k})}, D_\mu^{(\Lambda_{n_k})}) }{\abs{\Lambda_{n_k}}}
        = 2\mathfrak{H}^2(B_\mu, D_\mu),
    \end{split}
    \end{equation*}
    and the reverse inequality is the general lower bound \eqref{equation:contraction_time_marginal_minimum_entropy_problem_hellinger_infinite_volume_lower_bound}. This proves \eqref{equation:contraction_time_marginal_minimum_entropy_problem_hellinger_existence_of_minimizer_in_infinite_volume}.

    It remains to prove the estimate \eqref{equation:contraction_time_marginal_minimum_entropy_problem_hellinger_existence_of_minimizer_in_infinite_volume_entropy_comparison}. For the free evolution as reference measure, \zcref{lemma:contraction_time_marginal_minimum_entropy_problem_hellinger_existence_of_minimizer_entropy_bound_block_construction} and the linearity in \(T\) of the relative entropy of the stationary Markov process $W^{(\Lambda_n)}$, which follows from the chain rule, give
    \begin{equation}\label{equation:contraction_time_marginal_minimum_entropy_problem_hellinger_existence_of_minimizer_in_infinite_volume_entropy_comparison_iid_blocks}
        H(P^{(n)})
        \leq \sup_{T > 0} \frac{\Ent(W^{(\Lambda_n)}_{[0, T]} \lvert  \mu_{\Lambda_n} Q^{\Lambda_n})}{T\abs{\Lambda_n}} = \frac{\Ent(W^{(\Lambda_n)}_{[0, 1]} \lvert  \mu_{\Lambda_n} Q^{\Lambda_n})}{\abs{\Lambda_n}}.
    \end{equation}
    For the interaction term \(P^{(n)}[F]\) in \(H^\beta(P^{(n)})\), we show below that
\begin{equation}\label{equation:contraction_time_marginal_minimum_entropy_problem_hellinger_existence_of_minimizer_in_infinite_volume_entropy_comparison_2}
    \begin{split}
        &\Big\vert P^{(n)}[F] - \frac{1}{\abs{\Lambda_n}}\mu_{\Lambda_n}\Big[\int_{\Lambda_n}\d x\, [1-b_{\Lambda_n}(x, \eta)] \Big] - \frac{1}{\abs{\Lambda_n}}W^{(\Lambda_n)}\Big[ \sum_{\substack{(t, x) \text{ birth}\colon (t, x) \in (0, 1] \times \Lambda_n}} \hspace{-1cm}\log b_{\Lambda_n}(x, \omega_{t-}) \Big] \Big\vert \\
        &\leq \Big\vert P^{(n)}\big[1-b(0, \omega_{0})\big] - \frac{1}{\abs{\Lambda_n}}\mu_{\Lambda_n}\Big[\int_{\Lambda_n}\d x\, [1-b_{\Lambda_n}(x, \eta)] \Big]\Big\vert  + \Big\vert P^{(n)}\Big[\sum_{\substack{(t, x) \text{ birth}\colon (t, x) \in (0, 1] \times [0, 1)^d}} \hspace{-1cm}\log b(x, \omega_{t-}) \Big] \\
        &\qquad\qquad -  \frac{1}{\abs{\Lambda_n}}W^{(\Lambda_n)}\Big[ \sum_{\substack{(t, x) \text{ birth}\colon (t, x) \in (0, 1] \times \Lambda_n}}\hspace{-1cm} \log b_{\Lambda_n}(x, \omega_{t-}) \Big] \Big\vert \lesssim_{b, \mu} \frac{\abs{\partial^\range \Lambda_n}}{\abs{\Lambda_n}},
    \end{split}
    \end{equation}
    where $b_{\Lambda_n}(x,\eta)=b(x,\eta_{\Lambda_n})$ is the birth rate with empty boundary configuration. Then \eqref{equation:contraction_time_marginal_minimum_entropy_problem_hellinger_existence_of_minimizer_in_infinite_volume_entropy_comparison} follows from
    \begin{equation*}
    \begin{split}
        H^\beta(P^{(n)})
        &= H(P^{(n)}) - P^{(n)}[F] \\
        &\leq \frac{1}{\abs{\Lambda_n}}\bigg(\Ent(W^{(\Lambda_n)}_{[0, 1]} \lvert  \mu_{\Lambda_n} Q^{\Lambda_n}) - \mu_{\Lambda_n}\Big[\int_{\Lambda_n} [1-b_{\Lambda_n}(x, \eta)] \, \d x \Big] \\
        &\qquad\qquad - W^{(\Lambda_n)}\Big[ \sum_{\substack{(t, x) \text{ birth}\colon (t, x) \in (0, 1] \times \Lambda_n}} \log b_{\Lambda_n}(x, \omega_{t-}) \Big]\bigg) + c_{b, \mu} \frac{\abs{\partial^\range \Lambda_n}}{\abs{\Lambda_n}} \\
        &= \frac{\Ent(W^{(\Lambda_n)}_{[0, 1]} \lvert  \mu_{\Lambda_n} Q^{\Lambda_n,\beta})}{\abs{\Lambda_n}} + c_{b, \mu} \frac{\abs{\partial^\range \Lambda_n}}{\abs{\Lambda_n}},
    \end{split}
    \end{equation*}
    where we used for the last equality a change of measure as in \eqref{equation:interacting_case_conditional_entropy_rep_finite_volume}, and stationarity.

    Let us finally establish the bound \eqref{equation:contraction_time_marginal_minimum_entropy_problem_hellinger_existence_of_minimizer_in_infinite_volume_entropy_comparison_2}. We only explain how to obtain a bound of boundary order for the second term, since the first one is similar but easier.
    To compare the two expressions in the second term, we represent both on the same probability space. Let \(U\) be the uniform shift of the origin in \(\Lambda_n\), and write \(\mathbb{E}\) for the expectation with respect to \(U\) and the independent copies of \(W^{(\Lambda_n)}\), so that
    \begin{equation*}
        P^{(n)}\Big[\sum_{\substack{(t, x) \text{ birth}\colon (t, x) \in (0, 1] \times [0, 1)^d}} \log b(x, \omega_{t-}) \Big]
        = \mathbb{E}\Big[\sum_{\substack{(t, x) \text{ birth}\colon (t, x) \in (0, 1] \times [0, 1)^d}} \log b(x, \omega_{t-})  \Big]
    \end{equation*}
    and
    \begin{equation*}
        \frac{1}{\abs{\Lambda_n}} W^{(\Lambda_n)}\Big[\sum_{\substack{(t, x) \text{ birth}\colon (t, x) \in (0, 1] \times \Lambda_n}} \log b_{\Lambda_n}(x, \omega_{t-}) \Big]
        = \mathbb{E}\Big[\sum_{\substack{(t, x) \text{ birth}\colon (t, x) \in (0, 1] \times [0, 1)^d}} \log b_{U}(x, \omega_{t-})  \Big],
    \end{equation*}
    where \(b_U(x, \eta)\) denotes the birth rate with empty boundary configuration in the shifted copy of \(\Lambda_n\) that contains \(x\).
    By the finite range of \(b\),
    \begin{equation*}
    \begin{split}
    \mathbb{E}\Big[\sum_{\substack{(t, x) \text{ birth}\colon\\ (t, x) \in (0, 1] \times [0, 1)^d}} \abs{\log b(x, \omega_{t-}) - \log b_U(x, \omega_{t-})} \Big] 
    &\leq 2\Vert \log b\Vert_\infty\, \mathbb{E}\Big[\sum_{\substack{(t, x) \text{ birth}\colon\\ (t, x) \in (0, 1] \times [0, 1)^d}} \1_{\bigcup_{j \in \Z^d} (U + jn + \partial^\range \Lambda_n)}(x) \Big] \\
        &= \frac{2 \Vert\log b\Vert_\infty}{\abs{\Lambda_n}}\, W^{(\Lambda_n)}\Big[\sum_{\substack{(t, x) \text{ birth}\colon (t, x) \in (0, 1] \times \Lambda_n}} \1_{ \partial^\range \Lambda_n}(x) \Big].
    \end{split}
    \end{equation*}
    It remains to note that, for every \(A \subset \Lambda_n\), jump compensation, stationarity and \eqref{equation:contraction_time_marginal_minimum_entropy_problem_hellinger_finite_volume_attainment_symmetry_density_calculation} give
    \begin{equation*}
    \begin{split}
        W^{(\Lambda_n)}\Big[ \sum_{\substack{(t, x) \text{ birth}\colon (t, x) \in (0, 1] \times A}} 1 \Big]&= \mu_{\Lambda_n}\Big[\int_A b_{\Lambda_n}(x, \eta) \sqrt{\frac{1- p_{\Lambda_n}(\eta, x)}{p_{\Lambda_n}(\eta, x)}}\, \d x   \Big]
        = \sqrt{B_\mu^{(\Lambda_n)} D_\mu^{(\Lambda_n)}}(\Omega_{\Lambda_n} \times A) \\
        &\leq \sqrt{B_\mu^{(\Lambda_n)}(\Omega_{\Lambda_n} \times A) } \sqrt{D_\mu^{(\Lambda_n)}(\Omega_{\Lambda_n} \times A) }
        \leq \sqrt{\Vert b\Vert_\infty \abs{A}} \sqrt{\rho(\mu)\abs{A}}.
    \end{split}
    \end{equation*}
    For \(A = \partial^\range \Lambda_n\), this shows that the right-hand side of the previous display is at most
    \begin{equation*}
        2\Vert \log b\Vert_\infty \sqrt{\Vert b\Vert_\infty \rho(\mu)}\, \frac{\abs{\partial^\range \Lambda_n}}{\abs{\Lambda_n}},
    \end{equation*}
    which is of boundary order.
\end{proof}

\section{Weak large deviation principles and counterexamples}\label{section:counterexamples}
In this section we prove \zcref{theorem:weak-ldps} and show that, already for the free evolution, its weak large deviation principles in~(\ref{item:weak-ldp-simple}) and~(\ref{item:weak-ldp-Tocc}) cannot be strengthened to full ones. The upper bounds are soft consequences of \zcref{theorem:full-ldp,thm:marg-ldp}, and so are the lower bounds on simple configurations, since $\PP(E_{\rm s})$ and $\PP(\Oms)$ carry the topologies induced by the weak topologies on $\PP(E)$ and $\PP(\Om)$. The lower bound in $\mathcal T_{\rm occ}$ for the free evolution requires us to revisit the construction of \zcref{section:proof-lower-bound}, since the intensities of the stationary block law $\bar P$ and of the empirical process have to be controlled as well. The transition to $\beta\neq0$ and to the time marginals then follows the arguments of \zcref{subsection:strategy-interacting,subsection:strategy-contraction}. Throughout, we use the notation of \zcref{subsection:strategy-lower-bound,section:proof-lower-bound}, and we identify a law on $E_{\rm s}$ with the law on $E$ that it induces, which is supported by $E_{\rm s}$.

\subsection{Upper bounds}\label{subsection:weak-upper-bounds}
Let $\beta\in\R$ and $\xi\in\Omst$. In the graphical construction of \zcref{section:appendix-graphical}, every birth occurs at an atom of a Poisson random measure, so that, almost surely, the locations of all births are pairwise distinct and distinct from the points of $\xi$. Hence, if $\xi\in\Omsst$, then almost surely the configurations $\omega_t$ and $\omega_{t-}$ are simple for all $t\ge0$, and the same holds for the periodised path $\omega^{(n)}$, so that $R_n\in\PP(E_{\rm s})$ and $(R_n)_0\in\PP(\Oms)$. Moreover, by~\eqref{eq:rho} below, $\rho(R_n)=\rho((R_n)_0)$ is the total occupation time $A_n=\int_0^n\d t\, \omega_t(\Lambda_n)$ divided by $n\abs{\Lambda_n}$, which is finite almost surely, so that $R_n\in\PPP(E)$ and $(R_n)_0\in\PPP(\Om)$.

The canonical embedding of $\PP(E_{\rm s})$ into $\PP(E)$ is continuous for the weak topologies, because $E_{\rm s}$ carries the subspace topology, so that bounded continuous functions on $E$ restrict to bounded continuous functions on $E_{\rm s}$. The embedding of $(\PPP(E),\mathcal T_{\rm occ})$ into $\PP(E)$ is continuous, since $\mathcal T_{\rm occ}$ is finer than the weak topology. Hence both embeddings map compact sets to weakly compact sets, and the upper bound of \zcref{theorem:full-ldp}, applied to these images, yields the upper bounds in \zcref{theorem:weak-ldps}(\ref{item:weak-ldp-simple}) and~(\ref{item:weak-ldp-Tocc}). By the same continuity, $H^\beta$ is lower semicontinuous on $\PP(E_{\rm s})$ and on $(\PPP(E),\mathcal T_{\rm occ})$, see \zcref{lemma:interacting_case_rate_function_further_properties}, so that it is a rate function there. In the same way, the embeddings of $\PP(\Oms)$ and of $\PPP(\Om)$ into $\PP(\Om)$ are continuous, and \zcref{thm:marg-ldp} yields the upper bounds in \zcref{theorem:weak-ldps}(\ref{item:weak-ldp-marginals}), as well as the lower semicontinuity of $\mathfrak{I}^\beta$ on both spaces.

It remains to prove the lower bounds.

\subsection{The lower bound on simple configurations}\label{subsection:weak-lower-simple}
Let $\beta\in\R$ and $\xi\in\Omsst$. We equip $E_{\rm s}$ with the restriction of the metric $d_E$. Every open subset of $\PP(E_{\rm s})$ is the trace of a weakly open subset of $\PP(E)$, so that the lower bound in \zcref{theorem:weak-ldps}(\ref{item:weak-ldp-simple}) follows from that of \zcref{theorem:full-ldp}. Indeed, let $G\subset\PP(E_{\rm s})$ be open. By \zcref{lemma:open-neighbourhoods}(i), applied with $S=E_{\rm s}$, every $P\in G$ has a neighbourhood in $\PP(E_{\rm s})$ of the form $\{\widetilde P\colon|\widetilde P[f_i]-P[f_i]|<\delta,\ 1\le i\le p\}\subset G$ with $f_i\in\mathrm{BL}_1(E_{\rm s})$, and replacing the functions $f_i$ by extensions in $\mathrm{BL}_1(E)$, as provided by \zcref{lemma:open-neighbourhoods}(ii), defines a weakly open neighbourhood of $P$ in $\PP(E)$ whose elements supported by $E_{\rm s}$ belong to $G$. The union $G'$ of these neighbourhoods is weakly open in $\PP(E)$, contains $G$, and satisfies $G'\cap\PP(E_{\rm s})\subset G$. Since $R_n\in\PP(E_{\rm s})$ almost surely, see \zcref{subsection:weak-upper-bounds}, the lower bound of \zcref{theorem:full-ldp} gives
\begin{equation*}
    \liminf_{n\to\infty}\frac1{n|\Lambda_n|}\log Q^\beta_\xi(R_n\in G)=\liminf_{n\to\infty}\frac1{n|\Lambda_n|}\log Q^\beta_\xi(R_n\in G')\ge-\inf_{G'}H^\beta\ge-\inf_G H^\beta.
\end{equation*}
This argument does not apply in $\mathcal T_{\rm occ}$, which is finer than the weak topology, so that its open sets are in general not traces of weakly open sets.

\subsection{Intensities and the lower bound in \texorpdfstring{$\mathcal T_{\rm occ}$}{T\_occ}}\label{subsection:weak-lower-occupation}
We first collect some facts about intensities. Recall the pairing $\langle\eta,g\rangle=\int\eta(\d x)\,g(x)$ from \zcref{section:appendix-topology}. For $g\in C_c(\R^d)$ and $\omega\in E$, let
\begin{equation*}
    a_g(\omega):=\int_0^1\d t\,\langle\omega_t,g\rangle.
\end{equation*}

\begin{lemma}[Occupation functionals]\label{lem:occupation-functional}
\begin{enumerate}[(i)]
    \item For every $g\in C_c(\R^d)$, the function $a_g$ is finite and continuous on $E$.
    \item If $g\ge0$, then $P[a_g]=\rho(P)\int\d x\, g(x)$ for every $P\in\PP(E)$. Consequently, the intensity $\rho$ is weakly lower semicontinuous on $\PP(E)$.
    \item For every $n\in\N$ and every legal birth-and-death path $\omega\in D(\R_+;\Om)$,
    \begin{equation}\label{eq:rho}
        \rho(R_{n,\omega})=\frac{1}{n\abs{\Lambda_n}}\int_0^n\d t\,\omega_t(\Lambda_n).
    \end{equation}
    In particular, $\rho(R_n)=A_n/(n\abs{\Lambda_n})$, with the total occupation time $A_n$ from \zcref{subsection:probabilistic-control}.
\end{enumerate}
\end{lemma}

\begin{proof}
(i) The function $t\mapsto\langle\omega_t,g\rangle$ is c\`adl\`ag, hence bounded on $[0,1]$, so that $a_g(\omega)$ is finite. Let $\omega^j\to\omega$ in $E$, and choose $T>1$ such that $\omega$ is continuous at $\pm T$. Then there are increasing homeomorphisms $\Psi_j$ of $[-T,T]$ with $\sup_{|t|\le T}|\Psi_j(t)-t|\to0$ and $\sup_{|t|\le T}d_\Omega(\omega^j_{\Psi_j(t)},\omega_t)\to0$. Since $\eta\mapsto\langle\eta,g\rangle$ is continuous and the closure of $\{\omega_t\colon|t|\le T\}$ is compact, it follows that $\sup_{|t|\le T}|\langle\omega^j_{\Psi_j(t)},g\rangle-\langle\omega_t,g\rangle|\to0$. Consequently, $\langle\omega^j_s,g\rangle\to\langle\omega_s,g\rangle$ at every continuity time $s\in(-T,T)$ of $\omega$, hence for almost every $s\in[0,1]$, and the integrands are bounded uniformly in $j$ and $s\in[0,1]$. Dominated convergence gives $a_g(\omega^j)\to a_g(\omega)$, regardless of possible jumps of $\omega$ at the times $0$ and $1$.

(ii) For every $t$, the measure $B\mapsto P[\omega_t(B)]$ is translation invariant, and it equals $\rho(P)\abs{B}$ by stationarity. Hence $P[\langle\omega_t,g\rangle]=\rho(P)\int\d x\, g(x)$ for $g\ge0$, and Tonelli's theorem gives the formula for $P[a_g]$. Choosing $g\ge0$ with $\int\d x\,  g(x)=1$, we obtain $\rho(P)=\sup_{m\in\N}P[a_g\wedge m]$ by monotone convergence. By~(i), each map $P\mapsto P[a_g\wedge m]$ is weakly continuous, so that $\rho$ is weakly lower semicontinuous.

(iii) Since $\omega^{(n)}$ is periodic in space and $\omega^{(n)}_s(\Lambda_n)=\omega_s(\Lambda_n)$ for $0\le s<n$, Fubini's theorem gives $\int_{\Lambda_n}\d x\,(\theta_x\omega^{(n)}_s)([0,1)^d)=\omega_s(\Lambda_n)$. Integrating over $s\in[0,n)$ and dividing by $n\abs{\Lambda_n}$ gives~\eqref{eq:rho}.
\end{proof}

We now turn to the intensity of the block construction. Let $P\in\PP(E)$ with $H(P)<\infty$, so that $\rho(P)<\infty$ by \zcref{lem:finite-entropy-intensities}, and fix the cube $A=(0,a)^d$. Conditioning on $\omega_0(A)\le K$ changes the mean number of particles in $A$ at later times by at most $e_K$, uniformly in time, where
\begin{equation*}
    e_K := \sup_{t\geq 0} \big| P[\omega_t(A)\lvert  \omega_0(A)\leq K]-\rho(P)|A|\big|.
\end{equation*}

\begin{lemma}[Occupation under the block law]\label{lem:lb-density}
Let $P$, $a$, $K$, $\epsilon$ and $T$ be as in \zcref{lemma:block-law}. Then the following assertions hold.
\begin{enumerate}[(i)]
    \item For every $L>0$,
    \begin{equation}\label{eq:lb-uniform-population-mean}
        e_K \leq \frac{ P[\omega_0(A)\1_{\{\omega_0(A)>L\}}] +(L+\rho(P)|A|)P(\omega_0(A)>K) }{P(\omega_0(A)\leq K)},
    \end{equation}
    and consequently, for fixed $a$, $e_K\to0$ as $K\to\infty$.
    \item The block law has finite mean occupation, $\mathcal C_{\mathrm{block}}[\int_0^S\d t\,\omega_t(A)]<\infty$, and the mean occupation during the preparation and clearance phases satisfies
    \begin{equation}\label{eq:lb-connector-bounds}
        \mathcal C_{\mathrm{block}}\left[ \int_0^1\d t\,  \omega_t(A) + \int_{T+1}^{T+2}\d t \, \omega_t(A)\right]
        \leq K+|A|+\frac{\rho(P)|A|}{P(\omega_0(A)\leq K)}.
    \end{equation}
    \item The stationary block law has intensity
    \begin{equation}\label{eq:lb-cycle-density}
        \rho(\bar P) = \frac1{|A|S}\,\mathcal C_{\mathrm{block}}\Big[\int_0^S\d t\, \omega_t(A)\Big],
    \end{equation}
    and
    \begin{equation}\label{eq:lb-density-approximation}
        |\rho(\bar P)-\rho(P)| \le \frac{e_K}{|A|} + \frac{1}{T+2}\Big(\frac{K}{|A|}+1+\frac{\rho(P)}{P(\omega_0(A)\leq K)}+2\rho(P)\Big).
    \end{equation}
\end{enumerate}
\end{lemma}

\begin{proof}
(i) Let $f\ge0$ be measurable with $P[f]<\infty$, and let $L>0$. Splitting according to whether $f\circ\theta_z>L$ and using stationarity, we obtain
\begin{equation}\label{eq:lb-uniform-tail}
    \sup_{z\in\R^d\times\R} P[(f\circ\theta_z)\1_{\{\omega_0(A)> K\}}] \le P[f\1_{\{f>L\}}]+LP(\omega_0(A)>K).
\end{equation}
For $t\ge0$, we write
\begin{equation*}
    P[\omega_t(A)\lvert \omega_0(A)\le K]-\rho(P)|A|=\frac{\rho(P)|A|P(\omega_0(A)>K)-P[\omega_t(A)\1_{\{\omega_0(A)>K\}}]}{P(\omega_0(A)\le K)},
\end{equation*}
and~\eqref{eq:lb-uniform-tail}, applied to $f=\omega_0(A)$ and $z=(0,t)$, gives~\eqref{eq:lb-uniform-population-mean}. Since $P(\omega_0(A)>K)\to0$ as $K\to\infty$, we obtain $\limsup_{K\to\infty}e_K\le P[\omega_0(A)\1_{\{\omega_0(A)>L\}}]$ for every $L>0$, and the right-hand side tends to zero as $L\to\infty$, since $\rho(P)<\infty$.

(ii) During the preparation phase, the particles that are born or die before time $1$ form a Poisson point process that is independent of the configuration at time $1$, so that the tilt at time $1$ does not change their law. Their expected number is $|A|\int_0^1(1-\e^{-(1-t)})\,\d t=\e^{-1}|A|$. Moreover, at most $K$ particles are present at time $1$. Since the preparation phase has duration one, its integrated occupation is bounded by the number of individual particles present during this phase, whose expectation is at most $K+|A|$. During the clearance phase, there are no births, so that its integrated occupation is at most the number of particles present at time $1+T$. Since the smoothing preserves the number of particles in $A$, the expectation of this number is $P[\omega_T(A)\lvert\omega_0(A)\le K]$, which is bounded by $\rho(P)|A|/P(\omega_0(A)\le K)$ by~\eqref{eq:lb-terminal-mean}. This proves~\eqref{eq:lb-connector-bounds}. For the same reasons, the mean occupation during the bulk phase is $\int_0^T\d t\, P[\omega_t(A)\lvert\omega_0(A)\le K]\le T\rho(P)|A|/P(\omega_0(A)\le K)$, so that the total mean occupation is finite.

(iii) The identity~\eqref{eq:lb-cycle-density} follows from the definition~\eqref{eq:lb-stationary-cycle-law} of $\bar P$ by space-time averaging and Fubini's theorem, since the translates of $A$ tile $\R^d$ up to a null set. By the definition of $e_K$, the mean occupation during the bulk phase differs from $\rho(P)|A|T$ by at most $Te_K$, and by~\eqref{eq:lb-connector-bounds}, the mean occupation during the preparation and clearance phases lies between zero and $K+|A|+\rho(P)|A|/P(\omega_0(A)\le K)$. Since $\rho(P)|A|S=\rho(P)|A|T+2\rho(P)|A|$, dividing by $|A|S=|A|(T+2)$ gives~\eqref{eq:lb-density-approximation}.
\end{proof}

Under the auxiliary laws, the intensity of the empirical process concentrates around that of $\bar P$.

\begin{lemma}[Intensity of the empirical process]\label{lem:lb-occupation-concentration}
Let $\xi\in\Omst$, and let $P$, $a$, $K$, $\epsilon$ and $T$ be as in \zcref{lemma:block-law}. Then
\begin{equation}\label{eq:lb-occupation-concentration}
    \rho(R_n)\to \rho(\bar P) \qquad\text{in }\widehat Q_n\text{-probability}.
\end{equation}
\end{lemma}

\begin{proof}
Let $\widehat\omega$ denote the path under $\widehat Q_n$, and let $\zeta^{(i)}$, $i\in I_n$, be the block paths used in its construction. After the initial clearance, the corridor $\Lambda_n\setminus V_n$ and the final time interval are empty, and the tile faces carry no particles. Hence~\eqref{eq:rho} gives
\begin{equation}\label{eq:lb-empirical-occupation}
    \rho(R_{n}) = \frac{1}{n|\Lambda_n|}\int_0^1\d t\, \widehat\omega_t(\Lambda_n) + \frac1{n|\Lambda_n|}\sum_{i\in I_n}\int_0^S\d t\, \zeta_t^{(i)}(A).
\end{equation}
During the initial clearance there are no births, so that the first term is at most $\xi(\Lambda_n)/(n|\Lambda_n|)$, which tends to zero since $\xi\in\Omst$. The block integrals $\int_0^S\d t\,\zeta^{(i)}_t(A)$ are independent and identically distributed, with finite mean by \zcref{lem:lb-density}(ii). Truncating them at a level $L>0$ and using the variance bound for the truncated, independent random variables, we obtain
\begin{equation*}
    \widehat Q_n\Big[\Big| \frac1{|I_n|}\sum_{i\in I_n}\int_0^S\d t\,  \zeta_t^{(i)}(A) - \mathcal C_{\mathrm{block}}\Big[\int_0^S\d t\, \omega_t(A)\Big] \Big|\Big]\leq 2 \mathcal C_{\mathrm{block}}\Big[\Big(\int_0^S\d t\, \omega_t(A)-L\Big)_+\Big] + \frac{L}{\sqrt{|I_n|}}.
\end{equation*}
Letting first $n\to\infty$ and then $L\to\infty$ shows that the left-hand side tends to zero. Since $|I_n|/(n|\Lambda_n|)\to (|A|S)^{-1}$, \eqref{eq:lb-cycle-density} and~\eqref{eq:lb-empirical-occupation} yield~\eqref{eq:lb-occupation-concentration}.
\end{proof}

\begin{proof}[Proof of the lower bound in \zcref{theorem:weak-ldps}(\ref{item:weak-ldp-Tocc}) for $\beta=0$]
Let $\xi\in\Omst$, let $G\subset\PPP(E)$ be $\mathcal T_{\rm occ}$-open, let $P\in G$ with $H(P)<\infty$, and let $\alpha>0$. It suffices to show that the conclusions~(i)--(iii) of \zcref{proposition:change-of-measure-sanov} hold for $G$; the proof of \zcref{thm:ldp-lower-bound-free} then applies verbatim. By the definition of $\mathcal T_{\rm occ}$ and by \zcref{lemma:open-neighbourhoods}(i), there are $f_1,\dots,f_p\in\mathrm{BL}_1(E)$ and $\delta\in(0,1]$ such that every $\widetilde P\in\PPP(E)$ with
\begin{equation*}
    \big| \widetilde P[f_i]-P[f_i]\big|<\delta,\quad 1\leq i\leq p, \qquad\text{and}\qquad \lvert\rho(\widetilde P)-\rho(P)\rvert<\delta
\end{equation*}
belongs to $G$. We choose the parameters as in the proof of \zcref{lemma:G-typical-and-parameters}(b), with two additional requirements: when choosing $K$, we also require $e_K/|A|\le\delta/16$, which is possible by \zcref{lem:lb-density}(i), and when choosing $T$, we also require
\begin{equation*}
    \frac{1}{T+2}\Big(\frac{K}{|A|}+1+\frac{\rho(P)}{P(\omega_0(A)\leq K)}+2\rho(P)\Big)\le\frac{\delta}{16}.
\end{equation*}
Then the entropy bound~(iii) holds, $\sup_{f\in\mathrm{BL}_1(E)}|\bar P[f]-P[f]|<\delta/2$, and $|\rho(\bar P)-\rho(P)|\le\delta/8$ by~\eqref{eq:lb-density-approximation}. By \zcref{lemma:G-typical-and-parameters}(a) and \zcref{lem:lb-occupation-concentration}, with $\widehat Q_n$-probability tending to one, $|R_n[f_i]-\bar P[f_i]|<\delta/2$ for $1\le i\le p$ and $|\rho(R_n)-\rho(\bar P)|<\delta/2$. On this event $R_n\in G$, which proves~(ii), while~(i) holds by \zcref{lemma:auxiliary-entropy}.
\end{proof}

\subsection{Interacting dynamics and time marginals}\label{subsection:weak-interacting-marginals}
It remains to prove the lower bound in~(\ref{item:weak-ldp-Tocc}) for $\beta\neq0$ and the lower bounds for the time marginals.

\begin{proof}[Proof of \zcref{theorem:weak-ldps}]
The upper bounds and the lower semicontinuity of the rate functions were shown in \zcref{subsection:weak-upper-bounds}, the lower bound in~(\ref{item:weak-ldp-simple}) in \zcref{subsection:weak-lower-simple}, and the lower bound in~(\ref{item:weak-ldp-Tocc}) for $\beta=0$ in \zcref{subsection:weak-lower-occupation}.

\emph{Lower bound in~(\ref{item:weak-ldp-Tocc}) for $\beta\neq0$.} By \zcref{lemma:interacting_case_boundary_conditions_play_no_role}, it suffices to consider the periodic dynamics $Q^\beta_n$. Step~2 of the proof of \zcref{theorem:full-ldp} for $\beta\neq0$ applies verbatim: with the tolerance $a>0$ from there, the set $G_j=G\cap\{\widetilde P\colon \widetilde P[F_j]>P[F_j]-a\}$ is open in $(\PPP(E),\mathcal T_{\rm occ})$, because $F_j$ is bounded and continuous on $E$, and the lower bound for $\beta=0$ from \zcref{subsection:weak-lower-occupation} takes the place of \zcref{thm:ldp-lower-bound-free}.

\emph{Lower bounds in~(\ref{item:weak-ldp-marginals}).} First let $\xi\in\Omsst$. Then $(R_n)_0\in\PP(\Oms)$ almost surely, and the argument of \zcref{subsection:weak-lower-simple}, with $(R_n)_0$, $\Oms$, $\Om$, $\mathfrak{I}^\beta$ and \zcref{thm:marg-ldp} in place of $R_n$, $E_{\rm s}$, $E$, $H^\beta$ and \zcref{theorem:full-ldp}, gives the lower bound on $\PP(\Oms)$. Now let $\xi\in\Omst$ and let $G\subset\PPP(\Om)$ be open. The map $P\mapsto P_0$ is continuous from $(\PPP(E),\mathcal T_{\rm occ})$ to $\PPP(\Om)$, because it is weakly continuous and $\rho(P_0)=\rho(P)$. Hence $\{P\in\PPP(E)\colon P_0\in G\}$ is $\mathcal T_{\rm occ}$-open, and the lower bound in~(\ref{item:weak-ldp-Tocc}) yields
\begin{equation*}
    \liminf_{n\to\infty}\frac1{n|\Lambda_n|}\log Q^\beta_\xi((R_n)_0\in G)\ge-\inf_{\mu\in G}\inf\{H^\beta(P)\colon P\in\PPP(E),\,P_0=\mu\}=-\inf_G\mathfrak{I}^\beta,
\end{equation*}
where the last identity holds because $H^\beta(P)<\infty$ implies $H(P)<\infty$ by~\eqref{equation:interacting_case_lower_bound_via_free_entropy}, and hence $\rho(P)<\infty$ by \zcref{lem:finite-entropy-intensities}.
\end{proof}

If $\xi\in\Omsst$, the lower bound in~(\ref{item:weak-ldp-Tocc}) also holds for the laws in $\PPP(E)$ that are supported by $E_{\rm s}$, equipped with the topology induced by $\mathcal T_{\rm occ}$. Indeed, every open subset of this space is the trace of a $\mathcal T_{\rm occ}$-open set, and $R_n\in\PP(E_{\rm s})$ almost surely, so that the estimate at the end of \zcref{subsection:weak-lower-simple} applies, with the lower bound in~(\ref{item:weak-ldp-Tocc}) in place of that of \zcref{theorem:full-ldp}.

For a law $\mu$ with $\rho(\mu)<\infty$ that is supported by simple configurations, the minimum in~\eqref{equation:contraction_rate_function} can also be taken over the laws supported by $E_{\rm s}$. Indeed, if $P_0=\mu$ and $H^\beta(P)<\infty$, then $H(P)<\infty$, so that $P_{A,[0,T]}\ll P_{A,0}Q^A$ for every cube $A$ and every $T>0$, as in the proof of \zcref{lem:finite-entropy-intensities}. Under $P_{A,0}Q^A$, the configurations in $A$ stay simple during $[0,T]$, since $P_{A,0}$ is supported by simple configurations and the free evolution almost surely creates no particles at occupied locations. Stationarity and a countable exhaustion by translated space-time boxes therefore give $P(E_{\rm s})=1$.

\subsection{Counterexamples}\label{subsection:counterexamples}
We now show that the weak large deviation principles in~(\ref{item:weak-ldp-simple}) and~(\ref{item:weak-ldp-Tocc}) of \zcref{theorem:weak-ldps} cannot be strengthened to full ones, already for the free evolution. Throughout this subsection, $\beta=0$. The main tool is the following lemma. A genuine birth or death changes the configuration by a single particle, whereas at a temporal reset of the periodised path all particles alive at time $n$ are replaced at once by the initial configuration. Large downward jumps of the paths charged by the empirical process can therefore only come from the resets, whose frequency vanishes as $n\to\infty$, so that every rate function satisfying the lower bound must be infinite on laws that charge such jumps. Fix $g\in C_c(\R^d)$ with $0\le g\le1$, $g=1$ on $[-2,2]^d$, and $\operatorname{supp}g\subset[-3,3]^d$. The size of a downward jump of the scalar path $t\mapsto\langle\omega_t,g\rangle$ at time $t$ is $\langle\omega_{t-},g\rangle-\langle\omega_t,g\rangle$.

\begin{lemma}[Large downward jumps]\label{lem:periodisation-jumps}
The jump events
\begin{equation*}
\begin{split}
    &\Big\{\omega\colon \langle\omega_{t-},g\rangle-\langle\omega_t,g\rangle>\tfrac32\text{ for some } t\in(-1,1)\Big\}\text{ and }\\
    &\Big\{\omega\colon \langle\omega_{t-},g\rangle-\langle\omega_t,g\rangle\ge2\text{ for some } t\in[-\tfrac12,\tfrac12]\Big\}
\end{split}
\end{equation*}
are open and closed, respectively, in $E$ and in $E_{\rm s}$. Let $\xi\in\Omst$, and suppose that the large deviation lower bound under $Q_\xi$ at speed $n|\Lambda_n|$ holds with a rate function $I$, either on $\PP(E_{\rm s})$ with the weak topology, where $\xi\in\Omsst$, or on $(\PPP(E),\mathcal T_{\rm occ})$, or on the subset of laws in $\PPP(E)$ supported by $E_{\rm s}$, where $\xi\in\Omsst$. Then $I(P)=\infty$ for every $P$ in the respective space that assigns positive probability to the open jump event.
\end{lemma}

\begin{proof}
\emph{The jump events.} Suppose that $\omega^j\to\omega$ in $E$, and choose $T>1$ such that $\omega$ is continuous at $\pm T$. As in the proof of \zcref{lem:occupation-functional}(i), there are increasing homeomorphisms $\Psi_j$ of $[-T,T]$ and $\varepsilon_j\to0$ such that
\begin{equation*}
    \sup_{|t|\le T}|\Psi_j(t)-t|\longrightarrow0,
    \qquad
    \sup_{|t|\le T}\big|\langle\omega^j_{\Psi_j(t)},g\rangle-\langle\omega_t,g\rangle\big|\le\varepsilon_j.
\end{equation*}
Passing to left limits, we obtain, for $-T<t<T$,
\begin{equation*}
    \Big|\big(\langle\omega^j_{\Psi_j(t)-},g\rangle-\langle\omega^j_{\Psi_j(t)},g\rangle\big)-\big(\langle\omega_{t-},g\rangle-\langle\omega_t,g\rangle\big)\Big|\le2\varepsilon_j.
\end{equation*}
Thus a downward jump of size strictly greater than $3/2$ of $\omega$ at a time $t\in(-1,1)$ persists for $\omega^j$ at the time $\Psi_j(t)\in(-1,1)$, for all sufficiently large $j$. This proves that the first event is open.

To see that the second event is closed, suppose that every $\omega^j$ has a downward jump of size at least two at some time $t_j\in[-1/2,1/2]$, and put $s_j:=\Psi_j^{-1}(t_j)$. By the last display, $\omega$ has a downward jump of size at least $2-2\varepsilon_j>1$ at $s_j$, for all sufficiently large $j$. A real-valued c\`adl\`ag function has only finitely many jumps of size greater than one in a compact interval, since otherwise a sequence of such jump times would converge monotonically to some time, and the corresponding one-sided limit would force the values on both sides of these jumps to converge to the same limit. Consequently, $s_j=s$ is constant along a subsequence. Since $|s_j-t_j|\le\sup_{|t|\le T}|\Psi_j(t)-t|\to0$, we have $s\in[-1/2,1/2]$, and the jump of $\omega$ at $s$ has size at least $2$. Since $E$ is metrisable, the second event is closed. The assertions for $E_{\rm s}$ follow, since it carries the subspace topology.

\emph{Infinite rate.} Under $Q_\xi$, the birth and death times in $\Delta_n$ are pairwise distinct, almost surely, since the births form a Poisson point process and the lifetimes are independent and continuously distributed. For $n\ge10$, no translate of $\operatorname{supp}g\subset[-3,3]^d$ meets two spatial copies of the same particle. Since $0\le g\le1$, every jump of $t\mapsto\langle(\theta_{(x,s)}\omega^{(n)})_t,g\rangle$ that does not occur at a temporal reset therefore has size at most one, and only the resets, which are $n$ units of time apart, can produce the open jump event. For every spatial shift, the temporal shifts in $[0,n)$ that place a reset in $(-1,1)$ form a set of Lebesgue measure two, so that the $R_n$-probability of the open jump event is at most $2/n$, $Q_\xi$-almost surely.

Now let $P$ be such that the open jump event has $P$-probability $p>0$. By the Portmanteau theorem, the laws that assign probability larger than $p/2$ to this open event form an open neighbourhood $G$ of $P$ in the weak topology, and hence in each of the spaces in the statement. By the preceding paragraph, $Q_\xi(R_n\in G)=0$ for all $n\ge10$ with $n>4/p$. If the lower bound holds with the rate function $I$, then
\begin{equation*}
    -\infty=\liminf_{n\to\infty}\frac1{n|\Lambda_n|}\log Q_\xi(R_n\in G)\ge-\inf_{P'\in G}I(P'),
\end{equation*}
so that $\inf_GI=\infty$, and in particular $I(P)=\infty$.
\end{proof}

We now construct closed sets of finite exponential cost all of whose elements charge the open jump event. The first example uses paths with pairs of particles at distance of order $1/n$, whose weak limit has double points and therefore lies outside $\PP(E_{\rm s})$.

\begin{example}[No full large deviation principle on simple configurations]\label{example:no-full-ldp-simple}
For $\xi=\emptyset$, the empirical processes $(R_n)_{n\in\N}$ under $Q_\emptyset$ satisfy no full large deviation principle at speed $n|\Lambda_n|$ on $\PP(E_{\rm s})$, equipped with the weak topology, with any rate function.
\end{example}

\begin{proof}
Suppose that a full large deviation principle holds with a rate function $I$.

\emph{Step 1: The sets $K_n$.}
For $n\ge10$, enumerate the centres of the unit cubes $-\overline{n/2}+k+[0,1)^d$, $k\in\{0,\ldots,n-1\}^d$, which partition $\Lambda_n$, as $c_1,\ldots,c_{|\Lambda_n|}$. For $1\le j\le|\Lambda_n|$ and $1\le i\le2|\Lambda_n|$, set
\begin{equation*}
\begin{split}
    B_{n,2j-1}&:=B_{1/(100n)}\Big(c_j-\Big(\tfrac{1}{10n},0,\ldots,0\Big)\Big),\, 
    B_{n,2j}:=B_{1/(100n)}\Big(c_j+\Big(\tfrac{1}{10n},0,\ldots,0\Big)\Big),\\
    I_{n,i}&:=\Big[\tfrac{4i-3}{8|\Lambda_n|},\tfrac{4i-2}{8|\Lambda_n|}\Big].
\end{split}
\end{equation*}
The closed balls are pairwise disjoint and lie inside their unit cubes, and the time intervals lie in $(0,1)$ and are separated by gaps of positive length. Let $K_n$ be the set of all empirical processes $R_{n,\omega}$ of paths of the form
\begin{equation*}
    \omega_t=\sum_{i=1}^{2|\Lambda_n|}\1_{[t_i,\infty)}(t)\,\delta_{x_i},\quad t\ge0,\qquad\text{with } x_i\in B_{n,i},\ t_i\in I_{n,i},\ 1\le i\le2|\Lambda_n|.
\end{equation*}
These paths start with the empty configuration, and their only events are the prescribed births: two particles appear close to the centre of every unit cube during the time interval $(0,1)$ and stay alive until the reset at time $n$, see \zcref{fig:counterexample}. Their periodisations belong to $E_{\rm s}$, so that $K_n\subset\PP(E_{\rm s})$.

\emph{Step 2: Compactness of $K_n$.}
The parameter space $\prod_{i=1}^{2|\Lambda_n|}(B_{n,i}\times I_{n,i})$ is compact. For convergent parameters, consider the increasing piecewise linear time change of $[0,n]$ that fixes $0$ and $n$ and maps each limiting birth time to the corresponding approximating birth time, and extend it to $\R$ by translating its graph by $(hn,hn)$, $h\in\Z$. These time changes converge uniformly to the identity, and after this alignment the configurations converge vaguely, uniformly in time, since only finitely many particle positions contribute in every bounded spatial region. Since the limiting positions again lie in the disjoint balls $B_{n,i}$, the limiting path belongs to $E_{\rm s}$, and the periodised paths converge in $E_{\rm s}$. Every fixed space-time translation preserves this convergence, so that dominated convergence, applied to bounded continuous test functions in the definition~\eqref{eq:empirical-process} of $R_{n,\omega}$, proves that the map from the parameters to $R_{n,\omega}$ is continuous. Thus $K_n$ is compact.

\emph{Step 3: Closedness of $K=\bigcup_{n\ge10}K_n$.}
Let $X$ be uniformly distributed on $[0,1)^d$, and let $\gamma$ be the random constant path with $\gamma_t=2\sum_{z\in\Z^d}\delta_{z+X}$ for all $t\in\R$, which has a double point at every site of a randomly shifted lattice; its law is space-time stationary and assigns probability zero to $E_{\rm s}$. We claim that
\begin{equation}\label{eq:counterexample-uniform-convergence}
    \lim_{n\to\infty}\sup_{P\in K_n}\sup_{f\in\mathrm{BL}_1(E)}\big|P[f]-\mathbb E[f(\gamma)]\big|=0.
\end{equation}
Let $P=R_{n,\omega}\in K_n$, let $T>0$ and $J\in\N$, and let $n>2T+1$. The temporal shifts $s\in[0,n)$ for which $[s-T,s+T]$ meets the birth phase $[0,1]+n\Z$ form a set of measure $2T+1$, and all resets are contained in this exceptional set. For every other shift $s$ and every $x\in\Lambda_n$, the path $\theta_{(x,s)}\omega^{(n)}$ is constant on $[-T,T]$, and the two particles associated with each translated cube centre lie within distance $11/(100n)$ of this centre. We compare this path with the constant path that has two particles exactly at every translated cube centre. Since every $g_\ell$ is $1$-Lipschitz, and since the support of each of $g_1,\ldots,g_J$ is within distance one of a number of lattice points that is bounded independently of $n$, the definition of $d_\Omega$ gives a bound $C_J/n+2^{-J}$ on the distance of the two configurations, with a constant $C_J$ that depends only on $J$. By~\eqref{eq:lb-local-metric}, the $d_E$-distance of the two paths is at most $\e^{-T}+C_J/n+2^{-J}$. If $x$ is uniformly distributed on $\Lambda_n$, the translated cube centres form a uniformly shifted copy of the lattice $\Z^d$, so that the comparison path has the law of $\gamma$. Averaging over all shifts, we obtain, for every $f\in\mathrm{BL}_1(E)$,
\begin{equation*}
    \big|P[f]-\mathbb E[f(\gamma)]\big|\le\frac{2(2T+1)}{n}+\e^{-T}+\frac{C_J}{n}+2^{-J}.
\end{equation*}
Letting first $n\to\infty$ and then $T,J\to\infty$ proves~\eqref{eq:counterexample-uniform-convergence}.

Now let $P_j\in K_{n_j}$ converge weakly to some $P$ in $\PP(E_{\rm s})$. If $(n_j)_{j\in\N}$ were unbounded, then, along a subsequence, $P_j$ would converge weakly in $\PP(E)$ to the law of $\gamma$ by~\eqref{eq:counterexample-uniform-convergence}, whereas the continuity of the embedding of $\PP(E_{\rm s})$ into $\PP(E)$ forces $P_j\to P$ in $\PP(E)$. This is impossible, since $P(E_{\rm s})=1$, whereas the law of $\gamma$ assigns probability zero to $E_{\rm s}$. Hence $(n_j)_{j\in\N}$ is bounded, a subsequence of $(P_j)_{j\in\N}$ lies in a single set $K_n$, and the compactness of $K_n$ shows that $P\in K_n\subset K$. Since the weak topology on $\PP(E_{\rm s})$ is metrisable, $K$ is closed.

\emph{Step 4: Infinite rate on $K$.}
Let $P=R_{n,\omega}\in K_n$. For every spatial shift, the lattice of translated cube centres has a point in $[-1/2,1/2]^d$, whose two associated particles lie in $[-2,2]^d$, where $g=1$. At a temporal reset, all particles disappear, since these paths start with the empty configuration. Hence the scalar path $t\mapsto\langle\omega^{(n)}_t,g\rangle$, seen from any spatial shift, has a downward jump of size at least two at every reset. The temporal shifts placing a reset in $(-1,1)$ form a fraction $2/n$ of all shifts, so that $P$ assigns probability $2/n>0$ to the open jump event. By \zcref{lem:periodisation-jumps}, $I(P)=\infty$ for every $P\in K$.

\emph{Step 5: Probability estimate and contradiction.}
Under $Q_{\emptyset}$, consider the event that there is exactly one birth in each of the sets $B_{n,i}\times I_{n,i}$, no other birth in $\Delta_n$, and that every one of these births has a lifetime of at least $n$. By the independence properties of the marked Poisson point process of births, this event has probability
\begin{equation*}
    \e^{-3n|\Lambda_n|}\prod_{i=1}^{2|\Lambda_n|}|B_{n,i}|\,|I_{n,i}|
    =\e^{-3n|\Lambda_n|} \left(\frac{|B_{1/(100n)}(0)|}{8|\Lambda_n|}\right)^{2|\Lambda_n|}.
\end{equation*}
On this event, the path agrees on $\Lambda_n\times[0,n)$ with a path in the definition of $K_n$, so that $R_n\in K_n$. Since $|B_{1/(100n)}(0)|$ is a constant multiple of $n^{-d}$ and $|\Lambda_n|=n^d$, the logarithm of this probability equals $-3n|\Lambda_n|-O(|\Lambda_n|\log n)$. Therefore
\begin{equation*}
    \liminf_{n\to\infty}\frac1{n|\Lambda_n|}\log Q_{\emptyset}(R_n\in K)\ge-3,
\end{equation*}
whereas the upper bound for the closed set $K$, together with Step~4, gives
\begin{equation*}
    \limsup_{n\to\infty}\frac1{n|\Lambda_n|}\log Q_{\emptyset}(R_n\in K)=-\infty.
\end{equation*}
This is a contradiction.
\end{proof}

\begin{figure}
    \centering
    % Schematic for Example (no full LDP on simple configurations); time horizontal, space vertical,
% as in the authors' figures. Not to scale: the birth phase is magnified and the time axis is broken.
\begin{tikzpicture}[x=1cm,y=1cm,line cap=round]
    % birth phase (0,1) and its periodic copy (n,n+1)
    \fill[black!12] (0,0) rectangle (2,6);
    \fill[black!12] (8,0) rectangle (10,6);
    % unit cells
    \foreach \y in {0,1.5,3,4.5,6} {\draw[black!35] (0,\y) -- (10.6,\y);}
    \draw[thick] (0,0) -- (0,6);
    % temporal reset at time n
    \draw[thick,dashed] (8,-0.15) -- (8,6.15);
    \node[above] at (8,6.15) {\small reset};
    % lifelines: particle i is born at time t_i in (0,1), drawn at x=2t_i,
    % the two particles of cell j sit slightly below and above its centre
    \foreach \i/\x in {1/0.094,2/0.344,3/0.594,4/0.844,5/1.094,6/1.344,7/1.594,8/1.844} {
        \pgfmathtruncatemacro{\j}{(\i+1)/2}
        \pgfmathsetmacro{\yc}{1.5*\j-0.75}
        \pgfmathsetmacro{\s}{ifthenelse(mod(\i,2)==1,-1,1)}
        \pgfmathsetmacro{\y}{\yc+0.13*\s}
        \draw[line width=0.9pt] (\x,\y) -- (8,\y);
        \fill (\x,\y) circle (1.6pt);
        \pgfmathsetmacro{\xx}{8+\x}
        \draw[line width=0.9pt,black!55] (\xx,\y) -- (10.6,\y);
        \fill[black!55] (\xx,\y) circle (1.6pt);
    }
    % break of the time axis (not to scale)
    \fill[white] (4.75,-0.05) rectangle (5.05,6.05);
    % space axis
    \draw[<->,>=latex] (-0.35,0) -- (-0.35,6);
    \node[left] at (-0.4,3) {$\Lambda_n$};
    % time axis
    \draw (-0.1,-0.75) -- (4.75,-0.75);
    \draw[->,>=latex] (5.05,-0.75) -- (10.8,-0.75);
    \draw (4.65,-0.87) -- (4.85,-0.63);
    \draw (4.95,-0.87) -- (5.15,-0.63);
    \node[below] at (5.3,-0.8) {time};
    \foreach \x/\l in {0/0,2/1,8/n,10/n+1} {
        \draw[black!60] (\x,-0.1) -- (\x,-0.25);
        \node[below] at (\x,-0.25) {$\l$};
    }
\end{tikzpicture}
    \caption{Schematic picture of a path in the proof of \zcref{example:no-full-ldp-simple} for $d=1$, with space drawn vertically and the time axis broken and not to scale. During the time interval $(0,1)$, two particles are born at distance of order $1/n$ from the centre of every unit cell, at distinct times, and they stay alive until the temporal reset at time $n$, which removes all of them at once; the periodised path then repeats the same births. As $n\to\infty$, the pairs merge into double points.}
    \label{fig:counterexample}
\end{figure}
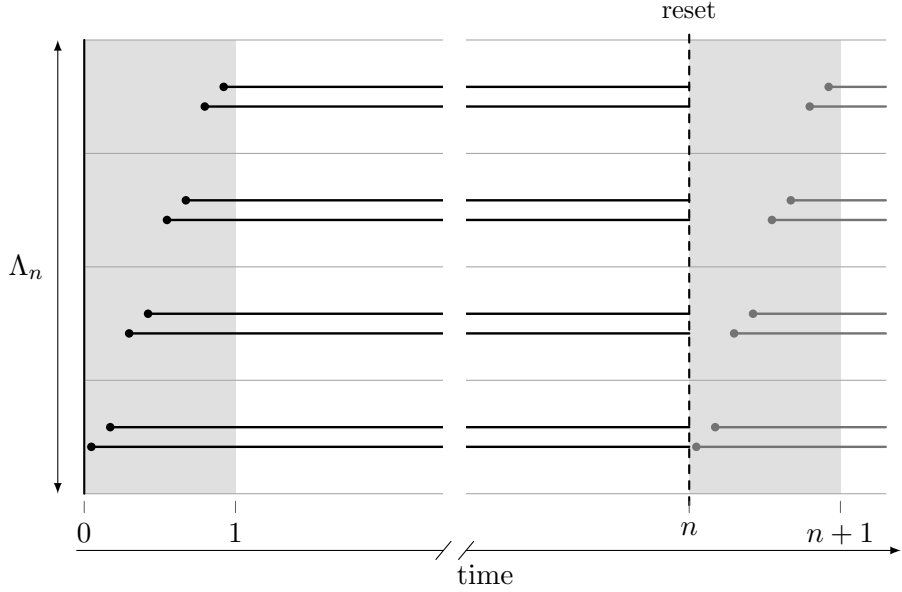

In the second example, the particles are concentrated in a small region, so that their weak limit, the law of the empty path, is excluded by a constraint on the intensity.

\begin{example}[No full large deviation principle with intensities]\label{example:no-full-ldp-occupation}
For every $\xi\in\Omst$, the empirical processes $(R_n)_{n\in\N}$ under $Q_\xi$ satisfy no full large deviation principle at speed $n|\Lambda_n|$ on $(\PPP(E),\mathcal T_{\rm occ})$ with any rate function. If $\xi\in\Omsst$, the same holds on the subset of laws in $\PPP(E)$ supported by $E_{\rm s}$, equipped with the topology induced by $\mathcal T_{\rm occ}$.
\end{example}

\begin{proof}
Suppose that a full large deviation principle holds with a rate function $I$. With $g$ as in \zcref{lem:periodisation-jumps}, the function $f:=a_g\wedge1$ is bounded and continuous on $E$ by \zcref{lem:occupation-functional}(i). For $P\in\PPP(E)$, \zcref{lem:occupation-functional}(ii) gives $P[a_g]=\rho(P)\int g(x)\,\d x$, so that $P[f]=0$ implies $\rho(P)=0$. For $n\in\N$, let $F_n$ be the set of all $P\in\PPP(E)$ such that
\begin{equation*}
    \rho(P)\ge1,\qquad P[f]\le\frac3n+\frac{7^d}{|\Lambda_n|},
\end{equation*}
and such that $P$ assigns probability at least $1/(n|\Lambda_n|)$ to the closed jump event of \zcref{lem:periodisation-jumps}. Each $F_n$ is $\mathcal T_{\rm occ}$-closed, by the continuity of $\rho$ and of $P\mapsto P[f]$, and by the Portmanteau theorem for the closed jump event. The union $F=\bigcup_{n\in\N}F_n$ is closed as well. Indeed, let $P\notin F$. If $P[f]>0$, the neighbourhood $\{P'\colon P'[f]>P[f]/2\}$ of $P$ is disjoint from $F_n$ for all sufficiently large $n$, so that it meets only finitely many of the closed sets $F_n$; since none of them contains $P$, some neighbourhood of $P$ is disjoint from $F$. If $P[f]=0$, then $\rho(P)=0$, and the neighbourhood $\{P'\colon\rho(P')<1\}$ of $P$ is disjoint from $F$. Since the closed jump event is contained in the open one, \zcref{lem:periodisation-jumps} gives $I(P)=\infty$ for every $P\in F$.

Under $Q_\xi$, consider the event that all initial particles in $\Lambda_n$ die before time $1$, that there are exactly $2|\Lambda_n|$ births in $[-1/8,1/8]^d\times[1,2]$, each with a lifetime of at least $n$, and that there is no other birth in $\Delta_n$. This event has probability
\begin{equation*}
    (1-\e^{-1})^{\xi(\Lambda_n)}\e^{-n|\Lambda_n|}\frac{(4^{-d}\e^{-n})^{2|\Lambda_n|}}{(2|\Lambda_n|)!},
\end{equation*}
whose logarithm, divided by $n|\Lambda_n|$, tends to $-3$, since $\xi\in\Omst$ and $|\Lambda_n|\log n=o(n|\Lambda_n|)$. We show that $R_n\in F_n$ on this event, for all sufficiently large $n$. First, by~\eqref{eq:rho}, $\rho(R_n)\ge2(1-2/n)\ge1$ for $n\ge4$. Second, the initial particles contribute to $f(\theta_{(x,s)}\omega^{(n)})$ only for a fraction at most $3/n$ of the temporal shifts, and every other non-zero value requires a spatial shift in a set of volume at most $7^d$, so that $R_n[f]\le3/n+7^d/|\Lambda_n|$. Third, for spatial shifts in $[-1/2,1/2]^d$, all the newly born particles have weight one just before a temporal reset, whereas the initial particles that reappear at the reset contribute at most $\xi([-4,4]^d)$. Thus the downward jump at the reset has size at least two for sufficiently large $n$. The temporal shifts placing the reset in $[-1/2,1/2]$ form a fraction $1/n$ of all shifts, so that $R_n$ assigns probability at least $1/(n|\Lambda_n|)$ to the closed jump event. Hence $R_n\in F_n\subset F$ for all sufficiently large $n$, and
\begin{equation*}
    \liminf_{n\to\infty}\frac1{n|\Lambda_n|}\log Q_\xi(R_n\in F)\ge-3,
\end{equation*}
whereas the upper bound for the closed set $F$ gives $-\infty$, a contradiction. For simple $\xi$, the newly born particles occupy distinct locations, almost surely, and the same argument applies to the subset of laws supported by $E_{\rm s}$.
\end{proof}

We conclude with a comment on exponential tightness. The compact sets $K_\alpha$ from \zcref{prop:exponential-tightness} may contain laws of infinite intensity, but a single additional estimate removes them. By \zcref{lem:occupation-functional}, the sets $\{P\colon\rho(P)\le c\}$ are weakly closed, and $\rho(R_n)=A_n/(n|\Lambda_n|)$. The bound~\eqref{eq:mgf} with $\lambda=1/2$ and $\psi = 0$ and the exponential Chebyshev inequality give
\begin{equation*}
    \limsup_{n\to\infty}\frac1{n|\Lambda_n|}\log Q_{\xi}(\rho(R_n)>\sigma)\le1-\frac\sigma2,\qquad\sigma>0.
\end{equation*}
Hence $K_\alpha\cap\{P\colon\rho(P)\le2(\alpha+1)\}$ is a weakly compact subset of $\PPP(E)$ that satisfies~\eqref{eq:exponential-tightness}, so that the empirical processes are exponentially tight in $\PPP(E)$ with the topology induced by the weak topology. They are, however, not exponentially tight in $(\PPP(E),\mathcal T_{\rm occ})$: otherwise, by~\cite[Lemma~1.2.18]{LDPbook}, the weak large deviation principle of \zcref{theorem:weak-ldps}(\ref{item:weak-ldp-Tocc}) would be a full one, contradicting \zcref{example:no-full-ldp-occupation}.

\subsection*{Disclosure of AI use.} 
This project was initiated in February 2026 as a collaboration between the authors. By the end of August 2026, we had developed, without the use of AI, the main proof strategies and key ideas, including the ideas of the block-path construction in the lower bound and of the transfer from the free evolution to the interacting process, and had worked out some of the details.
The details of the proofs were subsequently developed through a collaboration between the authors and AI, primarily using ChatGPT 6 Astra. The counterexamples in \zcref{subsection:counterexamples} were constructed entirely by AI and then checked and rewritten by the authors. 
We take full responsibility for the mathematical content and correctness of the paper.

\subsection*{Acknowledgements.} 

The authors would like to thank David Dereudre for the idea of introducing regularising corridors in the observation window, which led us to the construction of block paths. While working on this project JK was also employed at TU Braunschweig.  

\bibliographystyle{alpha}
\bibliography{references}

\appendix
\section{The topology of the path space}\label{section:appendix-topology}
In this appendix we collect the metrics on $\Om$ and on $E$ that we use for quantitative estimates, together with some elementary facts about bounded Lipschitz test functions and a local comparison estimate for paths.

For $\eta\in\Om$ and a bounded measurable function $g\colon\R^d\to\R$ with bounded support, we write $\langle\eta,g\rangle:=\int\eta(\d x)\,g(x)$. We fix a countable family $(g_j)_{j\ge1}\subset C_c(\R^d)$ of Lipschitz functions that determines vague convergence, normalised so that
\begin{equation*}
    \|g_j\|_\infty\le1, \qquad |g_j(x)-g_j(y)|\le\|x-y\|_\infty, \qquad x,y\in\R^d.
\end{equation*}
The metric
\begin{equation*}
    d_\Omega(\eta,\eta') := \sum_{j\ge1}2^{-j}\bigl(1\wedge|\langle\eta,g_j\rangle-\langle\eta',g_j\rangle|\bigr)
\end{equation*}
induces the vague topology on $\Om$.
For an open set $B \subset \mathbb{R}^d$, we equip $\Omega_B$ with the vague
topology on $B$, that is, the topology generated by the maps
$\eta \mapsto \langle \eta, g \rangle$ with $g \in C_c(B)$, and not with the 
topology induced by~$\Omega$. With this topology, the restriction map $\Omega \to \Omega_B$, $\eta \mapsto \eta_B$, is
continuous, since every $g \in C_c(B)$ extends by zero to a function in
$C_c(\mathbb{R}^d)$. 
Accordingly, a metric $d_{\Omega_B}$ is defined as $d_\Omega$, with a countable
family in $C_c(B)$ that determines vague convergence on $B$. The path spaces
$D([u,v];\Omega_B)$ carry the Skorokhod $J_1$ topology built from
$d_{\Omega_B}$.
For $u>0$, let $d_{J_1,u}$ be the Skorokhod $J_1$ metric on $D([-u,u];\Om)$ built from $d_\Omega$, that is,
\begin{equation*}
    d_{J_1,u}(\gamma,\gamma') := \inf_{\Psi}\max\Big\{\sup_{|t|\le u}|\Psi(t)-t|,\ \sup_{|t|\le u}d_\Omega\bigl(\gamma_t,\gamma'_{\Psi(t)}\bigr)\Big\},
\end{equation*}
where the infimum runs over all increasing homeomorphisms $\Psi$ of $[-u,u]$. On $E$ we use the bounded metric
\begin{equation}\label{eqn:definition-j1-metric}
    d_E(\omega,\omega') := \int_0^\infty\d u\, \e^{-u} \left( 1\wedge d_{J_1,u} \bigl(\omega|_{[-u,u]},\omega'|_{[-u,u]}\bigr) \right),
\end{equation}
which induces the local Skorokhod $J_1$ topology. On $\Oms$ and $E_{\rm s}$ we use the restrictions of $d_\Omega$ and $d_E$, respectively.

For a metric space $(S,d)$, we write $\mathrm{BL}_1(S)$ for the set of all functions $f\colon S\to\R$ with $\|f\|_\infty\le1$ and $|f(x)-f(y)|\le d(x,y)$ for all $x,y\in S$; for $S=E$, this is the set~\eqref{eq:BL1}. We use these functions for $S\in\{E,E_{\rm s},\Om,\Oms\}$, equipped with the metrics above. The following lemma shows that they suffice to describe the topology of weak convergence, and that they can be extended from a subspace.

\begin{lemma}[Bounded Lipschitz neighbourhoods]\label{lemma:open-neighbourhoods}
    Let $(S,d)$ be a metric space.
    \begin{enumerate}[(i)]
        \item Let $P$ be a Borel probability measure on $S$, and let $G$ be an open neighbourhood of $P$ in the space of Borel probability measures on $S$, equipped with the topology of weak convergence, or in a subset of this space, equipped with the induced topology. Then there are $p\in\N$, $f_1,\dots,f_p\in\mathrm{BL}_1(S)$ and $\delta>0$ such that every $\widetilde P$ in this space with $|\widetilde P[f_i]-P[f_i]|<\delta$ for $1\le i\le p$ belongs to $G$.
        \item If $S'\subset S$ is equipped with the restriction of $d$, then every $f\in\mathrm{BL}_1(S')$ has an extension $\widetilde f\in\mathrm{BL}_1(S)$, that is, $\widetilde f=f$ on $S'$.
    \end{enumerate}
\end{lemma}
In particular, every open neighbourhood $G$ of $P$ in $\PP(E)$ contains a set $U(P;f_1,\dots,f_p;\delta)$ of the form~\eqref{eqn:open-u-lipschitz-test}, and the same holds with $E_{\rm s}$, $\Om$ or $\Oms$ in place of $E$.

\begin{proof}
(i) The topology of weak convergence is generated by the maps $\widetilde P\mapsto\widetilde P[f]$, $f\in C_b(S)$, so that $G$ contains a finite intersection of sets of the form $\{\widetilde P\colon|\widetilde P[f]-P[f]|<\varepsilon\}$ with $f\in C_b(S)$ and $\varepsilon>0$. It therefore suffices to show that each such set contains a set of the form described in~(i). Replacing $f$ by an affine transformation of $f$, we may assume that $0\le f\le1$. For $k\in\N$, consider the Lipschitz approximations
\begin{equation*}
    f_k(x):=\inf_{y\in S}\bigl(f(y)+k\,d(x,y)\bigr),\qquad h_k(x):=\sup_{y\in S}\bigl(f(y)-k\,d(x,y)\bigr),\qquad x\in S.
\end{equation*}
They are $k$-Lipschitz, satisfy $0\le f_k\le f\le h_k\le1$, and, since $f$ is continuous, $f_k\uparrow f$ and $h_k\downarrow f$ pointwise as $k\to\infty$. By monotone convergence, there is $k$ with $P[h_k]-P[f_k]<\varepsilon/2$. If $|\widetilde P[f_k]-P[f_k]|<\varepsilon/4$ and $|\widetilde P[h_k]-P[h_k]|<\varepsilon/4$, then
\begin{equation*}
    \widetilde P[f]\le\widetilde P[h_k]<P[h_k]+\frac\varepsilon4<P[f_k]+\frac{3\varepsilon}4\le P[f]+\frac{3\varepsilon}4,
\end{equation*}
and, in the same way, $\widetilde P[f]>P[f]-3\varepsilon/4$. Since $f_k/k$ and $h_k/k$ belong to $\mathrm{BL}_1(S)$, this proves~(i) with $\delta=\varepsilon/(4k)$.

(ii) The function
\begin{equation*}
    \widetilde f(x)=\min\Big\{1,\,\inf_{y\in S'}\bigl(f(y)+d(x,y)\bigr)\Big\},\qquad x\in S,
\end{equation*}
is $1$-Lipschitz as the minimum of $1$ and an infimum of $1$-Lipschitz functions. The infimum is bounded from below by $-1$, and it coincides with $f$ on $S'$, since $f(y)+d(x,y)\ge f(x)$ for $x,y\in S'$. Since $f\le1$, taking the minimum with $1$ does not change the values on $S'$. Hence $\widetilde f\in\mathrm{BL}_1(S)$ and $\widetilde f=f$ on $S'$.
\end{proof}

The following estimates show that agreement of two paths on a sufficiently large bounded space-time box makes their $d_E$-distance small, without any control of the particles outside this box. Here $B_r=(-r,r)^d$ denotes the open box of side length $2r$.

\begin{lemma}[Local comparison]\label{lem:lb-local-metric}
For every $R>0$ and all paths $\omega,\omega'\in E$,
\begin{equation}\label{eq:lb-local-metric}
    d_E(\omega,\omega') \leq \e^{-R} +\sup_{|t|\leq R}d_{\Omega}(\omega_t,\omega'_t).
\end{equation}
Moreover, fix $R>0$ and $J\in\N$, and choose $r>1$ such that the supports of $g_1,\ldots,g_J$ lie in $(-r+1,r-1)^d$.
\begin{enumerate}[(i)]
    \item If $\omega,\omega'\in E$ coincide on $B_r$ at all times $|t|\le R$, then
    \begin{equation}\label{eq:lb-cylinder-agreement}
        d_E(\omega,\omega')\leq \e^{-R}+2^{-J}.
    \end{equation}
    \item Let $\omega,\omega'$ be legal birth-and-death paths. Suppose that, at each time $t\in[-R,R]$, there is a one-to-one matching between atoms of $\omega_t$ and atoms of $\omega'_t$, counted with multiplicity, which covers all atoms of either configuration in the supports of $g_1,\ldots,g_J$, and that matched positions differ by at most $\epsilon\le1$ in the sup-norm. Then
    \begin{equation}\label{eq:lb-lifeline-metric}
        \sup_{|t|\leq R}d_{\Omega}(\omega_t,\omega'_t) \leq 2^{-J}+\epsilon L,
    \end{equation}
    where $L$ is the number of individual particles of $\omega$ present in $B_r$ at some time during $[-R,R]$.
\end{enumerate}
\end{lemma}

\begin{proof}
For $u\le R$, the identity time change shows that
\begin{equation*}
    d_{J_1,u}(\omega|_{[-u,u]},\omega'|_{[-u,u]})\le\sup_{|t|\le R}d_\Omega(\omega_t,\omega'_t),
\end{equation*}
while the integrand in~\eqref{eqn:definition-j1-metric} is at most $\e^{-u}$ for $u>R$. This gives~\eqref{eq:lb-local-metric}.

(i) If $\omega$ and $\omega'$ coincide on $B_r$ at time $t$, then $\langle\omega_t,g_j\rangle=\langle\omega'_t,g_j\rangle$ for $j\le J$, and the remaining terms in the definition of $d_\Omega$ contribute at most $\sum_{j>J}2^{-j}=2^{-J}$. Together with~\eqref{eq:lb-local-metric}, this gives~\eqref{eq:lb-cylinder-agreement}.

(ii) Fix $t\in[-R,R]$ and $j\le J$. Every matched pair that contributes to $\langle\omega_t,g_j\rangle-\langle\omega'_t,g_j\rangle$ has one of its two positions in the support of $g_j$, so that the position of the atom of $\omega_t$ lies in $B_r$, since matched positions differ by at most $\epsilon\le1$. As $g_j$ is $1$-Lipschitz with respect to the sup-norm, each such pair contributes at most $\epsilon$, and since every atom of $\omega_t$ in $B_r$ is an individual particle of $\omega$ present in $B_r$ at time $t$, there are at most $L$ such pairs. Hence $|\langle\omega_t,g_j\rangle-\langle\omega'_t,g_j\rangle|\le\epsilon L$ for $j\le J$, and summing over $j$ gives~\eqref{eq:lb-lifeline-metric}.
\end{proof}

\section{The graphical construction}\label{section:appendix-graphical}
In this appendix we recall the pathwise construction of the birth-and-death dynamics as a thinning of a space-time Poisson random measure, following~\cite[Theorem~2.13]{garcia2006spatial}. It defines the laws $Q^\beta_\xi$, and it shows that, almost surely, all births occur at distinct locations, which are distinct from the points of the initial configuration.

Let $\Ncal$ be a Poisson random measure on $\R^d\times[0,\infty)^3$ with intensity measure $\d x\otimes \d u \otimes\e^{-r}\d r\otimes\d s$. An atom $(x,u,r,s)$ of $\Ncal$ proposes the birth of a particle at the location $x$ at time $s$, which is accepted if $u\le b(x,X_{s-})$, where $X_{s-}$ is the current configuration; in this case, the new particle has lifetime $r$, which is unit-exponentially distributed. Since $b$ is bounded and has finite range, the assumptions of~\cite[Theorem~2.13]{garcia2006spatial} are satisfied. Following its proof, we first prescribe the lifetimes of the initial particles. For $\xi = \sum_{i\ge 1}\delta_{x_i} \in \Om$, let $\underline\xi = \sum_{i\ge 1}\delta_{(x_i,\tau_i)}$ be the point process on $\R^d\times[0,\infty)$ obtained by associating a lifetime $\tau_i>0$ with each point $x_i$ of $\xi$.

\begin{proposition}[Graphical representation]\label{prop:graphical-representation}
    Let $\Fcal = (\Fcal_t)_{t\ge 0}$ be a filtration such that $\Ncal$ is $\Fcal$-compatible, that is, for every bounded Borel set $A\subset\R^d\times [0,\infty)^2$ and all $s,t\geq 0$, the random variable $\Ncal(A\times[0,t])$ is $\Fcal_t$-measurable, and $\Ncal(A\times(t,t+s])$ is independent of $\Fcal_t$. Then there is a unique $\Fcal$-adapted process $X\Ssup{\underline\xi}=(X\Ssup{\underline\xi}_t)_{t\ge 0}$ with values in $\Om$ such that, for every Borel set $B \subset \R^d$ and every $t\ge0$,
    \begin{equation}\label{eq:sdeBD}
    \begin{split}
        X\Ssup{\underline\xi}_t(B) = &\int_{B\times [0,\infty)^2 \times [0,t]}\hspace{-0.5cm}\Ncal(\d x,\d u,\d r,\d s)\, \1_{[0,b(x,X\Ssup{\underline{\xi}}_{s-})]}(u)\,\1_{(t-s,\infty)}(r) + \int_{B\times [0,\infty)}\underline\xi(\d x,\d r)\, \1_{(t,\infty)}(r).
    \end{split}
    \end{equation}
\end{proposition}

The first term in~\eqref{eq:sdeBD} counts the accepted births in $B$ during $[0,t]$ whose particles are still alive at time $t$, and the second term counts the initial particles in $B$ that are still alive at time $t$. If the initial lifetimes $\tau_i$ are chosen as independent unit-exponential random variables, independent of $\Ncal$, we obtain a Markov process $X\Ssup{\xi}$ that solves the martingale problem for the generator $\LL$ in~\eqref{eq:generator}, see~\cite[Proposition~2.9 and Corollary~2.11]{garcia2006spatial} and~\cite[Theorem~4.4.2]{EK86}. Its paths take values in $D(\R_+;\Om)$, and we write $Q^\beta_\xi$ for its law on this space; for $\beta=0$, we recover the free evolution $Q_\xi=Q^0_\xi$. Finite-volume versions of the dynamics, such as the dynamics in $\Lambda_n$ with periodic boundary conditions or the free evolution $Q^A_\xi$ in a cube $A$, are obtained in the same way, by restricting the spatial component of $\Ncal$ and of $\xi$ to the respective volume and by using the corresponding birth rate.

Only atoms $(x,u,r,s)$ of $\Ncal$ with $u\le\Vert b\Vert_\infty$ can lead to births. For every $M>0$, the locations $x$ of those atoms with $s\le M$ form a Poisson point process on $\R^d$ with intensity $M\Vert b\Vert_\infty\,\d x$, so that they are almost surely pairwise distinct and avoid the countable set of points of $\xi$. Letting $M\to\infty$, we see that, almost surely, all births occur at pairwise distinct locations that are not occupied by initial particles. In particular, if $\xi$ is simple, then almost surely the configurations $X\Ssup{\xi}_t$ and their left limits $X\Ssup{\xi}_{t-}$ are simple for all $t\ge0$.

\section{The one-cell representation}\label{section:appendix-one-cell}
In this appendix we sketch the proof of \zcref{prop:cond-entropy-representation}, with the notation of \zcref{subsection:one-cell}.

\begin{proof}[Proof sketch of \zcref{prop:cond-entropy-representation}]
Let $c_T(P)$ be the mean conditional entropy of the path in $C\times[0,T]$ given $\mathcal G^-_T$, relative to the free evolution in $C$ started from $\omega_0$, so that $c_1(P)=P[ h_{C\times[0,1]}(P_\omega\lvert  Q_{\omega_0})]$. For $I_m=\{-m,\ldots,m\}^d$, define $c_{m,T}(P)$ by replacing $\mathcal G^-_T$ with
\begin{equation*}
 \mathcal G_{m,T}
 =\bigvee_{z\in I_m}\mathcal F_{C_z\times[-mT,0]}
   \vee\bigvee_{\substack{z\in I_m\colon z\prec0}}
                       \mathcal F_{C_z\times[0,T]}.
\end{equation*}
These $\sigma$-algebras increase to $\mathcal G^-_T$ modulo null sets, and the entropy chain rule and martingale convergence give $c_{m,T}(P)\uparrow c_T(P)$. Partitioning $(0,n)^d\times[0,nT]$ into unit spatial cells and time slabs of length $T$, ordered chronologically and then lexicographically, and using that, conditional on the initial configuration, the free evolution factorises according to this order by spatial independence and the Markov property, we obtain
\begin{equation*}
    \left(1-\frac{2m}{n}\right)^d
    \left(1-\frac mn\right)c_{m,T}(P)
    \leq \frac{h_{(0,n)^d,nT}(P)}{n|\Lambda_n|}
    \leq c_T(P),\qquad n>2m.
\end{equation*}
Taking $n\to\infty$ and then $m\to\infty$ identifies the limit as $c_T(P)$. By \zcref{lem:entropy-superadditivity},
\begin{equation*}
    \frac{h_{(0,n)^d,nT}(P)}{n|\Lambda_n|T}
    \longrightarrow H(P),
\end{equation*}
so that $c_T(P)=T H(P)$. Conditioning initially on the entire temporal past increases the local conditional entropy while preserving its upper bound $n|\Lambda_n|c_1(P)$, which proves the second identity in \zcref{prop:cond-entropy-representation}.

For fixed $T$, the spatial chain rule with the entire temporal past retained gives
\begin{equation*}
    (n-2m)^dc_{m,T}(P)\leq P\left[ h_{(0,n)^d\times[0,T]}(P^p_\omega\lvert  Q_{\omega_0})\right] \leq n^dc_T(P),
\end{equation*}
and the same limits as above yield the third identity, since spatial stationarity allows us to use centred boxes. For the reconstruction of the path from the observations in the cells, one averages over spatial translates of the grid of cell boundaries: this shows that the grid carries no particles at any time in a prescribed countable dense set containing the endpoints, so that the observations in the open cells determine the c\`adl\`ag path. When $H(P)<\infty$, local absolute continuity with respect to the free evolution shows that the cell boundaries carry no particles at any time, so that open and closed observations are equivalent; when $H(P)=\infty$, the formulas for closed boxes follow by data processing.
\end{proof}

\end{document}